%% file: thesis.tex
\documentclass{config/myThesis}

\begin{document}

\include{config/macros}

  % =============================================================================
  % Front matter

	\frontmatter

\input{01-FrontMatter/titlepageIST}                  % Cover of the thesis

\input{01-FrontMatter/committeeIST}
	\input{01-FrontMatter/dedication}		              % Dedication
	\cleardoublepage

	\pagestyle{plain}                                 % Pages start (roman)

			% IST formatting
			\input{01-FrontMatter/resumo}                 % Abstract
	
	\input{01-FrontMatter/abstract}                   % Abstract
	\input{01-FrontMatter/acknowledgments}            % Acknowledgments

	\tableofcontents                                  % Table of contents
	\listoftables
	\listoffigures
	% =============================================================================

	% =============================================================================
	% Body of the text

	\mainmatter
	\pagestyle{fancy}

\input{02-MainMatter/introduction}

\input{02-MainMatter/relatedWork}

\input{02-MainMatter/globalClass}

	\input{02-MainMatter/partialClass}
	\input{02-MainMatter/conclusions}
	% =============================================================================

	% =============================================================================
	% Appendices

	\appendix
	\input{03-BackMatter/derivationRelatedWork}

\input{03-BackMatter/conjugateFunctions}

\input{03-BackMatter/derivationRelatedWorkKekatos}

	% =============================================================================

	% =============================================================================
	% Bibliography

	\backmatter
	\bibliographystyle{ieeetr}
	\addtotoc{Bibliography}
	{\singlespace \small	\bibliography{thesis}}
	% =============================================================================

\end{document}

%% file: config/macros.tex
% ==========================================================================
% Define time as Month, Year

\def\today{\ifcase\month\or
  January\or February\or March\or April\or May\or June\or
  July\or August\or September\or October\or November\or December\fi \space \number\year}
% ==========================================================================

% ==========================================================================
% Reference to Figures, Tables, etc
\newcommand{\fref}[1]{Figure~\ref{#1}}
\newcommand{\tref}[1]{Table~\ref{#1}}
\newcommand{\eref}[1]{Equation~\ref{#1}}
\newcommand{\cref}[1]{Chapter~\ref{#1}}
\newcommand{\sref}[1]{Section~\ref{#1}}
\newcommand{\ssref}[1]{Subsection~\ref{#1}}
\newcommand{\aref}[1]{Appendix~\ref{#1}}
\newcommand{\assref}[1]{Assumption~\ref{#1}}
% ==========================================================================

% ==========================================================================
% Bibliography shortcuts

% ********
% Journals

% IEEE
\newcommand{\TSP}{IEEE Trans. Signal Processing}
\newcommand{\SelectedTopicsSignalProcessing}{IEEE J. Selected Topics in Signal Processing}
\newcommand{\SigProcMagazine}{IEEE Sig. Proc. Mag.}
\newcommand{\ProcIEEE}{Proceedings of the IEEE}

\newcommand{\TAC}{IEEE Trans. Autom. Control}
\newcommand{\ControlSystemsMagazine}{IEEE Control Syst. Mag.}
\newcommand{\TControlSystemsTechnology}{IEEE Trans. Control Systems Technology}
\newcommand{\TManCybernetics}{IEEE Trans. Systems, Man, and Cybernetics}

\newcommand{\TINFO}{IEEE Trans. Info. Theory}

\newcommand{\TPowerSys}{IEEE Trans. Power Systems}

\newcommand{\SelectedAreasCommunications}{IEEE J. Selected Areas in Communications}
\newcommand{\TransWirelessCommunications}{IEEE Trans. Wireless Comm.}
\newcommand{\TransCommunications}{IEEE Trans. Communications}
\newcommand{\EuropeanTransTelecom}{European Trans. Telecommunications}

% Optimization
\newcommand{\JournalOptimTheory}{J. Optimization Theory and Appl.}
\newcommand{\NumerischeMathematik}{Numerische Mathematik}
\newcommand{\MathematicalProgramming}{Math. Program.}

% Other
\newcommand{\ComptesRendus}{Comptes Rendus de l'Academie des Sciences (Paris)}
% ********

% ***********
% Conferences

\newcommand{\ICASSP}{IEEE Intern. Conf. Acoustics, Speech, and Sig. Processing (ICASSP)}
\newcommand{\CDC}{IEEE Intern. Conf. Decision and Control (CDC)}
\newcommand{\Asilomar}{Asilomar Conf. Signals, Systems, and Computers}
\newcommand{\IPSN}{Intern. Conf. Information Proc. in Sensor Networks (IPSN)}
\newcommand{\WorkshopSignalProcAdvancesWirelessComm}{IEEE Workshop Signal Proc. Advances in Wireless Comunications}
\newcommand{\Eusipco}{European Signal Proc. Conf. (Eusipco)}
\newcommand{\AmericanControlConf}{American Control Conf.}
\newcommand{\Allerton}{Allerton Conf. Communications, Control, and Computing}
% ***********

% ***********
% Unpublished
\newcommand{\Arxiv}[1]{preprint: \url{#1}}
\newcommand{\OptimizationOnline}[1]{preprint: \url{#1}}
% ***********

% ==========================================================================

%% file: 01-FrontMatter/titlepageIST.tex
\newgeometry{top=1.0in,
bottom=1.0in,
left=0.8in,
right=0.8in}

\newcommand{\spacebetweenlinesT}{\vspace{0.1em}}     % Space between lines

\begin{titlepage}					

    \begin{center}

    \begin{pspicture}(17.5,3)
			\rput[lb](0,1.17){\includegraphics[width=3.2cm]{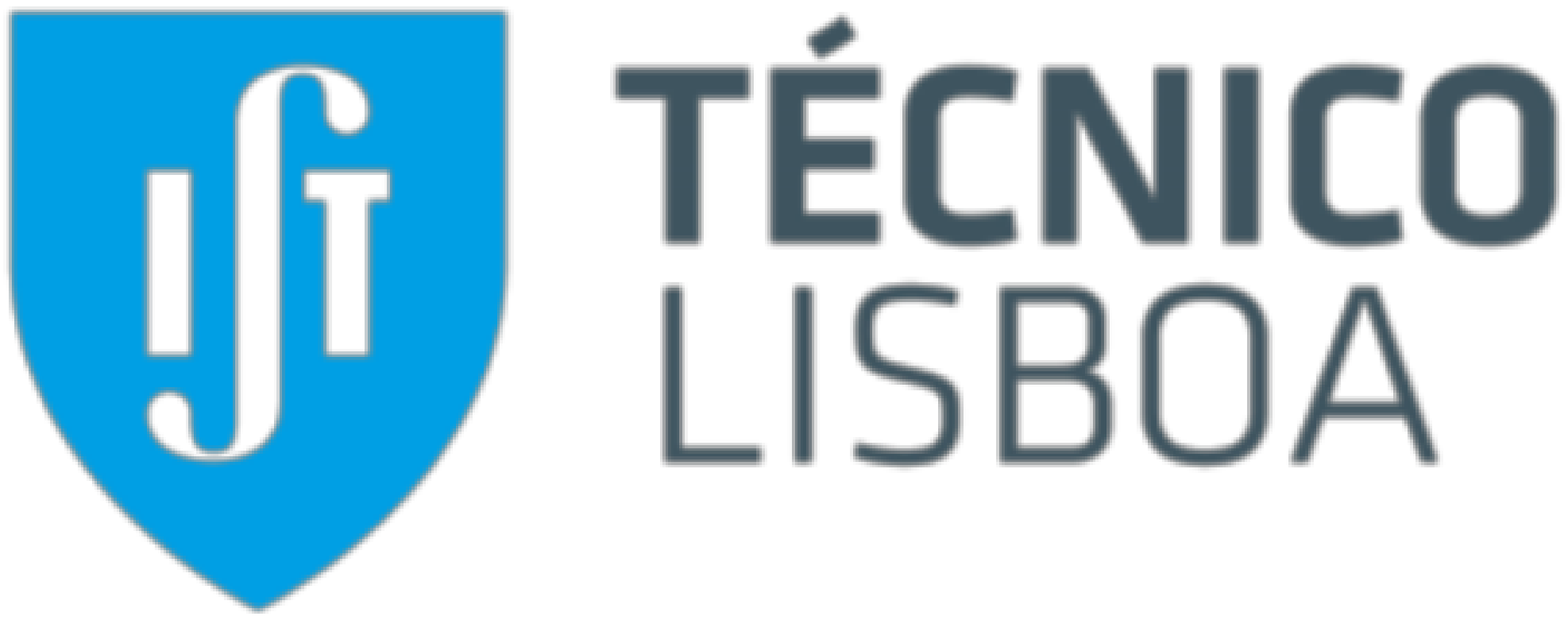}}
			\rput[lr](17.5,1.5){\includegraphics[width=3.0cm]{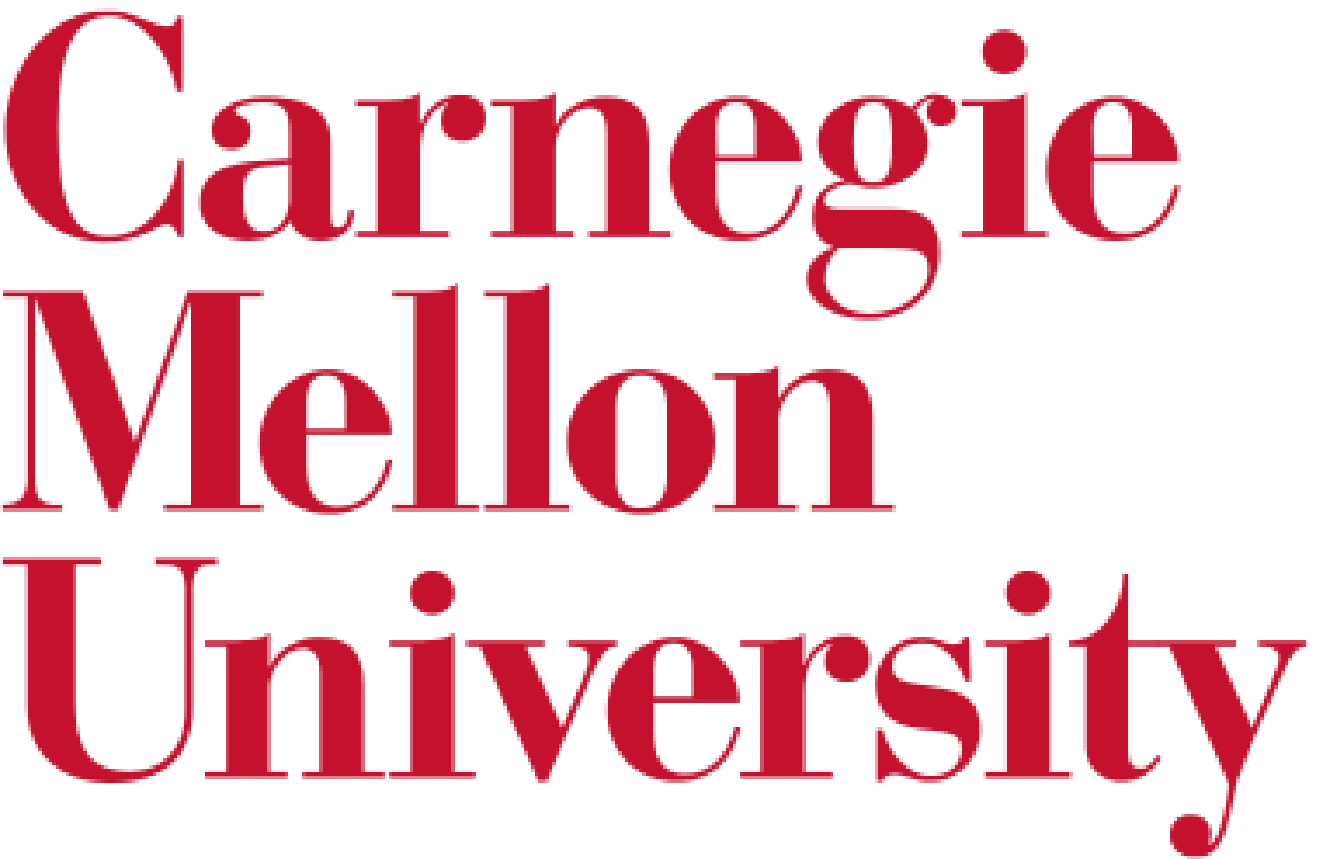}}
			\rput(8.75,2.27){\Large \bfseries \sffamily UNIVERSIDADE T\'ECNICA DE LISBOA}
			\rput(8.75,1.62){\Large \bfseries \sffamily INSTITUTO SUPERIOR T\'ECNICO}
			
			\rput(8.75,0.69){\Large \bfseries \sffamily CARNEGIE MELLON UNIVERSITY}
    	%\psgrid
    \end{pspicture}

    \hfill
    
    \psscalebox{1.05}{
		\begin{pspicture}(8,5.0)
				\def\nodesimp{
            \pscircle*[linecolor=black!50!white](0,0){0.3}
        }
        \rput(1,3.7){\rnode{N1}{\nodesimp}}
        \rput(0.3,2){\rnode{N2}{\nodesimp}}
        \rput(0.3,0.3){\rnode{N3}{\nodesimp}}
        \rput(2.5,1.7){\rnode{N4}{\nodesimp}}
        \rput(5,2){\rnode{N5}{\nodesimp}}
        \rput(3.8,3){\rnode{N6}{\nodesimp}}
        \rput(6.7,3.2){\rnode{N7}{\nodesimp}}
        \rput(3.5,0.5){\rnode{N8}{\nodesimp}}
        \rput(6.3,0.9){\rnode{N9}{\nodesimp}}
        \rput(4.1,4.1){\rnode{N10}{\nodesimp}}

        \ncline[nodesep=0.37cm]{-}{N1}{N2}
        \ncline[nodesep=0.37cm]{-}{N1}{N4}
        \ncline[nodesep=0.37cm]{-}{N2}{N3}
        \ncline[nodesep=0.37cm]{-}{N2}{N4}
        \ncline[nodesep=0.37cm]{-}{N4}{N5}
        \ncline[nodesep=0.37cm]{-}{N4}{N6}
        \ncline[nodesep=0.37cm]{-}{N5}{N6}
        \ncline[nodesep=0.37cm]{-}{N5}{N7}
        \ncline[nodesep=0.37cm]{-}{N3}{N4}
        \ncline[nodesep=0.37cm]{-}{N1}{N6}
        \ncline[nodesep=0.37cm]{-}{N6}{N7}
        \ncline[nodesep=0.37cm]{-}{N8}{N4}
        \ncline[nodesep=0.37cm]{-}{N8}{N5}
        \ncline[nodesep=0.37cm]{-}{N8}{N9}
        \ncline[nodesep=0.37cm]{-}{N8}{N3}
        \ncline[nodesep=0.37cm]{-}{N8}{N6}
        \ncline[nodesep=0.37cm]{-}{N9}{N5}
        \ncline[nodesep=0.37cm]{-}{N9}{N7}
        \ncline[nodesep=0.37cm]{-}{N10}{N6}
        \ncline[nodesep=0.37cm]{-}{N10}{N7}
        \ncline[nodesep=0.37cm]{-}{N10}{N1}
        \ncline[nodesep=0.37cm]{-}{N10}{N5}

        %\psgrid
		\end{pspicture}
		}

    \hfill
    
    {\huge \textbf{Communication-Efficient Algorithms \vspace{0.4em}\\For Distributed Optimization}}
    
    \vspace{0.7cm}
    
    {\large \textsc{Jo\~ao Filipe de Castro Mota}}

    \vfill
    
    {\large
			\textbf{Advisors}	\spacebetweenlinesT \\}
		{\raggedright \large 
			\hspace{2.02cm}Pedro Aguiar, Instituto Superior T\'ecnico, Technical University of Lisbon \spacebetweenlinesT \\
			\hspace{2.02cm}Markus P\"uschel, ETH Zurich \& Carnegie Mellon University \spacebetweenlinesT \\
			\hspace{2.02cm}Jo\~ao Xavier, Instituto Superior T\'ecnico, Technical University of Lisbon \\
		}
    
    \vfill

    {\large
		Thesis approved in public session to obtain the dual PhD degree in
		\spacebetweenlinesT \\
		{\large Electrical and Computer Engineering}
    }

    \vfill

		{\large \today}

    \end{center}
    %\vspace{1.0cm}

\end{titlepage}

\restoregeometry

%% file: 01-FrontMatter/committeeIST.tex
\newcommand{\spaceadvisors}{\vspace{0.1em}}

\vspace*{0.5cm}

\begin{flushleft}
	\textit{\large Doctoral Dissertation Committee:} \vspace{0.17em}\\

	Professor \textsc{Pedro Aguiar} (Advisor), Instituto Superior T\'ecnico, Technical University of Lisbon  \spaceadvisors \\

	Professor \textsc{Jos\'e M.\ F.\ Moura}, Carnegie Mellon University \spaceadvisors \\

	Professor \textsc{Markus P\"uschel} (Advisor), ETH Zurich \& Carnegie Mellon University \spaceadvisors \\

	Professor \textsc{Alejandro Ribeiro}, University of Pennsylvania \spaceadvisors \\

	Professor \textsc{Jo\~ao Xavier} (Advisor), Instituto Superior T\'ecnico, Technical University of Lisbon

\end{flushleft}

\vfill


\vspace{0.8cm}

{\large \bfseries Palavras-chave:} \,\,{\large Algoritmos distribu\'idos, optimiza\c{c}\~ao distribu\'ida, m\'etodo alternado dos multiplicadores de Lagrange, redes de sensores, aquisi\c{c}\~ao comprimida de sinais, controlo preditivo, m\'aquinas de aprendizagem autom\'atica, redes de fluxos, effici\^encia nas comunica\c{c}\~oes, colora\c{c}\~ao em redes.}

\restoregeometry

\cleardoublepage

%% file: 01-FrontMatter/dedication.tex
\begin{Dedication}
		\ldots to my parents, Lu\'is and Helena.
\end{Dedication}

%% file: 01-FrontMatter/resumo.tex
\begin{Resumo}

	Esta tese aborda o desenho de algoritmos distribu\'idos para resolver problemas de optimiza\c{c}\~ao. O cen\'ario \'e uma rede com~$P$ n\'os, onde cada n\'o tem acesso exclusivo a uma fun\c{c}\~ao de custo~$f_p$; todos os n\'os devem cooperar a fim de minimizar a soma de todas as fun\c{c}\~oes, $f_1 + \cdots + f_P$. In\'umeros problemas nas \'areas de processamento de sinal, controlo, e aprendizagem autom\'atica podem ser formulados desta maneira. Como crit\'erio de desempenho, adoptamos o n\'umero de comunica\c{c}\~oes entre os n\'os, j\'a que comunicar \'e  frequentemente a opera\c{c}\~ao que mais energia consome e, muitas vezes, tamb\'em a mais lenta. As duas principais contribui\c{c}\~oes desta tese s\~ao um esquema de classifica\c{c}\~ao de problemas de optimiza\c{c}\~ao distribu\'idos e um conjunto respectivo de algoritmos eficientes.

	A classe de problemas de optimiza\c{c}\~ao que consideramos \'e bastante geral, j\'a que assumimos que cada fun\c{c}\~ao pode depender, n\~ao necessariamente de todas as componentes da vari\'avel de optimiza\c{c}\~ao, mas de um n\'umero arbitr\'ario de componentes. Esta assump\c{c}\~ao permite-nos ir al\'em do que \'e normalmente assumido em optimiza\c{c}\~ao distribu\'ida e criar estrutura adicional que pode ser explorada para reduzir o n\'umero de comunica\c{c}\~oes. Esta estrutura forma a base do nosso esquema de classifica\c{c}\~ao, que identifica casos particulares mais simples; por exemplo, o problema mais comum em optimiza\c{c}\~ao distribu\'ida, onde cada fun\c{c}\~ao depende de todas as componentes.
	
	Os algoritmos que esta tese prop\~oe s\~ao distribu\'idos no sentido em que n\~ao h\'a nenhum n\'o central a controlar a rede ou a realizar c\'alculos de forma centralizada, todas as comunica\c{c}\~oes ocorrem exclusivamente entre n\'os vizinhos, e a informa\c{c}\~ao associada a cada n\'o \'e sempre processada localmente. Ilustramos os nossos algoritmos em v\'arias aplica\c{c}\~oes, entre as quais consenso de m\'edias, m\'aquinas de aprendizagem autom\'atica (\textit{support vector machines}), redes de fluxos, e v\'arios cen\'arios distribu\'idos em \textit{compressed sensing}. A tese tamb\'em prop\~oe um novo paradigma para modelar problemas de controlo distribu\'ido usando o conceito de \textit{model predictive control}. Atrav\'es de um conjunto extensivo de resultados experimentais, mostramos que os algoritmos propostos requerem menos comunica\c{c}\~oes para convergir do que os algoritmos distribu\'idos mais eficientes da literatura, incluindo algoritmos desenhados especificamente para uma aplica\c{c}\~ao particular.
 	 
\end{Resumo}

%% file: 01-FrontMatter/abstract.tex
\begin{Abstract}

 	This thesis is concerned with the design of distributed algorithms for solving optimization problems. The particular scenario we consider is a network with~$P$ compute nodes, where each node~$p$ has exclusive access to a cost function~$f_p$. We design algorithms in which all the nodes cooperate to find the minimum of the sum of all the cost functions, $f_1 + \cdots + f_P$. Several problems in signal processing, control, and machine learning can be posed as such optimization problems. Given that communication is often the most energy-consuming operation in networks and, many times, also the slowest one, it is important to design distributed algorithms with low communication requirements, that is, communication-efficient algorithms. The two main contributions of this thesis are a classification scheme for distributed optimization problems of the kind explained above and a set of corresponding communication-efficient algorithms.

 	The class of optimization problems we consider is quite general, since we allow that each function may depend on arbitrary components of the optimization variable, and not necessarily on all of them. In doing so, we go beyond the commonly used assumption in distributed optimization and create additional structure that can be explored to reduce the total number of communications. This structure is captured by our classification scheme, which identifies particular instances of the problem that are easier to solve. One example is the standard distributed optimization problem, in which all the functions depend on all the components of the variable.

 	All our algorithms are distributed in the sense that no central node coordinates the network, all the communications occur exclusively between neighboring nodes, and the data associated with each node is always processed locally. We show several applications of our algorithms, including average consensus, support vector machines, network flows, and several distributed scenarios for compressed sensing. We also propose a new framework for distributed model predictive control, which can be solved with our algorithms. Through extensive numerical experiments, we show that our algorithms outperform prior distributed algorithms in terms of communication-efficiency, even some that were specifically designed for a particular application.

\end{Abstract}

%% file: 01-FrontMatter/acknowledgments.tex
\begin{Acknowledgments}

	This thesis is the result of a complicated sequence of events. I would like to take the opportunity to thank here some of the people who, directly or indirectly, influenced, changed, or caused those events.

	The direct causers of the main events were, undoubtedly, my advisors: Jo\~ao Xavier, Pedro Aguiar, and Markus P\"uschel. The three of them gave me the support, the insight, and the knowledge that made this thesis possible. I learned a lot from them, both academically and non-academically. Most importantly, no matter how busy they were, they could always find time to answer my questions, to take care of bureaucracy that involved me, and to meet in our regular meetings. Also, I want to say that I had lots of fun in those yearly (work!) trips to several towns in Portugal. Thank you for all of that!

	I would like to thank my thesis committee members, Jos\'e Moura and Alejandro Ribeiro, for all the insight and suggestions. During my PhD, and especially during the years I spent in CMU, Jos\'e was always very supportive. On the few occasions that we discussed research, Jos\'e showed me how to look at my research from a different perspective. I would also like to thank Alejandro for arranging everything when I visited him in Philadelphia.

	The person who convinced me to enter the CMU/Portugal PhD program was Jo\~ao Paulo Costeira. He has always been in the background, doing whatever is needed to make this program great, and providing a comfortable layer between all the bureaucracy that lies under such a big program and the students (including myself). He has also put me in contact with people and projects from the real world! Another early causer of the events that led to this thesis was Victor Barroso, who invited me to participate in research meetings at ISR, and subsequently introduced me to $2/3$ of my future advisors.

	During my PhD, I had the opportunity to collaborate and to discuss research with several people. I would like to thank them for that. Some of these people are Michael Rabbat, Jo\~ao Miranda Lemos, Gabriela Hug, Andr\'e Martins, M\'ario Figueiredo, Petros Boufounos, Qing Ling, Ricardo Lima, Bruno Sinopoli, Stephen Boyd, Soummya Kar, Aurora Schmidt, Pedro Guerreiro, Ricardo Cabral, Christian Conte, Stefan Richter, Paul Goulart, Christian Berger, Jer\'onimo Rodrigues, Claudia Soares, Brian Swenson, Dusan Jakoveti\'c, Dragana Bajovic, Sabina Zejnilovic, Pinar Oguz, Dario Figueira, June Zhang, Qixing Liu, Matthias Althoff, Alysson Bessani, Paulo Oliveira, Bruce Krogh, Marija Ili\'c, Susana Brand\~ao, Nicholas O'Donoughue, Nikos Arechiga, Kyri Baker, Aliaksei Sandryhaila, Marek Telgarsky, Augusto Santos, Bernardo Pires, Ceyhun Eksin, Divyanshu Vats, Akshay Rajans, Ehsan Zamanizadeh, Jhi-Young Joo, Sanja Cvijic, Lu\'is Brand\~ao, Franz Franchetti, Xiahui Wang (Eeyore), Luca Parolini, Rohan Chabukswar, Joel Harley, Rodrigo Belo, Joya Deri, Jim Weimer, and S\'ergio Pequito. Also, thank you to Daniel McFarlin and Vas Chellappa for helping me out with several issues with GNU/Linux and MPI. Some of the experiments shown in this thesis were run on a computer cluster kindly provided by Florin Manolache.

	Conducting research while jumping back and forth over a large ocean is not possible without proper funding and an excellent team ``taking care of things.'' So I would like to thank the CMU/Portugal program and Funda\c{c}\~ao para a Ci\^encia e Tecnologia (FCT) for the grant SFRH/BD/33520/2008, provided through the Carnegie Mellon/Portugal Program and managed by the Information and Communication Technologies Institute (ICTI). Some work was partially funded by the FCT grants CMU-PT/SIA/0026/2009 and PEst-OE/EEI/LA0009/2009. I am also grateful to all the staff involved at CMU, IST, and the CMU/Portugal program, especially to Ana Mateus, Carolyn Patterson, Susana Santana, Alexandra Ara\'ujo, Ana Santos, Filomena Viegas, Lori Spears, Claire Bauerle, Tara Moe, Elaine Lawrence, and Samantha Goldstein.

	No single piece of this thesis would be possible without the early support of both of my parents, Lu\'is and Helena, who always encouraged me no matter what direction I chose. Their support has been a constant throughout my entire education and, for this and other reasons, the minimum I can give back is to dedicate this thesis to them.

	My sister, Renata, has also always provided constant encouragement, kept me in a good mood, and was a source of inspiration. My uncle Manuel and aunt F\'atima, and my uncle Alexandre and aunt Lili, and cousin Afonso have always encouraged me in my studies and instilled in me an interest in science from an early age. I would also like to thank the family friends Lourdes and Francisco Celestino for their support and friendship.

	Finally, I have no words to describe my gratitude to my wife, Kate. Thank you for all your support, kindness, and love. Thank you also for making sure that, during the writing of this thesis, I had proper nutrition, rest, and also fun; thanks for proofreading some parts of the thesis also. Now I promise that I'll do my homework for our piano lessons.

\end{Acknowledgments}

%% file: 02-MainMatter/introduction.tex
\chapter{Introduction}
\label{Ch:Introduction}

	Optimization theory has contributed to many fields in engineering by providing efficient algorithms that solve nontrivial real world problems. Notable examples can be found in signal processing, control engineering, and machine learning~\cite{Boyd04-ConvexOptimization,Bertsekas99-NonlinearProgramming,BenTal01-LecturesModernConvexOptimization}. On the other hand, over the last years, some computation platforms on which these algorithms may be executed have become distributed. For example, computers are now equipped with several processing devices, allowing for parallel computation. Also, complex systems such as power grids or water distribution systems are composed of several interconnected components, each with some processing power, and thus are distributed by nature. In addition, the data to be processed is often generated at different locations as, for example, in sensor networks, or in the internet. All these factors ask for algorithms that process data or control systems in a distributed way. However, it is challenging to design optimization algorithms that are matched to distributed resources. Part of the reason is that efficient centralized algorithms, such as for example interior-point methods, cannot be easily adapted to distributed scenarios. The high-level goal of this thesis is to advance the design of distributed algorithms for solving optimization problems.

	\section{Overview}
	\label{Sec:IntroOverview}

	\fref{Fig:IntroNetwork} will help us describe the main problem addressed by this thesis. The figure shows a network with~$10$ nodes, where each node~$p$ holds a function~$f_p$. Our goal is to make all nodes cooperate in order to find a minimizer of the sum of all the functions:
	\begin{equation}\label{Eq:IntroProb}\tag{P}
		\begin{array}{ll}
			\underset{x \in\mathbb{R}^n}{\text{minimize}} & f_1(x_{S_1}) + f_2(x_{S_2}) + \cdots + f_P(x_{S_P})\,.
		\end{array}
	\end{equation}
	where~$x \in \mathbb{R}^n$ is the optimization variable. Each function~$f_p$ in~\eqref{Eq:IntroProb} depends on the components of the variable~$x$ that are indexed by the set~$S_p \subseteq \{1,\ldots,n\}$, and we use~$x_{S_p}$ to denote those components. For example, if the function at node~$3$ depends on components~$x_1$, $x_5$, $x_8$, and~$x_{10}$, then~$S_3 = \{1,5,8,10\}$ and~$f_3(x_{S_3}) = f_3(x_1,x_5,x_8,x_{10})$. We require each function~$f_p$ to be private to node~$p$, i.e., no other node in the network has access to it. The edges of the network represent communication links; this means, for example, that node~$3$ in \fref{Fig:IntroNetwork} can communicate only with its neighbors: nodes~$2$, $4$, and~$8$. Given such a network, an algorithm that solves~\eqref{Eq:IntroProb} is considered \textit{distributed} if it uses no central node, no all-to-all communications, and if the privacy requirement for each function~$f_p$ is satisfied. In this thesis, we aim to solve~\eqref{Eq:IntroProb}, and related problems, with distributed algorithms that are communication-efficient, i.e., that use a minimal amount of communication. Communication-efficiency is an essential requirement, for example, when the nodes are battery-operated devices, such as in sensor-networks, since communication is usually very energy-demanding.

	\begin{figure}
  \centering
  \psscalebox{0.95}{
		\begin{pspicture}(7,4.8)
				\def\nodesimp{
            \pscircle*[linecolor=black!65!white](0,0){0.3}
        }
        \rput(1,3.7){\rnode{N1}{\nodesimp}}
        \rput(0.3,2){\rnode{N2}{\nodesimp}}
        \rput(0.3,0.3){\rnode{N3}{\nodesimp}}
        \rput(2.5,1.7){\rnode{N4}{\nodesimp}}
        \rput(5,2){\rnode{N5}{\nodesimp}}
        \rput(3.8,3){\rnode{N6}{\nodesimp}}
        \rput(6.7,3.2){\rnode{N7}{\nodesimp}}
        \rput(3.5,0.5){\rnode{N8}{\nodesimp}}
        \rput(6.3,0.9){\rnode{N9}{\nodesimp}}
        \rput(4.1,4.1){\rnode{N10}{\nodesimp}}

        \rput(1,3.7){\footnotesize \textcolor{black!3!white}{$1$}}      \rput[rb](0.68765248,3.94987802){$f_1$}
        \rput(0.3,2){\footnotesize \textcolor{black!3!white}{$2$}}      \rput[rb](0.017157,2.282843){$f_2$}
        \rput(0.3,0.3){\footnotesize \textcolor{black!3!white}{$3$}}    \rput[rt](0.017157,0.017157){$f_3$}
        \rput(2.5,1.7){\footnotesize \textcolor{black!3!white}{$4$}}    \rput[t](2.5,1.3){$f_4$}
        \rput(5,2){\footnotesize \textcolor{black!3!white}{$5$}}        \rput[t](5.0,1.6){$f_5$}
        \rput(3.8,3){\footnotesize \textcolor{black!3!white}{$6$}}      \rput[rb](3.51715729,3.28284271){$f_6$}
        \rput(6.7,3.2){\footnotesize \textcolor{black!3!white}{$7$}}    \rput[lb](6.982843,3.482843){$f_7$}
        \rput(3.5,0.5){\footnotesize \textcolor{black!3!white}{$8$}}    \rput[tr](3.21715729,0.21715729){$f_8$}
        \rput(6.3,0.9){\footnotesize \textcolor{black!3!white}{$9$}}    \rput[lt](6.58284271,0.61715729){$f_9$}
        \rput(4.1,4.1){\footnotesize \textcolor{black!3!white}{$10$}}   \rput[lb](4.41234752,4.34987802){$f_{10}$}

        \ncline[nodesep=0.37cm]{-}{N1}{N2}
        \ncline[nodesep=0.37cm]{-}{N1}{N4}
        \ncline[nodesep=0.37cm]{-}{N2}{N3}
        \ncline[nodesep=0.37cm]{-}{N2}{N4}
        \ncline[nodesep=0.37cm]{-}{N4}{N5}
        \ncline[nodesep=0.37cm]{-}{N4}{N6}
        \ncline[nodesep=0.37cm]{-}{N5}{N6}
        \ncline[nodesep=0.37cm]{-}{N5}{N7}
        \ncline[nodesep=0.37cm]{-}{N3}{N4}
        \ncline[nodesep=0.37cm]{-}{N1}{N6}
        \ncline[nodesep=0.37cm]{-}{N6}{N7}
        \ncline[nodesep=0.37cm]{-}{N8}{N4}
        \ncline[nodesep=0.37cm]{-}{N8}{N5}
        \ncline[nodesep=0.37cm]{-}{N8}{N9}
        \ncline[nodesep=0.37cm]{-}{N8}{N3}
        \ncline[nodesep=0.37cm]{-}{N8}{N6}
        \ncline[nodesep=0.37cm]{-}{N9}{N5}
        \ncline[nodesep=0.37cm]{-}{N9}{N7}
        \ncline[nodesep=0.37cm]{-}{N10}{N6}
        \ncline[nodesep=0.37cm]{-}{N10}{N7}
        \ncline[nodesep=0.37cm]{-}{N10}{N1}
        \ncline[nodesep=0.37cm]{-}{N10}{N5}

        %\psgrid
   \end{pspicture}
  }
  \bigskip
  \caption[Illustration of the main problem of the thesis.]{
		Illustration of the main problem of the thesis: each node in the network holds a private function and the goal is to minimize the sum of all the functions.
	}
  \label{Fig:IntroNetwork}
  \end{figure}
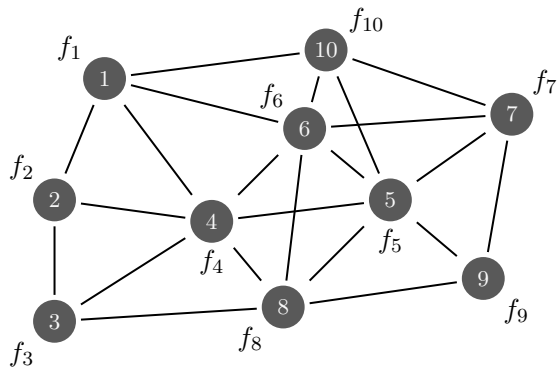

	\mypar{Simple example} Consider an inference problem on a sensor network~\cite{Rabbat04-DistributedOptimizationSensorNetworks,Bertsekas10-IncrementalGradientSubgradientProximal}, and suppose that each function~$f_p$ depends on all the components of the optimization variable~$x \in \mathbb{R}^n$, i.e., $S_p = \{1,\ldots,n\}$, for all~$p$ or, more compactly, $\cap_{p=1}^P S_p = \{1,\ldots,n\}$. While each node in the network represents a sensor with computing abilities, each edge indicates direct sensor communication, for instance, through a wireless connection. We want to estimate a parameter $\bar{\theta} \in \mathbb{R}^n$ (e.g., a set of environmental parameters~\cite{Akyildiz02-WirelessNetworksASurvey}), by using noisy measurements from all nodes. Let~$\theta_p$ be the measurement of~$\bar{\theta}$ taken at node~$p$. Assuming the noise is independent across nodes, finding the maximum log-likelihood estimate of~$\bar{\theta}$ can be written as~\eqref{Eq:IntroProb} with~$\cap_{p=1}^P S_p = \{1,\ldots,n\}$. For example, if the noise is Gaussian with zero mean and its covariance is the identity matrix, each~$f_p(x)$ is given by $(1/2)\|x - \theta_p\|^2$, and the resulting problem is known as the \textit{average consensus problem}~\cite{DeGroot74-ReachingConsensus}. In this case, the solution to~\eqref{Eq:IntroProb} is simply $x^\star = (1/P)\sum_{p=1}^P \theta_p$, that is, the maximum log-likelihood estimation of~$\bar{\theta}$ is the average of all the measurements. However, in our distributed scenario, node~$p$ is the only node who knows~$\theta_p$, and this makes computing the above average challenging. This simple example shows that to compute a solution of~\eqref{Eq:IntroProb} the nodes have to communicate, either by exchanging their private data or by exchanging their estimates of the problem's solution. What they exchange and how they do it is determined by the distributed algorithm they use.

	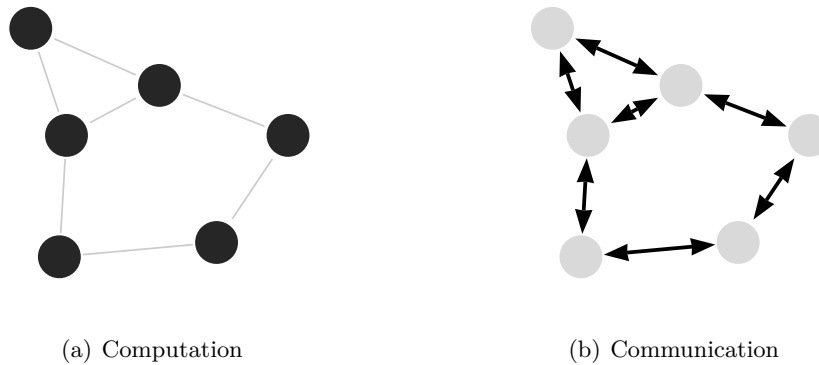
\begin{figure}
	\centering
	\subfigure[Computation]{\label{SubFig:IntroDistrAlgComputation}
		\psscalebox{0.95}{
			\begin{pspicture}(4.4,5)
				\def\nodesimp{
					\pscircle*[linecolor=black!85!white](0,0){0.3}
        }

				\rput(0.5,4.1){\rnode{C1}{\nodesimp}}   %\rput(0.5,4.1){\small \textcolor{white}{$1$}}
        \rput(1.0,2.6){\rnode{C2}{\nodesimp}}   %\rput(1.0,2.6){\small \textcolor{white}{$2$}}
        \rput(0.9,0.9){\rnode{C3}{\nodesimp}}   %\rput(0.9,0.9){\small \textcolor{white}{$3$}}
        \rput(3.1,1.1){\rnode{C4}{\nodesimp}}   %\rput(3.1,1.1){\small \textcolor{white}{$4$}}
        \rput(4.1,2.6){\rnode{C5}{\nodesimp}}   %\rput(4.1,2.6){\small \textcolor{white}{$5$}}
        \rput(2.3,3.3){\rnode{C6}{\nodesimp}}   %\rput(2.3,3.3){\small \textcolor{white}{$6$}}

				\psset{nodesep=0.33cm,linewidth=0.7pt,linecolor=black!20!white}
        \ncline{-}{C1}{C2}
        \ncline{-}{C1}{C6}
        \ncline{-}{C2}{C3}
        \ncline{-}{C2}{C6}
        \ncline{-}{C3}{C4}
        \ncline{-}{C4}{C5}
        \ncline{-}{C5}{C6}

				%\psgrid
			\end{pspicture}
		}
  }
  \hspace{2cm}
  \subfigure[Communication]{\label{SubFig:IntroDistrAlgCommunication}
		\psscalebox{0.95}{
			\begin{pspicture}(4.4,5)
				\def\nodesimp{
					\pscircle*[linecolor=black!15!white](0,0){0.3}
        }

				\rput(0.5,4.1){\rnode{C1}{\nodesimp}}   %\rput(0.5,4.1){\small \textcolor{black!35!white}{$1$}}
        \rput(1.0,2.6){\rnode{C2}{\nodesimp}}   %\rput(1.0,2.6){\small \textcolor{black!35!white}{$2$}}
        \rput(0.9,0.9){\rnode{C3}{\nodesimp}}   %\rput(0.9,0.9){\small \textcolor{black!35!white}{$3$}}
        \rput(3.1,1.1){\rnode{C4}{\nodesimp}}   %\rput(3.1,1.1){\small \textcolor{black!35!white}{$4$}}
        \rput(4.1,2.6){\rnode{C5}{\nodesimp}}   %\rput(4.1,2.6){\small \textcolor{black!35!white}{$5$}}
        \rput(2.3,3.3){\rnode{C6}{\nodesimp}}   %\rput(2.3,3.3){\small \textcolor{black!35!white}{$6$}}

				%\psset{nodesep=0.33cm,linewidth=0.7pt}
        %\ncline{-}{C1}{C2}
        %\ncline{-}{C1}{C6}
        %\ncline{-}{C2}{C3}
        %\ncline{-}{C2}{C6}
        %\ncline{-}{C3}{C4}
        %\ncline{-}{C4}{C5}
        %\ncline{-}{C5}{C6}

        %\psset{nodesep=0.34cm,arrowsize=6pt,arrowinset=0.1,linewidth=1.5pt,offset=-0.2cm}
        \psset{nodesep=0.32cm,arrowsize=7pt,arrowinset=0.05,linewidth=1.5pt}
        \ncline{<->}{C1}{C2}
        \ncline{<->}{C6}{C1}
        \ncline{<->}{C2}{C3}
        \ncline{<->}{C2}{C6}
        \ncline{<->}{C3}{C4}
        \ncline{<->}{C4}{C5}
        \ncline{<->}{C5}{C6}

				%\psgrid
			\end{pspicture}
		}
  }

  \caption[The two steps of a distributed algorithm.]{
		The two steps of a distributed algorithm. The nodes iteratively perform \text{(a)} computations and \text{(b)} broadcast the results of those computation to their neighbors.
	}
  \label{Fig:IntroDistributedAlg}
  \end{figure}

	\mypar{Distributed algorithms}
	A distributed algorithm computes a solution~$x^\star$ of~\eqref{Eq:IntroProb} while satisfying the requirement that each function~$f_p$ remains private to node~$p$. Typically, each iteration of a distributed algorithm consists of the two steps shown in \fref{Fig:IntroDistributedAlg}: \text{(a)} a computation step, and \text{(b)} a communication step. In the computation step, all nodes update their estimates of the components of~$x^\star$. Usually, each node~$p$ updates its estimates by combining information given by its private function~$f_p$ with information given by the estimates of its neighbors from the prior communication step. All these estimates are then exchanged in the subsequent communication step. Although all nodes in \fref{Fig:IntroDistributedAlg} are performing each of the two steps in parallel, this is not required for a distributed algorithm. Actually, as we will see, in environments such as wireless networks it might be impossible to perform the communication step \text{(b)} in parallel, because of packet collisions. In the average consensus example given above, a popular choice for the computation step~\text{(a)} is to linearly combine the estimate of node~$p$ with the estimates of its neighbors~$\mathcal{N}_p$. That is, the estimate of node~$p$, $x_p$, is updated as
	\begin{equation}\label{Eq:IntroAverageConsensus}
		x_p^{k+1} = a_{pp}\,x_p^k + \sum_{j \in \mathcal{N}_p}a_{pj}\, x_j^k\,,
	\end{equation}
	 where each $a_{pj}$ is a positive number, $a_{pp} + \sum_{j \in \mathcal{N}_p} a_{pj} = 1$, and $k$ denotes the iteration number. The computation scheme~\eqref{Eq:IntroAverageConsensus} implies that the nodes exchange their estimates $x_p^k$ at each communication step (\fref{SubFig:IntroDistrAlgCommunication}). This family of algorithms for the average consensus problem has been widely studied in the literature~\cite{DeGroot74-ReachingConsensus,Boyd04-FastLinearIterationsforDistributedAveraging,Erseghe11-FastConsensusByADMM,Olshevsky11-ConvergenceSpeedDistributedConsensusAveraging,Oreshkin10-OptimizationAnalysisDistrAveraging,Kar09-DistributedConsensusAlgsSN}.

	 In this thesis, we propose algorithms that solve not only the average consensus problem, but the entire class~\eqref{Eq:IntroProb}. We will see that this class contains several other problems that are relevant in signal processing, control theory, machine learning, and other areas. Solving~\eqref{Eq:IntroProb} in full generality, however, is challenging because the sets~$S_p$ are arbitrary. Our approach consists of identifying particular cases of~\eqref{Eq:IntroProb} that are easier to solve, designing algorithms for those cases, and then generalizing them to the most difficult cases. To do that, we introduce a scheme to classify instances of~\eqref{Eq:IntroProb}, as overviewed next. The outcome of our approach will be an algorithm solving~\eqref{Eq:IntroProb} in full generality. Despite its generality, our algorithm achieves performances better than prior distributed algorithms, even including some that were designed for a particular application.

	\begin{figure}
  \centering
  \subfigure[Global variable]{\label{SubFig:GlobalVar}
    \psscalebox{0.95}{
      \begin{pspicture}(5,5)
        \def\nodesimp{
          \pscircle*[linecolor=black!65!white](0,0){0.3}
        }

        \rput(0.4,4.1){\rnode{C1}{\nodesimp}}   \rput(0.4,4.1){\small \textcolor{white}{$1$}}
        \rput(0.9,2.6){\rnode{C2}{\nodesimp}}   \rput(0.9,2.6){\small \textcolor{white}{$2$}}
        \rput(0.8,0.9){\rnode{C3}{\nodesimp}}   \rput(0.8,0.9){\small \textcolor{white}{$3$}}
        \rput(3.0,1.1){\rnode{C4}{\nodesimp}}   \rput(3.0,1.1){\small \textcolor{white}{$4$}}
        \rput(4.0,2.6){\rnode{C5}{\nodesimp}}   \rput(4.0,2.6){\small \textcolor{white}{$5$}}
        \rput(2.2,3.3){\rnode{C6}{\nodesimp}}   \rput(2.2,3.3){\small \textcolor{white}{$6$}}

        \ncline[nodesep=0.33cm,linewidth=0.9pt]{-}{C1}{C2}
        \ncline[nodesep=0.33cm,linewidth=0.9pt]{-}{C1}{C6}
        \ncline[nodesep=0.33cm,linewidth=0.9pt]{-}{C2}{C3}
        \ncline[nodesep=0.33cm,linewidth=0.9pt]{-}{C2}{C6}
        \ncline[nodesep=0.33cm,linewidth=0.9pt]{-}{C3}{C4}
        \ncline[nodesep=0.33cm,linewidth=0.9pt]{-}{C4}{C5}
        \ncline[nodesep=0.33cm,linewidth=0.9pt]{-}{C5}{C6}

        \rput[lb](-0.11,4.45){ $f_1(x_1,x_2,x_3)$}
        \rput[lt](1.1,2.4){$f_2(x_1,x_2,x_3)$}
        \rput[lt](0.1,0.56){$f_3(x_1,x_2,x_3)$}
        \rput[t](3.7,0.74){$f_4(x_1,x_2,x_3)$}
        \rput[b](4.5,2.97){$f_5(x_1,x_2,x_3)$}
        \rput[b](2.9,3.68){$f_6(x_1,x_2,x_3)$}

        %\psgrid
      \end{pspicture}
    }
  }
  \hspace{2cm}
  \subfigure[Non-connected variable]{\label{SubFig:NonconnectedVar}
    \psscalebox{0.95}{
      \begin{pspicture}(5,5)
        \def\nodesimp{
          \pscircle*[linecolor=black!65!white](0,0){0.3}
        }

        \rput(0.4,4.1){\rnode{C1}{\nodesimp}}   \rput(0.4,4.1){\small \textcolor{white}{$1$}}
        \rput(0.9,2.6){\rnode{C2}{\nodesimp}}   \rput(0.9,2.6){\small \textcolor{white}{$2$}}
        \rput(0.8,0.9){\rnode{C3}{\nodesimp}}   \rput(0.8,0.9){\small \textcolor{white}{$3$}}
        \rput(3.0,1.1){\rnode{C4}{\nodesimp}}   \rput(3.0,1.1){\small \textcolor{white}{$4$}}
        \rput(4.0,2.6){\rnode{C5}{\nodesimp}}   \rput(4.0,2.6){\small \textcolor{white}{$5$}}
        \rput(2.2,3.3){\rnode{C6}{\nodesimp}}   \rput(2.2,3.3){\small \textcolor{white}{$6$}}

        \ncline[nodesep=0.33cm,linewidth=0.9pt]{-}{C1}{C2}
        \ncline[nodesep=0.33cm,linewidth=0.9pt]{-}{C1}{C6}
        \ncline[nodesep=0.33cm,linewidth=0.9pt]{-}{C2}{C3}
        \ncline[nodesep=0.33cm,linewidth=0.9pt]{-}{C2}{C6}
        \ncline[nodesep=0.33cm,linewidth=0.9pt]{-}{C3}{C4}
        \ncline[nodesep=0.33cm,linewidth=0.9pt]{-}{C4}{C5}
        \ncline[nodesep=0.33cm,linewidth=0.9pt]{-}{C5}{C6}

        \rput[lb](0.0,4.45){ $f_1(x_1,x_2)$}
        \rput[lt](1.1,2.4){$f_2(x_2,x_3)$}
        \rput[lt](0.2,0.56){$f_3(x_1,x_2,x_3)$}
        \rput[t](3.5,0.74){$f_4(x_1,x_3)$}
        \rput[b](4.5,2.97){$f_5(x_1,x_2)$}
        \rput[b](2.6,3.65){$f_6(x_2)$}

        %\psgrid
      \end{pspicture}
    }
  }
  \caption[Two instances of~\eqref{Eq:IntroProb} for a variable with~$3$ components.]{
		Two instances of~\eqref{Eq:IntroProb} for a variable with~$3$ components, $x = (x_1,x_2,x_3)$. In \text{(a)},
		the variable is~\textit{global} (and thus connected) because all the functions depend on all the components.
		In \text{(b)}, the variable is \textit{non-connected} because~$x_1$ induces a subgraph that is not connected.
  }
  \label{Fig:IntroExamplesVariable}
	\end{figure}
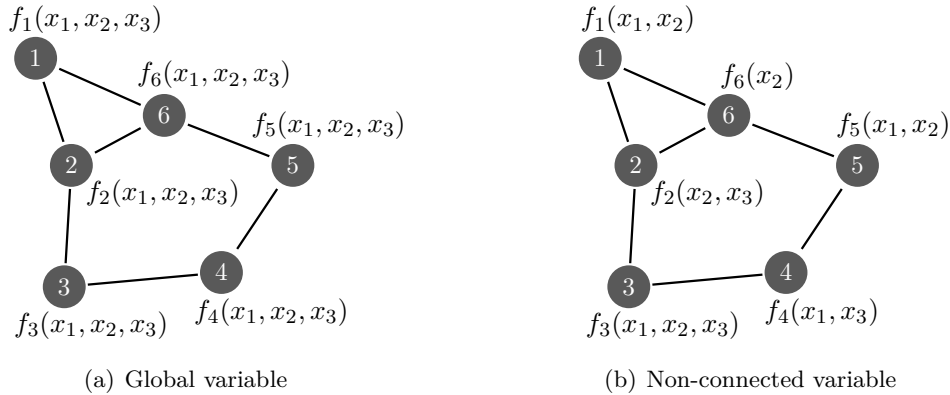

	\mypar{Classification scheme}
	The most popular instance of~\eqref{Eq:IntroProb} is illustrated in \fref{SubFig:GlobalVar}: each function depends on \textit{all} the components of the variable, $\cap_{p=1}^P S_p = \{1,\ldots,n\}$. Rewriting~\eqref{Eq:IntroProb} for this case, we have
	\begin{equation}\label{Eq:IntroGlobalProb}\tag{G}
		\underset{x}{\text{minimize}} \,\,\, f_1(x) + f_2(x) + \cdots + f_P(x)\,,
	\end{equation}
	which is the instance of~\eqref{Eq:IntroProb} for which most distributed algorithms have been designed. In our classification scheme, formally introduced later in \sref{Sec:ClassificationScheme} and visualized in \fref{Fig:IntroClassification}, we say that problem~\eqref{Eq:IntroGlobalProb} has a \textit{global variable}. Although many applications can be written as~\eqref{Eq:IntroGlobalProb}, many others are instances of~\eqref{Eq:IntroProb} with a non-global variable. In fact, our main motivation for considering the generic problem~\eqref{Eq:IntroProb} stems from its ability to model problems where each node is interested only in a subset of the problem's parameters or variables, rather than in all of them. This is typical in large-scale systems, for example, in large plants, in the power grid, and in the internet. A fundamental assumption we make is that if node~$p$ depends on components~$x_{S_p}$, then that node is interested in computing the optimal value for those components only, and not for any of the other components. For example, node~$4$ in \fref{SubFig:NonconnectedVar} depends on components~$x_1$ and~$x_3$, which means that it will compute the optimal value for these components, but not for~$x_2$. The flexibility introduced in~\eqref{Eq:IntroProb} by the sets~$S_p$, however, produces instances that are difficult to solve, given the previous assumption. \fref{SubFig:NonconnectedVar} shows an example: the component~$x_1$ appears in the functions of nodes~$1$, $3$, $4$, and~$5$, but not in the functions of nodes~$2$ and~$6$. This means that node~$1$ is ``isolated'' from all the other nodes that also depend on~$x_1$; indeed, nodes~$2$ and~$6$ are not interested in computing an optimal value for~$x_1$, let alone exchanging estimates of it. In other words, the subgraph of the nodes that depend on~$x_1$ is not connected and, for this reason, we say that the variable in this case is \textit{non-connected}. Of course, computing an optimal solution of~\eqref{Eq:IntroProb} in this case will invariably require selecting one of the nodes~$2$ or~$6$ to retransmit estimates of~$x_1$, so that all the nodes depending on this component can agree on an optimal value for it. In the small example of \fref{SubFig:NonconnectedVar}, it is indifferent to select either node~$2$ or node~$6$ for this task, but in larger networks, and for arbitrary sets~$S_p$, we should select the nodes in such a way that the total number of communications is minimized. Our solution for this problem involves computing Steiner trees and is explained in \cref{Ch:ConnectedNonConnected}.

	The concepts of global variable and non-connected variable are concepts of the classification scheme we introduce in this thesis. These concepts and the ones of connected, mixed, and star-shaped variable will be formally defined in \sref{Sec:ClassificationScheme}, but their relation can be visualized in \fref{Fig:IntroClassification}. Roughly, the variable of~\eqref{Eq:IntroProb} is divided into two classes: connected and non-connected. These are, in fact, the most relevant classes in our classification scheme for two reasons: they form a partition of the all the instances of the variable of~\eqref{Eq:IntroProb}, and addressing them requires completely different techniques. These two classes thus comprise the first level of our classification scheme. The second level consists of the following subclasses: global, star-shaped, and mixed. These subclasses neither are mutually disjoint nor do they cover all instances of the variable of~\eqref{Eq:IntroProb}. However, they are relevant both because they are much simpler instances of~\eqref{Eq:IntroProb}, and because they have been solved with several distributed algorithms. Most of the algorithms that solve these subclasses, however, cannot be easily generalized to solve the entire connected and non-connected classes. In this thesis, we propose an algorithm that solves~\eqref{Eq:IntroProb} for all classes and subclasses of variables.

\mypar{Overview of some applications}
	In this thesis we will consider several applications that arise in distributed contexts and that can be written as instances of~\eqref{Eq:IntroProb}. The recent field of compressed sensing~\cite{Donoho06-CompressedSensing,Candes06-RobustUncertaintyPrinciplesExactSignalReconstructionHighlyIncomplete} provides a rich collection of such problems: basis pursuit (BP)~\cite{Donoho98-AtomicDecompositionBasisPursuit}, basis pursuit denoising (BPDN)~\cite{Donoho98-AtomicDecompositionBasisPursuit}, and the least absolute shrinkage and selection operator (lasso)~\cite{Tibshirani96-RegressionShrinkageLasso}, among others. These compressed sensing problems are convex and provide heuristics for finding sparse solutions of linear systems. Although finding the sparsest solution of a linear system is NP-hard, compressed sensing theory establishes conditions under which the previous problems find an optimal (i.e., sparsest) solution. There is an increasing interest in solving compressed sensing in distributed scenarios, where either the columns or the rows of the matrix defining the linear system are spread over several nodes. We reformulate the above compressed sensing problems as~\eqref{Eq:IntroProb}, some with a global variable and others with a mixed one; some of these reformulations are novel and are presented in this thesis for the first time.

	We will see that training a support vector machine (SVM)\cite[Ch.7]{Bishop06-PatternRecognitionMachineLearning} requires solving an optimization problem that can be easily recast as~\eqref{Eq:IntroGlobalProb}. Roughly, given a database with two classes of datapoints, the goal in training an SVM is to find the hyperplane that best separates the two classes of datapoints. When the datapoints are distributed among several sites, training an SVM arises naturally as a distributed optimization problem. Therefore, solving this problem with a distributed algorithm has the advantages of not requiring the transmission of the private databases to a remote location, and of providing more robustness (if one node fails, the remaining nodes can still train the SVM, yet, with less data).

	Many systems can be modeled as networked dynamical systems~\cite{Slijak07-LargeScaleDynamicSystems}. Specifically, each system is seen as the node of a network and has associated a state, a control input, or both. The state of a given node is influenced not only by its own state and control input (or simply, input), but also by the states and inputs of its neighbors. An effective control strategy for this type of systems is distributed model predictive control (D-MPC)~\cite{Camponogara02-DistributedMPC}, which consists of the following. First, at each time instant, each node senses its own state; then, the nodes collectively solve an optimization problem that finds the best set of control inputs for a future time-horizon. These inputs are computed in such a way that their application to the systems will lead the nodes' states to a given goal and, at the same time, they will minimize some ``energy function.'' Although the nodes know an optimal set of inputs for all the time instants in the time-horizon, they will only use the input for the next time instant. The reason is to mitigate the impact of modeling and sensing errors. So, in the next time instant, after applying the previously computed input, each node senses its state and cooperates with the other nodes to solve the D-MPC optimization problem, now with new data. This procedure is repeated at each time instant. In this thesis, we provide a new framework for formulating D-MPC problems, and also communication-efficient algorithms to solve them.

	We also mention that several network flow problems can be recast as~\eqref{Eq:IntroProb} with a star-shaped variable. These are optimization problems formulated on directed networks where physical items can flow through the edges of the network. As a consequence, certain conservation laws have to be satisfied and are typically written as problem constraints. Network flow problems arise in several contexts~\cite{Ahuja93-NetworkFlows}, for example, in determining best energy policies in the power grid. After some reformulations, network flow problems can be recast as~\eqref{Eq:IntroProb} and, hence, can be solved with the algorithms we propose here.

	\mypar{Overview of the proposed algorithms}
	Problem reformulation plays a key role in the design of distributed optimization algorithms. In fact, we will see throughout this thesis that it impacts significantly the final algorithm. Our strategy for solving instances of~\eqref{Eq:IntroProb}, and ultimately~\eqref{Eq:IntroProb} in full generality, consists of reformulating those instances into a format such that well-known centralized optimization algorithms become naturally distributed.

	Our reformulations make use of a concept that has rarely appeared in high-level distributed algorithms, such as the ones considered in this thesis. That concept is \textit{network coloring}, an assignment of colors to the nodes of a network such that no two neighboring nodes have the same color (for convenience, instead of colors, we just use natural numbers). Assuming that a coloring scheme is available beforehand is realistic in many distributed scenarios, especially in wireless networks. For example, wireless networks require protocols known as media access control (MAC) to avoid packet collisions, i.e., that one node receives two messages at the same time and in the same frequency (assuming there is only one receive antenna). Some MAC protocols, such as time division multiple access (TDMA), rely on network coloring.

	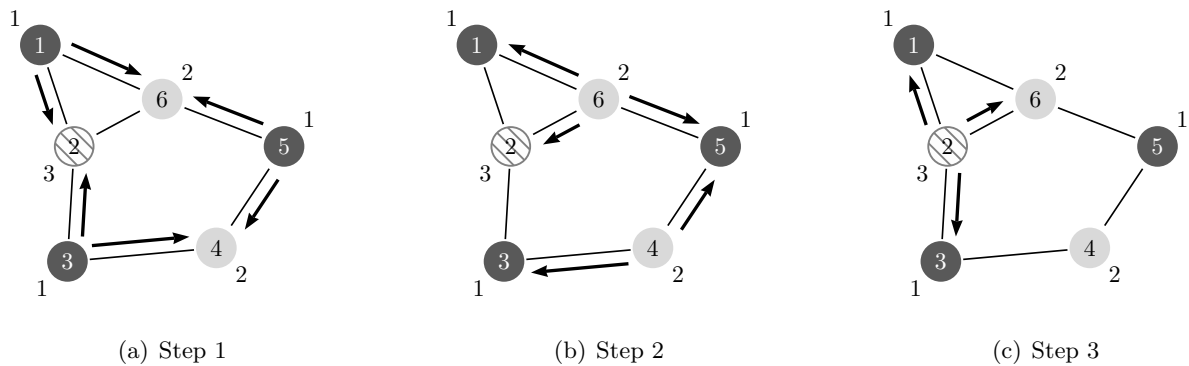
\begin{figure}
	\subfigure[Step 1]{\label{SubFig:Step1}
		\psscalebox{0.9}{
			\begin{pspicture}(4.9,5)
				\def\nodesimp{
					\pscircle*[linecolor=black!65!white](0,0){0.3}
        }
        \def\nodeB{
            \pscircle*[linecolor=black!15!white](0,0){0.3}
        }
        \def\nodeC{
            \pscircle[fillstyle=vlines*,linecolor=black!50!white,hatchcolor=black!50!white](0,0){0.3}
        }

				\rput(0.5,4.1){\rnode{C1}{\nodesimp}}   \rput(0.5,4.1){\small \textcolor{white}{$1$}}
        \rput(1.0,2.6){\rnode{C2}{\nodeC}}      \rput(1.0,2.6){\small \textcolor{black}{$2$}}
        \rput(0.9,0.9){\rnode{C3}{\nodesimp}}   \rput(0.9,0.9){\small \textcolor{white}{$3$}}
        \rput(3.1,1.1){\rnode{C4}{\nodeB}}      \rput(3.1,1.1){\small \textcolor{black}{$4$}}
        \rput(4.1,2.6){\rnode{C5}{\nodesimp}}   \rput(4.1,2.6){\small \textcolor{white}{$5$}}
        \rput(2.3,3.3){\rnode{C6}{\nodeB}}      \rput(2.3,3.3){\small \textcolor{black}{$6$}}

        \ncline[nodesep=0.33cm,linewidth=0.7pt]{-}{C1}{C2}
        \ncline[nodesep=0.33cm,linewidth=0.7pt]{-}{C1}{C6}
        \ncline[nodesep=0.33cm,linewidth=0.7pt]{-}{C2}{C3}
        \ncline[nodesep=0.33cm,linewidth=0.7pt]{-}{C2}{C6}
        \ncline[nodesep=0.33cm,linewidth=0.7pt]{-}{C3}{C4}
        \ncline[nodesep=0.33cm,linewidth=0.7pt]{-}{C4}{C5}
        \ncline[nodesep=0.33cm,linewidth=0.7pt]{-}{C5}{C6}

        \rput[rb](0.21715729,4.38284271){\small $1$}
        \rput[rt](0.71715729,2.31715729){\small $3$}
        \rput[rt](0.61715729,0.61715729){\small $1$}
        \rput[lt](3.38284271,0.81715729){\small $2$}
        \rput[lb](4.38284271,2.88284271){\small $1$}
        \rput[lb](2.58284271,3.58284271){\small $2$}

        \ncline[nodesep=0.38cm,arrowsize=5pt,arrowinset=0.1,linewidth=1.4pt,offset=-0.2cm]{->}{C1}{C2}
        \ncline[nodesep=0.38cm,arrowsize=5pt,arrowinset=0.1,linewidth=1.4pt,offset=+0.2cm]{->}{C1}{C6}
        \ncline[nodesep=0.38cm,arrowsize=5pt,arrowinset=0.1,linewidth=1.4pt,offset=-0.2cm]{->}{C3}{C2}
        \ncline[nodesep=0.38cm,arrowsize=5pt,arrowinset=0.1,linewidth=1.4pt,offset=+0.2cm]{->}{C3}{C4}
        \ncline[nodesep=0.38cm,arrowsize=5pt,arrowinset=0.1,linewidth=1.4pt,offset=+0.2cm]{->}{C5}{C4}
        \ncline[nodesep=0.38cm,arrowsize=5pt,arrowinset=0.1,linewidth=1.4pt,offset=-0.2cm]{->}{C5}{C6}

				%\psgrid
			\end{pspicture}
		}
  }
  \hfill
  \subfigure[Step 2]{\label{SubFig:Step2}
		\psscalebox{0.9}{
			\begin{pspicture}(4.9,5)
				\def\nodesimp{
					\pscircle*[linecolor=black!65!white](0,0){0.3}
        }
        \def\nodeB{
            \pscircle*[linecolor=black!15!white](0,0){0.3}
        }
        \def\nodeC{
            \pscircle[fillstyle=vlines*,linecolor=black!50!white,hatchcolor=black!50!white](0,0){0.3}
        }

				\rput(0.5,4.1){\rnode{C1}{\nodesimp}}   \rput(0.5,4.1){\small \textcolor{white}{$1$}}
        \rput(1.0,2.6){\rnode{C2}{\nodeC}}      \rput(1.0,2.6){\small \textcolor{black}{$2$}}
        \rput(0.9,0.9){\rnode{C3}{\nodesimp}}   \rput(0.9,0.9){\small \textcolor{white}{$3$}}
        \rput(3.1,1.1){\rnode{C4}{\nodeB}}      \rput(3.1,1.1){\small \textcolor{black}{$4$}}
        \rput(4.1,2.6){\rnode{C5}{\nodesimp}}   \rput(4.1,2.6){\small \textcolor{white}{$5$}}
        \rput(2.3,3.3){\rnode{C6}{\nodeB}}      \rput(2.3,3.3){\small \textcolor{black}{$6$}}

        \ncline[nodesep=0.33cm,linewidth=0.7pt]{-}{C1}{C2}
        \ncline[nodesep=0.33cm,linewidth=0.7pt]{-}{C1}{C6}
        \ncline[nodesep=0.33cm,linewidth=0.7pt]{-}{C2}{C3}
        \ncline[nodesep=0.33cm,linewidth=0.7pt]{-}{C2}{C6}
        \ncline[nodesep=0.33cm,linewidth=0.7pt]{-}{C3}{C4}
        \ncline[nodesep=0.33cm,linewidth=0.7pt]{-}{C4}{C5}
        \ncline[nodesep=0.33cm,linewidth=0.7pt]{-}{C5}{C6}

        \rput[rb](0.21715729,4.38284271){\small $1$}
        \rput[rt](0.71715729,2.31715729){\small $3$}
        \rput[rt](0.61715729,0.61715729){\small $1$}
        \rput[lt](3.38284271,0.81715729){\small $2$}
        \rput[lb](4.38284271,2.88284271){\small $1$}
        \rput[lb](2.58284271,3.58284271){\small $2$}

        \ncline[nodesep=0.38cm,arrowsize=5pt,arrowinset=0.1,linewidth=1.4pt,offset=-0.2cm]{->}{C6}{C1}
        \ncline[nodesep=0.38cm,arrowsize=5pt,arrowinset=0.1,linewidth=1.4pt,offset=+0.2cm]{->}{C6}{C2}
        \ncline[nodesep=0.38cm,arrowsize=5pt,arrowinset=0.1,linewidth=1.4pt,offset=+0.2cm]{->}{C6}{C5}
        \ncline[nodesep=0.38cm,arrowsize=5pt,arrowinset=0.1,linewidth=1.4pt,offset=+0.2cm]{->}{C4}{C3}
        \ncline[nodesep=0.38cm,arrowsize=5pt,arrowinset=0.1,linewidth=1.4pt,offset=-0.2cm]{->}{C4}{C5}

				%\psgrid
			\end{pspicture}
		}
  }
  \hfill
  \subfigure[Step 3]{\label{SubFig:Step3}
		\psscalebox{0.9}{
			\begin{pspicture}(4.9,5)
				\def\nodesimp{
					\pscircle*[linecolor=black!65!white](0,0){0.3}
        }
        \def\nodeB{
            \pscircle*[linecolor=black!15!white](0,0){0.3}
        }
        \def\nodeC{
            \pscircle[fillstyle=vlines*,linecolor=black!50!white,hatchcolor=black!50!white](0,0){0.3}
        }

				\rput(0.5,4.1){\rnode{C1}{\nodesimp}}   \rput(0.5,4.1){\small \textcolor{white}{$1$}}
        \rput(1.0,2.6){\rnode{C2}{\nodeC}}      \rput(1.0,2.6){\small \textcolor{black}{$2$}}
        \rput(0.9,0.9){\rnode{C3}{\nodesimp}}   \rput(0.9,0.9){\small \textcolor{white}{$3$}}
        \rput(3.1,1.1){\rnode{C4}{\nodeB}}      \rput(3.1,1.1){\small \textcolor{black}{$4$}}
        \rput(4.1,2.6){\rnode{C5}{\nodesimp}}   \rput(4.1,2.6){\small \textcolor{white}{$5$}}
        \rput(2.3,3.3){\rnode{C6}{\nodeB}}      \rput(2.3,3.3){\small \textcolor{black}{$6$}}

        \ncline[nodesep=0.33cm,linewidth=0.7pt]{-}{C1}{C2}
        \ncline[nodesep=0.33cm,linewidth=0.7pt]{-}{C1}{C6}
        \ncline[nodesep=0.33cm,linewidth=0.7pt]{-}{C2}{C3}
        \ncline[nodesep=0.33cm,linewidth=0.7pt]{-}{C2}{C6}
        \ncline[nodesep=0.33cm,linewidth=0.7pt]{-}{C3}{C4}
        \ncline[nodesep=0.33cm,linewidth=0.7pt]{-}{C4}{C5}
        \ncline[nodesep=0.33cm,linewidth=0.7pt]{-}{C5}{C6}

        \rput[rb](0.21715729,4.38284271){\small $1$}
        \rput[rt](0.71715729,2.31715729){\small $3$}
        \rput[rt](0.61715729,0.61715729){\small $1$}
        \rput[lt](3.38284271,0.81715729){\small $2$}
        \rput[lb](4.38284271,2.88284271){\small $1$}
        \rput[lb](2.58284271,3.58284271){\small $2$}

        \ncline[nodesep=0.38cm,arrowsize=5pt,arrowinset=0.1,linewidth=1.4pt,offset=+0.2cm]{->}{C2}{C1}
        \ncline[nodesep=0.38cm,arrowsize=5pt,arrowinset=0.1,linewidth=1.4pt,offset=+0.2cm]{->}{C2}{C6}
        \ncline[nodesep=0.38cm,arrowsize=5pt,arrowinset=0.1,linewidth=1.4pt,offset=+0.2cm]{->}{C2}{C3}

				%\psgrid
			\end{pspicture}
		}
  }
  \caption[Illustration of how the algorithms operate according to the coloring of the network.]{
		Illustration of how the algorithms operate according to the coloring of the network. The coloring scheme has three colors: nodes~$1$, $3$, and~$5$ have color~$1$, nodes~$4$ and~$6$ have color~$2$, and node~$2$ has color~$3$.
  }
  \label{Fig:IllustrationAlgs}
  \end{figure}

	Figure~\ref{Fig:IllustrationAlgs} shows how the algorithms we propose work as a function of the coloring scheme. The network in this figure has three colors: nodes~$1$, $3$, and~$5$ have color~$1$, nodes~$4$ and~$6$ have color~$2$, and node~$2$ has color~$3$. The algorithms we propose are iterative, and each iteration is divided into a number of steps equal to the number of colors. Figure~\ref{Fig:IllustrationAlgs} thus has~$3$ subfigures, each one corresponding to a step. In each step, all the nodes with the same color perform the same tasks in parallel, as illustrated in subfigures~\ref{SubFig:Step1}, \ref{SubFig:Step2}, and~\ref{SubFig:Step3}. These subfigures show the communication pattern occurring in each step. From an high-level point of view, the tasks performed by node~$p$ consist of:
	\begin{enumerate}
		\item finding new estimates for the components that~$f_p$ depends on, by solving
		$$
			\text{minimize} \,\,\, f_p(\cdot) + \text{``quadratic term,''}
		$$
		that is, node~$p$ minimizes the sum of~$f_p$ and a quadratic term. That quadratic term depends on the network structure as well as on previous estimates of the neighbors of node~$p$. Solving the above optimization problem corresponds to evaluating the proximity operator of the function~$f_p$ and, many times, this can be done in a simple way.

		\item sending the new estimates to the neighboring nodes.
	\end{enumerate}
	Finally, we note that the concept of network coloring required by our algorithms coincides with the concept of network coloring commonly used in low-level communication protocols, namely, MAC protocols~\cite[Ch.6]{Krishnamachari05-NetworkingWirelessSensors}. The goal of MAC protocols is to avoid packet collisions due to the hidden node and the exposed node problems~\cite[\S6.2.2]{Krishnamachari05-NetworkingWirelessSensors}. For example, in \fref{SubFig:Step1}, node~$6$ is receiving simultaneous messages from nodes~$1$ and~$5$. If the messages are in the same frequency and node~$6$ has one antenna only, this results in a packet collision and the nodes have to retransmit their messages. Time division multiple access (TDMA), for example, is a MAC protocol that avoids packet collisions by using a second-order coloring scheme: each node cannot have the same color as its neighbors and as its neighbors' neighbors. Such a coloring scheme works for our algorithms as well and, for this reason, the high-level structure of our algorithms is not altered by low-level protocols when they are implemented in networks that use TDMA as a MAC protocol.

	The strategy we use to derive our distributed algorithms consists of reformulating the problems we want to solve in such a way that we can apply well known centralized optimization algorithms. 	Regarding our choice for these algorithms, we will focus on the \textit{Alternating Direction Method of Multipliers} (ADMM), more specifically on an extended version of it: the multi-block, or extended, ADMM~\cite{Han12-NoteOnADMM}. ADMM was proposed in the seventies by~\cite{Glowinski75-ADMMFirst,Gabay76-ADMMFirst} to solve linearly constrained optimization problems, using a ``divide-and-conquer'' approach. In the eighties and nineties, ADMM was shadowed by the popular interior point methods, which solve small- and medium-sized problems very efficiently, but in a centralized way. Lately, ADMM has regained attention from the optimization community, because of its wide applicability and its ability to deal with large-scale and distributed scenarios. Notably, ADMM has been applied to solve some of the problems addressed in this thesis, in particular, instances of~\eqref{Eq:IntroProb} with a global variable and with a star-shaped variable. In spite of that, little is still known about its behavior. For instance, partial results on the convergence rate of ADMM, or a proof of the convergence of the multi-block ADMM, were established only very recently.
	%One goal of this thesis is to understand what factors influence the behavior of ADMM, in particular, its convergence rate.\mynote{To do!}

	\begin{figure}[h]
     \centering
     \begin{pspicture}(8.0,5.5)
       \rput[bl](0.25,0.70){\includegraphics[width=7.5cm]{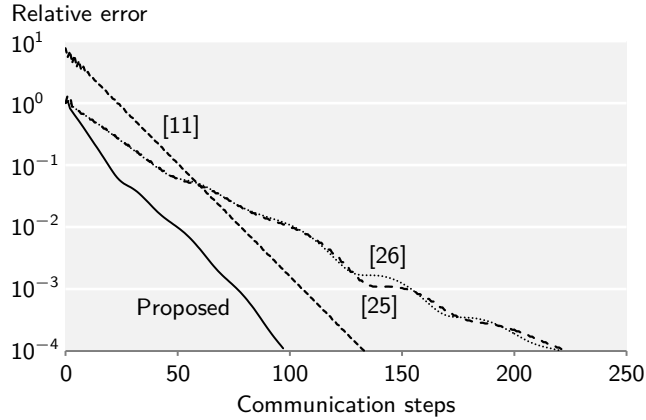}}
       \rput[b](4.00,-0.1){\footnotesize \textbf{\sf Communication steps}}
       \rput[bl](-0.45,5.17){\mbox{\footnotesize \textbf{{\sf Relative error}}}}

       \rput[r](0.20,4.89){\footnotesize $\mathsf{10^{1\phantom{-}}}$}
       \rput[r](0.20,4.08){\footnotesize $\mathsf{10^{0\phantom{-}}}$}
       \rput[r](0.20,3.26){\footnotesize $\mathsf{10^{-1}}$}
       \rput[r](0.20,2.46){\footnotesize $\mathsf{10^{-2}}$}
       \rput[r](0.20,1.63){\footnotesize $\mathsf{10^{-3}}$}
       \rput[r](0.20,0.81){\footnotesize $\mathsf{10^{-4}}$}

       \rput[t](0.280,0.59){\footnotesize $\mathsf{0}$}
       \rput[t](1.769,0.59){\footnotesize $\mathsf{50}$}
       \rput[t](3.260,0.59){\footnotesize $\mathsf{100}$}
       \rput[t](4.748,0.59){\footnotesize $\mathsf{150}$}
       \rput[t](6.240,0.59){\footnotesize $\mathsf{200}$}
			 \rput[t](7.734,0.59){\footnotesize $\mathsf{250}$}

			 \rput[rt](2.47,1.48){\footnotesize \textbf{\sf Proposed}}
       \rput[bl](1.55,3.60){\footnotesize \textbf{\sf \cite{Oreshkin10-OptimizationAnalysisDistrAveraging}}}
       \rput[rt](4.70,1.52){\footnotesize \textbf{\sf \cite{Schizas08-ConsensusAdHocWSNsPartI}}}
       \rput[bl](4.30,1.87){\footnotesize \textbf{\sf \cite{Zhu09-DistributedInNetworkChannelCoding}}}

       %\psgrid
     \end{pspicture}
     \caption[Comparison of the performance of a proposed algorithm with prior algorithms for the average consensus problem.]{
				Comparison of the performance of a proposed algorithm with prior algorithms for the average consensus problem. The network is a randomly generated geometric network with $2000$ nodes.
     }
     \label{Fig:IntroExperimentalResults}
	\end{figure}

\mypar{Example of performance results}
	\fref{Fig:IntroExperimentalResults} shows as example the performance of the algorithm we propose for~\eqref{Eq:IntroGlobalProb} when applied to the average consensus problem (scalar case, i.e., $n=1$). The plot shows the relative error of the solution estimates versus the number of communication steps. We say that a communication step has occurred whenever all the nodes have updated their estimates and transmitted them to their neighbors. This is equivalent to saying that the number of communication steps is the total number of communications divided by~$2E$, twice the number of edges. A communication, in this case, is defined as an ``edge usage.'' For example, in \fref{Fig:IllustrationAlgs}, there are~$6$ communications in \text{(a)}, $5$ communications in \text{(b)}, and~$3$ communications in \text{(c)}. And one communication step in that figure is comprised of steps~$1$, $2$, and~$3$. The number of communication steps, therefore, provides a direct measure of communication-efficiency. The network we used in the experiments of \fref{Fig:IntroExperimentalResults} has~$P = 2000$ nodes and was randomly generated as a geometric network with parameter $\sqrt{\log(P)/P} \simeq 0.06$. The figure shows that, among all algorithms, the proposed one required the least amount of communications (i.e., communication steps) to achieve any error between $10^0$ and $10^{-4}$. The other algorithms in the figure are~\cite{Schizas08-ConsensusAdHocWSNsPartI,Zhu09-DistributedInNetworkChannelCoding}, which solve the entire global problem class~\eqref{Eq:IntroGlobalProb}, and~\cite{Oreshkin10-OptimizationAnalysisDistrAveraging}, which is considered the most efficient consensus algorithm~\cite{Erseghe11-FastConsensusByADMM}, but it can only solve the average consensus problem and not any other problem in the class~\eqref{Eq:IntroGlobalProb}. Actually, if we consider the convergence rate, i.e., the slopes of the error lines, the proposed algorithm and~\cite{Oreshkin10-OptimizationAnalysisDistrAveraging} have roughly the same performance. In fact, they have the same slope, but~\cite{Oreshkin10-OptimizationAnalysisDistrAveraging} exhibits an offset, since it requires a special initialization. All the other algorithms were initialized with zeros.

	This experimental result reveals the surprising fact that, although the algorithms we propose solve an entire problem class, they can sometimes achieve the same performance as the best algorithms for a particular application. This is particularly surprising for the average consensus problem, since it is the simplest and the most thoroughly studied distributed problem.

\section{Goals of the thesis}

	We next summarize the goals of the thesis and then we explain each requirement in detail.

	\begin{center}
		\begin{minipage}{0.92\textwidth}
			\singlespace
			\begin{shaded}
				\medskip
				We aim to design, analyze, and implement algorithms that solve optimization problems of the form~\eqref{Eq:IntroProb} on networks. The algorithms should be
				\begin{list}{}{\setlength\itemindent{-0.15in}\setlength\itemsep{0.1in}}
          \item \textbf{Distributed:} no node has complete knowledge about the problem data and no central node is allowed; also, each node communicates only with its neighbors;
          \item \textbf{Communication-efficient:} the number of communications they use is minimized;
					\item \textbf{Network-independent:} the algorithms run on networks with arbitrary topology and their output is independent of the network.
				\end{list}
			\smallskip
			\end{shaded}
		\end{minipage}
	\end{center}

	\mypar{Distributed}
	A distributed algorithm only makes sense in an environment where both the data and the computing power are distributed. In such an environment, an algorithm is considered distributed if it fulfills three requirements. First, local data to a given node should remain private to that node. This enforces local computations, since any computation involving a piece of data has to be performed at the node where that data belongs to. In our problem~\eqref{Eq:IntroProb}, data of node~$p$ is encoded in the function~$f_p$ and, thus, we require~$f_p$ to be private to node~$p$; this means that no other node has full knowledge of~$f_p$ at any time during and before the execution of the algorithm.
	%Some algorithms, however, require each node to broadcast some piece of information to all the other nodes beforehand, but never the full information about the $f_p$'s.

	The second requirement is that there should not exist any central or special node. Such a node would coordinate all the other nodes and would make the data at a given node reachable to any other node in a very small number of hops. Actually, an algorithm that satisfies the first requirement but not the second one is usually called a \textit{parallel algorithm}~\cite{Bertsekas97-ParallelDistributed}. Indeed, according to~\cite{Bertsekas97-ParallelDistributed}, parallel algorithms run on systems where computing devices are at a small distance of each other and may be controlled by a central entity. Distributed algorithms, in contrast, run on systems where computing devices are located far apart, making centralized coordination inconvenient; in the latter, there is also little control on the network topology.

	Finally, the third requirement is that each node communicates only with neighboring nodes. Although this is equivalent to forbidding a central node, we explicitly state this requirement in order to exclude platforms that allow all-to-all communications. For example, an algorithm running on a computer cluster and using function calls from a message passing interface (MPI)~\cite{Dongarra96-MPITheCompleteReference} implementation, such as \verb#MPI_Bcast# or \verb#MPI_Reduce#, cannot be considered distributed; at most, it is parallel. We mention that sometimes distributed algorithms are also referred to as \textit{decentralized algorithms}.

	\mypar{Communication-efficient}
	In a centralized algorithm, the execution time and the closely related floating-point operation (FLOP) count are the most common performance metrics: the lower these metrics are, the more efficient an algorithm is. In distributed scenarios, however, other metrics arise. For example, computing accurate solutions is challenging in scenarios where there is communication noise. In that case, slower algorithms that are noise-resilient may be preferable to faster algorithms that are noise-sensitive. Another example is energy consumption. In many distributed scenarios, e.g., sensor networks, nodes rely on batteries and therefore have a limited source of energy. In these situations, increasing the lifespan of the network becomes the main priority. As communication in battery-operated devices is currently the most energy-consuming operation~\cite{Akyildiz02-WirelessNetworksASurvey,Fischione11-DesignPrinciplesWSN-chapter}, this priority translates into having algorithms with low communication requirements. The performance metric adopted in this thesis will then be the number of communications: the lower the number of communications an algorithm uses, the more efficient that algorithm will be. Hence, our goal will be to design distributed optimization algorithms that use the fewest communications possible.

	\mypar{Network-independent}
	The last requirement we impose on distributed algorithms is \textit{network independence}. This simply means that the output of the algorithm, i.e., the estimate of the solution returned by the algorithm, should be independent of the network topology. For instance, the algorithms should output the same solution estimate whether they are run on a densely or on a sparsely connected network. Naturally, the performance of the algorithms, i.e., the number of iterations or communications they use to compute that estimate, will in general vary with the network topology.

\section{A classification scheme for distributed optimization}
\label{Sec:ClassificationScheme}

	Our strategy for achieving the goals of this thesis is a divide-and-conquer one: first, we identify instances of~\eqref{Eq:IntroProb} that are easier to solve; then, we combine the solutions we designed for the simpler instances to solve~\eqref{Eq:IntroProb} in full generality. For convenience, we reproduce~\eqref{Eq:IntroProb} here:
  \begin{equation}\label{Eq:IntroProb}\tag{P}
		\begin{array}{ll}
			\underset{x \in\mathbb{R}^n}{\text{minimize}} & f_1(x_{S_1}) + f_2(x_{S_2}) + \cdots + f_P(x_{S_P})\,.
		\end{array}
	\end{equation}
	In this section, we formally introduce our classification scheme for the variable of problem~\eqref{Eq:IntroProb}. Before doing that, however, we need the concept of \textit{communication network}.

	\subsection{Communication network}
	\label{SubSec:CommunicationNetwork}

	The communication network is the physical network through which the computing devices, seen as network nodes, communicate. We represent the communication network with an undirected graph~$\mathcal{G}=(\mathcal{V},\mathcal{E})$, where $\mathcal{V}$ and $\mathcal{E}$ are the set of nodes and the set of edges, respectively. The cardinality of these sets, i.e., the number of nodes and the number of edges, will be denoted with~$P=|\mathcal{V}|$ and~$E=|\mathcal{E}|$, respectively. \fref{Fig:IntroNetwork} shows an example of a graph representing a communication network with~$P=10$ nodes and~$E = 21$ edges. An edge belongs to the communication network, say $(i,j) \in \mathcal{E}$, if and only if nodes~$i$ and~$j$ communicate directly. For example, nodes~$1$ and~$10$ in \fref{Fig:IntroNetwork} are neighbors: this means they can exchange messages with each other, because there is a communication link connecting them. We use the following convention: if~$(i,j) \in \mathcal{E}$, then~$i<j$. Throughout this thesis, we will assume that the communication network~$\mathcal{G}$ is connected and that its topology does not vary with time.

	\mypar{Functions associated to nodes}
  Associated with each node, there is a function depending on the components of a variable~$x \in \mathbb{R}^n$. The function at node~$p$ is denoted with $f_p: \mathbb{R}^{n_p} \xrightarrow{} \mathbb{R} \cup \{+\infty\}$, where~$n_p$ is the cardinality of the set~$S_p$. As explained before, we use~$S_p \subseteq \{1,\ldots,n\}$ to denote the components of~$x \in \mathbb{R}^n$ that function~$f_p$ depends on. Of course, $1 \leq n_p = |S_p| \leq n$. We assume that each node~$p$ is interested in computing the optimal value only of the components of~$x$ that are indexed by~$S_p$. We will see that in some situations, however, it is difficult, or even impossible, to solve instances of~\eqref{Eq:IntroProb} without forcing some nodes to receive and transmit components of~$x$ that are not indexed by their sets~$S_p$. To make our problem well-defined, we assume that each component of the variable appears in at least one of the nodes, that is, $\cup_{p=1}^P S_p = \{1,\ldots,n\}$.

  Unless otherwise stated, we that assume~$f_p$ is closed and convex~\cite{Boyd04-ConvexOptimization,BenTal01-LecturesModernConvexOptimization,Bertsekas99-NonlinearProgramming,Lemarechal04-FundamentalsConvexAnalysis,Nesterov04-IntroductoryLecturesConvexOptimization}, and not identically~$+\infty$. Note that our definition for each~$f_p$ allows it to take the value~$+\infty$; as a consequence, node~$p$ can impose constraints on the variable~$x$ implicitly, via indicator functions. An indicator function of a given set~$S \subset \mathbb{R}^{n}$ is defined as $\text{i}_S:\mathbb{R}^{n} \xrightarrow{} \mathbb{R} \cup \{+\infty\}$,
  $$
		\text{i}_S(x) =
		\left\{
			\begin{array}{ll}
				0 &,\,\, x \in S \\
				+\infty &,\,\, x \not\in S\,.
			\end{array}
		\right.
  $$
  Including an indicator function $\text{i}_S(x)$ in the objective of a minimization problem forces $x \in S$, since otherwise the optimal (minimal) value is~$+\infty$. Each function~$f_p$ is \textit{private}, i.e., at all times during and before the execution of the algorithm, only node~$p$ knows~$f_p$. As explained before, this privacy rule formalizes our wish to derive a distributed algorithm by enforcing local computations; namely, all computations involving~$f_p$ have to be done at node~$p$. This makes sense in scenarios where each~$f_p$ encodes a database that should be known only at node~$p$, or simply to make use of all the distributed computing resources as, for example, in a sensor network, where each sensor has some processing power available for computation.

	\subsection{Variable classification}

	Although each function is uniquely associated to a single node, the same does not happen for each component of~$x \in \mathbb{R}^n$, the optimization variable. This creates an additional structure and motivates our classification scheme. Essential to our classification scheme is the concept of induced subgraph.

	\mypar{Induced subgraph}
	Let~$x_l \in \mathbb{R}$ denote the $l$th component of the optimization variable~$x \in \mathbb{R}^n$.	We define the subgraph induced by~$x_l$ similarly to how~\cite[Ch.1]{Bollobas08-ModernGraphTheory} defines the subgraph induced by a set of nodes. In our case, these nodes are the ones whose functions depend on~$x_l$. To be more concrete, given a communication network~$\mathcal{G}$, the \textit{subgraph induced by $x_l$} is the subgraph $\mathcal{G}_l = (\mathcal{V}_l,\mathcal{E}_l) \subseteq \mathcal{G}$, where $\mathcal{V}_l$ is the set of nodes whose functions depend on~$x_l$, and an edge $(i,j)$ belongs to~$\mathcal{E}_l$ only if $(i,j) \in \mathcal{E}$ and both nodes~$i$ and~$j$ belong to~$\mathcal{V}_l$. As an example, \fref{SubFig:ClassLocalVariableHighlighted} highlights the subgraph induced by the component~$x_2$ in the setting of \fref{SubFig:ClassLocalVariable}: the set of nodes and the set of edges of this induced subgraph~$\mathcal{G}_2$ are, respectively, $\mathcal{V}_2 = \{1,2,3,6\}$ and $\mathcal{E}_2 = \{(1,2),(1,6),(2,3),(2,6)\}$. Note that neither~$f_4$ nor~$f_5$ depend on~$x_2$.

	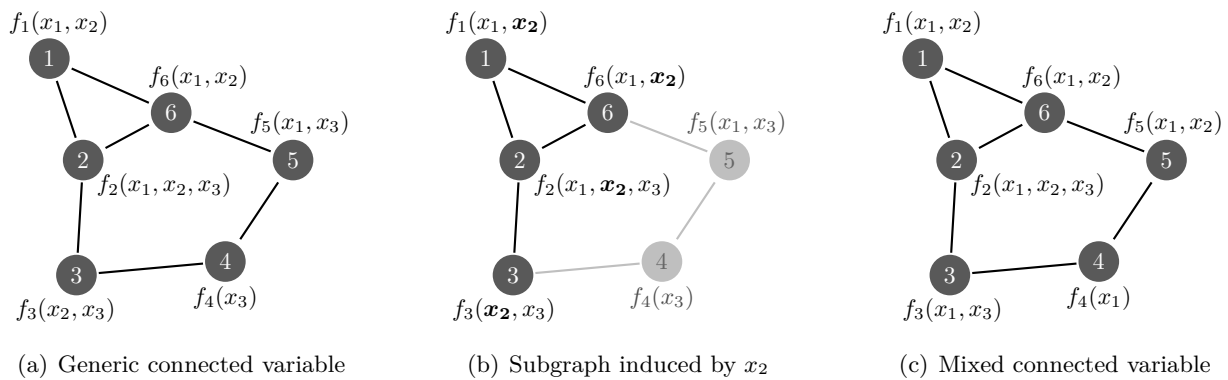
\begin{figure}
  \centering
  \subfigure[Generic connected variable]{\label{SubFig:ClassLocalVariable}
    \psscalebox{0.9}{
      \begin{pspicture}(4.9,4.8)
        \def\nodesimp{
          \pscircle*[linecolor=black!65!white](0,0){0.3}
        }

        \rput(0.5,4.1){\rnode{C1}{\nodesimp}}   \rput(0.5,4.1){\small \textcolor{white}{$1$}}
        \rput(1.0,2.6){\rnode{C2}{\nodesimp}}   \rput(1.0,2.6){\small \textcolor{white}{$2$}}
        \rput(0.9,0.9){\rnode{C3}{\nodesimp}}   \rput(0.9,0.9){\small \textcolor{white}{$3$}}
        \rput(3.1,1.1){\rnode{C4}{\nodesimp}}   \rput(3.1,1.1){\small \textcolor{white}{$4$}}
        \rput(4.1,2.6){\rnode{C5}{\nodesimp}}   \rput(4.1,2.6){\small \textcolor{white}{$5$}}
        \rput(2.3,3.3){\rnode{C6}{\nodesimp}}   \rput(2.3,3.3){\small \textcolor{white}{$6$}}

        \ncline[nodesep=0.33cm,linewidth=0.9pt]{-}{C1}{C2}
        \ncline[nodesep=0.33cm,linewidth=0.9pt]{-}{C1}{C6}
        \ncline[nodesep=0.33cm,linewidth=0.9pt]{-}{C2}{C3}
        \ncline[nodesep=0.33cm,linewidth=0.9pt]{-}{C2}{C6}
        \ncline[nodesep=0.33cm,linewidth=0.9pt]{-}{C3}{C4}
        \ncline[nodesep=0.33cm,linewidth=0.9pt]{-}{C4}{C5}
        \ncline[nodesep=0.33cm,linewidth=0.9pt]{-}{C5}{C6}

        \rput[lb](-0.1,4.45){\small $f_1(x_1,x_2)$}
        \rput[lt](1.2,2.4){\small $f_2(x_1,x_2,x_3)$}
        \rput[lt](0,0.56){\small $f_3(x_2,x_3)$}
        \rput[t](3.1,0.74){\small $f_4(x_3)$}
        \rput[b](4.2,2.97){\small $f_5(x_1,x_3)$}
        \rput[b](2.7,3.65){\small $f_6(x_1,x_2)$}

        %\psgrid
      \end{pspicture}
    }
  }
  \hfill
  \subfigure[Subgraph induced by $x_2$]{\label{SubFig:ClassLocalVariableHighlighted}
    \psscalebox{0.9}{
      \begin{pspicture}(4.9,4.8)
        \def\nodesimp{
          \pscircle*[linecolor=black!65!white](0,0){0.3}
        }
        \def\nodeshigh{
          \pscircle*[linecolor=black!25!white](0,0){0.3}
        }

        \rput(0.5,4.1){\rnode{C1}{\nodesimp}}    \rput(0.5,4.1){\small \textcolor{white}{$1$}}
        \rput(1.0,2.6){\rnode{C2}{\nodesimp}}    \rput(1.0,2.6){\small \textcolor{white}{$2$}}
        \rput(0.9,0.9){\rnode{C3}{\nodesimp}}    \rput(0.9,0.9){\small \textcolor{white}{$3$}}
        \rput(3.1,1.1){\rnode{C4}{\nodeshigh}}   \rput(3.1,1.1){\small \textcolor{black!60!white}{$4$}}
        \rput(4.1,2.6){\rnode{C5}{\nodeshigh}}   \rput(4.1,2.6){\small \textcolor{black!60!white}{$5$}}
        \rput(2.3,3.3){\rnode{C6}{\nodesimp}}    \rput(2.3,3.3){\small \textcolor{white}{$6$}}

        \ncline[nodesep=0.33cm,linewidth=0.9pt]{-}{C1}{C2}
        \ncline[nodesep=0.33cm,linewidth=0.9pt]{-}{C1}{C6}
        \ncline[nodesep=0.33cm,linewidth=0.9pt]{-}{C2}{C3}
        \ncline[nodesep=0.33cm,linewidth=0.9pt]{-}{C2}{C6}
        \ncline[nodesep=0.33cm,linewidth=0.9pt,linecolor=black!25!white]{-}{C3}{C4}
        \ncline[nodesep=0.33cm,linewidth=0.9pt,linecolor=black!25!white]{-}{C4}{C5}
        \ncline[nodesep=0.33cm,linewidth=0.9pt,linecolor=black!25!white]{-}{C5}{C6}

        \rput[lb](-0.1,4.45){\small $f_1(x_1,\boldsymbol{x_2})$}
        \rput[lt](1.2,2.4){\small $f_2(x_1,\boldsymbol{x_2},x_3)$}
        \rput[lt](0,0.56){\small $f_3(\boldsymbol{x_2},x_3)$}
        \rput[t](3.1,0.74){\small \textcolor{black!70!white}{$f_4(x_3)$}}
        \rput[b](4.2,2.97){\small \textcolor{black!70!white}{$f_5(x_1,x_3)$}}
        \rput[b](2.7,3.65){\small $f_6(x_1,\boldsymbol{x_2})$}

        %\psgrid
      \end{pspicture}
    }
  }
  \hfill
  \subfigure[Mixed connected variable]{\label{SubFig:ClassMixedVariable}
    \psscalebox{0.9}{
      \begin{pspicture}(4.9,4.8)
        \def\nodesimp{
          \pscircle*[linecolor=black!65!white](0,0){0.3}
        }

        \rput(0.5,4.1){\rnode{C1}{\nodesimp}}   \rput(0.5,4.1){\small \textcolor{white}{$1$}}
        \rput(1.0,2.6){\rnode{C2}{\nodesimp}}   \rput(1.0,2.6){\small \textcolor{white}{$2$}}
        \rput(0.9,0.9){\rnode{C3}{\nodesimp}}   \rput(0.9,0.9){\small \textcolor{white}{$3$}}
        \rput(3.1,1.1){\rnode{C4}{\nodesimp}}   \rput(3.1,1.1){\small \textcolor{white}{$4$}}
        \rput(4.1,2.6){\rnode{C5}{\nodesimp}}   \rput(4.1,2.6){\small \textcolor{white}{$5$}}
        \rput(2.3,3.3){\rnode{C6}{\nodesimp}}   \rput(2.3,3.3){\small \textcolor{white}{$6$}}

        \ncline[nodesep=0.33cm,linewidth=0.9pt]{-}{C1}{C2}
        \ncline[nodesep=0.33cm,linewidth=0.9pt]{-}{C1}{C6}
        \ncline[nodesep=0.33cm,linewidth=0.9pt]{-}{C2}{C3}
        \ncline[nodesep=0.33cm,linewidth=0.9pt]{-}{C2}{C6}
        \ncline[nodesep=0.33cm,linewidth=0.9pt]{-}{C3}{C4}
        \ncline[nodesep=0.33cm,linewidth=0.9pt]{-}{C4}{C5}
        \ncline[nodesep=0.33cm,linewidth=0.9pt]{-}{C5}{C6}

        \rput[lb](-0.1,4.45){\small $f_1(x_1,x_2)$}
        \rput[lt](1.2,2.4){\small $f_2(x_1,x_2,x_3)$}
        \rput[lt](0.2,0.56){\small $f_3(x_1,x_3)$}
        \rput[t](3.1,0.74){\small $f_4(x_1)$}
        \rput[b](4.2,2.97){\small $f_5(x_1,x_2)$}
        \rput[b](2.7,3.65){\small $f_6(x_1,x_2)$}

        %\psgrid
      \end{pspicture}
    }
  }
  \caption[Example of a generic connected variable and a mixed connected variable.]{
    Example of \text{(a)} a generic connected variable and \text{(c)} a mixed connected variable. \text{(b)} highlights the subgraph induced by the component~$x_2$ in \text{(a)}. The communication network is the same in all cases. In \text{(c)}, the component~$x_1$ is global, and~$x_2$ and~$x_3$ induce connected subgraphs.
  }
  \label{Fig:IntroClassificationExamplesVariable}
  \end{figure}

	\mypar{Component-wise classification of \boldmath{$x$}}
	We classify each component~$x_l$ according to its induced subgraph~$\mathcal{G}_l$ the following way: $x_l$ is
	\begin{itemize}
		\item \textit{connected} if $\mathcal{G}_l$ is a connected subgraph, and is \textit{non-connected} otherwise;
		\item \textit{global} if its induced subgraph coincides with the communication network, i.e., $\mathcal{G}_l = \mathcal{G}$;
		\item \textit{star-shaped} if $\mathcal{G}_l$ is a star graph.
	\end{itemize}
	By star graph we mean a graph in which there exists a node who is a neighbor of all the other nodes; the remaining nodes can also be neighbors between themselves. For example, the subgraph induced by variable~$x_2$ in \fref{SubFig:ClassLocalVariableHighlighted} is a star, because every node is a neighbor of node~$2$; therefore, $x_2$ is star-shaped. If a component is star-shaped, it can be handled in a centralized way, since the node in the center of the star can act as a central node. It can be checked that component~$x_1$ in \fref{SubFig:ClassLocalVariable} is also star-shaped, with node~$6$ in the center, but component~$x_3$ is not. All the components of the variable in that figure, however, are connected, since the respective subgraphs are connected. Naturally, a star-shaped variable is always connected. An example of a global component is given in \fref{SubFig:ClassMixedVariable}: the subgraph induced by~$x_1$ coincides with communication graph and, thus, $x_1$ is global. In other words, all the functions in \fref{SubFig:ClassMixedVariable} depend on~$x_1$. Again, a global component is always connected, since its induced subgraph coincides with the communication network, which we assume connected. Unless the communication network is a star, a global variable is never star-shaped. We had already illustrated a non-connected component in \fref{SubFig:NonconnectedVar}: the subgraph induced by~$x_1$ in that network is not connected, and thus~$x_1$ non-connected.

	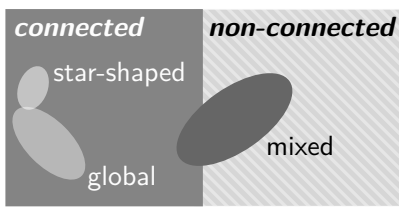
\begin{figure}
  \centering
  \psscalebox{1.05}{
    \begin{pspicture}(5.0,2.4)

				\psframe[fillstyle=vlines*,fillcolor=black!16!white,hatchcolor=black!10!white,hatchwidth=1.2pt,hatchsep=1.8pt,hatchangle=45,linewidth=0pt,linecolor=white](2.4,0)(5,2.5)
				\psframe*[linecolor=black!48!white,linewidth=0pt](0,0)(2.5,2.5)

				\rput[tl](0.1,2.41){\small \textsf{\textcolor{white}{\textbf{\textit{connected}}}}}
				\rput[tr](4.95,2.41){\small \textsf{\textcolor{black}{\textbf{\textit{non-connected}}}}}

				\psset{blendmode=2}

				\psset{linestyle=none,fillstyle=shape}
				\rput{-45}(0.55,0.8){\psellipse[fillcolor=black!30!white](0.0,0.0)(0.6,0.3)}
				\rput[b](1.47,0.2){\small \textsf{\textcolor{white}{global}}}

				\rput{70}(0.35,1.5){\psellipse[fillcolor=black!15!white](0.0,0.0)(0.3,0.2)}
				\rput[l](0.61,1.68){\small \textsf{\textcolor{white}{star-shaped}}}

				\rput{35}(2.9,1.1){\psellipse*[linecolor=black!60!white](0.0,0.0)(0.85,0.4)}
				\rput[tl](3.3,0.9){\small \textsf{\textcolor{black}{mixed}}}

        %\psgrid
   \end{pspicture}
  }

  \caption[Our classification scheme for the variable of problem~\eqref{Eq:IntroProb}.]{
		Our classification scheme for the variable~$x \in \mathbb{R}^n$ of problem~\eqref{Eq:IntroProb}. The variable is either connected or non-connected. Global and star-shaped variables are particular instances of a connected variable, and a mixed variable can be connected or non-connected.
	}
  \label{Fig:IntroClassification}
	\end{figure}

	\mypar{Classification of \boldmath{$x$}}
	With the component-wise classification of components, we are now in conditions to classify the full optimization variable~$x \in \mathbb{R}^n$ in~\eqref{Eq:IntroProb}. The proposed classification scheme is shown in \fref{Fig:IntroClassification}. There, the variable~$x$ is either
	\begin{itemize}
		\item \textit{connected} if all the components~$x_l$ of~$x \in \mathbb{R}^n$ are connected, for~$l = 1,\ldots,n$; or
		\item \textit{non-connected} if~$x$ has at least one non-connected component.
	\end{itemize}
	For example, while the variable in \fref{SubFig:ClassLocalVariable} is connected, because~$x_1$, $x_2$, and~$x_3$ are connected, the variable in \fref{SubFig:NonconnectedVar} is non-connected, because~$x_1$ is non-connected (in spite of~$x_2$ and~$x_3$ being connected). Note that the connected and non-connected classes partition the entire class of the variable~$x$ (see \fref{Fig:IntroClassification}). This distinction between a connected and a non-connected variable is the most important one in our classification scheme. In fact, we will see in \cref{Ch:ConnectedNonConnected} that they have to be addressed with different techniques.

	To the best of our knowledge, no algorithm has ever been designed (purposefully) to solve~\eqref{Eq:IntroProb} with a generic connected or non-connected variable. However, there are algorithms solving it with global, mixed (connected), and star-shaped variables, defined as follows. The variable~$x \in \mathbb{R}^n$ is
	\begin{itemize}
		\item \textit{global} if all its components are global: $\cap_{p=1}^P S_p = \{1,\ldots,n\}$;
		\item \textit{mixed} if it has at least one global component and at least one non-global component: $\cap_{p=1}^P S_p \neq \Bigl\{\emptyset, \{1,\ldots,n\}\Bigr\}$;
		\item \textit{star-shaped} if all its components are star-shaped.
	\end{itemize}
	We have already seen an example of a global variable in \fref{SubFig:GlobalVar}. When~\eqref{Eq:IntroProb} has a global variable, it can be written simply as
  \begin{equation}\label{Eq:IntroGlobalProb}\tag{G}
		\underset{x}{\text{minimize}} \,\,\, f_1(x) + f_2(x) + \cdots + f_P(x)\,.
	\end{equation}
	Because of the assumption that the communication network is always connected, a global variable is also always connected. A mixed variable, in turn, can be either connected or non-connected. It is connected if all the non-global components are connected, and non-connected if at least one of the non-global components is non-connected. \fref{SubFig:ClassMixedVariable} shows an example of a mixed variable that is connected, since the non-global components~$x_2$ and~$x_3$ are connected. Problem~\eqref{Eq:IntroProb} with a mixed variable can be written as
	\begin{equation}	\label{Eq:IntroMixed}\tag{M}
		\begin{array}{ll}
			\underset{x = (y,z) \in\mathbb{R}^n}{\text{minimize}} & f_1(y,z_{S_1}) + f_2(y,z_{S_2}) + \cdots + f_P(y,z_{S_P})\,,
		\end{array}
	\end{equation}
	where the variable~$x$ was decomposed into its global components~$y$ and into its non-global components~$z$. 	Finally, all the components of a star-shaped variable are like~$x_2$ in \fref{SubFig:ClassLocalVariableHighlighted}. Note that in \fref{Fig:IntroClassification} the star-shaped class intersects with the global class; this happens when the variable is global and the communication network is a star.

	Summarizing, our classification scheme partitions the variable of problem~\eqref{Eq:IntroProb} into two classes, shown as rectangles of \fref{Fig:IntroClassification}: connected and non-connected. These classes are the most fundamental ones, since they require different solution methods. The subclasses shown as ellipsoids in \fref{Fig:IntroClassification} identify easier instances of~\eqref{Eq:IntroProb}. Also, each one of these subclasses has been addressed with prior distributed optimization algorithms. In reality, while several algorithms have been proposed for the global and the star-shaped subclasses, we only found one distributed algorithm, in~\cite{Tan06-DistributedOptimizationCoupledSystemsNUM}, solving an instance of~\eqref{Eq:IntroProb} with a mixed variable. That instance is actually a very particular one: the variable is connected and all the non-global components are star-shaped. The classification scheme of \fref{Fig:IntroClassification} will also guide us throughout the thesis: we first address the global subclass, which is not only the subclass for which most of the distributed optimization algorithms have been proposed, but also the simplest one; in particular, the notation required to handle problems in this subclass is simpler, since we do not need the indexing sets~$S_p$. Then, we address the connected class, by generalizing the algorithm for the global class. And, finally, we generalize the connected class algorithm to handle both a connected and a non-connected variable, that is, to handle any instance of~\eqref{Eq:IntroProb}.

\section{Contributions}

	We list the main contributions of this thesis:
	\begin{itemize}
		\item We provide a classification scheme for the class~\eqref{Eq:IntroProb} of distributed optimization problems. Although it borrows some aspects from factor graphs~\cite{Kschischang01-FactorGraphsAndTheSum-ProductAlgorithm} (actually the same aspects that are used in~\cite{Boyd11-ADMM}), it establishes a relation between the (abstract) optimization problem to be solved and the (concrete) computational platform, in our case, the communication network. This classification scheme plays a fundamental role in the thesis, not only by providing a framework to develop our algorithms, but also by allowing us to organize prior work and applications.

		\item We develop a set of distributed algorithms to solve~\eqref{Eq:IntroProb} that are communication-efficient. The order in which we present our algorithms in the next chapters goes from the most specific to the most general, a pattern that corresponds to the order in which they were developed. More specifically, we first present an algorithm for the global class~\eqref{Eq:IntroGlobalProb}, then we present an algorithm for the connected class and, lastly, we present an algorithm that solves any instance of~\eqref{Eq:IntroProb}. All these algorithms are distributed, network-independent, and communication-efficient. In particular, we will see that they usually require less communications to converge than prior distributed algorithms.

		\item We apply our algorithms to several application problems from engineering and computer science. Some of these applications are novel, i.e., to the best of our knowledge, they have never been solved with distributed algorithms. This includes several compressed sensing problems and distributed model predictive control (D-MPC). Actually, we propose a new framework for D-MPC that considerably extends the modeling capability of the prior framework.

		\item To assess the performance of our algorithms, we provide extensive benchmarks with prior algorithms. This required an implementation of all the algorithms, including prior distributed optimization algorithms, and also algorithms that are specific to a given application.
	\end{itemize}

%The classification scheme uses two layers of information: the network topology, and the set of dependencies of each function on the components of the variable.

\section{Organization}

	The remainder of the thesis is organized as follows.

	\begin{itemize}
		\item In \cref{Ch:relatedWork}, we provide background on distributed and parallel algorithms for optimization including, for example, decomposition methods and the alternating direction method of multipliers. Although these methods are not distributed, they work as building blocks for distributed algorithms, which are presented subsequently. We also discuss prior distributed algorithms, organized according to the subclasses defined by our classification scheme.

		\item In \cref{Ch:GlobalVariable}, we present our algorithm for the global class, which is not only the most common class, but also the simplest one conceptually and notationally. This will allow us to introduce our main ideas without complicated notation. The chapter starts by stating the problem formally and discussing the general assumptions we make. Then, several applications are given, some of which are novel. Finally, the algorithm is derived and experimental results are shown.

		\item \cref{Ch:ConnectedNonConnected} has a similar structure, but now addresses our main problem~\eqref{Eq:IntroProb} in full generality. It starts with restating the problem and discussing the assumptions. Next, several potential applications are shown, including a new framework for D-MPC. Also, we present in that chapter the only algorithm we found in the literature that, after some adaptations, can also solve our problem in full generality. After that, we derive our algorithm, first assuming a general connected variable, and then moving to a non-connected one. The chapter ends with the presentation of several experimental results.

		\item Our conclusions and possible directions for future work are presented in \cref{Ch:conclusions}. There, we also restate our major contributions and discuss current limitations of our algorithms.

	\end{itemize}

% 	In \cref{Ch:relatedWork} we use the classification scheme proposed in \cref{Ch:Classification} to organize and review prior work on distributed optimization. We start with algorithms and techniques that are not distributed, but work as building blocks of distributed algorithms. Then, we review prior work on each of the classes.
%
% 	In Chapters~\ref{Ch:GlobalVariable}, \ref{Ch:ConnectedNonConnected}, and \ref{Ch:MixedVariable}, we address the global, the partial, and the mixed classes, respectively. In each of these chapters, we state the problem to be solved and discuss the assumptions we make. Then we describe in detail application problems for the class. Next, the algorithm for the respective class is derived and analyzed, and experimental results are shown. 	The structure of \cref{Ch:GlobalVariable} (global class) differs slightly from the structure of Chapters \ref{Ch:ConnectedNonConnected} and \ref{Ch:MixedVariable}, since the global class has no subclasses. The advantage of presenting the global class first is that it is simpler both conceptually and notationally. We also mention that the algorithm derived in~\cref{Ch:ConnectedNonConnected} for the partial class is the most general, since it contains the algorithms for the other classes as particular cases. For this reason, we do not provide a complete derivation of the algorithm for the mixed class in \cref{Ch:MixedVariable}. It is important, though, to address each problem class separately, as shown in \cref{Ch:Classification}.

%% file: 02-MainMatter/relatedWork.tex
\chapter{Background and Related Work}
\label{Ch:relatedWork}

	Distributed and parallel algorithms not only are relevant in the real world, but also are challenging to design. They provide several advantages over their centralized counterparts, for example, the ability to process distributed data and, notably, significant computational speed-ups. In the context of optimization, parallel algorithms, including decomposition methods, date back to the sixties with the works of Dantzig and Wolfe~\cite{Dantzig60-DecompositionPrincipleLPs}, Benders~\cite{Benders62-PartitioningProcedures}, and Everett~\cite{Everett63-GeneralizedLagrangeMultiplierMethod}. Research on distributed optimization algorithms started later, in the mid-eighties, with the work of Tsitsiklis, Bertsekas, and Athans~\cite{Tsitsiklis86-DistributedAsynchronousDeterministicAndStochasticGradientAlgorithms} and has boomed in the past ten years, motivated by the widespread of sensor networks~\cite{Rabbat04-DistributedOptimizationSensorNetworks}. Nowadays, distributed optimization finds application in sensor networks (localization~\cite{Soares12-DCOOLNET}, clustering~\cite{Forero11-DistributedClusteringUsingWSN}, etc), in cognitive radio~\cite{Bazerque10-DistributedSpectrumSensingCognitiveRadio}, in machine learning~\cite{Forero10-ConsensusBasedDistributedSVMs,Vazquez06-DistributedSVMs,Flouri08-ConsensusAlgsSVM,Boyd11-ADMM}, and in the control of complex systems such as irrigation canals~\cite{Conte12-ComputationalAspectsDistributedMPC} and the power grid~\cite{Kekatos12-DistributedRobustPowerStateEstimation,Boyd12-DynamicNetworkEnergyManagementProximalMessagePassing,Nedic13-DistributedConstrainedOptimPrimalDual,Anese13-DistributedOptimalPowerFlowForSmartMicrogrids,Necoara11-ParallelAndDistributedOptimizationMethodsEstimationAndControlInNetworks}.

	In this chapter, we use the classification scheme developed in the previous chapter to organize existing work on distributed optimization. Since most of this work builds upon parallel methods, we first overview relevant work on parallel methods, with special emphasis on the \textit{Alternating Direction Method of Multipliers} (ADMM), since it will play a key role in this thesis.

\section{Building blocks: non-distributed, parallel algorithms}

	We start by reviewing some methods that, although not distributed, work as building blocks of distributed algorithms. There are three subsections: one dedicated to decomposition methods, another dedicated to block-coordinate minimization methods, and the last one dedicated to augmented Lagrangian methods, which includes ADMM.

	\subsection{Decomposition methods}
	\label{SubsSec:DecompositionMethods}

		Decomposition methods are the precursors of distributed optimization methods. Their goal, as the name indicates, is to decompose a complex problem into smaller, simpler ones. Yet, they are not considered distributed, because they generally require a master node coordinating several slave nodes. The prototypical problem they solve is
		\begin{equation}\label{Eq:RelatedWorkDecompositionMethods}
			\begin{array}{ll}
				\underset{x_1,\ldots,x_P}{\text{minimize}} & f_1(x_1) + f_2(x_2) + \cdots + f_P(x_P) \\
				\text{subject to} & A_1 x_1 + A_2 x_2 + \cdots + A_P x_P = b\,,
			\end{array}
		\end{equation}
		where the variable is $x = (x_1,\ldots,x_P) \in \mathbb{R}^n$, with $x_p \in \mathbb{R}^{n_p}$, and $n_1 + \cdots + n_P = n$. Each function~$f_p:\mathbb{R}^{n_p} \xrightarrow{} \mathbb{R}$ is assumed convex. Problem~\eqref{Eq:RelatedWorkDecompositionMethods} is coupled through its constraint $Ax = \begin{bmatrix}A_1 & A_2 & \cdots & A_P\end{bmatrix}x = b \in \mathbb{R}^m$, which is always assumed feasible. %Problem~\eqref{Eq:RelatedWorkDecompositionMethods} is also known as a monotropic program.

		Similarly to distributed methods, decomposition methods solve~\eqref{Eq:RelatedWorkDecompositionMethods} by assigning a pair~$(f_p,A_p)$ to one device (or node) but, in contrast to distributed methods, all devices (or nodes) are controlled by a master node. Occasionally, the structure of the matrix~$A$ allows discarding the master node and the decomposition method becomes distributed.
		%In \cref{Ch:MixedVariable} we will see how to solve~\eqref{Eq:RelatedWorkDecompositionMethods} in a fully distributed manner.
		Decomposition methods are divided into primal and dual methods, and comprehensive references on the topic are~\cite[\S6.4]{Bertsekas99-NonlinearProgramming}, \cite[Ch.3]{Bertsekas97-ParallelDistributed}, and~\cite{Palomar06-TutorialDecompositionMethods}.

		\mypar{Primal decomposition}
		To solve~\eqref{Eq:RelatedWorkDecompositionMethods} through primal decomposition, we rewrite it as
		\begin{equation}\label{Eq:RelatedWorkPrimalDecomposition}
			\begin{array}[t]{ll}
				\underset{y_1,\ldots,y_P}{\text{minimize}} & \phi_1(y_1) + \phi_2(y_2) + \cdots + \phi_P(y_P)\\
				\text{subject to} & y_1 + y_2 + \cdots + y_P = b\,,
			\end{array}
		\end{equation}
		where each~$y_p \in \mathbb{R}^m$ is a new variable, and each function $\phi_p:\mathbb{R}^m \xrightarrow{} \mathbb{R} \cup \{+\infty\}$ is defined as
		\begin{equation}\label{Eq:RelatedWorkPrimalDecompositionFunct}
			\phi_p(y_p) :=
			\begin{array}[t]{cl}
				\underset{x_p}{\inf} & f_p(x_p) \\
				\text{s.t.} & y_p = A_p x_p\,.
			\end{array}
		\end{equation}
		For simplicity, we assume that each~$A_p$ has full row rank, which implies that~$\phi_p$ is defined over all~$\mathbb{R}^m$. Given a master node and~$P$ slave nodes, the master node solves the master problem~\eqref{Eq:RelatedWorkPrimalDecomposition} and delegates to slave node~$p$ the task of handling computations involving~$\phi_p$. Typically, the master problem~\eqref{Eq:RelatedWorkPrimalDecomposition} is solved with a first-order minimization method, such as the projected subgradient method. It can be shown that the subgradient of~$\phi_p$ at a point~$y_p$ is given by~$-\lambda_p$, where~$\lambda_p$ is the (optimal) dual variable associated to the constraint of~\eqref{Eq:RelatedWorkPrimalDecompositionFunct}; see sections 5.4.4 and 6.4.2 of~\cite{Bertsekas99-NonlinearProgramming} for more details. Therefore, in primal decomposition, the master node updates $y = (y_1,\ldots,y_P)$ as
		\begin{equation}\label{Eq:RelatedWorkPrimalDecompositionUpdates}
			y^{k+1} = \Bigl[y^k + \alpha_k \lambda^k\Bigr]_{\{1_n^\top y = b\}}\,,
		\end{equation}
		where $\bigl[\cdot\bigr]_{\{1_n^\top y = b\}}$ denotes the projection onto the set $\{y \in \mathbb{R}^n\,:\,1_n^\top y = b \}$, $\alpha_k$ is a positive stepsize, $1_n \in \mathbb{R}^n$ is a vector of ones, and $\lambda^k = (\lambda_1^k,\ldots,\lambda_P^k)$ is the vector of dual variables at iteration~$k$. At each iteration, the master node sends~$y_p^k$ to slave node~$p$, who then solves the problem in~\eqref{Eq:RelatedWorkPrimalDecompositionFunct} and returns $\lambda_p^k$ to the master node. The master node, in turn, updates~$y$ as in~\eqref{Eq:RelatedWorkPrimalDecompositionUpdates} and moves on to the next iteration. According to our previous definitions, primal decomposition is not distributed since it requires a master node playing the role of a central node.

		\mypar{Dual decomposition} Dual decomposition methods, rather than solving~\eqref{Eq:RelatedWorkDecompositionMethods} directly, solve its dual problem instead:
		\begin{equation}\label{Eq:RelatedWorkDualDecomposition}
				\underset{\lambda}{\text{minimize}} \,\,\, f_1^\star(A_1^\top \lambda) + f_2^\star(A_2^\top \lambda) + \cdots + f_P^\star(A_P^\top \lambda) - b^\top \lambda\,,
		\end{equation}
		where $\lambda \in \mathbb{R}^{m}$ is the dual variable and $f_p^\star:\mathbb{R}^m \xrightarrow{} \mathbb{R}$ is the convex conjugate of $f_p$, defined as
		\begin{equation}\label{Eq:RelatedWorkDualDecompositionConjugate}
			f_p^\star(\lambda) =
			\underset{x_p}{\sup}\,\,\, \lambda^\top x_p - f_p(x_p)\,.
		\end{equation}
		While~\eqref{Eq:RelatedWorkDecompositionMethods} is coupled through its constraint, \eqref{Eq:RelatedWorkDualDecomposition} is coupled through its objective (since all conjugate functions depend on~$\lambda$). As in the primal decomposition, given a master node and~$P$ slave nodes, the master node solves the master problem~\eqref{Eq:RelatedWorkDualDecomposition} and delegates to slave node~$p$ the task of handling~$f_p^\star$. Whenever each function~$f_p$ is strictly convex, there is only one minimizer $x_p(\lambda)$ of the problem in~\eqref{Eq:RelatedWorkDualDecompositionConjugate} for a given~$\lambda$. Hence, in this case, after the master node finds a dual solution~$\lambda^\star$ to~\eqref{Eq:RelatedWorkDualDecomposition}, the $p$th block of the optimal primal solution of~\eqref{Eq:RelatedWorkDecompositionMethods} can be found in the $p$th slave node as $x_p(\lambda^\star)$. In other words, when each function~$f_p$ is strictly convex, a primal solution is immediately available after solving the dual problem~\eqref{Eq:RelatedWorkDualDecomposition}.
		Again, the master problem~\eqref{Eq:RelatedWorkDualDecomposition} can be solved with first-order minimization methods, such as the subgradient method. The subgradient of~$f_p^\star \circ A_p^\top$, where~$\circ$ denotes composition, at a point~$\lambda$ is $A_px_p(\lambda)$, where~$x_p(\lambda)$ solves the problem in~\eqref{Eq:RelatedWorkDualDecompositionConjugate} \cite[Prop.B.25(b)]{Bertsekas99-NonlinearProgramming}. Hence, in dual decomposition, the master node updates $\lambda$ as
		\begin{equation}\label{Eq:RelatedWorkDualDecompositionUpdates}
			\lambda^{k+1} = \lambda^k - \alpha_k (A_1x_1(\lambda^k) + A_2x_2(\lambda^k) + \cdots + A_Px_P(\lambda^k) - b)\,,
		\end{equation}
		where~$\alpha_k > 0$ is the stepsize at iteration~$k$. At each iteration, the master node sends~$\lambda^k$ to all slave nodes and each slave node~$p$, in turn, returns $A_p x_p(\lambda^k)$ to the master node. Similarly to primal decomposition, dual decomposition also requires a central node (the master node) and, therefore, it is not distributed. Other dual decomposition methods are the Dantzig-Wolfe decomposition~\cite{Dantzig60-DecompositionPrincipleLPs},\cite[\S6.4.1]{Bertsekas99-NonlinearProgramming} and the Benders decomposition~\cite{Benders62-PartitioningProcedures}.

		Whenever each~$f_p$ is strongly convex with parameter~$\mu$, $f_p^\star$ is differentiable and its gradient $\nabla f_p^\star$ is Lipschitz continuous with constant $1/\mu$~\cite[Th. 4.2.2]{Lemarechal04-FundamentalsConvexAnalysis}, \cite{Vandenberghe11-DualDecomposition-lecs}. In that case, a faster algorithm can be applied, for example, the gradient method or even Nesterov's fast gradient method. These are explained next.

		\mypar{First-order minimization methods}
		Decomposition methods generally use first-order methods to solve the master problem, i.e., methods that use only first-order (sub)derivatives. Consider, for example,
		\begin{equation}\label{Eq:RelatedWorkFirstOrder}
			\begin{array}{ll}
				\underset{x}{\text{minimize}} & f(x) \\
				\text{subject to} & x \in X\,,
			\end{array}
		\end{equation}
		where $X \subseteq \mathbb{R}^n$ is a closed convex set and~$f:\mathbb{R}^n\xrightarrow{} \mathbb{R}$ is a convex function. If~$f$ is not differentiable, an appropriate method to find a minimizer of~\eqref{Eq:RelatedWorkFirstOrder} is the projected subgradient method:
		\begin{equation}\label{Eq:RelatedWorkSubgradientMethod}
			x^{k+1} = \Bigl[x^k - \alpha_k d^k\Bigr]_X\,,
		\end{equation}
		where $x^k$ is the estimate at iteration~$k$, $d^k$ is the subgradient of~$f$ at the point~$x^k$, i.e., $d^k \in \partial f(x^k)$,\footnote{The subdifferential~$\partial f$ of a convex function~$f$ at a point~$x$ is defined as $\partial f(x) = \{d\,:\, f(y) \geq f(x) + d^\top (y - x)\,,\, \forall_y\}$. Any point $d$ belonging to the subdifferential~$\partial f(x)$ is called subgradient of the function~$f$ at the point~$x$.} $\bigl[\cdot\bigr]_X$ is the projection operator onto the set~$X$, and~$\alpha_k > 0$ is the stepsize at iteration~$k$. The projected subgradient method is non-descent, that is, it does not guarantee that the cost function $f(x^k)$ decreases at every iteration. However, under the assumption that~$f$ is Lipschitz continuous, i.e., that there exists~$L>0$ such that $\|f(y) - f(x)\| \leq L\|y - x\|$ holds for all~$x,y$, and under an appropriate choice for the stepsize sequence $\{\alpha_k\}_{k=0}^{\infty}$, the best cost function estimate $f_{\text{best}}^{k} := \min_{0 \leq l \leq k}\, f(x^l)$ converges to the optimal value~$f^\star$ of~\eqref{Eq:RelatedWorkFirstOrder}. This is guaranteed, for example, by a square summable but not summable stepsize sequence, for instance, $\alpha_k = 1/(1+k)$. A constant stepsize sequence~$\alpha_k = \alpha$, for all~$k$, in contrast, only guarantees that~$f_{\text{best}}^{k}$ converges to a neighborhood of~$f^\star$. Even when convergence is guaranteed, the method is rather slow, since $f_{\text{best}}^{k} - f^\star$ converges to zero at rate $O(1/\sqrt{k})$. Extensive information about subgradient methods can be found in~\cite[\S8.2]{Bertsekas03-ConvexAnalysisAndOptimization},\cite[Ch.3]{Nesterov04-IntroductoryLecturesConvexOptimization}~\cite{Vandenberghe11-SubgradientMethod-lecs,Boyd07-SubgradientMethod-lecs}.

		When the function~$f$ is continuously differentiable and its gradient~$\nabla f$ is Lipschitz-continuous with constant~$L$, i.e., $\|\nabla f(y) - \nabla f(x)\| \leq L\|y - x\|$ for all $x,y \in \mathbb{R}^n$, more efficient methods can be applied. In fact, for a differentiable function, $\partial f(x) = \{\nabla f(x)\}$, and the iterations~\eqref{Eq:RelatedWorkSubgradientMethod} become the projected gradient method. In contrast with subgradient methods, gradient methods are descent and converge even with a fixed stepsize $\alpha_k= \alpha \in (0,1/L]$, for all~$k$. Moreover, $f(x^k) - f^\star$ converges to zero at rate $O(1/k)$. Gradient methods are studied extensively in~\cite[Ch.1,2]{Bertsekas99-NonlinearProgramming}\cite[Ch.1,2]{Nesterov04-IntroductoryLecturesConvexOptimization}\cite{Beck02-ConvergenceRateAnalysisGradientBasedAlgorithms-thesis,Beck09-FISTA,Beck10-ConvexOptimizatoinSignalProcessingCommunications-BookChapter}.

		Surprisingly, a small modification of the projected gradient yields a method whose error $f(x^k) - f^\star$ decreases at rate $O(1/k^2)$, as discovered by Nesterov. The problem assumptions are the same as in the projected gradient method. An instance of Nesterov's method is
		\begin{equation}\label{Eq:RelatedWorkNesterovAlg}
			\begin{array}{l}
				x^{k+1} = \Bigl[y^k - \alpha_k \,\nabla f(y^k)\Bigr]_X \vspace{0.3cm}\\
				y^{k+1} = x^{k+1} + \frac{k-1}{k+2}(x^{k+1} - x^k)\,,
			\end{array}
		\end{equation}
		which requires no significant additional computation with respect to~\eqref{Eq:RelatedWorkSubgradientMethod}. Yet, it not only has better bounds on the rate of convergence, but it also converges much faster in practice.
		For more information about accelerated first-order methods, see~\cite{Nesterov04-IntroductoryLecturesConvexOptimization,Beck09-FISTA,Beck10-ConvexOptimizatoinSignalProcessingCommunications-BookChapter,Tseng08-AcceleratedProximalGradientMethodsConvexConcaveOptimization,Nesterov10-FirstOrderMethodsConvexOptimizationInexactOracle,Nesterov05-SmoothMinimizationOfNonSmoothFunctions,Aspremont08-SmoothOptimizationWithApproximateGradient,Vandenberghe11-FastProximalGradientMethods-lecs}.

	\subsection{Block-coordinate minimization methods}

	Block-coordinate methods are appropriate when fixing some of the variables in an optimization problem makes the problem easier to solve. Consider, for example,
	\begin{equation}\label{Eq:RelatedWorkBlockCoordinate}
		\begin{array}{cl}
			\underset{x = (x_1,\ldots,x_P)}{\text{minimize}} & f(x_1,x_2,\ldots,x_P) \\
			\text{subject to} & x \in X_1 \times X_2 \times \cdots \times X_P\,,
		\end{array}
	\end{equation}
	where the variable is $x = (x_1,\ldots,x_P) \in \mathbb{R}^{n}$ with~$x_p \in \mathbb{R}^{n_p}$ and $n_1 + \cdots + n_P = n$. The function $f:\mathbb{R}^n \xrightarrow{} \mathbb{R}$ is assumed convex, and each set~$X_p \subseteq \mathbb{R}^{n_p}$ is assumed closed and convex. Block-coordinate minimization methods solve~\eqref{Eq:RelatedWorkBlockCoordinate} via a sequence of minimization problems with respect to one block variable while the other blocks are fixed, i.e.,
	$$
		\begin{array}{ll}
			\underset{\xi}{\text{minimize}} & f(x_1,\ldots,x_{p-1},\xi,x_{p+1},\ldots,x_P) \\
			\text{subject to} & \xi \in X_p\,.
		\end{array}
	$$
	Two important types of block-coordinate methods are nonlinear Jacobi and nonlinear Gauss-Seidel.

	\mypar{Nonlinear Jacobi}
	The nonlinear Jacobi method is defined as
	\begin{equation}\label{Eq:RelatedWorkNonlinearJacobi}
		x_p^{k+1} =
		\begin{array}[t]{cl}
			\underset{x_p}{\arg\min} & f(x_1^k, \ldots, x_{p-1}^k, x_p, x_{p+1}^k, \ldots, x_P^k)\,, \\
			\text{s.t.} & x_p \in X_p
		\end{array}
		\qquad
		p = 1,\ldots,P\,,
	\end{equation}
	where the $p$th minimization is taken with respect to~$x_p$. Since updating~$x_p^k$ to~$x_p^{k+1}$ requires all the other block components to be fixed at~$x_j^k$, for $j \neq p$, which were found in the previous iteration, the updates can be carried out in parallel. Convergence of the nonlinear Jacobi method to a minimizer of~\eqref{Eq:RelatedWorkBlockCoordinate} is guaranteed whenever~$f$ is differentiable and the mapping $x - \gamma \nabla f(x)$ is a contraction for any~$\gamma > 0$~\cite[Prop.3.10]{Bertsekas97-ParallelDistributed}. Another Jacobi-type method requiring milder assumptions is the diagonal quadratic approximation~\cite{Zakarian95-NonlinearJacobi-thesis,Mota08-DistributedAlgorithmsSparseApproximation-thesis}.

	\mypar{Nonlinear Gauss-Seidel}
	The nonlinear Gauss-Seidel method is defined as
	\begin{equation}\label{Eq:RelatedWorkBlockCoordinateNGS}
		x_p^{k+1} =
		\begin{array}[t]{cl}
			\underset{x_p}{\arg\min} & f(x_1^{k+1}, \ldots, x_{p-1}^{k+1}, x_p, x_{p+1}^k, \ldots, x_P^k)\,, \\
			\text{s.t.} & x_p \in X_p
		\end{array}
		\qquad
		p = 1,\ldots,P\,.
	\end{equation}
	In contrast with Jacobi methods, updating~$x_p$ at iteration~$k$ requires knowing the current estimates of the first~$p-1$ blocks, i.e., $x_j^{k+1}$ for~$j < p$. Hence, all updates have to be carried out sequentially. The order of the sequence, however, can change from iteration to iteration, and the convergence to a minimizer of~\eqref{Eq:RelatedWorkBlockCoordinate} is guaranteed whenever each problem in~\eqref{Eq:RelatedWorkBlockCoordinateNGS} has a unique solution and a regularity condition is satisfied~\cite{Tseng01-ConvergenceBlockCoordinateDescent}. For example, differentiability and strict convexity of~$f$ implies that regularity condition is satisfied (see the errata of proposition 2.7.1 of~\cite{Bertsekas99-NonlinearProgramming}, available at~\url{http://www.athenasc.com/nlperrata.pdf}).

	\subsection{Augmented Lagrangian methods}
	\label{SubSec:AugmentedLagrangianMethods}

	Augmented Lagrangian methods are important tools for distributed optimization, even though they were not designed for that purpose. They date back to penalty methods, where a constrained problem is solved via a sequence of unconstrained problems. Although relying on duality, augmented Lagrangian methods are guaranteed to find a primal solution even when the cost function is not strictly convex. This gives them a clear advantage over ``simple'' duality-based methods, such as dual decomposition. On the other hand, they do not distribute as easily as ``simple'' duality-based methods, because of the augmented term in the augmented Lagrangian.

	\mypar{Method of multipliers}
	Discovered independently by Hestenes~\cite{Hestenes69-MethodOfMultipliers} and by Powell~\cite{Powell69-MethodNonlinearConstraintsMinimizationProblems-BookChapter}, the method of multipliers solves the constrained problem
	\begin{equation}\label{Eq:RelatedWorkAugmentedLagrangianMoM}
		\begin{array}{ll}
			\underset{x}{\text{minimize}} & f(x) \\
			\text{subject to} & Ax = b\,,
		\end{array}
	\end{equation}
	where the function $f:\mathbb{R}^n \xrightarrow{} \mathbb{R} \cup \{+\infty\}$ is closed and convex, $b \in \mathbb{R}^m$, and the linear system $Ax = b$ is feasible. Using~$\lambda \in \mathbb{R}^m$ to denote the dual variable, the augmented Lagrangian of~\eqref{Eq:RelatedWorkAugmentedLagrangianMoM} is
	$$
		L_\rho(x;\lambda) = f(x) + \lambda^\top (Ax - b) + \frac{\rho}{2}\|Ax - b\|^2\,,
	$$
	where $\rho > 0$ is the augmented Lagrangian parameter. Note that the augmented Lagrangian differs from the ordinary Lagrangian in the augmented term $(\rho/2)\|Ax - b\|^2$. The method of multipliers solves~\eqref{Eq:RelatedWorkAugmentedLagrangianMoM} by minimizing the augmented Lagrangian with respect to~$x$, keeping the dual variable~$\lambda$ fixed at~$\lambda^k$, and then by updating~$\lambda$ in a gradient-based way. That is, it iterates
	\begin{align}
		x^{k+1} &= \underset{x}{\arg\min}\,\,\, L_\rho(x;\lambda^k) \label{Eq:RelatedWorkMoM1}\\
		\lambda^{k+1} &= \lambda^k + \rho (Ax^{k+1} - b)\,.
		\label{Eq:RelatedWorkMoM2}
	\end{align}
	Note that~\eqref{Eq:RelatedWorkMoM2} is indeed a gradient iteration: the dual function $L_\rho(\lambda) := \inf_x\, L_\rho(x;\lambda)$ is differentiable and its gradient is given by $Ax(\lambda) - b$, where~$x(\lambda)$ minimizes~$L_\rho(\cdot;\lambda)$.\footnote{This is true even when $A$ does not have full column rank. To see that, write $L_\rho(\lambda)$ as $L_\rho(\lambda) = \inf_z\, \Psi(z) + \lambda^\top z + \frac{\rho}{2}\|z\|^2$, where~$\Psi(z) := \inf_x\{f(x)\,:\, z = Ax - b\}$ is a convex function~\cite[\S3.2.5]{Boyd04-ConvexOptimization}. The quadratic term $\|z\|^2$ makes the objective strictly convex and, therefore, the problem defining~$L_\rho(\lambda)$ in terms of~$z$ has a unique minimizer $z(\lambda)$ for each~$\lambda$. It follows that the subdifferential of~$L_\rho$ is the singleton $\{z(\lambda)\}$. For each~$\lambda$, there can be several~$x(\lambda)$'s solving the problem defining~$\Psi(z(\lambda))$.} Furthermore, it can be shown that the gradient $Ax(\lambda) - b$ is Lipschitz continuous with constant~$1/\rho$ \cite[Th. 4.2.2]{Lemarechal04-FundamentalsConvexAnalysis}, \cite{Vandenberghe11-DualDecomposition-lecs}.	Rockafellar~\cite{Rockafellar76-AugmentedLagrangiansAndApplicationsOfProximalPointAlgorithm} showed that the iterations~\eqref{Eq:RelatedWorkMoM1}-\eqref{Eq:RelatedWorkMoM2} are actually an application of the proximal minimization algorithm to the dual problem of~\eqref{Eq:RelatedWorkAugmentedLagrangianMoM} (see also~\cite[\S3.4.4]{Bertsekas97-ParallelDistributed} and~\cite[Ch.3]{Eckstein89-SplittingMethodsMonotoneOperatorsParallelOptimization-thesis}). Therefore, the conditions under which the method of multipliers converges are very mild; see~\cite{Bertsekas76-MethodMultipliersSurvey,Bertsekas96-ConstrainedOptimizationLagrangeMultiplierMethods,Rockafellar73-MultiplierMethodHestenesPowellAppliedConvexProgramming}~\cite[\S3.4.4]{Bertsekas97-ParallelDistributed}\cite[\S4.2]{Bertsekas99-NonlinearProgramming} for a detailed analysis and for related methods. Nevertheless, the optimization problem in~\eqref{Eq:RelatedWorkMoM1} is usually nonseparable, because of the augmented term $(\rho/2)\|Ax - b\|^2$. This makes the method of multipliers difficult to apply in distributed optimization. We next present an alternative that, while preserving the good convergence properties of the method of multipliers, it suits distributed optimization better.

	\mypar{Alternating Direction Method of Multipliers}
		The Alternating Direction Method of Multipliers (ADMM) is an augmented Lagrangian method introduced in the mid-seventies by Glowinski and Marrocco~\cite{Glowinski75-ADMMFirst} and by Gabay and Mercier~\cite{Gabay76-ADMMFirst}. It solves
		\begin{equation}\label{Eq:RelatedWorkADMMProb}
			\begin{array}{ll}
				\underset{x_1,x_2}{\text{minimize}} & f_1(x_1) + f_2(x_2) \\
				\text{subject to} & A_1 x_1 + A_2 x_2 = b\,,
			\end{array}
		\end{equation}
		where $f:\mathbb{R}^{n_1} \xrightarrow{} \mathbb{R} \cup \{+\infty\}$ and $g:\mathbb{R}^{n_2} \xrightarrow{} \mathbb{R} \cup \{+\infty\}$ are closed convex functions, and $A_1 \in \mathbb{R}^{m \times n_1}$ and $A_2 \in \mathbb{R}^{m \times n_2}$ are full column rank matrices. The augmented Lagrangian of~\eqref{Eq:RelatedWorkADMMProb} is
		$$
			L_\rho(x_1,x_2;\lambda) = f_1(x_1) + f_2(x_2) + \lambda^\top (A_1x_1 + A_2x_2 - b) + \frac{\rho}{2}\|A_1x_1 + A_2x_2 - b\|^2\,,
		$$
		where the parameter~$\rho > 0$ is assumed fixed. ADMM minimizes $L_\rho$ first with respect to~$x_1$, then with respect to~$x_2$, and it finally updates the dual variable~$\lambda$ as in the method of multipliers:
		\begin{align}
			x_1^{k+1} &= \underset{x_1}{\arg\min}\,\,\, L_\rho(x_1,x_2^k;\lambda^k)
			\label{Eq:RelatedWorkADMMIter1}
			\\
			x_2^{k+1} &= \underset{x_2}{\arg\min}\,\,\, L_\rho(x_1^{k+1},x_2;\lambda^k)
			\label{Eq:RelatedWorkADMMIter2}
			\\
			\lambda^{k+1} &= \lambda^k + \rho (A_1x_1^{k+1} + A_2x_2^{k+1} - b)\,.
			\label{Eq:RelatedWorkADMMIter3}
		\end{align}
		ADMM can be seen as the application of the method of multipliers to problem~\eqref{Eq:RelatedWorkADMMProb}, where the minimization with respect to the primal variable~$(x_1,x_2)$ consists of just one Gauss-Seidel pass. Surprisingly, it solves~\eqref{Eq:RelatedWorkADMMProb} with the same accuracy level as the method of multipliers does, by using a few more iterations; see~\cite{Eckstein12-AugmentedLagrangianADMM} for a detailed comparison between ADMM and the method of multipliers. Curiously, both methods are instances of the proximal point algorithm~\cite{Martinet70-RegularisationInequations,Martinet72-DeterminationPointFixe,Rockafellar76-MonotoneOperatorsProximalPointAlgorithms}: while the method of multipliers results from applying iteratively the resolvent operator to the subdifferential of the dual function of~\eqref{Eq:RelatedWorkAugmentedLagrangianMoM}~\cite{Rockafellar76-AugmentedLagrangiansAndApplicationsOfProximalPointAlgorithm}, ADMM results from applying iteratively the Douglas-Rachford operator~\cite{Douglas56-NumericalSolutionHeatEquation,Lions79-SplittingAlgorithms-DouglasRachford} to the subdifferential of the dual function of~\eqref{Eq:RelatedWorkADMMProb}, as discovered by Gabay~\cite{Gabay83-ApplicationsMethodMultipliersVariationalInequalities-BookChapter}. An excellent account on this topic, including an introduction to monotone operator theory, is given by Eckstein~\cite[Ch.3]{Eckstein89-SplittingMethodsMonotoneOperatorsParallelOptimization-thesis} (see also~\cite{Eckstein92-DouglasRachford}). Alternative proofs for the convergence of ADMM that do not use any monotone operator theory include~\cite[\S3.4.4]{Bertsekas97-ParallelDistributed} and~\cite{Boyd11-ADMM,Mota11-ADMMProof}. Roughly, ADMM converges whenever $f$ and~$g$ are closed and convex, \eqref{Eq:RelatedWorkADMMProb} is solvable, and strong duality holds. When~$A_1$ and~$A_2$ do not have full column rank, the sequence $(x_1^k,x_2^k)$ might not converge, even though $f_1(x_1^k) + f_2(x_2^k)$ and $\lambda^k$ converge~\cite[p.260]{Bertsekas97-ParallelDistributed}. Regarding the augmented Lagrangian parameter~$\rho$, the proofs of the convergence hold for any positive, fixed~$\rho$. Since, in practice, the value of~$\rho$ significantly affects the performance of the algorithm, it is common to use heuristics to adapt $\rho$ along the iterations~\cite{Boyd12-DynamicNetworkEnergyManagementProximalMessagePassing,Boyd11-ADMM}. These heuristics, however, cannot be easily implemented in distributed environments, because they require information from all the nodes at each iteration.

		Until recently, the known proofs for the convergence of ADMM did not allow to derive a convergence rate. It was known, however, that ADMM converged linearly for linear programs~\cite{Eckstein90-ADMMForLP}. More recently, a series of works has derived bounds for the convergence rate of ADMM, many times, under assumptions stronger than the ones required to prove plain convergence. For example, \cite{He12-OnTheConvergenceRateOfTheDouglasRachfordADM} proved that the primal and the dual variables converge in an ergodic sense at rate of~$O(1/k)$. The same rate was established in~\cite{He12-OnNonErgodicADMMConvRate} in a non-ergodic sense. The work~\cite{Goldstein12-FastAlternatingDirectionOptimizationMethods} proved that the cost function of the dual problem converges to the optimal value at rate $O(1/k)$ and, as a consequence, the square of the primal and dual residuals~\cite{Boyd11-ADMM} converge to zero at the same rate. It is assumed, however, that at least one of the functions~$f_1$ or~$f_2$ is strongly convex. Inspired by Nesterov's gradient method, \cite{Goldstein12-FastAlternatingDirectionOptimizationMethods} also proposes a modification to ADMM whose dual cost function converges at rate~$O(1/k^2)$. Note that both $O(1/k)$ and $O(1/k^2)$ are sublinear rates.\footnote{We say that a sequence $\{x^k\}$ converges linearly (more appropriately, R-linearly) to $x^\star$ if there exists $M>0$ and~$c>1$ such that $\|x^k - x^\star\| \leq \frac{M}{c^k}$, for a sufficiently large~$k$.} When ADMM is applied to the average consensus problem (after a suitable reformulation to make it distributed, as we will see later), linear convergence can be proved~\cite{Erseghe11-FastConsensusByADMM}. For general quadratic problems, \cite{Boley12-LinearConvergenceOfADMMOnAModelProblem} conjectured that linear convergence also holds, which was later proved in~\cite{Ghadimi13-OptimalParameterSelectionForADMM}. More recently, Deng and Yin~\cite{Deng12-GlobalLinearConvergenceGeneralizedADMM} showed that a generalized version of ADMM converges linearly in terms of the primal and the dual estimates when at least one of the functions~$f_1$ or~$f_2$ is strongly convex, differentiable, and has a Lipschitz continuous gradient. Work that establishes convergence rates for modified versions of ADMM includes~\cite{Tseng91-ApplicationsSplittingAlgorithm,Goldfarb09-FastMultipleSplittingAlgs,Goldfarb10-FastAlternatingLinearizationMethodsSumTwoConvexFunctions,Wei13-OnTheConvergenceAsynchronousDistributedADMM}.

		The recent stream of theoretical work on ADMM in recent years has been motivated by its application in many areas. For example, ADMM has been applied to image processing~\cite{Afonso11-AugmentedLagrangianConstrainedOptimizationImagingInverseProblems-CSALSA}, to localization~\cite{Soares12-DCOOLNET}, and to several statistical and machine learning problems~\cite{Martins11-AnAugmentedLagrangianApproachMAP,Boyd11-ADMM}. Reference~\cite{Boyd11-ADMM}, in particular, provides a survey on ADMM from an optimization perspective and describes many applications in statistics and machine learning.

		\mypar{Multi-block ADMM}
		The multi-block ADMM is a natural generalization of ADMM when, instead of the variable being partitioned into two blocks, $x_1$ and~$x_2$, as in~\eqref{Eq:RelatedWorkADMMProb}, it is partitioned into a finite number~$C$. Sometimes this method is also known as generalized ADMM or extended ADMM. Since there are other methods named generalized ADMM, we will refer to it as multi-block ADMM or as extended ADMM. More specifically, the multi-block ADMM solves
		\begin{equation}\label{Eq:RelatedWorkADMMProbExtended}
			\begin{array}{ll}
				\underset{x_1,\ldots,x_C}{\text{minimize}} & f_1(x_1) + f_2(x_2) + \cdots + f_C(x_C) \\
				\text{subject to} & A_1x_1 + A_2 x_2 + \cdots + A_C x_C = b\,,
			\end{array}
		\end{equation}
		by iterating
		\begin{align}
			x_1^{k+1} &= \underset{x_1}{\arg\min}\,\,\, L_\rho(x_1,x_2^k,\ldots,x_C^k;\lambda^k)
			\label{Eq:RelatedWorkADMMProbExtendedADMMIter1}
			\\
			x_2^{k+1} &= \underset{x_2}{\arg\min}\,\,\, L_\rho(x_1^{k+1},x_2, x_3^k,\ldots,x_C^k;\lambda^k)
			\label{Eq:RelatedWorkADMMProbExtendedADMMIter2}
			\\
			&\phantom{1}\vdots \\
			x_C^{k+1} &= \underset{x_C}{\arg\min}\,\,\, L_\rho(x_1^{k+1},x_2^{k+1},\ldots,x_{C-1}^{k+1},x_C;\lambda^k)
			\label{Eq:RelatedWorkADMMProbExtendedADMMIter3}
			\\
			\lambda^{k+1} &= \lambda^k + \rho \,\sum_{c=1}^C A_c x_c^{k+1} \,.
			\label{Eq:RelatedWorkADMMProbExtendedADMMIter4}
		\end{align}
		Note that~\eqref{Eq:RelatedWorkADMMProbExtended} is the same problem as~\eqref{Eq:RelatedWorkDecompositionMethods}, the problem solved by decomposition methods.	In this case, the augmented Lagrangian is
		$$
			L_\rho(x_1,x_2,\ldots,x_C;\lambda) = \sum_{c=1}^C f_c(x_c) + \lambda^\top \Bigl(\sum_{c=1}^C A_c x_c - b\Bigr) + \frac{\rho}{2}\Bigl\|\sum_{c=1}^C A_cx_c - b\Bigr\|^2\,.
		$$
		It is assumed that each function~$f_c:\mathbb{R}^{n_c}\xrightarrow{}\mathbb{R}\cup \{+\infty\}$ is closed and convex, and that each matrix~$A_c \in \mathbb{R}^{m\times n_c}$ has full column rank. When~$C=2$, the multi-block ADMM~\eqref{Eq:RelatedWorkADMMProbExtendedADMMIter1}-\eqref{Eq:RelatedWorkADMMProbExtendedADMMIter4} becomes the $2$-block ADMM~\eqref{Eq:RelatedWorkADMMIter1}-\eqref{Eq:RelatedWorkADMMIter3}. The only known proof of convergence of the multi-block ADMM is due to Han and Yuan~\cite{Han12-NoteOnADMM} and it assumes that all functions~$f_1$, \ldots, $f_C$ are strongly convex. The following theorem summarizes the known convergence results for the multi-block ADMM, including its particular version, the $2$-block ADMM.
		\begin{theorem}[\cite{Mota11-ADMMProof,Han12-NoteOnADMM}]\label{Teo:RelatedWorkConvergenceADMM}
			\hfill

			\medskip
			\noindent
			Let~$f_c:\mathbb{R}^{n_c}\xrightarrow{} \mathbb{R} \cup \{+\infty\}$ be a closed convex function over~$\mathbb{R}^{n_c}$, not identically $+\infty$, and let~$A_c$ be an $m\times n_c$ matrix, for~$c=1,\ldots,C$. Assume that~\eqref{Eq:RelatedWorkADMMProbExtended} is solvable and that either
			\newcounter{TeoConvADMM}
			\begin{list}{(\alph{TeoConvADMM})}{\usecounter{TeoConvADMM}}
				\item $C=2$ and each~$A_c$ has full column-rank, or
				\item $C\geq2$, each~$f_c$ is strongly convex with modulus $\mu_c$ and
				\begin{equation}\label{Eq:RelatedWorkExtendedADMMConditionRho}
					0 < \rho < \underset{c=1,\ldots,C}{\min} \,\,\,\frac{2\mu_c}{3(C-1)\,\sigma_{\max}^2(A_c)}\,,
				\end{equation}
				where~$\sigma_{\max}(\cdot)$ denotes the largest singular value of a matrix.
			\end{list}
			Then, the sequence~$\{(x_1^k,\ldots,x_C^k,\lambda^k)\}$ generated by~\eqref{Eq:RelatedWorkADMMProbExtendedADMMIter1}-\eqref{Eq:RelatedWorkADMMProbExtendedADMMIter4} converges to $(x_1^\star, \ldots, x_C^\star, \lambda^\star)$, where $(x_1^\star, \ldots, x_C^\star)$ solves~\eqref{Eq:RelatedWorkADMMProbExtended} and~$\lambda^\star$ solves the dual problem of~\eqref{Eq:RelatedWorkADMMProbExtended}:
			$
				\min_\lambda b^\top \lambda + \sum_{c=1}^C f_c^\star(-A_c^\top \lambda)
			$,
			where~$f_c^\star$ is the convex conjugate of~$f_c$, $c=1,\ldots,C$.
		\end{theorem}
		A proof for case \text{(a)} can be found in~\cite{Mota11-ADMMProof}, which generalizes the proofs of~\cite{Bertsekas97-ParallelDistributed,Boyd11-ADMM}. A proof for case \text{(b)} can be found in~\cite{Han12-NoteOnADMM}. It is believed that the multi-block ADMM~\eqref{Eq:RelatedWorkADMMProbExtendedADMMIter1}-\eqref{Eq:RelatedWorkADMMProbExtendedADMMIter4} still converges for any finite~$C>2$ whenever each function $f_c$ is closed and convex and each matrix~$A_c$ has full column rank, i.e., that the generalization of Theorem~\ref{Teo:RelatedWorkConvergenceADMM} under case~\text{(a)} still holds. This belief is based on empirical evidence~\cite{He12-AlternatingDirectionMethodGaussianBackSubstitution}, but its proof remains still an open problem. So far, there are only proofs of convergence for similar algorithms that are either slower~\cite{Luo13-LinearConvergenceADMM} or that cannot be implemented (at least, straightforwardly) in distributed scenarios~\cite{He12-AlternatingDirectionMethodGaussianBackSubstitution}. In fact, \cite{Luo13-LinearConvergenceADMM} proves that a modification of the iterates~\eqref{Eq:RelatedWorkADMMProbExtendedADMMIter1}-\eqref{Eq:RelatedWorkADMMProbExtendedADMMIter4} converges linearly when each function~$f_c$ is strictly convex, differentiable, and has a Lipschitz continuous gradient. That modification consists of changing the stepsize~$\rho$ in~\eqref{Eq:RelatedWorkADMMProbExtendedADMMIter4} to a smaller number. When that number is sufficiently small, linear convergence can be proved. However, in practice, reducing the stepsize makes the algorithm slower. We note that the distributed algorithms proposed in this thesis are based on the multi-block ADMM, and that we started using them~\cite{Mota11-DistributedBP-ArxivV2} even before there was a proof of convergence~\cite{Han12-NoteOnADMM}.

\section{Distributed algorithms}

	To the best of our knowledge, the problem we aim to solve, \eqref{Eq:IntroProb}, has been considered before only with the following types of variable: global, star-shaped, and mixed (where all non-global components are star-shaped); see \fref{Fig:IntroClassification} from \cref{Ch:Introduction} for a visualization of the relation between these types of variables. We next review distributed algorithms that were designed for these types of variables, or for applications that can be written as~\eqref{Eq:IntroProb} with such variables.

	We mention that~\cite[\S7.2]{Boyd11-ADMM} proposes an algorithm based on the $2$-block ADMM for solving~\eqref{Eq:IntroProb} with a generic variable. However, it either requires a platform supporting all-to-all communications (equivalently, a central node), or running, at each iteration, a consensus algorithm on each induced subgraph~\cite[\S10.1]{Boyd11-ADMM}. This makes that algorithm not distributed in our sense. Actually, that algorithm becomes distributed only when the variable is star-shaped. We also mention that we found only one distributed algorithm in the literature that can solve~\eqref{Eq:IntroProb} when the variable is non-global and non-star-shaped, but still connected. That algorithm, also based on the $2$-block ADMM, was proposed in~\cite{Kekatos12-DistributedRobustPowerStateEstimation} for state estimation of power systems, a problem formulated as~\eqref{Eq:IntroProb} with a star-shaped variable. In \cref{Ch:ConnectedNonConnected}, we generalize that algorithm for a generic connected variable, and then for a non-connected variable. This means that the algorithm in~\cite{Kekatos12-DistributedRobustPowerStateEstimation} can also solve~\eqref{Eq:IntroProb} in full generality, after proper modifications. Our experimental results, however, show that it always requires more communications to converge to a solution of~\eqref{Eq:IntroProb} than the algorithm we propose.

	\subsection{Global class}

		Among all the classes, the global problem class~\eqref{Eq:IntroGlobalProb} is the most well studied. For convenience, we recall that~\eqref{Eq:IntroGlobalProb} is written as
		\begin{equation}\label{Eq:GlobalProblem}\tag{G}
			\begin{array}{ll}
				\underset{x \in \mathbb{R}^n}{\text{minimize}} & f_1(x) + f_2(x) + \cdots + f_P(x)\,,
			\end{array}
		\end{equation}
		where all functions depend on all the components of the variable~$x$. Although several applications can be posed naturally as~\eqref{Eq:GlobalProblem}, the application that triggered the interest on the design of distributed algorithms for~\eqref{Eq:GlobalProblem} was the average consensus~\cite{DeGroot74-ReachingConsensus}. Indeed, this was the motivating application in~\cite{Rabbat04-DistributedOptimizationSensorNetworks}, which designed probably the first distributed algorithm for the class~\eqref{Eq:GlobalProblem}, an incremental subgradient algorithm. Other important pioneer work includes gradient- and subgradient-based algorithms by Nedi\'c, Ozdaglar, and collaborators~\cite{Nedic07-OnTheRateConvergenceDistributedSubgradientMultiAgent-CDC,Nedic09-DistributedSubgradientMethodsMultiAgentOptimization,Lobel11-DistributedMultiAgentOptimization,Nedic10-ConvexOptimizatoinSignalProcessingCommunications-BookChapter,Ram09-AsynchronousGossipAlgorithmsForStochasticOptimization}, whose work was inspired by~\cite{Tsitsiklis86-DistributedAsynchronousDeterministicAndStochasticGradientAlgorithms,Tsitsiklis84-ProblemsInDecentralizedDecisionMaking-thesis}. At the same time, the first distributed, ADMM-based algorithm was proposed by Schizas, Ribeiro, and Giannakis~\cite{Schizas08-ConsensusAdHocWSNsPartI}.

		We next review four categories of algorithms for~\eqref{Eq:GlobalProblem}: incremental, (sub)gradient-based, double-looped, and ADMM-based. Special emphasis will be given to the latter, since the algorithms we propose are also based on ADMM.

		\mypar{Incremental methods} An incremental (sub)gradient method solves problems with the format~\eqref{Eq:GlobalProblem} with the same scheme as the (sub)gradient method~\eqref{Eq:RelatedWorkSubgradientMethod}. However, instead of using the (sub)gradient of the entire objective~$f_1 + \cdots + f_P$, it only uses the (sub)gradient of one function~$f_p$ at a time. More concretely, it consists of
		\begin{equation}\label{Eq:RelatedWorkIncrementalGradient}
			x^{k+1} = \Bigl[ x^k - \alpha_k \nabla \tilde{f}_{i_k}(x^k)\Bigr]_{X_p}\,,
		\end{equation}
		where we decomposed each $f_p = \tilde{f}_p + \text{i}_{X_p}$ into its real-valued part~$\tilde{f}_p$ and into its infinity-valued (or constraint-enforcing) part $\text{i}_{X_p}$. To simplify notation, we assumed in~\eqref{Eq:RelatedWorkIncrementalGradient} that each function~$\tilde{f}_p$ is differentiable; if not, just replace $\nabla \tilde{f}_p(x^k)$ by any subgradient of~$\tilde{f}_p$ at the point~$x^k$. The sequence~$\{i_k\}$ takes values in $\{1,2,\ldots,P\}$ and determines the order of the updates, which can be deterministic or randomized. Surveys about incremental methods, including convergence analysis, can be found in~\cite[\S1.5.2,\S6.3.2]{Bertsekas99-NonlinearProgramming} and~\cite{Bertsekas10-IncrementalGradientSubgradientProximal}. Roughly, incremental (sub)gradient methods progress faster than their non-incremental counterparts far from the solution, but are slower near the solution~\cite{Bertsekas10-IncrementalGradientSubgradientProximal}. Since they use the (sub)gradient of only one function at each iteration, they can be implemented naturally in a distributed scenario, with a single node performing the update~\eqref{Eq:RelatedWorkIncrementalGradient} at each time instant, in a round-robin fashion. This was done in~\cite{Rabbat04-DistributedOptimizationSensorNetworks,Rabbat05-QuantizedIncrementalAlgorithmsDistributedOptimization} for a deterministic sequence~$\{i_k\}$ and in~\cite{Johansson09-RandomizedIncrementalSubgradient} for a randomized one. The work~\cite{Bertsekas10-IncrementalGradientSubgradientProximal} surveys these methods and, in addition, presents an unified view of incremental (sub)gradient methods, incremental proximal methods, and their combination. In general, incremental methods have slow convergence rates; and, in distributed optimization, they have the disadvantage of making just a single node active at each time instant. The algorithms we propose here, besides exhibiting faster convergence rates, have a higher degree of parallelism, even though not all nodes are active at the same time, i.e., they are not fully parallel.

		\mypar{(Sub)gradient-based}
		If we apply the (sub)gradient algorithm~\eqref{Eq:RelatedWorkSubgradientMethod} directly to problem~\eqref{Eq:GlobalProblem}, the resulting algorithm is non-distributed, since updating~$x$ at iteration~$k$ requires the (sub)gradients of all the functions at the point~$x^k$. Therefore, using (sub)gradient algorithms to solve~\eqref{Eq:GlobalProblem} in a distributed way requires either reformulating~\eqref{Eq:GlobalProblem} into another equivalent problem, or changing the (sub)gradient algorithm.

		The first option was taken in~\cite{Rabbat05-GeneralizedConsensusAlgs}, where~\eqref{Eq:GlobalProblem} was rewritten as
		\begin{equation}\label{Eq:RelatedWorkRabbat05DualDecomposition}
			\begin{array}{ll}
				\underset{x_1,\ldots,x_P}{\text{minimize}} & f_1(x_1) + f_2(x_2) + \cdots + f_P(x_P) \\
				\text{subject to} & x_p - x_j \geq 0\,,\quad j \in \mathcal{N}_p\,,\quad p = 1,\ldots,P\,,
			\end{array}
		\end{equation}
		where each variable $x_p \in \mathbb{R}^n$, held at node~$p$, is a clone of the original variable~$x \in \mathbb{R}^n$. Recall that~$\mathcal{N}_p$ denotes the set of neighbors of node~$p$. This reformulation increases the size of the optimization variable in~\eqref{Eq:IntroGlobalProb} from~$n$ to~$Pn$ and adds~$2E$ constraints (two constraints per each edge $(i,j) \in \mathcal{E}$: $x_i - x_j \geq 0$ and $x_j - x_i \geq 0$, which implies $x_i = x_j$ and, thus, the equivalence between~\eqref{Eq:GlobalProblem} and~\eqref{Eq:RelatedWorkRabbat05DualDecomposition}). Note that problem~\eqref{Eq:RelatedWorkRabbat05DualDecomposition} has the same format as~\eqref{Eq:RelatedWorkDecompositionMethods}, the problem that decomposition methods solve. Indeed, \cite{Rabbat05-GeneralizedConsensusAlgs} then applies the dual decomposition method described in \ssref{SubsSec:DecompositionMethods} (the generalization of the dual decomposition from equality-constrained problems to inequality-constrained ones is straightforward). If we rewrite the constraints of~\eqref{Eq:RelatedWorkRabbat05DualDecomposition}  in matrix form, the matrices corresponding to each~$A_p$ in~\eqref{Eq:RelatedWorkDecompositionMethods} have a special format: the nonzero entries correspond either to the variable of node~$p$ or to the variables of its neighbors. This is makes dual decomposition yield a distributed algorithm. Since in~\cite{Rabbat05-GeneralizedConsensusAlgs} each function is assumed strictly convex, the dual function is differentiable and the gradient algorithm can be applied to solve the dual problem.

		The second option (of changing the (sub)gradient algorithm) was taken in a series of works, including~\cite{Nedic07-OnTheRateConvergenceDistributedSubgradientMultiAgent-CDC,Nedic09-DistributedSubgradientMethodsMultiAgentOptimization,Lobel11-DistributedMultiAgentOptimization,Nedic10-ConvexOptimizatoinSignalProcessingCommunications-BookChapter}. These works study the convergence of a (sub)gradient algorithm coupled with a consensus scheme:
		\begin{equation}\label{Eq:RelatedWorkSubgradientConsensus}
			x_p^{k+1} = \biggl[\sum_{j \in \mathcal{N}_p \cup \{p\}}a_{pj}^kx_j^k - \alpha_k \nabla f_p(x^k)\biggr]_{X_p}\,,
		\end{equation}
		where we used the same simplifications as in~\eqref{Eq:RelatedWorkIncrementalGradient}. In~\eqref{Eq:RelatedWorkSubgradientConsensus}, $a_{pj}^k > 0$ models the influence node~$j$ exerts on node~$p$ at iteration~$k$. So, at each iteration~$k$, node~$p$ receives the estimates $x_j^k$ from its neighbors $\mathcal{N}_p$, averages them with its own estimate~$x_p^k$, and then performs a projected (sub)gradient step, where the (sub)gradient that is used is the one given by its private function~$f_p$. It is generally assumed that $\sum_j a_{ij}^k = 1$. When all functions $f_p$ are zero and all the sets $X_p$ are the full space $\mathbb{R}^n$, \eqref{Eq:RelatedWorkSubgradientConsensus} becomes the familiar consensus scheme~\eqref{Eq:IntroAverageConsensus}; and when the network is reduced to a single node, \eqref{Eq:RelatedWorkSubgradientConsensus} becomes the familiar (sub)gradient algorithm~\eqref{Eq:RelatedWorkSubgradientMethod}. The linearity of the algorithm~\eqref{Eq:RelatedWorkSubgradientConsensus} and the nonexpansiveness property of the projection operator allow an extensive study of the algorithm. In particular, there are proofs of convergence even when the network edges appear and disappear randomly over time. The resulting algorithm, however, inherits the slow convergence properties of the (sub)gradient algorithm, making it communication-inefficient. Variations of~\eqref{Eq:RelatedWorkSubgradientConsensus} have also been explored~\cite{Johansson08-SubgradientMethodsAndConsensusAlgorithms-CDC,Zhu12-OnDistributedConvexOptimizationUnderInequalityEqualityConstraints,Zhu10-OnDistributedOptimizationUnderInequalityPrimalDualMethods,Sayed12-DiffusionAdaptationDistributedOptimization}. For example, \cite{Johansson08-SubgradientMethodsAndConsensusAlgorithms-CDC} considers the update $x_p^{k+1} = \bigl[\sum_{j \in \mathcal{N}_p \cup \{p\}}a_{pj}^k (x_j^k - \alpha_k \nabla f_j(x^k))\bigr]_{X_p}$ and, thus, the (sub)gradient update occurs before transmission; the work~\cite{Zhu12-OnDistributedConvexOptimizationUnderInequalityEqualityConstraints,Zhu10-OnDistributedOptimizationUnderInequalityPrimalDualMethods} applies~\eqref{Eq:RelatedWorkSubgradientConsensus} to the dual of a constrained optimization problem.

		After noticing that \eqref{Eq:RelatedWorkSubgradientConsensus} is the application of the (sub)gradient algorithm~\eqref{Eq:RelatedWorkSubgradientMethod} to a problem related (but not equivalent) to~\eqref{Eq:GlobalProblem}, \cite{Jakovetic11-FastDistributedMethods} proposed an improvement based on Nesterov's fast gradient algorithm~\eqref{Eq:RelatedWorkNesterovAlg}. The problem algorithm~\eqref{Eq:RelatedWorkSubgradientConsensus} actually solves is
		\begin{equation}\label{Eq:RelatedWorkJakoveticReform}
			\underset{x = (x_1,\ldots,x_P)}{\text{minimize}}\,\, f_1(x_1) + \cdots + f_P(x_P) + \frac{\alpha \,k}{2} \sum_{(i,j) \in \mathcal{E}} \|x_i - x_j\|^2\,,
		\end{equation}
		where~$x_p \in \mathbb{R}^n$ is the copy of~$x$ held by node~$p$, $\alpha > 0$ is a constant, and~$k$ is the iteration number. The first term of the objective of~\eqref{Eq:RelatedWorkJakoveticReform} is the original objective~\eqref{Eq:GlobalProblem}, where the variable~$x$ was replaced by its copy~$x_p$ at the $p$th function~$f_p$; the second term is a consensus-inducing term, in the sense that different values of the copies between neighbors are penalized. As the iterations go on, the second term becomes more important, forcing the nodes to achieve a consensus on their copies. 	Note that problems~\eqref{Eq:GlobalProblem} and~\eqref{Eq:RelatedWorkJakoveticReform} are not equivalent and that this provides an additional reason why algorithms based on~\eqref{Eq:RelatedWorkSubgradientConsensus} usually converge slowly. The algorithms we propose in this thesis reformulate~\eqref{Eq:GlobalProblem} into problems that are equivalent to the original one and, thus, do not have this drawback.

		Other distributed algorithms that solve~\eqref{Eq:GlobalProblem} with algorithms based on (sub)gradient methods include~\cite{Duchi12-DualAveragingDistributedOptimization}, which hinges on a dual averaging algorithm by Nesterov~\cite{Nesterov09-PrimalDualSubgradientMethodsConvexProblems}, and \cite{Ram09-AsynchronousGossipAlgorithmsForStochasticOptimization}, which studies a gossip-based version of~\eqref{Eq:RelatedWorkSubgradientConsensus}, i.e., only two neighboring nodes communicate at each time instant. The work~\cite{Ghadimi13-AcceleratedGradientMethodsNetworkedOptimization} solves~\eqref{Eq:RelatedWorkDualDecomposition}, i.e., the dual of~\eqref{Eq:RelatedWorkDecompositionMethods}, using Polyak's heavy-ball method~\cite{Polyak87-IntroductionToOptimization}.

		\mypar{Double-looped algorithms}
		A reformulation of~\eqref{Eq:GlobalProblem} similar to~\eqref{Eq:RelatedWorkRabbat05DualDecomposition}, but that uses half the constraints, is
		\begin{equation}\label{Eq:RelatedWorkEdgeBasedReformulation}
			\begin{array}{ll}
				\underset{x_1,\ldots,x_P}{\text{minimize}} & f_1(x_1) + f_2(x_2) + \cdots + f_P(x_P) \\
				\text{subject to} & x_i = x_j\,,\quad (i,j) \in \mathcal{E}\,,
			\end{array}
		\end{equation}
		where the copies associated to each node are enforced to be the same through the edges of the network. Similarly to what we saw for~\eqref{Eq:RelatedWorkRabbat05DualDecomposition}, if we apply dual decomposition to~\eqref{Eq:RelatedWorkEdgeBasedReformulation}, the result is a distributed algorithm. Unless it is assumed that each function~$f_p$ is strictly convex, it is not possible, however, to recover a primal solution after having solved the dual problem. An alternative is to use augmented Lagrangian methods, for example, the method of multipliers~\eqref{Eq:RelatedWorkMoM1}-\eqref{Eq:RelatedWorkMoM2}. The augmented term, however, precludes the minimization~\eqref{Eq:RelatedWorkMoM1} from being carried out in a distributed way. A known workaround is to use an additional loop: an iterative algorithm such as the nonlinear Jacobi~\eqref{Eq:RelatedWorkNonlinearJacobi} or the nonlinear Gauss-Seidel~\eqref{Eq:RelatedWorkBlockCoordinateNGS}. In fact, this has been done for solving problem~\eqref{Eq:RelatedWorkDecompositionMethods} in~\cite{Ruszczynski92-AugmentedLagrangianDecomposition} (method of multipliers concatenated with the diagonal quadratic approximation) and in~\cite{Zakarian95-NonlinearJacobi-thesis} (method of multipliers concatenated with the nonlinear Jacobi method). In our work~\cite{Mota11-BasisPursuitInSensorNetworks-ICASSP}, which is not included in this thesis, we applied Nesterov's gradient algorithm~\eqref{Eq:RelatedWorkNesterovAlg} to both loops, for solving basis pursuit, a problem that can be written as~\eqref{Eq:GlobalProblem}, as we will see in the next chapter. Another relevant work is~\cite{Jakovetic11-CooperativeConvexOptimization}, which solves~\eqref{Eq:GlobalProblem} with the method of multipliers concatenated with a randomized nonlinear Gauss-Seidel method, and uses a reformulation identical to~\eqref{Eq:RelatedWorkEdgeBasedReformulation}; see~\cite{Necoara13-RandomCoordinateDescentAlgorithmsMultiAgentConvexOptimization} for related work. A difficulty that arises when implementing double-looped algorithms is determining a distributed, robust stopping criterion for the inner loop. Implementing double-looped algorithms in a communication-efficient manner is therefore very challenging.

		\mypar{ADMM-based}
		If we apply ADMM to the reformulations~\eqref{Eq:RelatedWorkRabbat05DualDecomposition}, \eqref{Eq:RelatedWorkEdgeBasedReformulation}, and similar ones, we get, in general, distributed algorithms that do not suffer the lack of parallelism of incremental methods, the slow rates of convergence of (sub)gradient-based methods, and the cumbersome two loops of double-looped algorithms.

		\begin{algorithm}
			\small
			\caption{\cite{Schizas08-ConsensusAdHocWSNsPartI}}
			\label{Alg:Schizas}
			\algrenewcommand\algorithmicrequire{\textbf{Initialization:}}
			\begin{algorithmic}[1]
			\Require Choose $\rho \in \mathbb{R}$; for all $p \in \mathcal{V}$, set $x_p^0 = \mu_p^0 = \eta_p^0 = 0_n \in \mathbb{R}^n$ and $\tau_p = 1/(\rho(D_p + 1))$; set $k=0$

			\Repeat
			\ForAll{$p \in \mathcal{V}$ [in parallel]}

			\Statex

			\State Compute
				$
					z_p^{k+1} = \tau_p \,\mu_p^k + \frac{1}{D_p+1}\sum_{j \in \mathcal{N}_p^+} x_j^k
				$
				and exchange $z_p^{k+1}$ with neighbors $\mathcal{N}_p$
				\label{SubAlg:RelatedWorkSchizasUpdateZ}

			\Statex

			\State Compute
				$
					x_p^{k+1} = \text{prox}_{\,\tau_p f_p}\Bigl(\frac{1}{D_p + 1} \sum_{j \in \mathcal{N}_p^+}z_j^{k+1} - \tau_p \, \eta_p^k\Bigr)
				$
				and exchange $x_p^{k+1}$ with neighbors $\mathcal{N}_p$
				\label{SubAlg:RelatedWorkSchizasUpdateX}

			\Statex

			\State Update the dual variables:
				\begin{align*}
					\mu_p^{k+1} &= \mu_p^k + \frac{1}{\tau_p} \biggl(\frac{1}{D_p + 1}\sum_{j \in \mathcal{N}_p^+} x_j^{k+1} - z_p^{k+1} \biggr)
					\\
					\eta_p^{k+1} &= \eta_p^k + \frac{1}{\tau_p} \biggl(x_p^{k+1} - \frac{1}{D_p+1}\sum_{j \in \mathcal{N}_p^+} z_j^{k+1} \biggr)
				\end{align*}
				\label{SubAlg:RelatedWorkSchizasUpdateDual}

			\vspace{-0.2cm}

			\State $k \gets k+1$
			\EndFor
			\Until{some stopping criterion is met}
    \end{algorithmic}
    \end{algorithm}

		As said before, the first distributed algorithm based on ADMM was proposed in~\cite{Schizas08-ConsensusAdHocWSNsPartI}, for solving a particular instance of~\eqref{Eq:GlobalProblem} in the context of estimation. That algorithm, however, can be easily generalized to solve the entire class~\eqref{Eq:GlobalProblem} and is shown as Algorithm~\ref{Alg:Schizas}, explained later. \aref{App:ADMMAlgsGlobalClass} shows the derivation of Algorithm~\ref{Alg:Schizas}: we show this derivation for completeness and because, to our best knowledge, there is no reference in the literature where the algorithm is derived to solve the entire class~\eqref{Eq:GlobalProblem}. The derivation applies the $2$-block ADMM to the following reformulation of~\eqref{Eq:GlobalProblem}:
		\begin{equation}\label{Eq:RelatedWorkSchizasReformulation}
			\begin{array}{ll}
				\underset{\bar{x},\bar{z}}{\text{minimize}} & f_1(x_1) + f_2(x_2) + \cdots + f_P(x_P) \\
				\text{subject to} & x_p = z_j\,, \quad j \in \mathcal{N}_p^+\,,\quad p = 1,\ldots,P\,,
			\end{array}
		\end{equation}
		where each node~$p$ has two copies of~$x$: $x_p \in \mathbb{R}^n$ and~$z_p \in \mathbb{R}^n$.\footnote{As pointed out in~\cite{Schizas08-ConsensusAdHocWSNsPartI}, if there are cliques in the network and, in each clique, only one node is chosen to have the second copy of~$x$, say~$z_p$, problems~\eqref{Eq:GlobalProblem} and~\eqref{Eq:RelatedWorkSchizasReformulation} are still equivalent. In that case, we can even go further and reduce each clique to one node. Since this is a very specific case, we will ignore it and assume that there are no cliques or, if there are, that each node has two copies of~$x$ anyway.} The optimization variable is $(\bar{x},\bar{z}) = (x_1,\ldots,x_P,z_1,\ldots,z_P) \in (\mathbb{R}^n)^{2P}$, which makes problem~\eqref{Eq:RelatedWorkSchizasReformulation} have~$2P$ times more variables than the original problem~\eqref{Eq:GlobalProblem}. In~\eqref{Eq:RelatedWorkSchizasReformulation}, we used $\mathcal{N}_p^+ = \mathcal{N}_p \cup \{p\}$ to denote the extended neighborhood of node~$p$, i.e., its set of neighbors~$\mathcal{N}_p$ and itself. Problem~\eqref{Eq:RelatedWorkSchizasReformulation} then has $2E + P$ constraints, since there are~$2$ constraints per edge $(i,j) \in \mathcal{E}$, $x_i = z_j$ and $x_j = z_i$, and each node~$p$ constrains~$x_p = z_p$. Regarding Algorithm~\ref{Alg:Schizas}, it is fully parallel, as all nodes perform the same tasks at the same time. In the initialization, $\rho$ is the augmented Lagrangian parameter and is assumed fixed and known by all the nodes. At each node~$p$, there is an auxiliary variable~$\tau_p$ that depends on~$\rho$ and on~$D_p = |\mathcal{N}_p|$, the number of neighbors of node~$p$. The algorithm consists of three operations, in two of each there is a communication step. Specifically, in step~\ref{SubAlg:RelatedWorkSchizasUpdateZ} (resp.\ \ref{SubAlg:RelatedWorkSchizasUpdateX}) each node updates $z_p$ (resp.\ $x_p$) and exchanges it with its neighbors. Note that updating~$x_p$ in step~\ref{SubAlg:RelatedWorkSchizasUpdateX} requires the variables~$z_j$ from the neighbors $j \in \mathcal{N}_p$. Note also that while the update of~$z_p$ is linear and independent of the function~$f_p$, the update of~$x_p$ involves the prox operator of a scaled version of~$f_p$. The prox operator of a  closed convex function~$f:\mathbb{R}^q \xrightarrow{} \mathbb{R} \cup \{+\infty\}$ is defined as
		\begin{equation}\label{Eq:DefinitionProxOperator}
			\text{prox}_f(x) = \underset{y}{\arg\min}\,\, f(y) + \frac{1}{2}\|y - x\|^2\,.
		\end{equation}
		This operator, introduced in~\cite{Moreau62-FonctionsConvexesDualesEtPointsProximaux}, arises in ADMM-based algorithms, since each ADMM subproblem (cf.\ \eqref{Eq:RelatedWorkADMMIter1}-\eqref{Eq:RelatedWorkADMMIter2}) is a quadratic problem that can always be written in terms of the prox operator. The prox operator has many properties; see~\cite{Pesquet11-ProximalSplittingMethodsInSignalProcessing-BookChapter} for an extensive list. After performing step~\ref{SubAlg:RelatedWorkSchizasUpdateX} in Algorithm~\ref{Alg:Schizas}, node~$p$ updates two dual variables, $\mu_p$ and~$\eta_p$, using the new values of~$x_j$ and~$z_j$, for~$j \in \mathcal{N}_p^+$. The convergence of the algorithm is guaranteed by the convergence results for the $2$-block ADMM~\eqref{Eq:RelatedWorkADMMIter1}-\eqref{Eq:RelatedWorkADMMIter3}.

    \begin{algorithm}
			\small
			\caption{\cite{Zhu09-DistributedInNetworkChannelCoding}}
			\label{Alg:Zhu}
			\algrenewcommand\algorithmicrequire{\textbf{Initialization:}}
			\begin{algorithmic}[1]
			\Require Choose $\rho \in \mathbb{R}$; for all $p \in \mathcal{V}$, set $x_p^0 = \mu_p^0 = 0_n \in \mathbb{R}^n$ and $\tau_p = 1/(2\rho D_p)$; set~$k=0$
			\Repeat

			\ForAll{$p \in \mathcal{V}$ [in parallel]}

			\Statex

			\State Compute
				$
					x_p^{k+1} = \text{prox}_{\tau_p f_p}\Bigl(\frac{1}{2D_p} \sum_{j \in \mathcal{N}_p} (x_p^k + x_j^k) - \tau_p \mu_p^k\Bigr)
				$
				and exchange $x_p^{k+1}$ with neighbors~$\mathcal{N}_p$
				\label{SubAlg:RelatedWorkZhuUpdateX}

			\Statex

			\State Update the dual variable
				$
					\mu_p^{k+1} = \mu_p^k + \frac{1}{2\tau_p}\bigl(x_p^{k+1} - \frac{1}{D_p}\sum_{j \in \mathcal{N}_p}x_j^{k+1}\bigr)
				$
				\label{SubAlg:RelatedWorkZhuUpdateDual}

			\vspace{0.2cm}
			\State $k \gets k+1$

			\EndFor
    \Until{some stopping criterion is met}
    \end{algorithmic}
    \end{algorithm}

    The second distributed algorithm based on ADMM was proposed in~\cite{Zhu09-DistributedInNetworkChannelCoding} to solve the average consensus problem, in the context of channel decoding. As~\cite{Schizas08-ConsensusAdHocWSNsPartI}, it can also be easily generalized to solve the entire class~\eqref{Eq:GlobalProblem}. Indeed, that algorithm was used in~\cite{Erseghe11-FastConsensusByADMM,Bazerque10-DistributedSpectrumSensingCognitiveRadio,Forero10-ConsensusBasedDistributedSVMs,Mateos10-DistributedSparseLinearRegression} to solve several other problems in signal processing and machine learning that can be recast as~\eqref{Eq:GlobalProblem}. Algorithm~\ref{Alg:Zhu} shows an adaptation of the algorithm proposed in~\cite{Zhu09-DistributedInNetworkChannelCoding} to solve the entire class~\eqref{Eq:GlobalProblem}; its derivation is shown in \aref{SubApp:DerivationZhu}. As in Algorithm~\ref{Alg:Schizas}, the derivation applies the $2$-block ADMM, but to a different reformulation of~\eqref{Eq:GlobalProblem}. Namely, starting with the equivalent problem~\eqref{Eq:RelatedWorkEdgeBasedReformulation}, \cite{Zhu09-DistributedInNetworkChannelCoding} adds~$E$ new variables, each one associated to an edge~$(i,j) \in \mathcal{E}$ of the network, and rewrites~\eqref{Eq:RelatedWorkEdgeBasedReformulation} as
    \begin{equation}\label{Eq:RelatedWorkZhuReformulation}
    	\begin{array}{ll}
    		\underset{\bar{x},\bar{z}}{\text{minimize}} & f_1(x_1) + f_2(x_2) + \cdots + f_P(x_P) \\
    		\text{subject to} & x_i = z_{ij}\,, \quad (i,j) \in \mathcal{E} \\
    		                  & x_j = z_{ij}\,, \quad (i,j) \in \mathcal{E}\,,
    	\end{array}
    \end{equation}
    where~$(\bar{x},\bar{z}) = (x_1,\ldots,z_P,\ldots,z_{ij},\ldots) \in (\mathbb{R}^n)^{P+E}$ is the optimization variable. Problem~\eqref{Eq:RelatedWorkZhuReformulation} then has~$P+E$ more variables (each of size~$n$) than~\eqref{Eq:GlobalProblem} and introduces~$2E$ constraints. In \aref{SubApp:DerivationZhu}, we show how the application of the $2$-block ADMM~\eqref{Eq:RelatedWorkADMMIter1}-\eqref{Eq:RelatedWorkADMMIter3} to~\eqref{Eq:RelatedWorkZhuReformulation} yields Algorithm~\ref{Alg:Zhu}. Note that the application of the same algorithm to a different problem reformulation yields a different, yet more efficient, algorithm. In particular, Algorithm~\ref{Alg:Zhu} has only one communication step per iteration, whereas Algorithm~\ref{Alg:Schizas} has two. The communication step occurs in step~\ref{SubAlg:RelatedWorkZhuUpdateX}, where each node~$p$ updates its estimate~$x_p$ by computing the prox operator of~$\tau_p f_p$, and then broadcasts the new estimate to its neighbors~$\mathcal{N}_p$. Note that~$\bar{z}$, the variable that was introduced in~\eqref{Eq:RelatedWorkZhuReformulation}, is absent of Algorithm~\ref{Alg:Zhu} since, as shown in \aref{SubApp:DerivationZhu}, it can be eliminated. The notable work~\cite{Erseghe11-FastConsensusByADMM} provides a thorough analysis of Algorithms~\ref{Alg:Schizas} and~\ref{Alg:Zhu} applied to the average consensus problem. Namely, it establishes linear convergence, proposes a scheme to select the augmented Lagrangian parameter~$\rho$, and studies the factors that influence their convergence. More recently, the work~\cite{Yin13-LinearlyConvergentDecentralizedConsensusOptimizationADMM-ICASSP,Yin13-LinearConvergenceADMMDecentralizedOptimization}, based on the results of~\cite{Deng12-GlobalLinearConvergenceGeneralizedADMM}, establishes the linear convergence of Algorithm~\ref{Alg:Zhu} whenever each function~$f_p$ is strongly convex, differentiable, and its gradient is Lipschitz continuous. It also studies the factors that influence the convergence rate of the algorithm and, based on that study, proposes a scheme to select the augmented Lagrangian parameter~$\rho$. Although that scheme gives a reasonable value for~$\rho$, it does not give the optimal one, i.e., it is usually possible to select a better one by trial-and-error. This partly explains why in the experimental results presented in this thesis we always try several values for~$\rho$, through grid search, and select the one that yields the best result.

    The algorithm we propose for~\eqref{Eq:GlobalProblem}, rather than using the $2$-block ADMM, applies the multi-block ADMM~\eqref{Eq:RelatedWorkADMMProbExtendedADMMIter1}-\eqref{Eq:RelatedWorkADMMProbExtendedADMMIter4} directly to reformulation~\eqref{Eq:RelatedWorkEdgeBasedReformulation}. Although we cannot establish a convergence rate (since that is still an open problem for the multi-block ADMM), we show through extensive experimental results that the resulting algorithm outperforms both Algorithms~\ref{Alg:Schizas} and~\ref{Alg:Zhu} in terms of the number of communications.

    \mypar{Other splitting methods}
    We already mentioned that ADMM is an application of the Douglas-Rachford splitting operator to finding the zeros of a given monotone operator. Besides ADMM, other splitting algorithms can be applied and yield distributed optimization algorithms. One example is in~\cite{Ling12-MultiBlockAlternatingDirectionMethodParallelSplittingConsensusOptimization}, which applies a parallel splitting scheme directly to reformulation~\eqref{Eq:RelatedWorkEdgeBasedReformulation}, as the algorithm we propose. Our experimental results show, however, that our algorithm outperforms the algorithm proposed in~\cite{Ling12-MultiBlockAlternatingDirectionMethodParallelSplittingConsensusOptimization} in terms of the number of communications.%\mynote{Elaborate more}

    We also mention that~\cite{Iutzeler13-AsynchronousDistributedOptimizationRandomizedADMM} proposed an asynchronous distributed algorithm for~\eqref{Eq:GlobalProblem} using a randomized version of the Douglas-Rachford operator. Their experimental results show, however, that the resulting algorithm requires more communications to converge than by using the synchronous version. A gossip-based distributed ADMM-based algorithm has been recently proposed in~\cite{Wei13-OnTheConvergenceAsynchronousDistributedADMM} and has been shown to converge with rate~$O(1/k)$.

	\subsection{Star-shaped class}

	Somehow differently from the global class~\eqref{Eq:GlobalProblem}, distributed algorithms for the star-shaped class have been motivated mainly by specific applications, and not by the goal of solving an entire class of optimization problems. Such a motivating applications include network utility maximization (NUM), network flow problems, state estimation in power systems, and distributed model predictive control (D-MPC). For this reason, we will organize this section application-wise rather than algorithm-wise. Some applications, most notably D-MPC, arise naturally in scenarios where the variable is non-global and non-star-shaped. In fact, one of the contributions of this thesis is a new framework for D-MPC that uses a generic connected, or even non-connected, variable; this will be addressed in \cref{Ch:ConnectedNonConnected}.

	\mypar{Network utility maximization}
	Consider a network whose edges have a finite transmission capacity and whose nodes are either packet sources, packet sinks, or packet re-transmitters. Each source sends packets to one sink through a specific, pre-chosen route along the network. Associated to each source~$s$ there is an utility function~$U_s$ (increasing and concave) that depends on~$x_s$, the rate at which source~$s$ sends packets. The goal of network utility maximization (NUM), proposed in~\cite{Kelly97-ChargingRateControlElasticTraffic,Kelly98-RateControlCommunicationNetworksShadowPrices}, is to maximize the sum of the utilities of all the sources, while satisfying the link capacity constraints:
	\begin{equation}\label{Eq:RelatedWorkNUM}
		\begin{array}{ll}
			\underset{x = (x_1,\ldots,x_S)}{\text{maximize}} & \sum_{s=1}^S U_s(x_s) \\
			\text{subject to} & Rx = c \\
			                  & x \geq 0\,,
		\end{array}
	\end{equation}
	where the $l$th row of the routing matrix~$R$ has ones in entries corresponding to sources that use link~$l$ and zeros elsewhere. The $l$th entry of vector~$c$ has the capacity of link~$l$. Note that problem~\eqref{Eq:RelatedWorkNUM} is a particular instance of~\eqref{Eq:RelatedWorkDecompositionMethods}. It has been used to model congestion control on the Internet~\cite{Kelly97-ChargingRateControlElasticTraffic,Low99-OptimizationFlowControl,Low02-UnderstandingVegas} and scheduling problems~\cite{Shakkotai07-NetworkOptimizationAndControl}; see also the surveys \cite{Palomar06-TutorialDecompositionMethods,Chiang07-LayeringAsOptimizationDecomposition}. If we build an auxiliary network indicating which links are used by each source then, as we will see in \cref{Ch:ConnectedNonConnected}, a dual problem of~\eqref{Eq:RelatedWorkNUM} can be written as~\eqref{Eq:IntroProb} with a star-shaped variable.  Actually, if we apply a gradient or a subgradient method directly to that dual problem, we obtain a distributed algorithm because all the induced subgraphs are stars. This is done in~\cite{Low99-OptimizationFlowControl}, which proposes and analyzes synchronous and asynchronous versions of the gradient method for a dual problem of~\eqref{Eq:RelatedWorkNUM}; curiously, the TCP/IP Vegas protocol, which was designed as an ad hoc congestion control protocol, is interpreted in~\cite{Low02-UnderstandingVegas} as a gradient method solving that dual problem. With the goal of improving the speed to convergence, Newton-like methods have also been proposed, for example, a diagonally scaled version of the gradient method with Hessian information in the diagonal~\cite{Low00-OptimizationFlowNewton}, and a Newton method where the descent direction is computed approximately~\cite{Bertsekas11-CentralizedDistributedNewton,Wei10-DistributedNewtonMethod,Wei11-DistributedNewtonMethodNUM}. More recently, \cite{Beck13-OptimalDistributedGradientMethodsNetworkResourceAllocationProblems} took advantage of the strong concavity of typical utility functions, which implies that their conjugate is differentiable with Lipschitz continuous gradients, and proposed applying Nesterov's gradient method~\eqref{Eq:RelatedWorkNesterovAlg} with a choice for a Lipschitz constant that does not require knowing all the utilities at a central location. Then, it proved that the primal estimates converge at rate $O(1/k)$ to their optimal values.

	In all these methods, the communication between the source nodes and the used links can be done implicitly, i.e., without sending additional numbers over the network: only by increasing or decreasing the sending rate at which each source sends its packets, and by discarding or not packets that arrive to a given link, an implicit communication can be established. The algorithm we propose for~\eqref{Eq:IntroProb}, in contrast, requires explicit communication between the source nodes and the links; however, it exhibits faster convergence to the equilibrium.

	\mypar{Distributed model predictive control}
	Model predictive control (MPC), also known as receding horizon control, is an efficient control scheme for discrete-time systems. Dating back to the early sixties~\cite{Zadeh62-OnOptimalControlAndLinearProgramming,Propoi63-UseOfLPMethodsForSynthesizingSampledDataAutomaticSystems}, MPC became very popular in the petro-chemical industry in the early eighties, as surveyed in~\cite{Morari89-ModelPredictiveControlTheoryAndPractice}. The interest in applying MPC to distributed systems, however, arose later, in the nineties~\cite{Acar92-SomeExamplesDecentralizedRecedingHorizonControl-CDC,Fawal98-OptimalControlComplexIrrigationSystemsAugmentedLagrangian}. The setting is a network of systems, each of which has associated a state, a control input, or both. Each system interacts with neighboring systems in two ways: through system dynamics and through communication. Interaction through system dynamics means that the state of each system is influenced by the states and control inputs of neighboring systems; sometimes, neighboring systems also have coupled goals (or efficiency measures). Interaction through communication refers to the ability that each system has to exchange messages with neighboring systems and, thus, it corresponds to what we call communication network. MPC in this scenario is usually referred to as distributed MPC (D-MPC). The goal in each instance of D-MPC is to make the systems cooperate to find an optimal set of inputs, i.e., control inputs that drive the state of each system from an initial (measured) state to a predefined goal, while minimizing the energy to do so. This can be cast as an optimization problem with the format of~\eqref{Eq:IntroProb}, as we will see in \cref{Ch:ConnectedNonConnected}. To the best of our knowledge, all prior work on D-MPC has assumed that interaction through dynamics coincides with interaction through communication. That is, if two systems have coupled dynamics, i.e., the state of one of them is influenced by the state or input of the other, then they necessarily communicate directly. According the classification scheme introduced in \cref{Ch:Introduction}, the variable in this case is star-shaped. In this thesis, we introduce a new framework for D-MPC, where coupled systems do not necessarily need to communicate directly. We also present potential applications for this new framework.

	Early work on D-MPC has focused on studying stability and performance of heuristics whose solutions are not guaranteed to be optimal. For example, \cite{Jia01-DistributedModelPredictiveControl} proposes a one-step scheme where each system solves a local optimization problem that incorporates state predictions from its neighbors; this is preceded by a communication step, where state predictions are exchanged between neighboring nodes. For related methods, see~\cite{Camponogara02-DistributedMPC,Keviczky06-DecentralizedRecedingHorizonControl,Venkat05-StabilityOptimalityDistributedMPC}.

	D-MPC has also been tackled with optimization-based algorithms, not always completely distributed, that find exact solutions. For example, \cite{Fawal98-OptimalControlComplexIrrigationSystemsAugmentedLagrangian} proposes an augmented Lagrangian method where the augmented term is linearized, a method now known as split inexact Uzawa method in the image processing community~\cite{Zhang10-BregmanizedNonlocalRegularizationForDeconvolutionSparseReconstructions,Zhang10-UnifiedPrimalDualAlgorithmFrameworkBasedBregmanIteration}. The resulting algorithm is not distributed, since it requires a central node. Distributed algorithms for D-MPC include dual decomposition with the subgradient method~\cite{Wakasa08-DecentralizedMPCDualDecomp} (as described in \ssref{SubsSec:DecompositionMethods}), distributed interior-point methods~\cite{Camponogara11-DistributedOptimizationMPCLinearDynamicNetworks}, and more recently, fast gradient methods~\cite{Conte12-ComputationalAspectsDistributedMPC} and ADMM~\cite{Conte12-ComputationalAspectsDistributedMPC,Summers12-DistributedModelPredictiveControlViaADMM}. In particular, \cite{Conte12-ComputationalAspectsDistributedMPC,Summers12-DistributedModelPredictiveControlViaADMM} apply the ADMM method proposed in~\cite{Boyd11-ADMM}, which becomes distributed whenever the variable is star-shaped. This is, in fact, the case since, as mentioned before, all prior work on D-MPC assumes that interaction through dynamics coincides with interaction through communication.

	The algorithms we propose for D-MPC require less communications to achieve convergence than all these algorithms. In addition, they solve D-MPC in scenarios that have never been considered before: problems with a connected variable that is neither global nor star-shaped, and problems with a non-connected variable. Both cases model systems that are coupled through their dynamics, but cannot communicate directly.

	\mypar{Network flows}
	Beyond NUM and D-MPC, there is an extensive literature on network flow problems, some of which can be formulated as~\eqref{Eq:IntroProb} as well. In a typical network flow problem, each component of the optimization variable is associated to an edge of the network, and the function at each node depends on the variables associated to its incident edges. Hence, the variable is star-shaped; actually, each induced subgraph is very simple: it consists of two nodes and an edge connecting them. The first optimization algorithm solving a network flow problem was Dantzig's simplex method~\cite[Ch.19-20]{Dantzig63-LinearProgrammingAndExtensions}. Extensive information about network flows, including specialized algorithms (most of them centralized), can be found in the surveys~\cite{Ahuja91-SomeRecentAdvancesInNetworkFlows,Goldberg89-NetworkFlowAlgorithms} and in the books~\cite{Ahuja93-NetworkFlows,Bertsekas98-NetworkOptimization}.

	Regarding distributed algorithms for network flows, dual decomposition methods generally yield distributed algorithms. For example, by assuming strict convexity on the cost functions, \cite{Bertsekas87-DistributedAsynchronousRelaxationMethodsConvexNetworkFlowProblems} computes the dual of a network flow problem and proposes to solve it with an asynchronous Gauss-Seidel method. The application of a subgradient method to a similar problem is analyzed in~\cite{Boyd04-SimultaneousRoutingAndResourceAllocationViaDualDecomposition}. More recently, \cite{Trdlicka10-DistributedMultiCommodityNetworkFlowAlgorithmForEnergyOptimalRoutingWSN} proposed a double-looped algorithm, where the outer loop uses the proximal minimization algorithm (to overcome the lack of strict convexity of the primal objective) and the inner loop uses the gradient method. The work~\cite{Jadbabaie09-DistributedNewtonMethodNetworkOptimization,Zargham11-AcceleratedDualDescentNetworkOptimization,Zargham12-AcceleratedDualDescentForNetworkFlowOptimization,Bertsekas11-CentralizedDistributedNewton} proposes a distributed algorithm for network flows based on Newton's method, where the Newton direction is computed approximately. Although the resulting method requires the cost functions to be strongly convex and twice differentiable, it is proven to converge superlinearly to a neighborhood of the problem's solution. This contrasts with the algorithm we propose for~\eqref{Eq:GlobalProblem}, which only requires the cost functions to be convex, possibly non-differentiable. Our algorithm thus requires assumptions much less restrictive that the assumptions of methods based on dual decomposition or on Newton's algorithm. Additionally, as will be shown in \cref{Ch:ConnectedNonConnected}, the algorithm we propose requires less communications to converge than the algorithm in~\cite{Jadbabaie09-DistributedNewtonMethodNetworkOptimization,Zargham11-AcceleratedDualDescentNetworkOptimization,Zargham12-AcceleratedDualDescentForNetworkFlowOptimization}.

	\subsection{Mixed class}

	The mixed problem class~\eqref{Eq:IntroMixed}, reproduced here for convenience,
	\begin{equation}	\label{Eq:IntroMixed}\tag{M}
		\begin{array}{ll}
			\underset{x = (y,z) \in\mathbb{R}^n}{\text{minimize}} & f_1(y,z_{S_1}) + f_2(y,z_{S_2}) + \cdots + f_P(y,z_{S_P})\,,
		\end{array}
	\end{equation}
	has rarely appeared in literature, despite its generality and possible applications. One instance of~\eqref{Eq:IntroMixed} has appeared in~\cite{Tan06-DistributedOptimizationCoupledSystemsNUM} (see also~\cite[\S IV-B]{Palomar06-TutorialDecompositionMethods}) as a dual of a NUM problem with coupled objectives. We will look at this problem with more detail in \cref{Ch:ConnectedNonConnected}. Such a problem can model cooperative systems, e.g., systems where the rate allocated to one source depends on the rate allocated to the cluster that source belongs to, or competitive systems, e.g., wireless power control or digital subscribed line (DSL) spectrum management where signal-to-interference ratios (SIR) are dependent on transmit powers of other users. The method proposed in~\cite{Tan06-DistributedOptimizationCoupledSystemsNUM} is distributed and consists of solving that dual problem (which has the format of~\eqref{Eq:IntroMixed}) with a gradient method. Actually, the application of the gradient method to~\eqref{Eq:IntroMixed} in~\cite{Tan06-DistributedOptimizationCoupledSystemsNUM} yields a distributed algorithm, because the non-global components, $z$ in~\eqref{Eq:IntroMixed}, are star-shaped.

  We will also use the framework of~\eqref{Eq:IntroMixed} to solve in a distributed way a compressed sensing problem with a data partitioning that has never been considered before. More concretely, basis pursuit denoising (BPDN), and a related problem that we call \textit{reversed lasso} have been solved in a distributed way with a row partition~\cite{Mota13-DADMM,Bazerque10-DistributedSpectrumSensingCognitiveRadio,Mateos10-DistributedSparseLinearRegression} and with a column partition~\cite{Mota13-DADMM}, respectively. The reverse cases, i.e., BPDN with a column partition and reversed lasso with a row partition, have never been solved before. We will show in \cref{Ch:ConnectedNonConnected} that reversed lasso with a row partition can be formulated as~\eqref{Eq:IntroMixed}, and therefore can be solved with the algorithms we propose here.

%% file: 02-MainMatter/globalClass.tex
\chapter{Global Class}
\label{Ch:GlobalVariable}

	This chapter addresses the global class~\eqref{Eq:GlobalProblem} and is based on the publications~\cite{Mota12-DADMM-ICASSP,Mota12-ConsensusOnColoredNetworks-CDC,Mota12-DistributedBP,Mota11-ADMMProof,Mota13-DADMM}. The chapter is organized into four sections. In \sref{Sec:GC:ProblemStatement}, the problem is formally stated and the assumptions are clearly identified. In \sref{Sec:GC:Applications}, we describe some applications that can be written as~\eqref{Eq:GlobalProblem}. Special emphasis is given to \ssref{SubSec:GC:SparseSolutions}, since it contains novel contributions, such as writing some distributed compressed sensing problems as~\eqref{Eq:GlobalProblem}. Then, in~\sref{Sec:GC:Derivation}, we propose our algorithm for the global class~\eqref{Eq:GlobalProblem} and analyze it. Finally, in \sref{Sec:GC:ExperimentalResults}, we show the performance of the proposed algorithm against prior algorithms by running extensive simulations. These show that, while solving the entire class~\eqref{Eq:GlobalProblem}, our algorithm is as efficient as algorithms that were specifically designed for particular applications and, often, it is even better.

	\section{Problem statement}
	\label{Sec:GC:ProblemStatement}

	The global problem class~\eqref{Eq:GlobalProblem} consists of minimizing the sum of~$P$ functions where each function depends on all the components of~$x$. For convenience, let us rewrite~\eqref{Eq:GlobalProblem} here:
	\begin{equation}\tag{G}
		\begin{array}{ll}
			\underset{x \in \mathbb{R}^n}{\text{minimize}} & f_1(x) + f_2(x) + \cdots + f_P(x)\,.
		\end{array}
	\end{equation}
	We make the following assumptions:
	\begin{assumption}\label{Ass:GPfunctions}
		Each function~$f_p : \mathbb{R}^n \xrightarrow{} \mathbb{R} \cup \{+\infty\}$ is closed and convex over~$\mathbb{R}^n$ and not identically~$+\infty$.
	\end{assumption}
	\begin{assumption}\label{Ass:GPsolvable}
		Problem~\eqref{Eq:GlobalProblem} is solvable, i.e., it has at least one solution~$x^\star$.
	\end{assumption}
	In \assref{Ass:GPfunctions} we use the concept of an extended real-valued function~$f$, which can take infinite values and is defined over all~$\mathbb{R}^n$. Such a function is closed and convex if its epigraph $\text{epi}\,f := \{(x,r) \in \mathbb{R}^n \times \mathbb{R}\,:\, f(x) \leq r\}$ is closed and convex, respectively~\cite[\S1.2]{Bertsekas03-ConvexAnalysisAndOptimization}, \cite[\S B.1]{Lemarechal04-FundamentalsConvexAnalysis}. Alternatively, a function is closed if it is lower semicontinuous or if all its sublevel sets are closed~\cite[Prop.1.2.2]{Bertsekas03-ConvexAnalysisAndOptimization}, \cite[Prop.1.2.2]{Lemarechal04-FundamentalsConvexAnalysis}. Considering extended real-valued functions simplifies the notation without losing generality: as explained before, each node~$p$ can constrain variable~$x$ to belong to a given set~$S$, i.e., $x \in S$, through an indicator function~$\text{i}_S$, defined as $\text{i}_S(x) = 0$ if~$x \in S$, and $\text{i}_S(x) = +\infty$ if $x\not\in S$.

	We associate problem~\eqref{Eq:GlobalProblem} to a communication network~$\mathcal{G} = (\mathcal{V},\mathcal{E})$ with~$P = |\mathcal{V}|$ nodes and $E = |\mathcal{E}|$ edges: the $p$th node of the network is the only node who knows function~$f_p$ or, in other words, function~$f_p$ is private to node~$p$. Regarding the network, we assume:
	\begin{assumption}\label{Ass:GPConnectedStatic}
		The network is connected and its topology does not vary with time.
	\end{assumption}
	\begin{assumption}\label{Ass:GPColoring}
		A coloring scheme~$\mathcal{C}$ of the network is available; each node knows its own color and the color of its neighbors.
	\end{assumption}
	The concept of network coloring was explained in~\sref{Sec:IntroOverview}: it is an assignment of numbers, called colors, to the nodes such that no neighboring nodes have the same color. Formally, each node is assigned a color in~$\mathcal{C} = \{1,\ldots,C\}$, where~$C :=|\mathcal{C}|$ is the total number of colors, and $\mathcal{C}(p)$ denotes the color of node~$p$. The coloring scheme~$\mathcal{C}$ is called \textit{proper} (or valid) if $\mathcal{C}(i) \neq \mathcal{C}(j)$, for all $(i,j) \in \mathcal{E}$.	Our goal is to \textit{design a distributed algorithm that solves~\eqref{Eq:GlobalProblem} while keeping the function~$f_p$ private to node~$p$}. Recall that a distributed algorithm is one that uses no central or special node and no all-to-all communications.

	\mypar{Discussion of the assumptions}
	Compared to prior algorithms for the global class~\eqref{Eq:GlobalProblem}, the problem Assumptions~\ref{Ass:GPfunctions} and~\ref{Ass:GPsolvable} are very general, while the network Assumptions~\ref{Ass:GPConnectedStatic} and~\ref{Ass:GPColoring} are more restrictive, yet realistic in some scenarios. In fact, what \assref{Ass:GPfunctions} asks is the problem to be convex, a minimal requirement to guarantee that we can find a global minimizer of~\eqref{Eq:GlobalProblem}. In \assref{Ass:GPsolvable}, we require that the problem is well-posed by having at least one solution.

	Regarding the network assumptions, assuming a fixed network topology as in \assref{Ass:GPConnectedStatic} is a common first step in distributed optimization. Some algorithms, however, are proven to converge under intermittent link failures, e.g., \cite{Rabbat05-GeneralizedConsensusAlgs,Lobel11-DistributedMultiAgentOptimization,Jakovetic11-CooperativeConvexOptimization}. These algorithms, in turn, require more assumptions on the functions in~\eqref{Eq:GlobalProblem}. In fact, there seems to be a curious tradeoff between the problem assumptions and the network assumptions: the algorithms that relax the network assumptions usually require more restrictive problem assumptions, and vice-versa. Regarding \assref{Ass:GPColoring}, this assumption is new in the context of distributed optimization and will underlie the construction of our algorithm. Recall that finding the minimum number of colors a network can be colored with is NP-hard~\cite{Garey79-ComputersAndIntractability}, except for bipartite networks. The minimum number of colors required to color a network~$\mathcal{G}$ is called the \textit{chromatic number} and is represented with~$\chi(\mathcal{G})$. Assuming that~$\chi(\mathcal{G})$ is known and that~$\chi(\mathcal{G}) > 2$ (i.e., the network is not bipartite), coloring~$\mathcal{G}$ with~$\chi(\mathcal{G})$ colors is NP-hard as well. Given its importance in wireless networks, there are several approximation algorithms to compute coloring schemes of networks, some of which are distributed~\cite{Kuhn06-ComplexityDistributedGraphColoring,Leith06-ConvergenceDistributedLearningAlgorithmsOptimalWirelessChannelAllocation,Duffy08-ComplexityAnalysisDecentralizedGraphColoring,Linial92-LocalityDistributedGraphAlgorithms}. For example, \cite{Kuhn06-ComplexityDistributedGraphColoring} proposes a coloring scheme that uses~$O(D_{\max})$ colors while requiring~$O(D_{\max}/\log^2(D_{\max}) + \log^\star(P))$ iterations to compute them, where~$D_{\max} := \max\{D_p\,:\, p \in \mathcal{V}\}$ is the maximum degree of a node in the network. Another coloring scheme using less iterations, but more colors, more specifically, $O(\log^\star(P))$ iterations and~$O(D_{\max}^2)$ colors, is proposed in~\cite{Linial92-LocalityDistributedGraphAlgorithms}. In this thesis, we assume that a coloring scheme with~$C$ colors is given and we will ignore how it was obtained. Consequently, the additional number of communications to obtain the scheme will also be ignored in the comparison with other algorithms. Although all the other algorithms use no coloring scheme (all nodes work in parallel), the comparison is fair for two reasons: first, if an algorithm is run several times on the same network, for example, for different data, coloring the network just needs to be done once, before the first instantiation; after running the algorithm several times, the coloring cost becomes diluted.	The second, and perhaps more important, reason is that in networks where the transmission medium is shared, for example, in wireless networks or even in Ethernet cables, the nodes cannot communicate in parallel without using a medium access control (MAC) protocol~\cite[Ch.5-6]{Kurose05-ComputerNetworking},\cite{Krishnamachari05-NetworkingWirelessSensors}. For example, in wireless networks, one node cannot receive two different messages from its neighbors at the same time and at the same frequency (unless it uses more than one receive antenna~\cite{Choi10-AchievingSingleChannelDuplex}). This creates the hidden and the exposed node problems~\cite[\S6.2.2]{Krishnamachari05-NetworkingWirelessSensors}, which are prevented by the use of MAC protocols. For data-intensive algorithms, such as the ones considered in this thesis, schedule-based MAC protocols are the most energy-efficient~\cite[\S6.7]{Krishnamachari05-NetworkingWirelessSensors}. Time division multiple access (TDMA) is such a protocol which, in addition, is also based on network coloring. The particular coloring scheme used by TDMA can also be used for the algorithms we propose; thus, our algorithms integrate naturally with TDMA. Prior algorithms for distributed optimization, in contrast, assume no particular MAC protocol. The second part of \assref{Ass:GPColoring} will be discussed when we introduce our algorithm; briefly, it allows discarding a centralized entity controlling all the nodes that have the same color (recall that they are not neighbors) and, because of that, it is essential in making our algorithm distributed.

	\section{Applications}
	\label{Sec:GC:Applications}

	There are many engineering problems that can be written as~\eqref{Eq:GlobalProblem}. Here, we will focus on problems that arise in networks and, consequently, that can be solved via distributed algorithms. We address two types of problems: inference problems, which include average consensus and support vector machines (SVMs), and sparse solutions of linear systems, which include several compressed sensing problems.

	\subsection{Inference problems}

		\mypar{Average consensus}
		Consider the scalar version of the inference problem described in \cref{Ch:Introduction}: a sensor network composed of~$P$ nodes is deployed to estimate a parameter~$\bar{\theta} \in \mathbb{R}$. The estimation uses measurements from all the sensors, which are assumed noisy. Let~$\theta_p$ denote the measurement at node~$p$. When the noise is independent across nodes, Gaussian, with zero mean, and identity covariance matrix, the maximum log-likelihood estimation of~$\bar{\theta}$ is given by \textit{average consensus}~\cite{DeGroot74-ReachingConsensus}:
		\begin{equation}\label{Eq:Consensus}
			\begin{array}{ll}
				\underset{x}{\text{minimize}} & \frac{1}{2}(x - \theta_1)^2 + \frac{1}{2}(x - \theta_2)^2 + \cdots + \frac{1}{2}(x - \theta_P)^2\,.
			\end{array}
		\end{equation}
		Average consensus has been widely studied in the literature, and many distributed algorithms have been proposed to solve it~\cite{Jadbabaie03-CoordinationGroupsMobileAutonomousAgents,Boyd04-FastLinearIterationsforDistributedAveraging,Boyd06-RandomizedGossipAlgorithms,OlfatiSaber07-ConsensusAndCooperationInNetworkMultiAgentSystems,Kar09-DistributedConsensusAlgsSN,Oreshkin10-OptimizationAnalysisDistrAveraging,Moura10-GossipSurvey,Olshevsky11-ConvergenceSpeedDistributedConsensusAveraging}. Curiously, most of these algorithms are not optimization-based, in the sense that they do not view the consensus problem as the distributed optimization problem~\eqref{Eq:Consensus}; rather, they simply solve it with a linear update scheme, such as~\eqref{Eq:IntroAverageConsensus}. Work that has addressed average consensus by devising a distributed optimization algorithm for~\eqref{Eq:Consensus} includes~\cite{Rabbat04-DistributedOptimizationSensorNetworks,Rabbat05-QuantizedIncrementalAlgorithmsDistributedOptimization,Schizas08-ConsensusAdHocWSNsPartI,Zhu09-DistributedInNetworkChannelCoding,Erseghe11-FastConsensusByADMM}. In particular, \cite{Erseghe11-FastConsensusByADMM} analyzes Algorithms~\ref{Alg:Schizas} and~\ref{Alg:Zhu}, described in \cref{Ch:relatedWork}, applied to consensus. Despite the vast quantity of algorithms for the average consensus, we will see that the algorithm we propose for the global class~\eqref{Eq:GlobalProblem} has a performance similar to the most efficient algorithms, if not better.

		\mypar{Support vector machine (SVM)}
		Another important inference problem is a support vector machine (SVM)~\cite[Ch.7]{Bishop06-PatternRecognitionMachineLearning}. Training an SVM consists of finding the parameters~$(s,r) \in \mathbb{R}^{n-1} \times \mathbb{R}$ of an hyperplane~$\{x \in \mathbb{R}^{n-1}\,:\, s^\top x = r\}$ that best separates two classes of points. These points are given as~$(x_k,y_k) \in \mathbb{R}^{n-1} \times \mathbb{R}$, where~$y_k \in \{-1,1\}$ indicates the class of the point~$x_k$. Finding these parameters usually involves solving an optimization problem, for example,
		\begin{equation}\label{Eq:SVM}
			\begin{array}{ll}
				\underset{s,r,\xi}{\text{minimize}} & \frac{1}{2} \|s\|^2 + \beta \,1_K^\top \xi \\
				\text{subject to} & y_k(s^\top x_k - r) \geq 1 - \xi_k\,,\quad k=1,\ldots,K \\
				& \xi \geq 0\,,
			\end{array}
		\end{equation}
		where~$K$ is the total number of points, $\beta > 0$ is a tradeoff parameter, and~$\xi \in \mathbb{R}^K$ is a vector of slack variables. In a network scenario, we assume each node knows~$m_p$ points, but all the nodes cooperate to solve the global problem~\eqref{Eq:SVM}. This problem can be written as~\eqref{Eq:GlobalProblem} by setting
		\begin{equation}\label{Eq:GlobalFunctionNodeSVM}
			f_p(s,r) =
			\begin{array}[t]{cl}
				\underset{\bar{\xi}_p}{\inf} & \frac{1}{2P}\|s\|^2 + \beta\, 1_{m_p}^\top \bar{\xi}_p \\
				\text{s.t.} & Y_p (X_p s - r 1_{m_p}) \geq 1_{m_p} - \bar{\xi}_p \\
				& \bar{\xi}_p \geq 0\,,
			\end{array}
		\end{equation}
		where~$Y_p$ is a diagonal matrix with the labels~$y_k$ of the points of node~$p$ in the diagonal, and~$X_p$ is an $m_p \times n$ matrix with each row containing~$x_k^\top$, ordered the same way as~$Y_p$. The variable~$x$ in~\eqref{Eq:GlobalProblem} corresponds to~$(s,r)$, since the slack variables~$\bar{\xi}_p$ are internal to each node. This distributed SVM problem has been solved in~\cite{Forero10-ConsensusBasedDistributedSVMs} with Algorithm~\ref{Alg:Zhu}. See~\cite{Predd09-ACollaborativeTrainingAlgorithmForDistributedLearning} for a related message-passing method.

		\subsection{Sparse solutions of linear systems}
		\label{SubSec:GC:SparseSolutions}

		Another application we consider is finding sparse solutions of distributed linear systems. This is mainly motivated by the recent field of compressed sensing~\cite{Donoho06-CompressedSensing,Candes06-RobustUncertaintyPrinciplesExactSignalReconstructionHighlyIncomplete}, which establishes a new paradigm for signal acquisition and sampling. Surveys on the topic include~\cite{Candes06-CompressiveSampling,Candes08-IntroductionToCompressiveSampling,Baraniuk07-CompressiveSensing,Bryan13-MakingDoWithLess}. While acquisition of signals in compressed sensing is usually simple, reconstructing them afterwards is more complicated and it involves solving an optimization problem. In noiseless scenarios, the most common problem is \textit{basis pursuit} (BP)~\cite{Donoho98-AtomicDecompositionBasisPursuit}:
		\begin{equation}\label{Eq:BasisPursuit}
			\begin{array}{ll}
				\underset{x \in \mathbb{R}^n}{\text{minimize}} & \|x\|_1 \\
				\text{subject to} & Ax = b\,,
			\end{array}
		\end{equation}
		where~$x \in \mathbb{R}^n$ is the variable and~$\|x\|_1$ denotes the $\ell_1$-norm of~$x$, defined as $\|x\|_1 := \sum_{i=1}^n |x_i|$. The matrix~$A \in \mathbb{R}^{m \times n}$ and the vector~$b \in \mathbb{R}^m$ are associated to the acquisition process, and we assume they are given. The linear system~$Ax = b$ is underdetermined, i.e., $m < n$, and the matrix~$A$ is usually assumed full rank, so that the linear system is feasible for any~$b$. This is common in compressed sensing, since the entries of~$A$ are usually drawn randomly and in an independent way. In noisy scenarios other problems are used. An example is \textit{basis pursuit denoising} (BPDN)~\cite{Donoho98-AtomicDecompositionBasisPursuit}:
		\begin{equation}\label{Eq:BPDN}
			\begin{array}{ll}
				\underset{x \in \mathbb{R}^n}{\text{minimize}} & \frac{1}{2}\|Ax - b\|^2 + \beta\|x\|_1\,,
			\end{array}
		\end{equation}
		where~$\beta > 0$ is a tradeoff parameter and~$\|z\|$ denotes the $\ell_2$-norm of~$z \in \mathbb{R}^q$, i.e., $\|z\| = \sqrt{\sum_{i=1}^q z_i^2}$. There is also a problem that we will call \textit{reversed lasso}~\cite{Fuchs05-RecoveryExactSparseRepresentationsPresenceBoundedNoise,Donoho06-StableRecoverySparseOvercompleteRepresentationsPresenceNoise,Tropp06-JustRelax}:
		\begin{equation}\label{Eq:ReverseLasso}
			\begin{array}{ll}
				\underset{x \in \mathbb{R}^n}{\text{minimize}} & \|x\|_1 \\
				\text{subject to} & \|Ax - b\| \leq \sigma\,,
			\end{array}
		\end{equation}
		where~$\sigma > 0$ is a known bound on the noise magnitude, and a problem called the \textit{least absolute shrinkage and selection operator} (lasso)~\cite{Tibshirani96-RegressionShrinkageLasso}:
		\begin{equation}\label{Eq:lasso}
			\begin{array}{ll}
				\underset{x \in \mathbb{R}^n}{\text{minimize}} & \frac{1}{2}\|Ax - b\|^2 \\
				\text{subject to} & \|x\|_1 \leq \gamma\,,
			\end{array}
		\end{equation}
		where~$\gamma > 0$ is a known parameter. Problems~\eqref{Eq:BasisPursuit}-\eqref{Eq:lasso} provide heuristics to find sparse solutions of the linear system~$A x = b$. In fact, it was established in~\cite{Natarajan95-SparseApproximateSolutionToLinearSystems} that finding a sparsest solution of that linear system is NP-hard. Such a problem would be written as~\eqref{Eq:BasisPursuit} with the cost function replaced by the cardinality of the vector~$x$, $\text{card}(x)$. We thus see that~\eqref{Eq:BasisPursuit} approximates the non-continuous, non-convex function~$\text{card}(x)$ by the convex function~$\|x\|_1$, resulting in a convex (and hence easier) problem. The same approximation motivates problems~\eqref{Eq:BPDN}-\eqref{Eq:lasso}, but in the scenario where~$b$ may not be expressed exactly as a linear combination of the columns of~$A$. The theory of compressed sensing establishes conditions on the matrix~$A$ under which approximating $\text{card}(x)$ by~$\|x\|_1$ in, for example, BP~\eqref{Eq:BasisPursuit} yields an exact approximation: this means that the NP-hard problem obtained from~\eqref{Eq:BasisPursuit} by replacing~$\|x\|_1$ with~$\text{card}(x)$ has the same solution as the convex problem BP~\eqref{Eq:BasisPursuit}. Surprisingly, some types of random matrices satisfy those conditions with overwhelming probability. For more details see, for example, \cite{Candes05-DecodingByLinearProgramming,Rudelson08-SparseReconstructionFourierGaussianMeasurements,Candes08-RIPAndItsImplicationsForCompressedSensing,Baraniuk08-SimpleProofRIP}.

		Problems~\eqref{Eq:BPDN}, \eqref{Eq:ReverseLasso}, and~\eqref{Eq:lasso} are all related through duality and, therefore, are equivalent in some sense, provided their parameters~$\beta$, $\sigma$, and~$\gamma$ are chosen appropriately. Among these problems, \eqref{Eq:ReverseLasso} is the one to which compressed sensing results apply directly~\cite{Candes08-RIPAndItsImplicationsForCompressedSensing,Tropp06-JustRelax}, in spite of never have been coined a specific name. Apparently, sometimes it is also called lasso~\cite{Candes08-IntroductionToCompressiveSampling}, but we avoid that name to prevent confusion with the original lasso~\eqref{Eq:lasso}. Instead, we will call it \textit{reversed lasso} since, compared to lasso, its objective and constraints are reversed. Note that when~$\sigma=0$, the reversed lasso becomes BP~\eqref{Eq:BasisPursuit}. We are interested in solving the compressed sensing problems~\eqref{Eq:BasisPursuit}-\eqref{Eq:lasso} in the distributed scenarios described next.

		\begin{figure}
			\centering
			\includegraphics[scale=1.15]{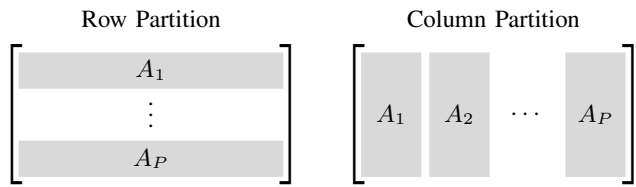}
			\medskip
			\caption[Row partition and column partition of~$A$ into~$P$ blocks.]{
				Row partition and column partition of~$A$ into~$P$ blocks. A block in the row (resp. column) partition is a set of rows (resp. columns).
				}
			\label{Fig:PartitionOfA}
		\end{figure}

		\mypar{Distributed scenarios: row and column partition}
		We consider two different scenarios for splitting the data in matrix~$A$ and vector~$b$ among the nodes of a network with~$P$ nodes. These are called \textit{row partition} and \textit{column partition}, and are visualized in \fref{Fig:PartitionOfA}. In the row (resp.\ column) partition, each node stores a block of rows (resp.\ columns) of~$A$. While in the row partition vector~$b$ is partitioned similarly to~$A$, with each node storing the corresponding subblock, in the column partition we assume all nodes know the full vector~$b$. More specifically, node~$p$ knows~$(A_p,b_p) \in \mathbb{R}^{m_p \times n} \times \mathbb{R}^{m_p}$ in the row partition and knows~$(A_p,b) \in \mathbb{R}^{m \times n_p} \times \mathbb{R}^m$ in the column partition. Naturally, we have $m = \sum_{p=1}^P m_p$ and $n = \sum_{p=1}^P n_p$.

		The row partition scenario arises naturally when applying compressed sensing in a sensor network. For instance, suppose the nodes of the network are interested in estimating a high-dimensional but sparse vector~$x \in \mathbb{R}^n$, for example, an ultra-wide band but spectrally sparse radio signal. Each node in the network is equipped with a low bandwidth antenna and, hence, any signal acquisition has to be done at a rate far below the Nyquist rate. By using a random demodulator~\cite{Candes08-IntroductionToCompressiveSampling,Tropp10-BeyondNyquist}, compressed sensing can be applied, and each row of the linear system~$Ax = b$ represents one measurement (performed at a low acquisition rate). Therefore, if node~$p$ takes~$m_p$ linear measurements of~$x$, we have exactly the row partition scenario. This setting appeared in~\cite{Nowak06-CompressiveWirelessSensing,Rabbat08-CompressedSensingNetworkedData}, where several applications are described. It is assumed there, however, that the signal reconstruction, i.e., solving one of the problems~\eqref{Eq:BPDN}-\eqref{Eq:lasso} is done in a centralized way, in a fusion center. The algorithms we propose in this thesis allow reconstructing the signal on the network, without using any fusion center. Furthermore, all nodes will know the signal when the algorithm finishes. Other applications include distributed target localization~\cite{Cevher08-DistributedTargetLocalizationSpatialSparsity} and distributed field reconstruction~\cite{Schmidt13-ScalableSensorNetworkFieldRobustBasisPursuit-thesis}.

		One application of the column partition is described in~\cite{Romberg08-EfficientSeismicForwardModeling}, in the context of forward modeling in geological applications. The goal is to find the Green's function, represented by a vector~$x$, of a model of the earth's surface. The authors of~\cite{Romberg08-EfficientSeismicForwardModeling} propose deploying a set of sources and a set of receivers over some geographical area and have all the sources emit a signal simultaneously. The receivers capture a linear superposition of all the emitted signals. The proposed way to find~$x$ is by solving BP~\eqref{Eq:BasisPursuit}, where a set of columns of~$A$ is associated to a source. This is clearly our column partition scenario. The distance between all the devices in this application makes a distributed solution convenient, such as the ones provided by our algorithms.

		We will see next how BP, BPDN, and lasso with a row partition are naturally recast as~\eqref{Eq:GlobalProblem}. Then, we will consider the less trivial case of a column partition, for all the problems~\eqref{Eq:BasisPursuit}-\eqref{Eq:lasso}. The only problem that will be missing is reversed lasso with a row partition. However, in \cref{Ch:ConnectedNonConnected}, we will be able to recast it as~\eqref{Eq:IntroProb}, not with a global variable, but with a mixed one.

		\mypar{Row partition: BP, BPDN, and lasso}
		Consider a row partition as shown in \fref{Fig:PartitionOfA}. Then, BP~\eqref{Eq:BasisPursuit} can be written as~\eqref{Eq:GlobalProblem} by setting as the function of node~$p$
		\begin{equation}\label{Eq:GlobalRecastBP}
			f_p(x) = \frac{1}{P}\|x\|_1 + \text{i}_{A_px=b_p}(x)\,,
		\end{equation}
		where $\text{i}_{A_px=b_p}(x)$ is the indicator function of the set $\{x\,:\, A_p x = b_p\}$. Similarly, BPDN~\eqref{Eq:BPDN} can be written as~\eqref{Eq:GlobalProblem} by setting as the function of node~$p$
		\begin{equation}\label{Eq:GlobalRecastBPDN}
			f_p(x) = \frac{1}{2}\|A_p x - b_p\|^2 + \frac{\beta}{P}\|x\|_1\,.
		\end{equation}
		Note that the parameter~$\beta$ and the number of nodes~$P$ is assumed to be known by all nodes. Lasso~\eqref{Eq:lasso} can also be written easily as~\eqref{Eq:GlobalProblem} by setting
		\begin{equation}\label{Eq:GlobalRecastLasso}
			f_p(x) = \frac{1}{2}\|A_p x - b_p\|^2 + \text{i}_{\|x\|_1 \leq \gamma}(x)
		\end{equation}
		as the function of node~$p$. Here, the parameter~$\gamma$ is also assumed to be known at all nodes. Each function~$f_p$ in~\eqref{Eq:GlobalRecastBP}-\eqref{Eq:GlobalRecastLasso} contains data that is known only by node~$p$: namely, the pair~$(A_p,b_p)$. All these functions are closed and convex. Furthermore, the extended real-valued function~\eqref{Eq:GlobalRecastBP} (resp.\ \eqref{Eq:GlobalRecastLasso}) is not identically~$+\infty$ whenever~$A_p$ has full rank (resp.\ $\gamma$ is positive). BPDN with a row partition was solved in~\cite{Bazerque10-DistributedSpectrumSensingCognitiveRadio,Mateos10-DistributedSparseLinearRegression} with Algorithm~\ref{Alg:Zhu}, viewing it as an instance of~\eqref{Eq:GlobalProblem} with~\eqref{Eq:GlobalRecastBPDN}.

		\mypar{Column partition: duality and regularization}
		We now turn into a column partition and recast all the problems~\eqref{Eq:BasisPursuit}-\eqref{Eq:lasso} as~\eqref{Eq:GlobalProblem}. We will need duality to do this. However, plain duality will not be enough to recover primal solutions from dual solutions, since the problems we dualize have cost functions that are not strictly convex. We will thus use regularization and, in the case of BP, the concept of \textit{exact regularization}. We introduce this concept together with a result by Friedlander and Tseng~\cite{Fridlander07-ExactRegularizationCVXPrograms}. Consider the following conic program
		\begin{equation}\label{Eq:ConicProgram}
			\begin{array}{ll}
				\underset{x}{\text{minimize}} & c^\top x \\
				\text{subject to} & x \in \mathcal{K} \\
				                  & Ax = b\,,
			\end{array}
		\end{equation}
		where~$c \in \mathbb{R}^n$, $A \in \mathbb{R}^{m \times n}$, and~$b \in \mathbb{R}^m$ are given, and $\mathcal{K} \subseteq \mathbb{R}^n$ is a nonempty, closed, convex cone. Problem~\eqref{Eq:ConicProgram} is assumed to have a nonempty solution set~$\mathcal{S}$. Consider now a regularization function~$\phi:\mathbb{R}^n \xrightarrow{} \mathbb{R}$ such that all sublevel sets of~$\mathcal{S}$, i.e., $\{x \in \mathcal{S}\,:\, \phi(x) \leq \alpha\}$, are bounded for all~$\alpha$. A result in~\cite{Fridlander07-ExactRegularizationCVXPrograms}, more specifically in corollary~2.3 of~\cite{Fridlander07-ExactRegularizationCVXPrograms}, states that when~$\mathcal{K}$ is polyhedral, i.e., $\mathcal{K} = \{x\,:\, v_i^\top x \leq 0\,,i=1,\ldots,q\}$ for some set of~$q$ vectors~$v_i \in \mathbb{R}^n$, then the regularization of~\eqref{Eq:ConicProgram} with~$\phi$ is exact. This means that there exists a $\bar{\delta} > 0$ such that the set of solutions of the regularized problem
		\begin{equation}\label{Eq:ConicProgramRegularized}
			\begin{array}{ll}
				\underset{x}{\text{minimize}} & c^\top x  + \frac{\delta}{2}\phi(x)\\
				\text{subject to} & x \in \mathcal{K} \\
				                  & Ax = b
			\end{array}
		\end{equation}
		is contained in the set of solutions~$\mathcal{S}$ of~\eqref{Eq:ConicProgram}, for all $0 \leq \delta \leq \bar{\delta}$ \cite[Cor.2.3]{Fridlander07-ExactRegularizationCVXPrograms}. As mentioned in~\cite{Fridlander07-ExactRegularizationCVXPrograms}, this is a generalization of exact regularization results for linear programs~\cite{Mangasarian79-NonlinearPerturbationLPs,Friedlander06-ExactRegularizationLPs}. Experimental results in~\cite{Fridlander07-ExactRegularizationCVXPrograms} suggest that the above result is true even when~$\mathcal{K}$ is not polyhedral, namely, when~$\mathcal{K}$ is the Lorenz cone~$\{(x,t)\;:\, \|x\| \leq t\}$ (also known as the ice-cream cone and as the second-order cone). That cone will actually arise in some of our problems for which, inspired by the results in~\cite{Fridlander07-ExactRegularizationCVXPrograms}, we will perform the above regularization. For BP, we will use the exact regularization result, since BP is equivalent to a linear program, which is the simplest instance of a conic program. Regarding the choice of~$\delta$, we are unaware of any method that finds~$\bar{\delta}$ without first solving the unregularized problem~\eqref{Eq:ConicProgram}. Therefore, we will choose~$\delta$ based on trial-and-error. According to our experiments, $\delta \in [10^{-3},10^{-1}]$ allows computing an optimal solution with reasonable accuracy most of the times. In~\cite[\S7]{Fridlander07-ExactRegularizationCVXPrograms}, it is reported that $\delta = 10^{-4}$ yielded an optimal solution in $85\%$ of their experiments.

		\mypar{Column partition: BP}
		We start with BP~\eqref{Eq:BasisPursuit}. Consider the regularization function~$\phi(x) = (1/4)\|x\|^2$ and the regularized problem
		\begin{equation}\label{Eq:RegularizedBP}
			\begin{array}{ll}
				\underset{x}{\text{minimize}} & \|x\|_1 + \frac{\delta}{2} \|x\|^2 \\
				\text{subject to} & Ax = b\,.
			\end{array}
		\end{equation}
		Then, by the previous discussion, the following theorem follows.
		\begin{theorem}
		\hfill

		\medskip
		\noindent
			Problem~\eqref{Eq:RegularizedBP} is an exact regularization of BP~\eqref{Eq:BasisPursuit}, i.e., there exists a $\bar{\delta} > 0$ such that the solution of~\eqref{Eq:RegularizedBP}, with $0 \leq \delta \leq \bar{\delta}$, is always a solution of BP.
		\end{theorem}
		\begin{proof}
		\hfill

		\medskip
		\noindent
			We use the exact regularization results of~\cite{Fridlander07-ExactRegularizationCVXPrograms,Mangasarian79-NonlinearPerturbationLPs,Friedlander06-ExactRegularizationLPs}, as explained before. First, we recast BP as a problem with the same format as~\eqref{Eq:ConicProgram}:
			\begin{equation}\label{Eq:ProofExactRegularizationBP}
				\begin{array}{ll}
					\underset{x,t}{\text{minimize}} & 1_n^\top t \\
					\text{subject to} & Ax = b \\
					                  & -t \leq x \leq t\,,
				\end{array}
			\end{equation}
			where~$t \in \mathbb{R}^n$ is an epigraph variable and~$1_n \in \mathbb{R}^n$ is the vector of ones. Problem~\eqref{Eq:ProofExactRegularizationBP} has the same format as~\eqref{Eq:ConicProgram} by making the correspondence $c = (0_n,1_n)$ and~$\mathcal{K} = \{(x,t)\,:\, -t \leq x \leq t\}$, which is a polyhedral cone that is nonempty, closed, and convex. The corresponding regularized problem~\eqref{Eq:ConicProgramRegularized} with $\phi(z) = (1/4)\|z\|^2$ is
			\begin{align*}
				\begin{array}[t]{ll}
					\underset{x,t}{\text{minimize}} & 1_n^\top t  + \frac{\delta}{4}\|x\|^2 + \frac{\delta}{4}\|t\|^2\\
					\text{subject to} & Ax = b \\
					                  & -t \leq x \leq t
				\end{array}
				\qquad
				&\Longleftrightarrow
				\qquad
				\begin{array}[t]{ll}
					\underset{x}{\text{minimize}} & \frac{\delta}{4}\|x\|^2 \\
					\text{subject to} & Ax = b
				\end{array}
				+
				\begin{array}[t]{cl}
					\underset{t}{\inf} & 1_n^\top t + \frac{\delta}{4}\|t\|^2 \\
					\text{s.t.} & -t \leq x \leq t
				\end{array}
				\\
				&\Longleftrightarrow
				\qquad
				\begin{array}[t]{ll}
					\underset{x}{\text{minimize}} & \|x\|_1 + \frac{\delta}{2}\|x\|^2 \\
					\text{subject to} & Ax = b\,,
				\end{array}
			\end{align*}
			which is~\eqref{Eq:RegularizedBP}. In the last equivalence, we used the fact that, for a fixed~$x$,
			\begin{equation}\label{Eq:ProofRegularizationBPIdentity}
				\begin{array}[t]{cl}
					\underset{t}{\inf} & 1_n^\top t + \frac{\delta}{4}\|t\|^2 \\
					\text{s.t.} & -t \leq x \leq t\,.
				\end{array}
				\quad
				=
				\quad
				\|x\|_1 + \frac{\delta}{4}\|x\|^2\,.
			\end{equation}
		\end{proof}
		Now, let~$\lambda \in \mathbb{R}^m$ be a dual variable associated to the constraint of~\eqref{Eq:RegularizedBP}. The dual problem is
		\begin{align}
			&
				\underset{\lambda}{\text{maximize}}\,\,\,\, b^\top \lambda +  \underset{x}{\inf}\,\, \Bigl[\|x\|_1 + \frac{\delta}{2}\|x\|^2 - \lambda^\top Ax\Bigr]
			\label{Eq:BPDual1}
			\\
			\Longleftrightarrow
			\qquad&
				\underset{\lambda}{\text{maximize}}\,\,\,\, b^\top \lambda +  \sum_{p=1}^P\,\underset{x_p}{\inf}\,\, \Bigl[\|x_p\|_1 + \frac{\delta}{2}\|x_p\|^2 - \lambda^\top A_px_p\Bigr]
			\label{Eq:BPDual2}
			\\
			\Longleftrightarrow
			\qquad&
			  \underset{\lambda}{\text{minimize}}\,\,\,\, -b^\top \lambda +  \sum_{p=1}^P\,\underset{x_p}{\sup}\,\, \biggl[ (A_p^\top \lambda)^\top x_p - \Bigl(\|x_p\|_1 + \frac{\delta}{2}\|x_p\|^2\Bigr) \biggr]
			\label{Eq:BPDual3}
			\\
			\Longleftrightarrow
			\qquad&
			  \underset{\lambda}{\text{minimize}}\,\,\,\, \sum_{p=1}^P \, \Bigl(h_p^\star(A_p^\top \lambda) - \frac{1}{P}b^\top \lambda\Bigr)\,,
			\label{Eq:BPDual4}
		\end{align}
		which has the format of~\eqref{Eq:GlobalProblem} with the function at node~$p$ given by~$f_p(\lambda) = h_p^\star(A_p^\top \lambda) - (1/P)b^\top \lambda$. From~\eqref{Eq:BPDual1} to~\eqref{Eq:BPDual2}, we used the column partition and the fact that all terms inside the infimum decouple. From~\eqref{Eq:BPDual2} to~\eqref{Eq:BPDual3}, we switched from a maximization problem to a minimization one. And, in~\eqref{Eq:BPDual4}, we defined $h_p^\star$ as being the convex conjugate of the function~$h_p(x_p) = \|x_p\|_1 + (\delta/2)\|x_p\|^2$, for each~$p$. Note that the global variable is the dual variable~$\lambda$; also, after an optimal value~$\lambda^\star$ has been found (or better, agreed by all the nodes), the $p$th component of the corresponding primal solution~$x^\star$ is available at the $p$th node. Each component is given by soft-thresholding:
		\begin{equation}\label{Eq:BPSoftSolution}
			x_i
			=
			\left\{
			  \begin{array}{ll}
			  	\frac{1}{\delta}\Bigl((A_p^\top \lambda)_i - 1\Bigr)_i &,\,\, (A_p^\top \lambda)_i > 1 \\
			  	\frac{1}{\delta}\Bigl((A_p^\top \lambda)_i + 1\Bigr)_i &,\,\, (A_p^\top \lambda)_i < -1 \\
			  	0 &\,\, -1 \leq (A_p^\top \lambda)_i \leq 1\,,
			  \end{array}
			\right.
		\end{equation}
		for~$i$ belonging to the indices of the columns of~$A_p$; see \aref{App:ConjugateFunctions} for the derivation of~\eqref{Eq:BPSoftSolution}.

		\mypar{Column partition: BPDN}
		We now move to BPDN with a column partition. We will also use regularization but, this time, we will not have an exact regularization result. To regularize BPDN~\eqref{Eq:BPDN} the same way as BP, we first rewrite it with the format of~\eqref{Eq:ConicProgram}:
		\begin{equation}\label{Eq:RegularizationBPDN}
			\begin{array}{ll}
				\underset{x,t,u,v}{\text{minimize}} & \frac{1}{2}v + \beta\,1_n^\top t \\
				\text{subject to} & \|u\|^2 \leq v \\
				                  & -t \leq x \leq t \\
				                  & u = Ax - b\,,
			\end{array}
		\end{equation}
		where $t \in \mathbb{R}^n$ and~$v \in \mathbb{R}$ are epigraph variables, and~$u \in \mathbb{R}^m$ is an auxiliary variable. Problem~\eqref{Eq:RegularizationBPDN} has the same structure as~\eqref{Eq:ConicProgram}, since its objective is linear, the last two constraints are also linear, and the cone~$\mathcal{K}$ is the Cartesian product $\mathcal{K} = \mathcal{K}_{x,t} \times \mathcal{K}_{u,v}$, where $\mathcal{K}_{x,t} = \{(x,t)\,:\, -t \leq x \leq t\}$ is polyhedral, but~$\mathcal{K}_{u,v} = \{(u,v)\,:\, \|u\|^2 \leq v\}$ is not. Using the function~$\phi(z) = (1/4)\|z\|^2$ to regularize~\eqref{Eq:RegularizationBPDN}, we obtain
		\begin{align}
			&
			  \begin{array}{ll}
				  \underset{x,t,u,v}{\text{minimize}} & \frac{1}{2}v + \beta 1_n^\top t  + \frac{\delta}{4}\Bigl(\|x\|^2 + \|t\|^2 + \|u\|^2 + v^2\Bigr)\\
				  \text{subject to} & \|u\|^2 \leq v \\
				                    & -t \leq x \leq t \\
				                    & u = Ax - b\,,
			  \end{array}
			  \label{Eq:RegularizationBPDN1}
			\\
			\Longleftrightarrow
			\qquad&
			  \begin{array}[t]{ll}
			  	\underset{x}{\text{minimize}} & \frac{\delta}{4}\|x\|^2 + \frac{\delta}{4}\|Ax - b\|^2
			  \end{array}
			  +
			  \begin{array}[t]{cl}
					\underset{t}{\inf} & \beta 1_n^\top t + \frac{\delta}{4}\|t\|^2 \\
					\text{s.t.} & -t \leq x \leq t
				\end{array}
				+
				\begin{array}[t]{cl}
					\underset{v}{\inf} & \frac{1}{2}v + \frac{\delta}{4}v^2\\
					\text{s.t.} & \|Ax - b\|^2 \leq v \\
				\end{array}
			\label{Eq:RegularizationBPDN2}
			\\
			\Longleftrightarrow
			\qquad&
			  \begin{array}{ll}
			  	\underset{x}{\text{minimize}} & (\frac{1}{2}+\frac{\delta}{4})\|Ax - b\|^2 + \beta \|x\|_1 + \frac{\delta}{2}\|x\|^2 + \frac{\delta}{4}\|Ax - b\|^4\,.
			  \end{array}
			\label{Eq:RegularizationBPDN3}
		\end{align}
		From~\eqref{Eq:RegularizationBPDN1} to~\eqref{Eq:RegularizationBPDN2}, we replaced~$u$ by~$Ax - b$. From~\eqref{Eq:RegularizationBPDN2} to~\eqref{Eq:RegularizationBPDN3}, we used~\eqref{Eq:ProofRegularizationBPIdentity} with the weight~$\beta$ and eliminated the epigraph variable~$v$.

		Although our next steps are also valid for~\eqref{Eq:RegularizationBPDN3}, we will discard the last term of its objective, for simplicity. That is, we will solve instead:
		\begin{equation}\label{Eq:RegularizationBPDNFinal}
			\begin{array}{ll}
			  	\underset{x}{\text{minimize}} & (\frac{1}{2}+\frac{\delta}{4})\|Ax - b\|^2 + \beta \|x\|_1 + \frac{\delta}{2}\|x\|^2\,.
			  \end{array}
		\end{equation}
		We now introduce an auxiliary variable~$y \in \mathbb{R}^m$ and write~\eqref{Eq:RegularizationBPDNFinal} equivalently as
		\begin{equation}\label{Eq:ManipulationBPDN1}
			\begin{array}{ll}
			  	\underset{x}{\text{minimize}} & (\frac{1}{2}+\frac{\delta}{4})\|y\|^2 + \beta \|x\|_1 + \frac{\delta}{2}\|x\|^2 \\
			  	\text{subject to} & Ax = b + y\,.
			  \end{array}
		\end{equation}
		Associate a dual variable~$\lambda \in \mathbb{R}^m$ to the constraint of~\eqref{Eq:ManipulationBPDN1} and compute the dual problem:
		\begin{align}
			&
			  	\underset{\lambda}{\text{maximize}} \,\,\, b^\top \lambda + \inf_y \,\, \Bigl((\frac{1}{2}+\frac{\delta}{4})\|y\|^2 + \lambda^\top y\Bigr)
			  	+
			  	\inf_x \,\, \Bigl(\beta \|x\|_1 + \frac{\delta}{2}\|x\|^2 - \lambda^\top A x\Bigr)
			\label{Eq:ManipulationBPDN2}
			\\
			\Longleftrightarrow
			\qquad&
			  	\underset{\lambda}{\text{minimize}} \,\,\, b^\top \lambda + \frac{1}{2+\delta}\|\lambda\|^2 + \sum_{p=1}^P \bar{h}_p^\star(A_p^\top \lambda)
			\label{Eq:ManipulationBPDN3}
			\\
			\Longleftrightarrow
			\qquad&
			    \underset{\lambda}{\text{minimize}} \,\,\, \sum_{p=1}^P \,\biggl[\bar{h}_p^\star(A_p^\top \lambda) + \frac{1}{P}b^\top \lambda + \frac{1}{(2+\delta)P}\|\lambda\|^2\biggr]\,,
			\label{Eq:ManipulationBPDN4}
		\end{align}
		which has the format of~\eqref{Eq:GlobalProblem} with $f_p(\lambda) = \bar{h}_p^\star(A_p^\top \lambda) + (1/P)b^\top \lambda + (1/((2+\delta)P))\|\lambda\|^2$ as the function of each node~$p$. From~\eqref{Eq:ManipulationBPDN2} to~\eqref{Eq:ManipulationBPDN3}, we switched from a maximization problem to a minimization one, and used the fact that the infimum problem in~$y$ has a closed-form expression. Also, the infimum in~$x$ was decomposed into blocks, and $\bar{h}_p^\star$ denotes the convex conjugate of the function~$\bar{h}_p(x_p) = \beta \|x_p\|_1 + (\delta/2)\|x_p\|^2$. From~\eqref{Eq:ManipulationBPDN3} to~\eqref{Eq:ManipulationBPDN4}, we just grouped terms. Note that solving~\eqref{Eq:ManipulationBPDN4} is not equivalent to solving BPDN for two reasons: first because we used regularization for which there are no exactness results and, second, because we ignored the quartic term in~\eqref{Eq:RegularizationBPDN3}.

		\mypar{Column partition: reversed lasso}
		Regarding reversed lasso~\eqref{Eq:ReverseLasso}, we will also regularize it and, again, we will not have any exact regularization guarantee. To do the regularization the same way as before, we first rewrite it with the format of~\eqref{Eq:ConicProgram}:
		\begin{equation}\label{Eq:RegularizationReversedLasso}
			\begin{array}{ll}
				\underset{x,t,u,v}{\text{minimize}} & 1_n^\top t \\
				\text{subject to} & \|u\| \leq v \\
													& -t \leq x \leq t \\
													& u = Ax - b \\
				                  & v = \sigma\,,
			\end{array}
		\end{equation}
		where~$t \in \mathbb{R}^n$ is, again, an epigraph variable, and~$u \in \mathbb{R}^m$ and~$v \in \mathbb{R}$ are auxiliary variables, introduced to make a cone appear. Problem~\eqref{Eq:RegularizationReversedLasso} has indeed the same structure as~\eqref{Eq:ConicProgram}, since the objective is linear, the last two constraints are linear equalities, and the cone~$\mathcal{K}$ is the Cartesian product of two cones: $\mathcal{K} = \mathcal{K}_{x,t} \times \mathcal{K}_{u,v}$, where $\mathcal{K}_{x,t} = \{(x,t)\,:\, -t \leq x \leq t\}$ is polyhedral, and~$\mathcal{K}_{u,v} = \{(u,v)\,:\, \|u\| \leq v\}$ is the Lorenz cone and, thus, not polyhedral. By regularizing~\eqref{Eq:RegularizationReversedLasso} with the function $\phi(z) = (1/4)\|z\|^2$, we obtain
		\begin{align}
			&
				\begin{array}{ll}
					\underset{x,t,u,v}{\text{minimize}} & 1_n^\top t + \frac{\delta}{4}\Bigl(\|x\|^2 + \|t\|^2 + \|u\|^2 + v^2\Bigr) \\
					\text{subject to} & \|u\| \leq v\,,\quad v = \sigma \\
														& -t \leq x \leq t \\
														& u = Ax - b \\
				\end{array}
			\label{Eq:RegularizarionReversedLasso1}
			\\
			\Longleftrightarrow
			\qquad&
			  \begin{array}[t]{ll}
					\underset{x}{\text{minimize}} & \frac{\delta}{4}\|x\|^2 \\
				\end{array}
				+
				\begin{array}[t]{cl}
					\underset{t}{\inf} & 1_n^\top t + \frac{\delta}{4}\|t\|^2 \\
					\text{s.t.} & -t \leq x \leq t
				\end{array}
				+
				\begin{array}[t]{cl}
					\underset{u}{\inf} & \frac{\delta}{4}\|u\|^2 \\
					\text{s.t.} & u = Ax - b \\
					            & \|u\| \leq \sigma
				\end{array}
			\label{Eq:RegularizarionReversedLasso2}
			\\
			\Longleftrightarrow
			\qquad&
			  \begin{array}[t]{ll}
			  	\underset{x}{\text{minimize}} & \|x\|_1 + \frac{\delta}{2}\|x\|^2  + \frac{\delta}{4}\|Ax - b\|^2\\
			  	\text{subject to} & \|Ax - b\| \leq \sigma\,.
			  \end{array}
			\label{Eq:RegularizarionReversedLasso3}
		\end{align}
		From~\eqref{Eq:RegularizarionReversedLasso1} to~\eqref{Eq:RegularizarionReversedLasso2}, we used the constraint~$v = \sigma$. From~\eqref{Eq:RegularizarionReversedLasso2} to~\eqref{Eq:RegularizarionReversedLasso3}, we used~\eqref{Eq:ProofRegularizationBPIdentity} and the fact that
		$$
			\begin{array}[t]{cl}
					\underset{u}{\inf} & \frac{\delta}{4}\|u\|^2 \\
					\text{s.t.} & u = Ax - b \\
					            & \|u\| \leq \sigma
			\end{array}
			\,\,
			=
			\,\,\,\,
			\text{i}_{\|Ax - b\| \leq \sigma}(x) + \frac{\delta}{4}\|Ax - b\|^2\,.
		$$
		In contrast with BP and similarly to BPDN, there is no proof that~\eqref{Eq:RegularizarionReversedLasso3} is an exact regularization of reversed lasso, although experimental results in~\cite{Fridlander07-ExactRegularizationCVXPrograms} suggest that exact regularization might occur for the Lorenz cone. In our experimental results, discussed later, we solved~\eqref{Eq:RegularizarionReversedLasso3} using $\delta = 10^{-2}$ and the corresponding solutions never differed more than $0.5\%$ from the ``true'' solution.

		We next introduce an auxiliary variable~$y \in \mathbb{R}^m$ in~\eqref{Eq:RegularizarionReversedLasso3}, yielding
		\begin{equation}\label{Eq:ReversedLassoDualization1}
			\begin{array}{ll}
				\underset{x,y}{\text{minimize}} & \|x\|_1 + \frac{\delta}{2}\|x\|^2  + \frac{\delta}{4}\|y\|^2\\
				\text{subject to} & \|y\| \leq \sigma \\
				                  & y = Ax - b\,.
			\end{array}
		\end{equation}
		Now, associate a dual variable~$\lambda \in \mathbb{R}^m$ to the last constraint of~\eqref{Eq:ReversedLassoDualization1} and compute the dual problem (without dualizing the first constraint). This gives
		\begin{align}
			&
			  \underset{\lambda}{\text{maximize}} \,\,\,\, b^\top \lambda
			  +
			  \begin{array}[t]{cl}
			  	\underset{y}{\inf} & \lambda^\top y  + \frac{\delta}{4}\|y\|^2\\
			  	\text{s.t.}        & \|y\| \leq \sigma
			  \end{array}
			  +
			  \,\,
			  \underset{x}{\inf} \,\, \Bigl[\|x\|_1 + \frac{\delta}{2}\|x\|^2 - \lambda^\top Ax\Bigr]
			\label{Eq:ReversedLassoDualization2}
			\\
			\Longleftrightarrow
			\qquad&
			  \underset{\lambda}{\text{maximize}} \,\,\,\, b^\top \lambda
			  -
			  \sigma \|\lambda\|
			  +
			  \,\,
			  \underset{x}{\inf} \,\, \Bigl[\|x\|_1 + \frac{\delta}{2}\|x\|^2 - \lambda^\top Ax\Bigr]
			\label{Eq:ReversedLassoDualization3}
			\\
			\Longleftrightarrow
			\qquad&
			  \underset{\lambda}{\text{minimize}} \,\,\,\, \sum_{p=1}^P \Bigl( h_p^\star(A_p^\top \lambda) + \frac{\sigma}{P}\|\lambda\| - \frac{1}{P}b^\top \lambda\Bigr)\,.
			\label{Eq:ReversedLassoDualization4}
		\end{align}
		From~\eqref{Eq:ReversedLassoDualization2} to~\eqref{Eq:ReversedLassoDualization3}, we noticed that the problem in~$y$ has a closed-form solution that can be computed by solving its dual problem. Namely, its optimal objective is $(\delta/2)\sigma^2 - \sigma \|\lambda\|$. From~\eqref{Eq:ReversedLassoDualization3} to~\eqref{Eq:ReversedLassoDualization4}, we made the column partition explicit and took exactly the same steps as in the manipulations~\eqref{Eq:BPDual1}-\eqref{Eq:BPDual4}, since the problem in~$x$ is exactly the same as in~\eqref{Eq:BPDual1}. In fact, notice that by setting~$\sigma = 0$ in~\eqref{Eq:ReversedLassoDualization4} we obtain~\eqref{Eq:BPDual4}, exactly the same way we obtain BP from reversed lasso in the primal domain. Problem~\eqref{Eq:ReversedLassoDualization4} has the format of~\eqref{Eq:GlobalProblem} and, similarly to BP, the $p$th block-component of the primal solution of~\eqref{Eq:RegularizarionReversedLasso3} can be obtained at node~$p$ after solving the dual problem~\eqref{Eq:ReversedLassoDualization4}: the expression for each component is~\eqref{Eq:BPSoftSolution}, the same as for BP. However, for the reversed lasso, we do not have the theoretical guarantee that, for a small enough~$\delta$, the solution of the regularized problem~\eqref{Eq:RegularizarionReversedLasso3} is also a solution of the original~\eqref{Eq:ReverseLasso}.

		\mypar{Column partition: lasso}
		Finally we address lasso. As with BPDN and the reversed lasso, the regularization we use here is not proven to be exact. Again, we start by rewriting~\eqref{Eq:lasso} as~\eqref{Eq:ConicProgram}:
		\begin{align}
			  \begin{array}[t]{ll}
					\underset{x}{\text{minimize}} & \frac{1}{2}\|Ax - b\|^2 \\
					\text{subject to} & \|x\|_1 \leq \gamma
				\end{array}
			\qquad
			&\Longleftrightarrow
			\qquad
				\begin{array}[t]{ll}
					\underset{x}{\text{minimize}} & \|Ax - b\| \\
					\text{subject to} & \|x\|_1 \leq \gamma
				\end{array}
			\notag
			\\
			&\Longleftrightarrow
			\qquad
			  \begin{array}[t]{ll}
			  	\underset{x,t,u,v}{\text{minimize}} & v \\
			  	\text{subject to} & \|u\| \leq v \\
			  	                  & \|x\|_1 \leq t \\
			  	                  & u = Ax - b \\
			  	                  & t = \gamma\,,
			  \end{array}
			\label{Eq:RegularizationLasso1}
		\end{align}
		where, for simplicity, we did not represent the constraint~$\|x\|_1 \leq t$ as a set of linear inequalities. This can indeed be done by writing~$2^n$ inequalities of the form $r_i^\top x \leq t$, where each~$r_i \in \mathbb{R}^n$ has $\pm 1$ in its entries; there are~$2^n$ such vectors. 	Therefore, \eqref{Eq:RegularizationLasso1} has the same format as~\eqref{Eq:ConicProgram}, where the objective is linear, the last two constraints are linear equations, and the first two constraints represent the cone~$\mathcal{K}$, which is the Cartesian product of a polyhedral closed convex cone $\mathcal{K}_{x,t} = \{(x,t)\,:\, r_i^\top x \leq t,\, i=1,\ldots,2^n\}$ and the Lorenz cone $\mathcal{K}_{u,v} = \{(u,v)\,:\, \|u\|\leq v\}$. We now regularize problem~\eqref{Eq:RegularizationLasso1} the same way we regularized the previous problems:
		\begin{align}
			&
				\begin{array}[t]{ll}
					\underset{x,t,u,v}{\text{minimize}} & v + \frac{\delta}{4}\Bigl(\|x\|^2 + t^2 + \|u\|^2 + v^2\Bigr)\\
					\text{subject to} & \|u\| \leq v \\
														& \|x\|_1 \leq t \\
														& u = Ax - b \\
														& t = \gamma
				\end{array}
			\label{Eq:RegularizationLasso2}
			\\
			\Longleftrightarrow
			\qquad&
				\begin{array}[t]{ll}
					\underset{x}{\text{minimize}} & \frac{\delta}{4}\|x\|^2 + \frac{\delta}{4}\|Ax - b\|^2\\
					\text{subject to} & \|x\|_1 \leq \gamma
				\end{array}
				+\,\,
				\begin{array}[t]{cl}
					\underset{v}{\inf} & \frac{\delta}{4}v^2 + v \\
					\text{s.t.} & \|Ax - b\| \leq v
				\end{array}
			\label{Eq:RegularizationLasso3}
			\\
			\Longleftrightarrow
			\qquad&
				\begin{array}[t]{ll}
					\underset{x}{\text{minimize}} & \|Ax - b\| + \frac{\delta}{2}\|Ax - b\|^2 + \frac{\delta}{4}\|x\|^2 \\
					\text{subject to} & \|x\|_1 \leq \gamma\,.
				\end{array}
			\label{Eq:RegularizationLasso4}
		\end{align}
		From~\eqref{Eq:RegularizationLasso2} to~\eqref{Eq:RegularizationLasso3}, we eliminated the linear constraints. From~\eqref{Eq:RegularizationLasso3} to~\eqref{Eq:RegularizationLasso4}, we used the fact that the optimal value of the problem in~$v$, for a fixed~$x$, is $(\delta/4)\|Ax - b\|^2 + \|Ax - b\|$. Now, introduce an auxiliary variable~$y \in \mathbb{R}^m$ in~\eqref{Eq:RegularizationLasso4}:
		$$
			\begin{array}{ll}
				\underset{x,y}{\text{minimize}} & \|y\| + \frac{\delta}{2}\|y\|^2 + \frac{\delta}{4}\|x\|^2 \\
				\text{subject to} & \|x\|_1 \leq \gamma \\
				                  & y = Ax - b\,,
			\end{array}
		$$
		and compute the dual problem by dualizing both constraints ($\mu$ and~$\lambda$ will be the dual variables associated to the first and second constraints, respectively). We get
		\begin{align}
			&
				\begin{array}[t]{ll}
					\underset{\lambda,\mu}{\text{maximize}} & b^\top \lambda - \gamma \mu \\
					\text{subject to} & \mu \geq 0
				\end{array}
				+
				\underset{y}{\inf}\,\, \biggl[\|y\| + \frac{\delta}{2}\|y\|^2 + \lambda^\top y\biggr]
				+
				\underset{x}{\inf}\,\, \biggl[\mu \|x\|_1 + \frac{\delta}{4}\|x\|^2 - \lambda^\top Ax\biggr]
			\notag
			\\
			\Longleftrightarrow
			\qquad&
			  \begin{array}[t]{ll}
					\underset{\lambda,\mu}{\text{minimize}} & \gamma \mu  - b^\top \lambda
						+ g^\star(-\lambda) \\
					\text{subject to} & \mu \geq 0
				\end{array}
				+
				\underset{x}{\sup} \,\,\bigg[(A^\top \lambda)^\top x - \mu \|x\|_1 - \frac{\delta}{4}\|x\|^2\biggr]
			\label{Eq:DualizationLasso1}
			\\
			\Longleftrightarrow
			\qquad&
			  \begin{array}[t]{ll}
					\underset{\lambda,\mu}{\text{minimize}} & \gamma \mu  - b^\top \lambda
						+ g^\star(-\lambda) \\
					\text{subject to} & \mu \geq 0
				\end{array}
				+
				\sum_{p=1}^P \, \underset{x_p}{\sup} \,\,\bigg[(A_p^\top \lambda)^\top x_p - \mu \|x_p\|_1 - \frac{\delta}{4}\|x_p\|^2\biggr]
			\label{Eq:DualizationLasso2}
			\\
			\Longleftrightarrow
			\qquad&
					\underset{\lambda,\mu}{\text{minimize}} \,\,\, \sum_{p=1}^P
						\biggl(
							\frac{1}{P}\bigl(\gamma \mu  - b^\top \lambda + g^\star(-\lambda) \bigr)
							+
							l_p(A_p^\top \lambda, \mu) + \text{i}_{\{\mu \geq 0\}}(\mu)
						\biggr)\,.
			\label{Eq:DualizationLasso3}
		\end{align}
		In~\eqref{Eq:DualizationLasso1}, $g^\star$ is the convex conjugate of $g(y) = \|y\| + (\delta/2)\|y\|^2$. We show in \aref{App:ConjugateFunctions} that
		\begin{equation}\label{Eq:ConvexConjugateNormPlusNormSquared}
			g^\star(\eta)
			=
			\underset{x}{\sup}\,\,\, \biggl(\eta^\top x - \|x\| - \frac{\delta}{2}\|x\|^2\biggr)
			=
			\left\{
				\begin{array}{ll}
					0 &,\,\, \|\eta\| \leq 1 \\
					\frac{1}{2\delta}\Bigl(\|\lambda\|^2 - 2\|\lambda\| + 1\Bigr) &,\,\, \|\eta\| > 1\,.
				\end{array}
			\right.
		\end{equation}
		From~\eqref{Eq:DualizationLasso1} to~\eqref{Eq:DualizationLasso2}, we just made the column partition explicit and, in~\eqref{Eq:DualizationLasso3}, we defined
		\begin{equation}\label{Eq:DualizationLasso4}
			l_p(\eta,\mu) = \underset{x_p}{\sup} \biggl[\eta^\top x_p - \mu\|x_p\|_1 - \frac{\delta}{4}\|x_p\|^2\biggr]\,.
		\end{equation}
		Note that~\eqref{Eq:DualizationLasso3} has the same format as~\eqref{Eq:GlobalProblem}. Because of regularization, the objective in the supremum problem in~\eqref{Eq:DualizationLasso4} is strictly concave, which means that, after the nodes agree on an optimal dual solution $(\lambda^\star,\mu^\star)$, the $p$th component of the primal solution of~\eqref{Eq:RegularizationLasso4} will be available at the $p$th node; see \aref{App:ConjugateFunctions} for the particular expression.

	\section{Algorithm derivation}
	\label{Sec:GC:Derivation}

		We now present our algorithm for the global class~\eqref{Eq:GlobalProblem}. As mentioned before, our strategy consists of reformulating~\eqref{Eq:GlobalProblem} as~\eqref{Eq:RelatedWorkEdgeBasedReformulation} and then we applying the multi-block ADMM. For convenience, we recall reformulation~\eqref{Eq:RelatedWorkEdgeBasedReformulation}
		\begin{equation}\label{Eq:GlobalEdgeBasedReformulation}
			\begin{array}{ll}
				\underset{x_1,\ldots,x_P}{\text{minimize}} & f_1(x_1) + f_2(x_2) + \cdots + f_P(x_P)\\
				\text{subject to} & x_i = x_j\,,\quad (i,j) \in \mathcal{E}\,,
			\end{array}
		\end{equation}
		where~$x_p \in \mathbb{R}^n$ is the copy of the original variable~$x \in \mathbb{R}^n$ and is held by node~$p$. The optimization variable is now the collection of all the copies: $\bar{x}=(x_1,\ldots,x_P) \in (\mathbb{R}^n)^P$. All these copies are forced to be equal through the constraints of~\eqref{Eq:GlobalEdgeBasedReformulation}, which state that, for each edge $(i,j)$ in the network, the copies of nodes~$i$ and~$j$ are equal. Since by \assref{Ass:GPConnectedStatic} the network is assumed connected, there are no isolated nodes and, hence, all the copies are equal. Consequently, problems~\eqref{Eq:GlobalProblem} and~\eqref{Eq:GlobalEdgeBasedReformulation} are equivalent.

		\mypar{Matrix representation}
		Recall that, according to \assref{Ass:GPColoring}, we assume the network has a coloring scheme~$\mathcal{C}$ with~$C = |\mathcal{C}|$ colors. We use $\mathcal{C}_c \subset \mathcal{V}$ to denote the set of nodes that have color~$c \in \mathcal{C}$, and $\mathcal{C}(p)$ to denote the color of node~$p$. Also, the number of nodes with color~$c$ is represented with $C_c = |\mathcal{C}_c|$. Without loss of generality and to simplify our derivation, we will assume that the nodes are numbered according to this coloring scheme as: $\mathcal{C}_1 = \{1,2,\ldots,C_1\}$, $\mathcal{C}_2 = \{C_1+1,C_1+2,\ldots,C_1+C_2\}$, \ldots, i.e., the first~$C_1$ nodes have color~$1$, the next~$C_2$ nodes have color~$\mathcal{C}_2$, and so on. Now, notice that the constraints in problem~\eqref{Eq:GlobalEdgeBasedReformulation} can be written in matrix format as $(B^\top \otimes I_n) \bar{x} = 0$, where~$B \in \mathbb{R}^{P \times E}$ is the node-arc incidence matrix, $\otimes$ is the Kronecker product, and~$I_n$ is the identity matrix in~$\mathbb{R}^n$. In the node-arc incidence matrix, each column is associated to an edge of the network~$(i,j) \in \mathcal{E}$, with $1$ in the $i$th entry, $-1$ in the $j$th entry, and zeros in the remaining entries. Given our assumption on the ordering of the nodes and the coloring scheme, we can write
		$
			(B^\top \otimes I_n) \bar{x} = (B_1^\top \otimes I_n) \bar{x}_1 + (B_2^\top \otimes I_n) \bar{x}_2 + \cdots + (B_C^\top \otimes I_n) \bar{x}_C
		$,
		where~$\bar{x}_c$ collects the copies of the nodes in~$\mathcal{C}_c$, i.e.,
		$$
			\bar{x} = (
				\underbrace{x_1,\ldots,x_{C_1}}_{\bar{x}_1},
				\underbrace{x_{C_1 + 1},\ldots,x_{C_1+C_2}}_{\bar{x}_2},
				\ldots,
				\underbrace{x_{P-C_p + 1}, \ldots,x_P}_{\bar{x}_C})\,,
		$$
		and the matrix~$B$ is partitioned by rows accordingly. 	Therefore, \eqref{Eq:GlobalEdgeBasedReformulation} can be written as
		\begin{equation}\label{Eq:GlobalEdgeBasedReformulationMatrix}
			\begin{array}{ll}
				\underset{(x_1,\ldots,x_P)}{\text{minimize}} & \sum_{p \in \mathcal{C}_1} f_p(x_p) + \cdots + \sum_{p \in \mathcal{C}_C} f_p(x_p) \\
				\text{subject to} & (B_1^\top \otimes I_n) \bar{x}_1 + (B_2^\top \otimes I_n) \bar{x}_2  + \cdots + (B_C^\top \otimes I_n) \bar{x}_C= 0\,,
			\end{array}
		\end{equation}
		where we also grouped the terms in the objective according to the colors of the nodes. We next apply the multi-block ADMM to~\eqref{Eq:GlobalEdgeBasedReformulationMatrix}.

	\mypar{Applying the multi-block ADMM}
	We introduced the multi-block ADMM in \ssref{SubSec:AugmentedLagrangianMethods}. Our reformulations of~\eqref{Eq:GlobalProblem} resulted in problem~\eqref{Eq:GlobalEdgeBasedReformulationMatrix}, which has the  format of~\eqref{Eq:RelatedWorkADMMProbExtended}, the problem the multi-block ADMM solves. If we apply the multi-block ADMM~\eqref{Eq:RelatedWorkADMMProbExtendedADMMIter1}-\eqref{Eq:RelatedWorkADMMProbExtendedADMMIter4} directly to~\eqref{Eq:GlobalEdgeBasedReformulationMatrix}, we will see that the update of~$\bar{x}_c$ yields~$C_c$ independent problems which can consequently be solved in parallel. For example, the first block variable~$\bar{x}_1$ is updated as
	\begin{equation}\label{Eq:GlobalDerivation1}
		\bar{x}_1^{k+1} = \underset{\bar{x}_1 = (x_1,\ldots,x_{C_1})}{\arg\min}
		\sum_{p \in \mathcal{C}_1}\, f_p(x_p) + {\lambda^k}^\top (B_1^\top \otimes I_n) \bar{x}_1 + \frac{\rho}{2}\biggl\|(B_1^\top \otimes I_n)\bar{x}_1 + \sum_{c=2}^C (B_c^\top \otimes I_n) \bar{x}_c^k\biggr\|^2\,,
	\end{equation}
	where the terms not depending~$\bar{x}_1$ were dropped. Developing the quadratic term in~\eqref{Eq:GlobalDerivation1},
	\begin{multline}\label{Eq:GlobalDerivation2}
		\biggl\|(B_1^\top \otimes I_n)\bar{x}_1 + \sum_{c=2}^C (B_c^\top \otimes I_n) \bar{x}_c^k\biggr\|^2\\
		=
		\bar{x}_1^\top (B_1 B_1^\top \otimes I_n) \bar{x}_1 + 2\,\bar{x}_1^\top \biggl(\sum_{c=2}^C (B_1B_c^\top \otimes I_n) \bar{x}_c^k\biggr) + \biggl\|\sum_{c=1}^C (B_c^\top \otimes I_n) \bar{x}_c^k\bigg\|^2\,.
	\end{multline}
	In the first term of~\eqref{Eq:GlobalDerivation2}, $B_1 B_1^\top$ is the first diagonal block (of size $C_1 \times C_1$) of the network Laplacian. Because the first~$C_1$ nodes have the same color and, hence, cannot be neighbors, the matrix $B_1 B_1^\top$ is diagonal. The $p$th entry in the diagonal is the degree~$D_p$ of node~$p$. Therefore, the first term of~\eqref{Eq:GlobalDerivation2} can be written as $\bar{x}_1^\top (B_1 B_1^\top \otimes I_n) \bar{x}_1 = \sum_{p \in \mathcal{C}_1} D_p \|x_p\|^2$. In the second term, $B_1B_c^\top$ is an off-diagonal block of the Laplacian and depicts the links between the nodes with color~$1$ and the nodes with color~$c$. Namely, if node~$i$ has color~$1$ and node~$j$ has color~$c$ and they are neighbors, i.e., $(i,j) \in \mathcal{E}$, then the $ij$th entry of~$B_1B_c^\top$ will be~$-1$. Therefore, the second term is written equivalently as $2\,\bar{x}_1^\top \Bigl(\sum_{c=2}^C (B_1B_c^\top \otimes I_n) \bar{x}_c^k\Bigr) = -2\sum_{p \in \mathcal{C}_1}\sum_{j \in \mathcal{N}_p} x_p^\top x_j^k$. Finally, the last term of~\eqref{Eq:GlobalDerivation2} does not depend on~$\bar{x}_1$ and hence can be dropped. These simplifications render problem~\eqref{Eq:GlobalDerivation1} equivalent to
	\begin{equation}\label{Eq:GlobalDerivation3}
		\bar{x}_1^{k+1} = \underset{\bar{x}_1 = (x_1,\ldots,x_{C_1})}{\arg\min} \,\,
		\sum_{p \in \mathcal{C}_1} f_p(x_p) + \Bigl(\gamma_p^k - \rho \sum_{j \in \mathcal{N}_p}x_j^k\Bigr)^\top x_p + \frac{\rho D_p}{2}\|x_p\|^2\,,
	\end{equation}
	where $\gamma_p^k := \sum_{j \in \mathcal{N}_p} \lambda_{pj}^k	$	was obtained from the second term of~\eqref{Eq:GlobalDerivation1} as
	\begin{align}
		  {\lambda^k}^\top (B_1^\top \otimes I_n)\bar{x}_1
		=
		  ((B_1 \otimes I_n)\lambda^k)^\top \bar{x}_1
		=
		  \sum_{p \in \mathcal{C}_1}\underbrace{\sum_{j \in \mathcal{N}_p} {\lambda_{pj}^k}^\top}_{{\gamma_p^k}\top} x_j^k\,.
		\label{Eq:GlobalDerivation4}
	\end{align}
	In the last equality in~\eqref{Eq:GlobalDerivation4}, we used the fact that the $p$th entry of the vector~$(B_1 \otimes I_n)\lambda^k$ is given by $\sum_{j \in \mathcal{N}_p} \lambda_{pj}^k$. Note that we decomposed the dual variable as $(\ldots,\lambda_{ij},\ldots)$, where~$\lambda_{ij}$ is associated to the constraint $x_i = x_j$, i.e., the edge between node~$i$ and node~$j$. Given our convention that $(i,j) \in \mathcal{E}$ implies that $i < j$ (see \ssref{SubSec:CommunicationNetwork}), $\lambda_{ij}$ is only defined for~$i<j$. It is clear that problem~\eqref{Eq:GlobalDerivation3} decomposes into $C_1$ problems that can be solved in parallel. Namely, node~$p$ updates its copy~$x_p$ as
	\begin{align}
		  x_p^{k+1}
		&=
		  \underset{x_p}{\arg\min} \,\,\, f_p(x_p) + \Bigl(\gamma_p^k - \rho \sum_{j \in \mathcal{N}_p}x_j^k\Bigr)^\top x_p + \frac{\rho D_p}{2}\|x_p\|^2
		\notag
		\\
		&=
		  \text{prox}_{\tau_p f_p}\biggl( \frac{1}{D_p}\sum_{j \in \mathcal{N}_p} x_j^k - \tau_p \gamma_p^k\biggr)\,,
		\label{Eq:GlobalDerivation5}
	\end{align}
	where the prox operator was defined in~\eqref{Eq:DefinitionProxOperator} and $\tau_p = 1/(\rho D_p)$. The problems with respect to the other block variables can be decomposed into parallel problems the same way. The only difference is the definition of~$\gamma_p^k$, which is different due to the nodes' ordering. Its general definition is
	\begin{equation}\label{Eq:GlobalDerivation6}
		\gamma_p^k =
		\sum_{\begin{subarray}{c}j \in \mathcal{N}_p \\ p < j\end{subarray}} \lambda_{pj}^k
		-
		\sum_{\begin{subarray}{c}j \in \mathcal{N}_p \\ p > j\end{subarray}} \lambda_{jp}^k \,.
	\end{equation}
	Note that, from~\eqref{Eq:GlobalDerivation5}, each node~$p$ needs to know the aggregate sum~$\gamma_p^k$, but not the individual~$\lambda_{ij}$'s. According to the multi-block ADMM iterations, namely~\eqref{Eq:RelatedWorkADMMProbExtendedADMMIter4}, each~$\lambda_{ij}$, for~$(i,j) \in \mathcal{E}$, is updated as $	\lambda_{ij}^{k+1} = \lambda_{ij}^k + \rho \,(x_i^{k+1} - x_j^{k+1})$. 	Replacing this update in the definition of~$\gamma_p^k$ in~\eqref{Eq:GlobalDerivation6}, we get
	\begin{align*}
		  \gamma_p^{k+1}
		&=
		  \sum_{\begin{subarray}{c}j \in \mathcal{N}_p \\ p < j\end{subarray}} \lambda_{pj}^{k+1}
			-
			\sum_{\begin{subarray}{c}j \in \mathcal{N}_p \\ p > j\end{subarray}} \lambda_{jp}^{k+1}
		\\
		&=
		  \sum_{\begin{subarray}{c}j \in \mathcal{N}_p \\ p < j\end{subarray}} \lambda_{pj}^k
		  +
		  \rho
		  \sum_{\begin{subarray}{c}j \in \mathcal{N}_p \\ p < j\end{subarray}}
		  \Bigl(x_p^{k+1} - x_j^{k+1}\Bigr)
		  -
		  \sum_{\begin{subarray}{c}j \in \mathcal{N}_p \\ p > j\end{subarray}} \lambda_{jp}^k
		  -
			\rho
			\sum_{\begin{subarray}{c}j \in \mathcal{N}_p \\ p > j\end{subarray}}
			\Bigl(x_j^{k+1} - x_p^{k+1}\Bigr)
		\\
		&=
			\underbrace{
		  \sum_{\begin{subarray}{c}j \in \mathcal{N}_p \\ p < j\end{subarray}} \lambda_{pj}^k
		  -
		  \sum_{\begin{subarray}{c}j \in \mathcal{N}_p \\ p > j\end{subarray}} \lambda_{jp}^k
		  }_{= \gamma_p^k}
		  +
		  \rho
		  \sum_{\begin{subarray}{c}j \in \mathcal{N}_p \\ p < j\end{subarray}}
		  \Bigl(x_p^{k+1} - x_j^{k+1}\Bigr)
		  +
		  \rho
		  \sum_{\begin{subarray}{c}j \in \mathcal{N}_p \\ p > j\end{subarray}}
		  \Bigl(x_p^{k+1} - x_j^{k+1}\Bigr)
		\\
		&=
		  \gamma_p^k + \rho \sum_{j \in \mathcal{N}_p} \Bigl(x_p^{k+1} - x_j^{k+1}\Bigr)\,.
	\end{align*}

	\begin{algorithm}
		\small
		\caption{Algorithm for the global class (D-ADMM)}
		\label{Alg:GlobalClass}
		\algrenewcommand\algorithmicrequire{\textbf{Initialization:}}
		\begin{algorithmic}[1]
			\Require Choose $\rho \in \mathbb{R}$; for all $p \in \mathcal{V}$, set $x_p^0 = \gamma_p^0 = 0_n \in \mathbb{R}^n$ and $\tau_p = 1/(\rho D_p)$; set $k=0$

			\Repeat
			\ForAll{$c=1,\ldots,C$}
			\ForAll{$p \in \mathcal{C}_c$ [in parallel]}

			\vspace{0.14cm}

			\State Compute the average
				$$
					z_p^k =\frac{1}{D_p}\biggl( \sum_{\begin{subarray}{c}j \in \mathcal{N}_p\\ \mathcal{C}(j) < \mathcal{C}(p) \end{subarray}} x_j^{k+1} + \sum_{\begin{subarray}{c}j \in \mathcal{N}_p\\ \mathcal{C}(j) > \mathcal{C}(p) \end{subarray}} x_j^{k} \biggr)
				$$
				\label{SubAlg:GlobalAverage}

			\Statex

			\State Update
				$
					x_p^{k+1} = \text{prox}_{\tau_p f_p}\Bigl( z_p^k - \tau_p \gamma_p^k\Bigr)
				$
				and send $x_p^{k+1}$ to neighbors $\mathcal{N}_p$
				\label{SubAlg:GlobalProx}

			\EndFor

			\EndFor

			\vspace{0.2cm}

			\ForAll{$p \in \mathcal{V}$ [in parallel]}

			\State Update the dual variable $\gamma_p^{k+1} = \gamma_p^k + \rho \sum_{j \in \mathcal{N}_p} \Bigl(x_p^{k+1} - x_j^{k+1}\Bigr)$
				\label{SubAlg:GlobalDual}
			\EndFor

			\vspace{0.2cm}

			\State $k \gets k+1$

			\Until{some stopping criterion is met}
		\end{algorithmic}
  \end{algorithm}

	\mypar{D-ADMM: algorithm for the global class} The resulting algorithm is shown as Algorithm~\ref{Alg:GlobalClass}, which we named D-ADMM in~\cite{Mota13-DADMM}, after Distributed-ADMM.
	Algorithm~\ref{Alg:GlobalClass} solves~\eqref{Eq:GlobalEdgeBasedReformulationMatrix}, and hence~\eqref{Eq:GlobalProblem}, by creating~$C$ groups of nodes according to the coloring scheme. The nodes within each group perform the same tasks in parallel, as illustrated before in Figure~\ref{Fig:IllustrationAlgs}. These tasks consist of computing the average of the solution estimates by the neighbors (step~\ref{SubAlg:GlobalAverage}), computing the prox of the scaled function~$\tau_p f_p$ at the point indicated in step~\ref{SubAlg:GlobalProx}, and then sending the new solution estimate to the neighbors. Note that in the computation of the average~$z_p^k$ of a given node~$p$, in step~\ref{SubAlg:GlobalAverage}, there are two kinds of estimates: ones that were computed in the current iteration~$k$, i.e., $x_j^{k+1}$ and ones that were computed in the previous iteration~$k-1$, i.e., $x_j^k$. The first kind are estimates of the neighbors with a color smaller than the color of node~$p$, that is, $\mathcal{C}(j)<\mathcal{C}(p)$. Node~$p$ has access to these estimates because the nodes with smaller colors have performed steps~\ref{SubAlg:GlobalAverage} and~\ref{SubAlg:GlobalProx} before. The second kind are estimates of the neighbors with a color larger than the color of node~$p$, $\mathcal{C}(j) > \mathcal{C}(p)$, and were transmitted in the previous iteration. Note that, in contrast with its derivation, Algorithm~\ref{Alg:GlobalClass} does not assume that the nodes are ordered according to their colors, thanks to the use of inequalities $\mathcal{C}(j) \lessgtr \mathcal{C}(p)$ instead of $j \lessgtr p$. After all nodes perform step~\ref{SubAlg:GlobalProx}, the dual variables~$\gamma_p$ are updated simultaneously at all nodes, as described in step~\ref{SubAlg:GlobalDual}.

	\begin{figure}
		\centering
		\subfigure[Undirected graph]{\label{SubFig:SelfCoordinationUndirected}
			\psscalebox{0.9}{
			\begin{pspicture}(4.9,5)
				\def\nodesimp{
					\pscircle*[linecolor=black!65!white](0,0){0.3}
        }
        \def\nodeB{
            \pscircle*[linecolor=black!15!white](0,0){0.3}
        }
        \def\nodeC{
            \pscircle[fillstyle=vlines*,linecolor=black!50!white,hatchcolor=black!50!white](0,0){0.3}
        }

				\rput(0.5,4.1){\rnode{C1}{\nodesimp}}   \rput(0.5,4.1){\small \textcolor{white}{$1$}}
        \rput(1.0,2.6){\rnode{C2}{\nodeC}}      \rput(1.0,2.6){\small \textcolor{black}{$2$}}
        \rput(0.9,0.9){\rnode{C3}{\nodesimp}}   \rput(0.9,0.9){\small \textcolor{white}{$3$}}
        \rput(3.1,1.1){\rnode{C4}{\nodeB}}      \rput(3.1,1.1){\small \textcolor{black}{$4$}}
        \rput(4.1,2.6){\rnode{C5}{\nodesimp}}   \rput(4.1,2.6){\small \textcolor{white}{$5$}}
        \rput(2.3,3.3){\rnode{C6}{\nodeB}}      \rput(2.3,3.3){\small \textcolor{black}{$6$}}

        \psset{nodesep=0.33cm,linewidth=0.7pt}
        \ncline{-}{C1}{C2}
        \ncline{-}{C1}{C6}
        \ncline{-}{C2}{C3}
        \ncline{-}{C2}{C6}
        \ncline{-}{C3}{C4}
        \ncline{-}{C4}{C5}
        \ncline{-}{C5}{C6}

        \rput[rb](0.21715729,4.38284271){\small $1$}
        \rput[rt](0.71715729,2.31715729){\small $3$}
        \rput[rt](0.61715729,0.61715729){\small $1$}
        \rput[lt](3.38284271,0.81715729){\small $2$}
        \rput[lb](4.38284271,2.88284271){\small $1$}
        \rput[lb](2.58284271,3.58284271){\small $2$}

				%\psgrid
			\end{pspicture}
			}
		}
		\hspace{2cm}
		\subfigure[Directed graph]{\label{SubFig:SelfCoordinationDirected}
			\psscalebox{0.9}{
			\begin{pspicture}(4.9,5)
				\def\nodesimp{
					\pscircle*[linecolor=black!65!white](0,0){0.3}
        }
        \def\nodeB{
            \pscircle*[linecolor=black!15!white](0,0){0.3}
        }
        \def\nodeC{
            \pscircle[fillstyle=vlines*,linecolor=black!50!white,hatchcolor=black!50!white](0,0){0.3}
        }

				\rput(0.5,4.1){\rnode{C1}{\nodesimp}}   \rput(0.5,4.1){\small \textcolor{white}{$1$}}
        \rput(1.0,2.6){\rnode{C2}{\nodeC}}      \rput(1.0,2.6){\small \textcolor{black}{$2$}}
        \rput(0.9,0.9){\rnode{C3}{\nodesimp}}   \rput(0.9,0.9){\small \textcolor{white}{$3$}}
        \rput(3.1,1.1){\rnode{C4}{\nodeB}}      \rput(3.1,1.1){\small \textcolor{black}{$4$}}
        \rput(4.1,2.6){\rnode{C5}{\nodesimp}}   \rput(4.1,2.6){\small \textcolor{white}{$5$}}
        \rput(2.3,3.3){\rnode{C6}{\nodeB}}      \rput(2.3,3.3){\small \textcolor{black}{$6$}}

        \psset{nodesep=0.33cm,linewidth=0.7pt,arrowsize=7pt,arrowinset=0.05}
        \ncline{->}{C1}{C2}
        \ncline{->}{C1}{C6}
        \ncline{<-}{C2}{C3}
        \ncline{<-}{C2}{C6}
        \ncline{->}{C3}{C4}
        \ncline{<-}{C4}{C5}
        \ncline{->}{C5}{C6}

        \rput[rb](0.21715729,4.38284271){\small $1$}
        \rput[rt](0.71715729,2.31715729){\small $3$}
        \rput[rt](0.61715729,0.61715729){\small $1$}
        \rput[lt](3.38284271,0.81715729){\small $2$}
        \rput[lb](4.38284271,2.88284271){\small $1$}
        \rput[lb](2.58284271,3.58284271){\small $2$}

				%\psgrid
			\end{pspicture}
			}
		}
		\caption[Construction of a directed graph from the coloring scheme of an undirected graph.]{
			Construction of a directed graph (in \text{(b)}) from the coloring scheme of an undirected graph (in \text{(a)}). The coloring scheme is $\mathcal{C}_1 = \{1,3,5\}$, $\mathcal{C}_2 = \{4,6\}$, and~$\mathcal{C}_3 = \{2\}$. From \text{(a)} to \text{(b)}, each edge gets assigned a direction, from the node with the smallest color to the node with the largest color.
		}
		\label{Fig:SelfCoordination}
  \end{figure}
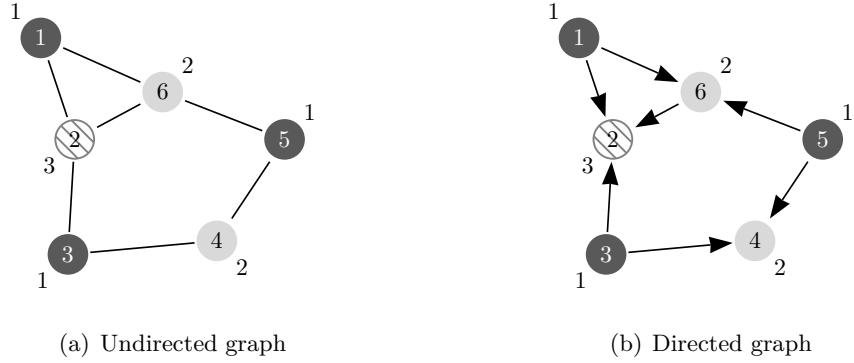

	Apparently, Algorithm~\ref{Alg:GlobalClass} needs some kind of central coordination to perform steps~\ref{SubAlg:GlobalAverage} and~\ref{SubAlg:GlobalProx}, because nodes with the same color, not being neighbors, should perform the same tasks in parallel. But, provided \assref{Ass:GPColoring} holds, i.e., that each node knows its own color and the colors of its neighbors, no central coordination is required. In that case, steps~\ref{SubAlg:GlobalAverage} and~\ref{SubAlg:GlobalProx} need not be performed exactly in parallel: as soon as node~$p$ has received the copies~$x_j^{k+1}$ from the neighbors with smaller colors, it can perform steps~\ref{SubAlg:GlobalAverage} and~\ref{SubAlg:GlobalProx} immediately. \fref{Fig:SelfCoordination} illustrates an alternative way to see this. \fref{SubFig:SelfCoordinationUndirected} shows a communication network and its coloring scheme: nodes~$1$, $3$, and~$5$ have color~$1$, nodes~$4$ and~$6$ have color~$2$, and node~$2$ has color~$3$. From these colors, we can assign directions to the edges of the network, as shown in \fref{SubFig:SelfCoordinationDirected}: the edge $(i,j) \in \mathcal{E}$ is assigned the direction $i \rightarrow j$ if the color of node~$i$ is smaller than the color of node~$j$, i.e., $\mathcal{C}(i) < \mathcal{C}(j)$, and the direction $i \leftarrow j$ otherwise. For example, node~$6$, with color~$2$ has incoming edges from nodes~$1$ and~$5$, both with color~$1$, and an outgoing edge to node~$2$, with color~$3$. Whenever node~$6$ receives, at each iteration, estimates from neighbors~$1$ and~$5$, it can immediately perform steps~\ref{SubAlg:GlobalAverage} and~\ref{SubAlg:GlobalProx} without ``talking'' at all with the nodes that have the same color; in this case, that is just node~$4$. This makes the algorithm distributed, since there is no central or coordinating node, the function~$f_p$ is only known at node~$p$, and there are no all-to-all communications. Furthermore, the algorithm is independent of the network. Regarding its convergence, we use Theorem~\ref{Teo:RelatedWorkConvergenceADMM} to prove:
  \begin{theorem}\label{Teo:GlobalConvergence}
		\hfill

		\medskip
		\noindent
    Let Assumptions~\ref{Ass:GPfunctions}-\ref{Ass:GPColoring} hold. Then, Algorithm~\ref{Alg:GlobalClass} produces a sequence~$(x_1^k,\ldots,x_P^k)$ convergent to~$(x^\star, \ldots,x^\star)$, where~$x^\star$ solves~\eqref{Eq:GlobalProblem}, when at least one of the following conditions is satisfied:
    \newcounter{TeoConvDADMM}
		\begin{list}{(\alph{TeoConvDADMM})}{\usecounter{TeoConvDADMM}}
      \item the coloring scheme uses two colors only (which implies that the network is bipartite);
      \item each function~$f_p$ is strongly convex with modulus~$\mu_p$ and
      \begin{equation}\label{Eq:ThmConditionRho}
				0 < \rho < \underset{c=1,\ldots,C}{\min}\,\,\, \frac{2\sum_{p \in \mathcal{C}_c}\mu_p}{3\,(C-1)\,\max_{p \in \mathcal{C}_c}D_p}\,.
      \end{equation}
    \end{list}
  \end{theorem}
  \begin{proof}

	We just need to show that~\eqref{Eq:GlobalEdgeBasedReformulationMatrix}, the problem to which we apply multi-block ADMM, satisfies the assumptions of Theorem~\ref{Teo:RelatedWorkConvergenceADMM}. First, note that Assumptions~\ref{Ass:GPfunctions} and~\ref{Ass:GPsolvable} and the equivalence between~\eqref{Eq:GlobalProblem} and~\eqref{Eq:GlobalEdgeBasedReformulationMatrix} imply that problem~\eqref{Eq:GlobalEdgeBasedReformulationMatrix} is solvable and that each function $\sum_{p \in \mathcal{C}_c} f_p(x_p)$ is closed and convex over~$(\mathbb{R}^n)^C$. Next, we show that condition \text{(a)} (resp.\ \text{(b)}) implies condition \text{(a)} (resp.\ \text{(b)}) of Theorem~\ref{Teo:RelatedWorkConvergenceADMM}.
		\newcounter{TeoConvDADMMProof}
		\begin{list}{(\alph{TeoConvDADMMProof})}{\usecounter{TeoConvDADMMProof}}
			\item
				We first see that Assumption~\ref{Ass:GPConnectedStatic} implies that each~$B_c^\top \otimes I_n$ has full column rank. Since the identity matrix~$I_n$ has always full rank, we just need to show that~$B_c^\top$ has full column-rank. If, on the other hand, we prove that~$B_cB_c^\top$ has full rank, then the result follows, because~$\text{rank}(B_cB_c^\top) = \text{rank}(B_c^\top)$. As mentioned before, $B_cB_c^\top$ is a diagonal matrix, where the diagonal contains the degrees of the nodes belonging to the subnetwork composed by the nodes in~$\mathcal{C}_c$. Since no node has degree~$0$ (cf. Assumption~\ref{Ass:GPConnectedStatic}), $B_cB_c^\top$ has full rank.	 We thus have shown that, independently of the coloring scheme, each matrix~$B_c^\top \otimes I_n$ has full column rank. Therefore, when the coloring scheme uses two colors, both requirements of point \text{(a)} in Theorem~\ref{Teo:RelatedWorkConvergenceADMM} are satisfied.

			\item
				When each function~$f_p$ is strongly convex with modulus~$\mu_p$ and~$\rho$ satisfies~\eqref{Eq:ThmConditionRho}, then each~$\sum_{p \in \mathcal{C}_c}f_p$ is strongly convex with modulus $\sum_{p \in \mathcal{C}_c} \mu_p$ \cite[Lem.\ 2.1.4]{Nesterov04-IntroductoryLecturesConvexOptimization} and conditions~\eqref{Eq:RelatedWorkExtendedADMMConditionRho} and~\eqref{Eq:ThmConditionRho} are equivalent. To see this last point, just note that
				\begin{align*}
					\sigma_{\max}(A_c)^2 = \lambda_{\max}(A_c^\top A_c) = \lambda_{\max}(B_c B_c^\top \otimes I_n)
					= \lambda_{\max}(B_c B_c^\top) = \max_{p \in \mathcal{C}_c}\, D_p\,,
				\end{align*}
				since, as we had seen before, each $B_c B_c^\top$ is a diagonal matrix with the degrees of the nodes with color~$c$ in the diagonal.
		\end{list}
  \end{proof}
  As stated before, it is believed that multi-block ADMM converges under condition \text{(a)} of Theorem~\ref{Teo:RelatedWorkConvergenceADMM} when~$C > 2$. This requires that each $B_c^\top \otimes I_n$ has full column rank, which we just proved in part \text{(a)} of the proof above. Translated to Algorithm~\ref{Alg:GlobalClass}, this belief means that algorithm converges for generic (non-bipartite) networks when~$f_p$ is not necessarily strongly convex, i.e., that Theorem~\ref{Teo:GlobalConvergence} holds even when neither condition~\text{(a)} nor condition~\text{(b)} are satisfied. Our simulations of Algorithm~\ref{Alg:GlobalClass} provide some experimental evidence strengthening that belief, as we will soon see.

  Note that the structure of Algorithms~\ref{Alg:Schizas} and \ref{Alg:Zhu}, which are based on the $2$-block ADMM, is similar to the structure of Algorithm~\ref{Alg:GlobalClass}: in all of them, a parameter~$\rho$ has to be chosen, and each node performs the same kind of computations, i.e., compute an average of the estimates of the neighbors and compute the prox of its private function. While in Algorithms~\ref{Alg:Schizas} and \ref{Alg:Zhu} all the nodes perform all the tasks in parallel, the nodes in Algorithm~\ref{Alg:GlobalClass} operate in a color-based way. Therefore, in environments where parallel communication is allowed, one iteration of Algorithm~\ref{Alg:GlobalClass} takes longer than one iteration of Algorithms~\ref{Alg:Schizas} and \ref{Alg:Zhu}. In environments where parallel communication is impossible, e.g., in wireless networks, Algorithms~\ref{Alg:Schizas} and~\ref{Alg:Zhu} have to implement a MAC protocol and, for example, operate in the same color-based way as Algorithm~\ref{Alg:GlobalClass}. In either case, simulation shows that Algorithm~\ref{Alg:GlobalClass} takes systematically less iterations to converge than Algorithms~\ref{Alg:Schizas} and \ref{Alg:Zhu}, for several different problems and several different networks. This means that it is more communication-efficient than the other algorithms, and hence more attractive in scenarios where the nodes are battery-operated.

	\section{Experimental results}
	\label{Sec:GC:ExperimentalResults}

	In this section, we provide some experimental results that compare the performance of the proposed algorithm with prior distributed optimization algorithms. The performance of all the algorithms will be measured in terms of communication steps, defined next.

	\mypar{Communication steps}
	We say that a communication step (CS) has occurred whenever all the nodes have
	transmitted to their neighbors a new solution estimate, usually computed by evaluating a prox operator, as in step~\ref{SubAlg:GlobalProx} of Algorithm~\ref{Alg:GlobalClass}. The number of CSs an algorithm uses to solve an optimization problem is intrinsic to the algorithm and does not take into account factors like MAC protocols, algorithm implementation, or computing platforms. Other performance measures, for example execution time, may give different results if we change any of these factors. Besides, the total number of communications can be easily obtained from the CSs by multiplying it by~$2E$, i.e., by twice the number of edges in the network. Note that Algorithm~\ref{Alg:Schizas} takes two CSs per iteration, while Algorithms~\ref{Alg:Zhu} and~\ref{Alg:GlobalClass} take only one.

	\begin{table}
    \centering
    \caption{
      Network models.
    }
    \label{Tab:NetworkModels}
    \renewcommand{\arraystretch}{1.3}
    \scriptsize
    %\resizebox{0.7\linewidth}{!}{
    \begin{tabular}{@{}llclp{10.5cm}@{}}
    \toprule[1pt]
    Name && Parameters && Description \\
    \midrule

    Erd\H os-R\'enyi~\cite{Erdos59-OnRandomGraphs} &&
    $p$  &&
    Every pair of nodes~$(i,j) \in \mathcal{E}$ is connected or not with probability~$p$ \\

    Watts-Strogatz~\cite{Watts98-CollectiveDynamicsSmallWorldNetworks} &&
    $(n,p)$ &&
    First, it creates a lattice where every node is connected to~$n$ nodes; then, it rewires every link with probability~$p$. Rewiring link $(i,j)$ means removing the link, and connecting node~$i$ or node~$j$ (chosen with equal probability) to another node in the network, chosen uniformly.\\

    Barabasi-Albert~\cite{Barabasi99-EmergenceOfScalingInRandomNetworks} &&
    ------ &&
    It starts with one node. At each step, one node is added to the network by connecting it to~$2$ existing nodes: the probability to connect it to node~$p$ is proportional to~$D_p$.\\

    Geometric~\cite{Penrose04-RandomGeometricGraphs} &&
    $d$ &&
    It drops~$P$ points, corresponding to the nodes of the network, randomly in a $[0,1]^2$ square; then, it connects nodes whose (Euclidean) distance is less than~$d$.\\

    Lattice &&
    ------ &&
    Creates a lattice of dimensions~$m\times n$; $m$ and~$n$ are chosen to make the lattice as square as possible.\\
    \bottomrule[1pt]
    \end{tabular}
    %}
    %\isdraft{\vspace{-0.6cm}}{}
   \end{table}

		\mypar{Networks}
		We generated several networks in our experiments, ranging from networks with~$10$ nodes to networks with~$2000$ nodes. The models we used to generate them are described in \tref{Tab:NetworkModels}. All models, except the lattice, are random, and yield networks with arbitrary topologies. Using these models, we created $40$ different networks, as shown in \tref{Tab:NetworkParametersADNC}. For each one of the models of \tref{Tab:NetworkModels}, we generated~$8$ networks with different numbers of nodes, from~$P=10$ nodes, to $P=2000$ nodes. All the networks were generated in Python~\cite{RossumPython} with the NetworkX library~\cite{Hagberg08-NetworkX}. The parameters we used to generate the Erd\H os-R\'enyi and the geometric networks are known to generate connected networks with high probability. To color the networks, we used a built-in function in Sage~\cite{Stein13-Sage}. The number of colors of each network and the average node degree are shown in \tref{Tab:NetworkParametersADNC}. For example, the network with the largest average degree was the geometric network with $2000$ nodes; the same network had the largest number of colors, $21$. Note that all the lattice networks were colored with two colors, indicating that they are, in fact, bipartite. Note also that these are the only networks for which Algorithm~\ref{Alg:GlobalClass} is proven to converge when the cost functions at each node are not strongly convex (cf.\ Theorem~\ref{Teo:GlobalConvergence}).

   \begin{table}
    \centering
    \caption{
      Network parameters, average degree, and number of colors.
    }
    \label{Tab:NetworkParametersADNC}
    \renewcommand{\arraystretch}{1.3}
    \scriptsize
    %\resizebox{0.7\linewidth}{!}{
    \newcommand{\vpc}{\vspace{-0.12cm}}
    \begin{tabular}{@{}clcrrrrrrrr@{}}
    \toprule[1pt]
    Number & Model & Parameters & \multicolumn{8}{c}{Average degree (top), Number of colors (bottom)}  \\
    \midrule
    & & & \multicolumn{8}{c}{Number of nodes $P$} \\
    \cmidrule(r){4-11}
          & & & 10\phantom{00} & 50\phantom{00} & 100\phantom{00} & 200\phantom{00} & 500\phantom{00} & 700\phantom{00} & 1000\phantom{00} & 2000\phantom{00} \\
    \midrule
     1 &
     Erd\H os-R\'enyi & $1.1\log(P)/P$ &
     \scriptsize{\begin{tabular}{rr} 3  \vpc\\ 3  \vpc\end{tabular}} &
     \scriptsize{\begin{tabular}{rr} 6  \vpc\\ 5  \vpc\end{tabular}} &
     \scriptsize{\begin{tabular}{rr} 5  \vpc\\ 5  \vpc\end{tabular}} &
     \scriptsize{\begin{tabular}{rr} 6  \vpc\\ 5  \vpc\end{tabular}} &
     \scriptsize{\begin{tabular}{rr} 12 \vpc\\ 7  \vpc\end{tabular}} &
     \scriptsize{\begin{tabular}{rr} 14 \vpc\\ 8  \vpc\end{tabular}} &
     \scriptsize{\begin{tabular}{rr} 18 \vpc\\ 9  \vpc\end{tabular}} &
     \scriptsize{\begin{tabular}{rr} 8  \vpc\\ 6  \vpc\end{tabular}} \\

    \midrule[0.3pt]

    2 &
    Watts-Strogatz & $(4,0.4)$ &
    \scriptsize{\begin{tabular}{rr} 4 \vpc\\ 3 \vpc\end{tabular}} &
    \scriptsize{\begin{tabular}{rr} 4 \vpc\\ 4 \vpc\end{tabular}} &
    \scriptsize{\begin{tabular}{rr} 4 \vpc\\ 4 \vpc\end{tabular}} &
    \scriptsize{\begin{tabular}{rr} 4 \vpc\\ 4 \vpc\end{tabular}} &
    \scriptsize{\begin{tabular}{rr} 4 \vpc\\ 5 \vpc\end{tabular}} &
    \scriptsize{\begin{tabular}{rr} 4 \vpc\\ 4 \vpc\end{tabular}} &
    \scriptsize{\begin{tabular}{rr} 4 \vpc\\ 4 \vpc\end{tabular}} &
    \scriptsize{\begin{tabular}{rr} 4 \vpc\\ 4 \vpc\end{tabular}} \\

   \midrule[0.3pt]

    3 &
    Barabasi-Albert & ------ &
    \scriptsize{\begin{tabular}{rr} 3 \vpc\\ 3 \vpc\end{tabular}} &
    \scriptsize{\begin{tabular}{rr} 4 \vpc\\ 3 \vpc\end{tabular}} &
    \scriptsize{\begin{tabular}{rr} 4 \vpc\\ 3 \vpc\end{tabular}} &
    \scriptsize{\begin{tabular}{rr} 4 \vpc\\ 4 \vpc\end{tabular}} &
    \scriptsize{\begin{tabular}{rr} 4 \vpc\\ 4 \vpc\end{tabular}} &
    \scriptsize{\begin{tabular}{rr} 4 \vpc\\ 4 \vpc\end{tabular}} &
    \scriptsize{\begin{tabular}{rr} 4 \vpc\\ 4 \vpc\end{tabular}} &
    \scriptsize{\begin{tabular}{rr} 4 \vpc\\ 4 \vpc\end{tabular}} \\

    \midrule[0.3pt]

    4 &
    Geometric & $\sqrt{\log(P)/P}$&
    \scriptsize{\begin{tabular}{rr} 4  \vpc\\ 5 \vpc\end{tabular}} &
    \scriptsize{\begin{tabular}{rr} 10 \vpc\\ 10 \vpc\end{tabular}} &
    \scriptsize{\begin{tabular}{rr} 12 \vpc\\ 11 \vpc\end{tabular}} &
    \scriptsize{\begin{tabular}{rr} 14 \vpc\\ 12 \vpc\end{tabular}} &
    \scriptsize{\begin{tabular}{rr} 18 \vpc\\ 18 \vpc\end{tabular}} &
    \scriptsize{\begin{tabular}{rr} 19 \vpc\\ 19 \vpc\end{tabular}} &
    \scriptsize{\begin{tabular}{rr} 20 \vpc\\ 17 \vpc\end{tabular}} &
    \scriptsize{\begin{tabular}{rr} 23 \vpc\\ 21 \vpc\end{tabular}} \\

    \midrule[0.3pt]

    5 &
    Lattice & ------&
    \scriptsize{\begin{tabular}{rr} 3 \vpc\\ 2 \vpc\end{tabular}} &
    \scriptsize{\begin{tabular}{rr} 3 \vpc\\ 2 \vpc\end{tabular}} &
    \scriptsize{\begin{tabular}{rr} 4 \vpc\\ 2 \vpc\end{tabular}} &
    \scriptsize{\begin{tabular}{rr} 4 \vpc\\ 2 \vpc\end{tabular}} &
    \scriptsize{\begin{tabular}{rr} 4 \vpc\\ 2 \vpc\end{tabular}} &
    \scriptsize{\begin{tabular}{rr} 4 \vpc\\ 2 \vpc\end{tabular}} &
    \scriptsize{\begin{tabular}{rr} 4 \vpc\\ 2 \vpc\end{tabular}} &
    \scriptsize{\begin{tabular}{rr} 4 \vpc\\ 2 \vpc\end{tabular}} \\

    \bottomrule[1pt]
    \end{tabular}
    %}
   \end{table}

   \mypar{Choosing \boldmath{$\rho$}}
   Almost all the algorithms we compare are based on augmented Lagrangian duality and, thus, are parametrized by a parameter~$\rho$. We are unaware of any method that selects a good~$\rho$ before executing the algorithm; as discussed in \cref{Ch:relatedWork}, the existing heuristics for adapting~$\rho$ during the execution of the algorithm cannot be implemented in a distributed setting. Therefore, for each algorithm that depends on~$\rho$, we execute the algorithm several times, one for a different value of~$\rho$, and select the one that leads to the best performance. In our experiments, we used two strategies for selecting~$\rho$. The simplest one just selects~$\rho$ out of a set of values, typically $\{10^{-4},10^{-3},10^{-2},10^{-1},1,10,10^2\}$. In the second strategy, for a given algorithm, we present the chosen value of~$\rho$ and give the precision value. We say that $\bar{\rho}$ was chosen with precision~$\xi >0$ for a given algorithm whenever both $\rho = \bar{\rho} - \xi$ and $\rho = \bar{\rho}+\xi$ lead to more CSs than $\rho = \bar{\rho}$. This definition is motivated by the fact that the number of CSs in augmented Lagrangian algorithms seems to vary with~$\rho$ in a convex way.

   Next, we present the results of our experiments for each of the applications of \sref{Sec:GC:Applications}. The simplest of these applications is average consensus and, for this reason, we study average consensus in more detail.

  \subsection{Average consensus}

  We designed two sets of experiments for the average consensus problem. In one of them, we fix the network and run several distributed algorithms, comparing how the error evolves along the iterations (or better, along the CSs). In the other set of experiments, we observe only the total number of CSs that each algorithm takes to achieve a predefined relative error. While the first set of experiments is run on a single network and for many algorithms, the second set of experiments is run for all the networks of \tref{Tab:NetworkParametersADNC} and only for the most competitive algorithms. Next, we describe the how the experiments were designed, then we state which algorithms we compare, and finally we describe the results for both sets of experiments.

  \mypar{Experimental setup}
  In consensus, each node~$p$ holds a scalar~$\theta_p$, and the goal is to compute the average of all the $\theta_p$'s. We generated each~$\theta_p$ independently from each other as a realization of a Gaussian distribution with mean~$10$ and standard deviation~$100$. Such a large standard deviation was chosen to ensure all the $\theta_p$'s differed significantly. We generated~$8$ sets of these numbers, each set for a network with a fixed number of nodes. This means that one set of $\theta_p$'s is used across networks with the same number of nodes, that is, the same set is used, for example, for a geometric network with~$200$ nodes and for a Barabasi-Albert network with~$200$ nodes.

	In all the algorithms we compare, each node requires an initialization of its solution estimate. In all our experiments, the estimate of node~$p$ is initialized with~$\theta_p$. We only do this special initialization for the average consensus problem; the reason is to make a fair comparison between algorithms that were designed specifically for consensus and that require this exact initialization, and between general-purpose algorithms, which do not require any special initialization. This contrasts with the results in \fref{Fig:IntroExperimentalResults}, in \cref{Ch:Introduction}, where some algorithms were initialized this way and others, including the algorithm we propose, were initialized with zeros. While those results are merely illustrative, they are not as fair as the ones we present next.

	\mypar{Algorithms for comparison}
	In our experiments, we compare the performance of Algorithm~\ref{Alg:GlobalClass} not only with other algorithms solving the problem class~\eqref{Eq:GlobalProblem}, but also with algorithms that were designed only for average consensus and that cannot solve any other problem in that class. Namely, the algorithms in~\cite{Olshevsky11-ConvergenceSpeedDistributedConsensusAveraging} and~\cite{Oreshkin10-OptimizationAnalysisDistrAveraging} are consensus algorithms and cannot be generalized (at least, straightforwardly) to solve other problems written as~\eqref{Eq:GlobalProblem}. The algorithm in~\cite{Oreshkin10-OptimizationAnalysisDistrAveraging} is actually considered the fastest consensus algorithm, among the synchronous and the asynchronous ones~\cite{Erseghe11-FastConsensusByADMM}. Since each iteration takes one CS, it is also the most communication-efficient algorithm for consensus. We will see next that the algorithm we propose, when applied to consensus, performs as well as~\cite{Oreshkin10-OptimizationAnalysisDistrAveraging}, and sometimes better. Note that our algorithm is general-purpose, in contrast with~\cite{Oreshkin10-OptimizationAnalysisDistrAveraging}, which is specific to consensus.

	Regarding general-purpose algorithms, we consider distributed algorithms based on the $2$-block ADMM, namely,
	\cite{Schizas08-ConsensusAdHocWSNsPartI} (written as Algorithm~\ref{Alg:Schizas}), \cite{Zhu09-DistributedInNetworkChannelCoding} (written as Algorithm~\ref{Alg:Zhu}), and~\cite{Ling12-MultiBlockAlternatingDirectionMethodParallelSplittingConsensusOptimization}. All these algorithms (and also ours) require computing the prox operator of the function~$f_p = (1/2)(x - \theta_p)^2$. This can be done in closed-form: $\text{prox}_{\tau f_p}(\eta) = (\tau\theta_p + \eta)/(1+\tau)$; see~\eqref{Eq:DefinitionProxOperator} for the definition of the prox operator. The algorithm in~\cite{Ling12-MultiBlockAlternatingDirectionMethodParallelSplittingConsensusOptimization} is slightly different from the other ADMM-based algorithms since, instead of just one tuning parameter, it has two: the augmented Lagrangian~$\rho$ and a stepsize~$\beta$. In our experiments, we set always $\beta = 0.9\mu$, just like the authors of~\cite{Ling12-MultiBlockAlternatingDirectionMethodParallelSplittingConsensusOptimization} did in their experiments. 	We also consider the (sub)gradient-based method~\cite{Nedic09-DistributedSubgradientMethodsMultiAgentOptimization}, which also solves the class~\eqref{Eq:GlobalProblem}. (Actually, the algorithm in~\cite{Nedic09-DistributedSubgradientMethodsMultiAgentOptimization} solves only unconstrained problems; to solve problems with constraints one has to consider the generalization in~\cite{Nedic10-ConstrainedConsensusOptimizationMultiAgentNetworks}.) We implemented the algorithm in~\cite{Nedic09-DistributedSubgradientMethodsMultiAgentOptimization} with uniform weights, i.e., each node averages equally the estimates of its neighbors, and with stepsize $1/(k + 1)$.

	\begin{figure}
     \centering
     \psscalebox{1.1}{
     \begin{pspicture}(8.0,5.4)
       \rput[bl](0.25,0.70){\includegraphics[width=7.5cm]{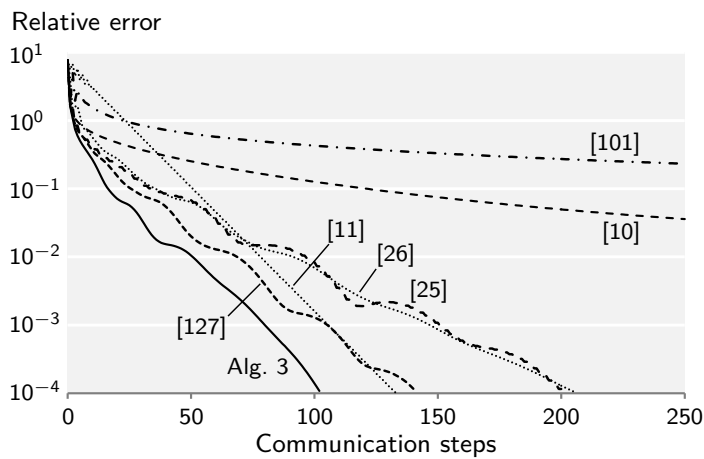}}
       \rput[b](4.00,0.001){\footnotesize \textbf{\sf Communication steps}}
       \rput[bl](-0.41,5.17){\mbox{\footnotesize \textbf{{\sf Relative error}}}}

       \rput[r](0.20,4.89){\scriptsize $\mathsf{10^{1\phantom{-}}}$}
       \rput[r](0.20,4.08){\scriptsize $\mathsf{10^{0\phantom{-}}}$}
       \rput[r](0.20,3.26){\scriptsize $\mathsf{10^{-1}}$}
       \rput[r](0.20,2.46){\scriptsize $\mathsf{10^{-2}}$}
       \rput[r](0.20,1.63){\scriptsize $\mathsf{10^{-3}}$}
       \rput[r](0.20,0.81){\scriptsize $\mathsf{10^{-4}}$}

       \rput[t](0.280,0.59){\scriptsize $\mathsf{0}$}
       \rput[t](1.769,0.59){\scriptsize $\mathsf{50}$}
       \rput[t](3.260,0.59){\scriptsize $\mathsf{100}$}
       \rput[t](4.748,0.59){\scriptsize $\mathsf{150}$}
       \rput[t](6.240,0.59){\scriptsize $\mathsf{200}$}
			 \rput[t](7.734,0.59){\scriptsize $\mathsf{250}$}

       \rput[rt](2.95,1.2){\scriptsize \textbf{\sf Alg.\ \ref{Alg:GlobalClass}}}
       \rput[rt](2.2,1.7){\scriptsize \textbf{\sf \cite{Ling12-MultiBlockAlternatingDirectionMethodParallelSplittingConsensusOptimization}}}
       \psline[linewidth=0.5pt]{-}(2.21,1.71)(2.6,2.1)
       \rput[lb](4.4,1.87){\scriptsize \textbf{\sf \cite{Schizas08-ConsensusAdHocWSNsPartI}}}
       \rput[lb](4.0,2.3){\scriptsize \textbf{\sf \cite{Zhu09-DistributedInNetworkChannelCoding}}}
       \psline[linewidth=0.5pt]{-}(3.8,2.0)(4.0,2.28)
       \rput[lb](3.4,2.6){\scriptsize \textbf{\sf \cite{Oreshkin10-OptimizationAnalysisDistrAveraging}}}
       \psline[linewidth=0.5pt]{-}(3.0,2.08)(3.43,2.58)
       \rput[rt](7.2,2.85){\scriptsize \textbf{\sf \cite{Olshevsky11-ConvergenceSpeedDistributedConsensusAveraging}}}
       \rput[rb](7.2,3.65){\scriptsize \textbf{\sf \cite{Nedic09-DistributedSubgradientMethodsMultiAgentOptimization}}}

       %\psgrid
     \end{pspicture}
     }
     \caption[Comparison of several algorithms for the average consensus problem in a geometric network with $P = 2000$ nodes.]{
			Comparison of several algorithms for the average consensus problem in a geometric network with $P = 2000$ nodes. The plot shows the relative error versus the number of CSs.
    }
     \label{Fig:GlobalExpConsensusAllAlgs}
  \end{figure}

	\begin{figure}
     \centering
     \subfigure[Network 1: Erd\H os-R\'enyi]{\label{SubFig:GlobalExpConsER}
     \begin{pspicture}(7.9,5.4)
       \rput[b](3.44,0.86){\includegraphics[scale=0.3]{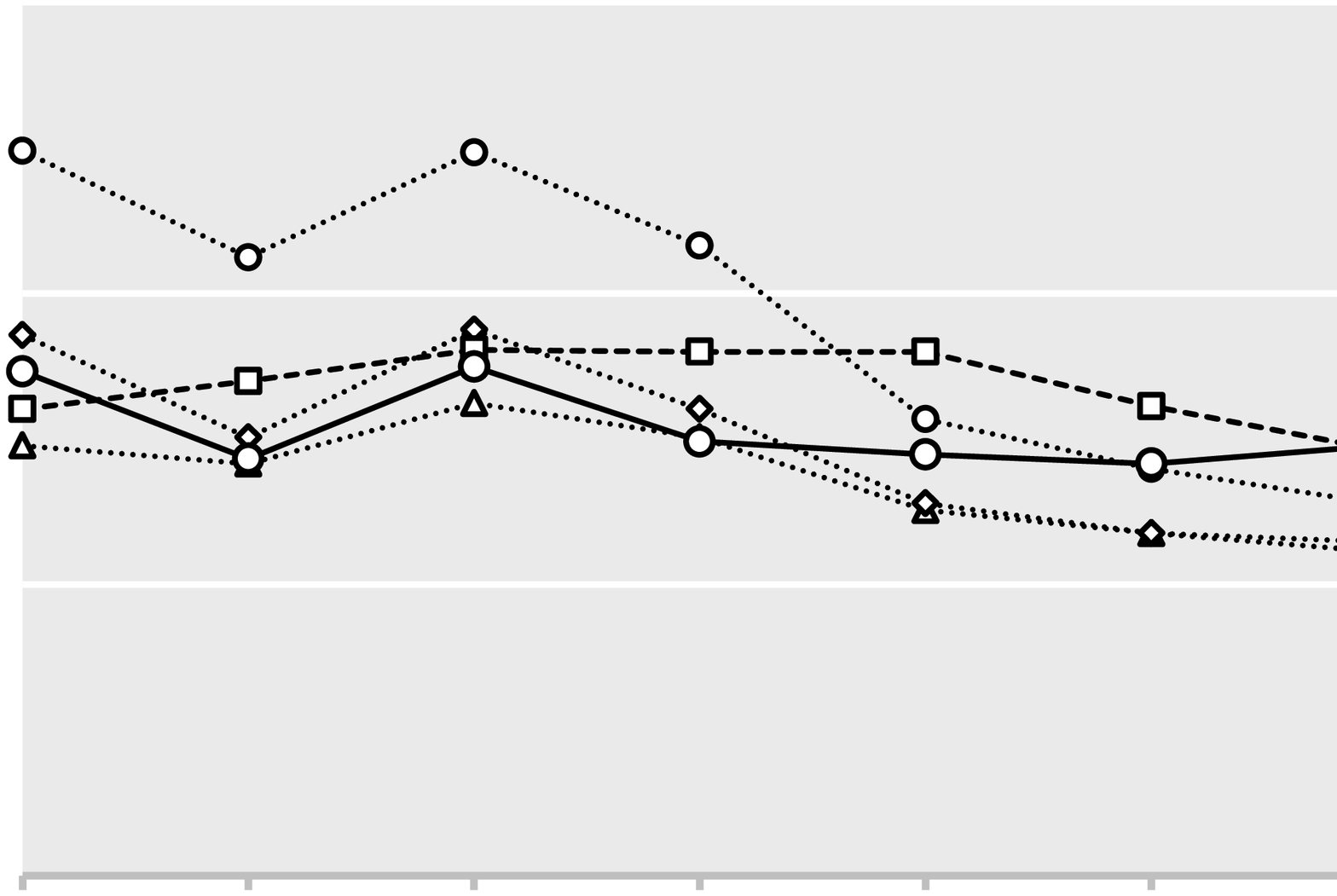}}
       \rput[b](3.95,0.1){\footnotesize \textbf{\sf Number of nodes}}
       \rput[bl](0.01,5.05){\mbox{\footnotesize \textbf{{\sf Communication steps}}}}
       \rput[r](0.44,0.93){\scriptsize $\mathsf{10^0}$}
       \rput[r](0.44,2.22){\scriptsize $\mathsf{10^1}$}
       \rput[r](0.44,3.48){\scriptsize $\mathsf{10^2}$}
       \rput[r](0.44,4.74){\scriptsize $\mathsf{10^3}$}

       \rput[t](0.610,0.72){\scriptsize $\mathsf{10}$}
       \rput[t](1.610,0.72){\scriptsize $\mathsf{50}$}
       \rput[t](2.592,0.72){\scriptsize $\mathsf{100}$}
       \rput[t](3.579,0.72){\scriptsize $\mathsf{200}$}
       \rput[t](4.560,0.72){\scriptsize $\mathsf{500}$}
       \rput[t](5.540,0.72){\scriptsize $\mathsf{700}$}
       \rput[t](6.520,0.72){\scriptsize $\mathsf{1000}$}
       \rput[t](7.505,0.72){\scriptsize $\mathsf{2000}$}

       \rput[lt](5.2,2.0){\scriptsize \textbf{\sf Alg.\ \ref{Alg:GlobalClass}}}
       \psline[linewidth=0.5pt]{-}(5,2.70)(5.19,1.98)
       \rput[lb](5.01,3.12){\scriptsize \textbf{\sf \cite{Zhu09-DistributedInNetworkChannelCoding}}}
       \rput[lb](0.83,3.2){\scriptsize \textbf{\sf \cite{Olshevsky11-ConvergenceSpeedDistributedConsensusAveraging}}}
       \rput[t](1.00,2.70){\scriptsize \textbf{\sf \cite{Oreshkin10-OptimizationAnalysisDistrAveraging}}}
       \rput[lb](2.7,4.1){\scriptsize \textbf{\sf \cite{Ling12-MultiBlockAlternatingDirectionMethodParallelSplittingConsensusOptimization}}}

       %\psgrid
     \end{pspicture}
     }
     \hfill
     \subfigure[Network 2: Watts-Strogatz]{\label{SubFig:GlobalExpConsWS}
     \begin{pspicture}(7.9,5.4)
       \rput[b](3.44,0.86){\includegraphics[scale=0.3]{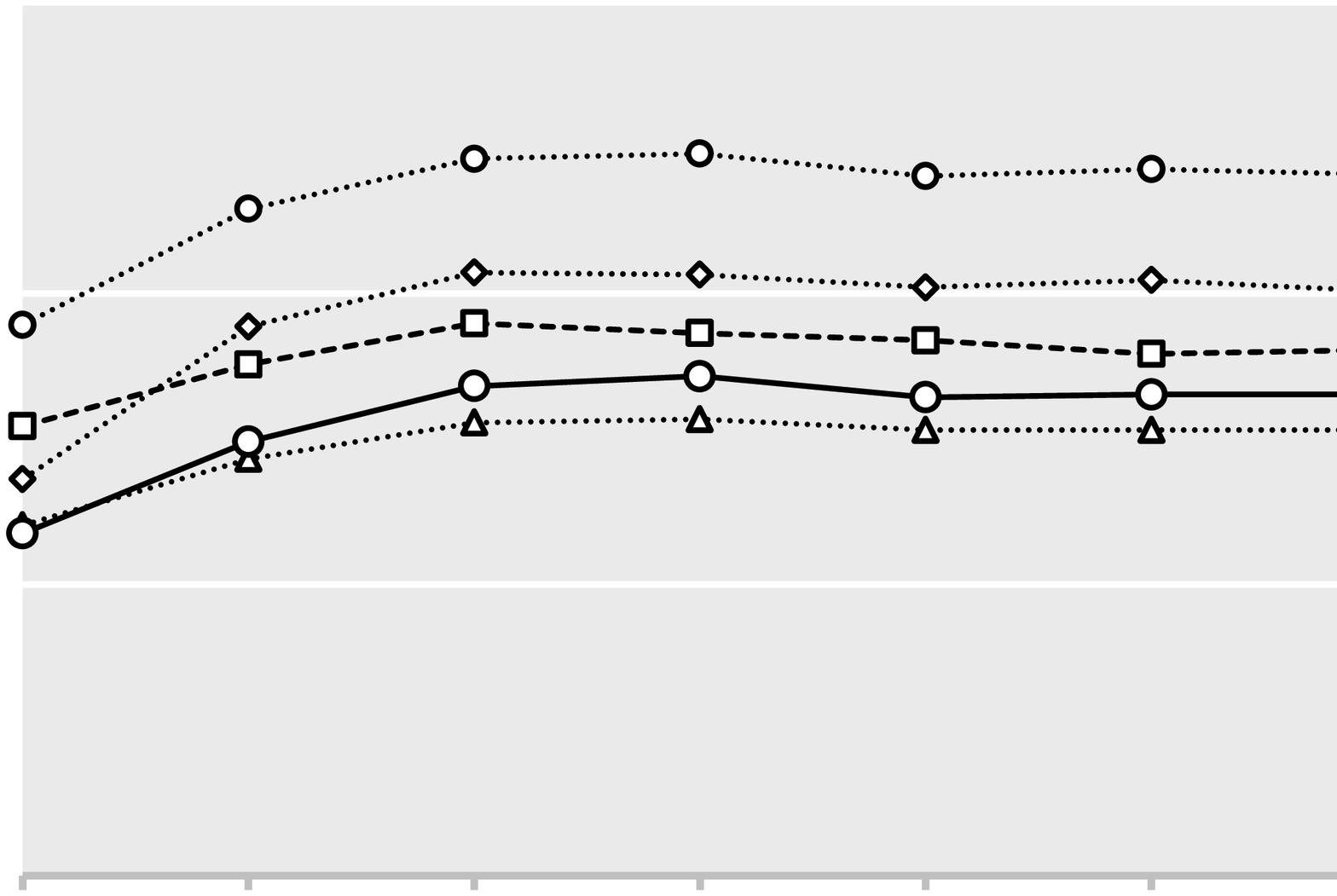}}
       \rput[b](3.95,0.1){\footnotesize \textbf{\sf Number of nodes}}
       \rput[bl](0.01,5.05){\mbox{\footnotesize \textbf{{\sf Communication steps}}}}
       \rput[r](0.44,0.93){\scriptsize $\mathsf{10^0}$}
       \rput[r](0.44,2.22){\scriptsize $\mathsf{10^1}$}
       \rput[r](0.44,3.48){\scriptsize $\mathsf{10^2}$}
       \rput[r](0.44,4.74){\scriptsize $\mathsf{10^3}$}

       \rput[t](0.610,0.72){\scriptsize $\mathsf{10}$}
       \rput[t](1.610,0.72){\scriptsize $\mathsf{50}$}
       \rput[t](2.592,0.72){\scriptsize $\mathsf{100}$}
       \rput[t](3.579,0.72){\scriptsize $\mathsf{200}$}
       \rput[t](4.560,0.72){\scriptsize $\mathsf{500}$}
       \rput[t](5.540,0.72){\scriptsize $\mathsf{700}$}
       \rput[t](6.520,0.72){\scriptsize $\mathsf{1000}$}
       \rput[t](7.505,0.72){\scriptsize $\mathsf{2000}$}

       \rput[rt](7.0,2.6){\scriptsize \textbf{\sf Alg.\ \ref{Alg:GlobalClass}}}
       \psline[linewidth=0.5pt]{-}(7.02,2.61)(7.3,3.03)
       \rput[l](7.60,3.26){\scriptsize \textbf{\sf \cite{Zhu09-DistributedInNetworkChannelCoding}}}
       \rput[l](7.60,3.64){\scriptsize \textbf{\sf \cite{Olshevsky11-ConvergenceSpeedDistributedConsensusAveraging}}}
       \rput[l](7.60,2.84){\scriptsize \textbf{\sf \cite{Oreshkin10-OptimizationAnalysisDistrAveraging}}}
       \rput[l](7.60,4.1){\scriptsize \textbf{\sf \cite{Ling12-MultiBlockAlternatingDirectionMethodParallelSplittingConsensusOptimization}}}

       %\psgrid
     \end{pspicture}
     }

     \bigskip

     \subfigure[Network 3: Barabasi-Albert]{\label{SubFig:GlobalExpConsBarabasi}
     \begin{pspicture}(7.9,5.4)
       \rput[b](3.44,0.86){\includegraphics[scale=0.3]{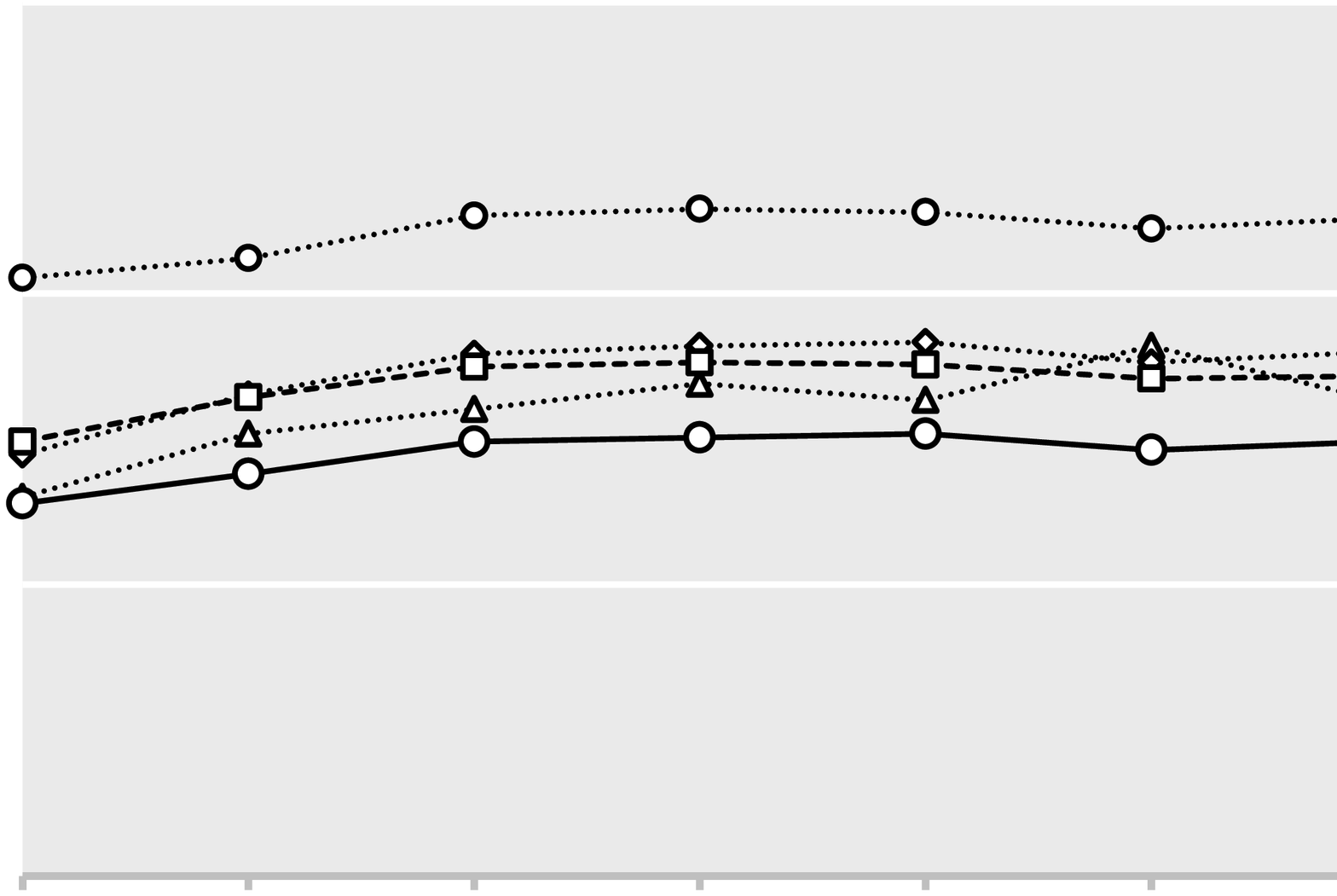}}
       \rput[b](3.95,0.1){\footnotesize \textbf{\sf Number of nodes}}
       \rput[bl](0.01,5.05){\mbox{\footnotesize \textbf{{\sf Communication steps}}}}
       \rput[r](0.44,0.93){\scriptsize $\mathsf{10^0}$}
       \rput[r](0.44,2.22){\scriptsize $\mathsf{10^1}$}
       \rput[r](0.44,3.48){\scriptsize $\mathsf{10^2}$}
       \rput[r](0.44,4.74){\scriptsize $\mathsf{10^3}$}

       \rput[t](0.610,0.72){\scriptsize $\mathsf{10}$}
       \rput[t](1.610,0.72){\scriptsize $\mathsf{50}$}
       \rput[t](2.592,0.72){\scriptsize $\mathsf{100}$}
       \rput[t](3.579,0.72){\scriptsize $\mathsf{200}$}
       \rput[t](4.560,0.72){\scriptsize $\mathsf{500}$}
       \rput[t](5.540,0.72){\scriptsize $\mathsf{700}$}
       \rput[t](6.520,0.72){\scriptsize $\mathsf{1000}$}
       \rput[t](7.505,0.72){\scriptsize $\mathsf{2000}$}

			 \rput[rt](7.4,2.73){\scriptsize \textbf{\sf Alg.\ \ref{Alg:GlobalClass}}}
			 \rput[bl](6.63,3.22){\scriptsize \textbf{\sf \cite{Olshevsky11-ConvergenceSpeedDistributedConsensusAveraging}}}
       \rput[bl](4.5,3.43){\scriptsize \textbf{\sf \cite{Zhu09-DistributedInNetworkChannelCoding}}}
       \psline[linewidth=0.5pt]{-}(4.3,3.17)(4.5,3.40)
			 \rput[bl](3.3,3.43){\scriptsize \textbf{\sf \cite{Oreshkin10-OptimizationAnalysisDistrAveraging}}}
			 \psline[linewidth=0.5pt]{-}(3.02,3.04)(3.3,3.40)
			 \rput[rb](2.3,3.8){\scriptsize \textbf{\sf \cite{Ling12-MultiBlockAlternatingDirectionMethodParallelSplittingConsensusOptimization}}}

       %\psgrid
     \end{pspicture}
     }
     \hfill
     \subfigure[Network 4: Geometric]{\label{SubFig:GlobalExpConsGeo}
     \begin{pspicture}(7.9,5.4)
       \rput[b](3.44,0.86){\includegraphics[scale=0.3]{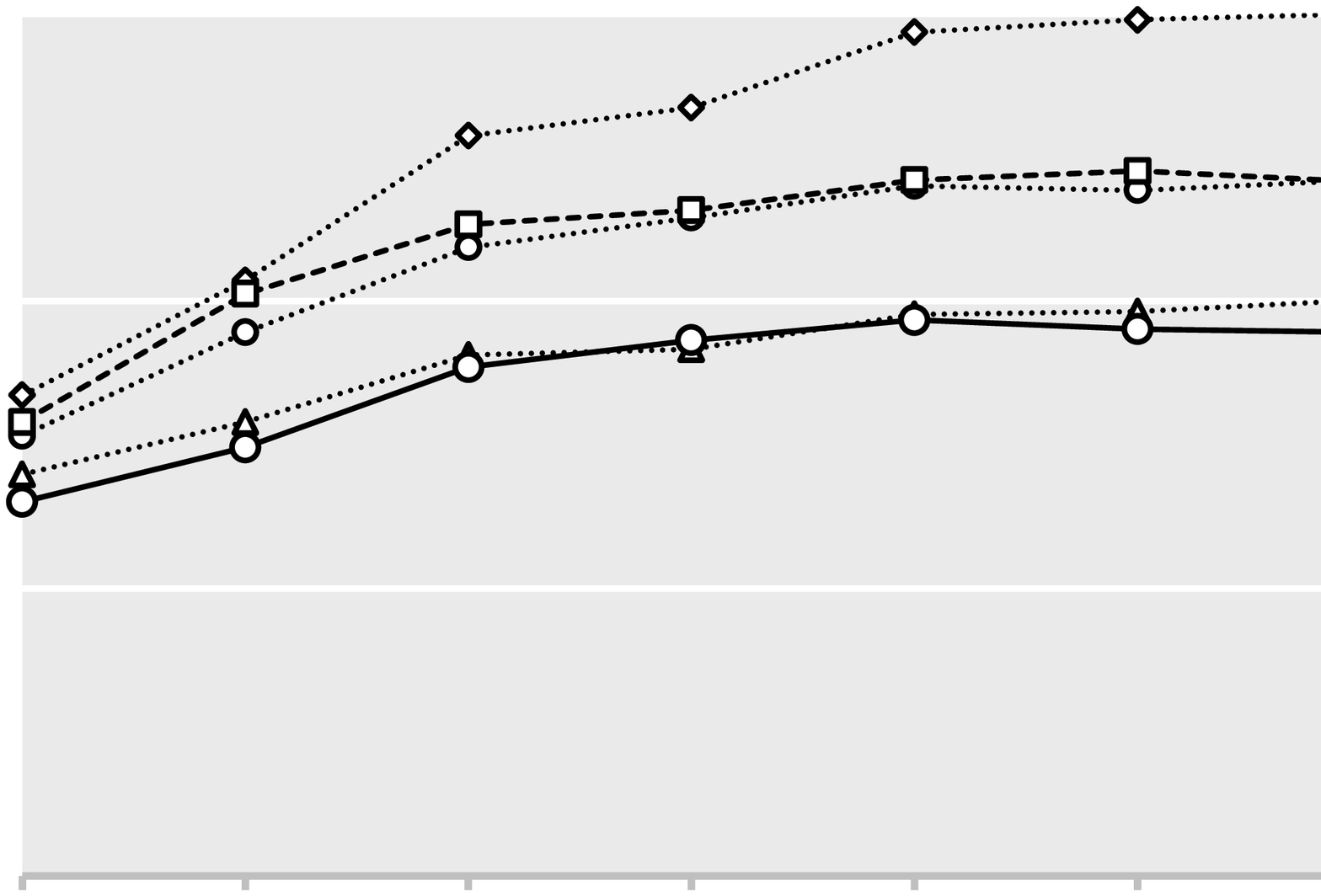}}
       \rput[b](3.95,0.1){\footnotesize \textbf{\sf Number of nodes}}
       \rput[bl](0.01,5.05){\mbox{\footnotesize \textbf{{\sf Communication steps}}}}
       \rput[r](0.44,0.93){\scriptsize $\mathsf{10^0}$}
       \rput[r](0.44,2.22){\scriptsize $\mathsf{10^1}$}
       \rput[r](0.44,3.48){\scriptsize $\mathsf{10^2}$}
       \rput[r](0.44,4.74){\scriptsize $\mathsf{10^3}$}

       \rput[t](0.610,0.72){\scriptsize $\mathsf{10}$}
       \rput[t](1.610,0.72){\scriptsize $\mathsf{50}$}
       \rput[t](2.592,0.72){\scriptsize $\mathsf{100}$}
       \rput[t](3.579,0.72){\scriptsize $\mathsf{200}$}
       \rput[t](4.560,0.72){\scriptsize $\mathsf{500}$}
       \rput[t](5.540,0.72){\scriptsize $\mathsf{700}$}
       \rput[t](6.520,0.72){\scriptsize $\mathsf{1000}$}
       \rput[t](7.505,0.72){\scriptsize $\mathsf{2000}$}

			 \rput[lt](6.6,3.33){\scriptsize \textbf{\sf Alg.\ \ref{Alg:GlobalClass}}}
			 \rput[rb](6.8,3.53){\scriptsize \textbf{\sf \cite{Oreshkin10-OptimizationAnalysisDistrAveraging}}}
       \rput[rb](1.70,4.10){\scriptsize \textbf{\sf \cite{Zhu09-DistributedInNetworkChannelCoding}}}
       \psline[linewidth=0.5pt]{-}(1.69,4.09)(2.1,3.70)
			 \rput[rb](3.20,4.3){\scriptsize \textbf{\sf \cite{Olshevsky11-ConvergenceSpeedDistributedConsensusAveraging}}}
			 \rput[lt](2.8,3.72){\scriptsize \textbf{\sf \cite{Ling12-MultiBlockAlternatingDirectionMethodParallelSplittingConsensusOptimization}}}

       %\psgrid
     \end{pspicture}
     }

     \bigskip

     \subfigure[Network 5: Lattice]{\label{SubFig:GlobalExpConsLattice}
     \begin{pspicture}(7.9,5.4)
       \rput[b](3.44,0.86){\includegraphics[scale=0.3]{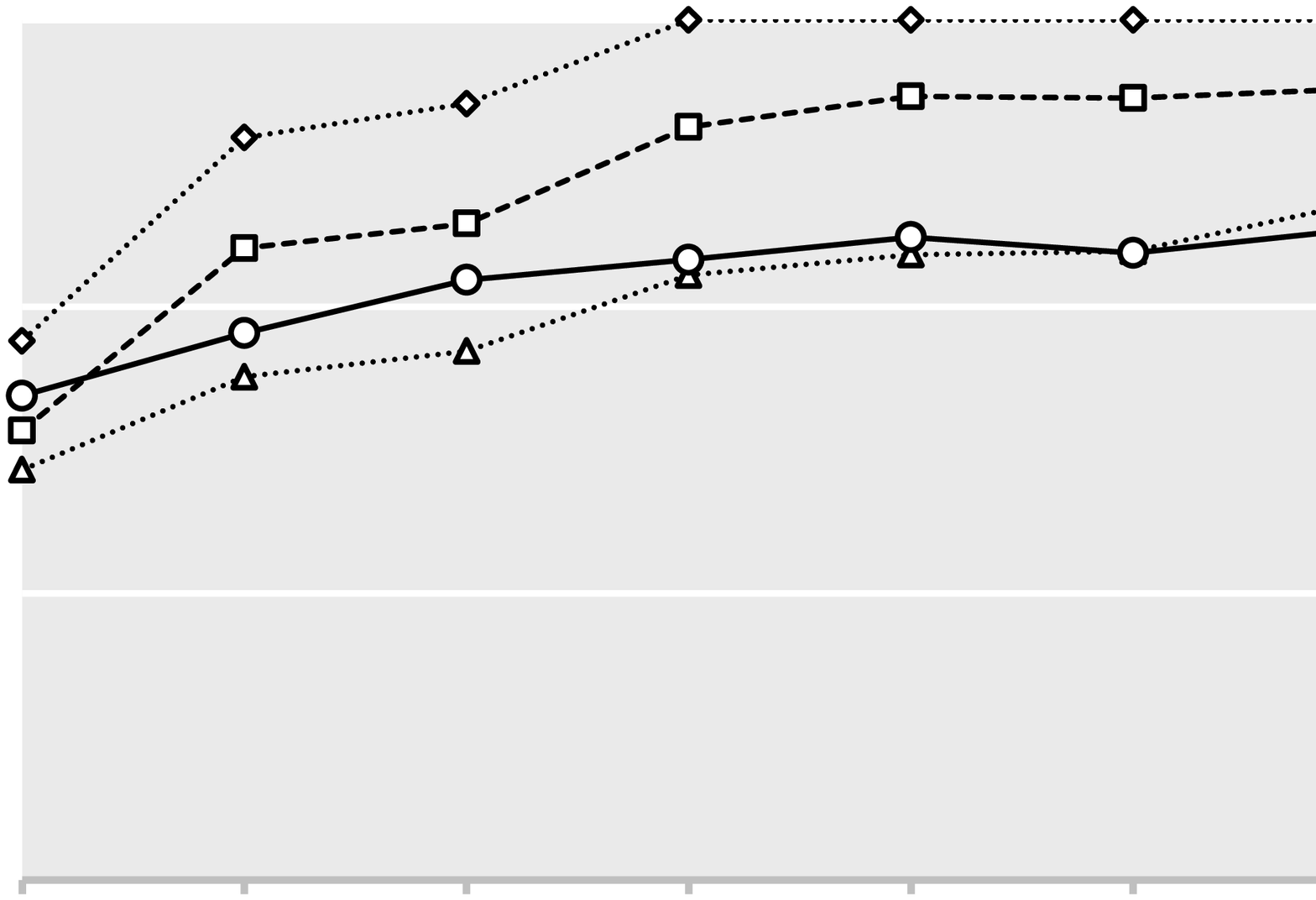}}
       \rput[b](3.95,0.1){\footnotesize \textbf{\sf Number of nodes}}
       \rput[bl](0.01,5.05){\mbox{\footnotesize \textbf{{\sf Communication steps}}}}
       \rput[r](0.44,0.93){\scriptsize $\mathsf{10^0}$}
       \rput[r](0.44,2.22){\scriptsize $\mathsf{10^1}$}
       \rput[r](0.44,3.48){\scriptsize $\mathsf{10^2}$}
       \rput[r](0.44,4.74){\scriptsize $\mathsf{10^3}$}

       \rput[t](0.610,0.72){\scriptsize $\mathsf{10}$}
       \rput[t](1.610,0.72){\scriptsize $\mathsf{50}$}
       \rput[t](2.592,0.72){\scriptsize $\mathsf{100}$}
       \rput[t](3.579,0.72){\scriptsize $\mathsf{200}$}
       \rput[t](4.560,0.72){\scriptsize $\mathsf{500}$}
       \rput[t](5.540,0.72){\scriptsize $\mathsf{700}$}
       \rput[t](6.520,0.72){\scriptsize $\mathsf{1000}$}
       \rput[t](7.505,0.72){\scriptsize $\mathsf{2000}$}

			 \rput[lt](6.66,3.73){\scriptsize \textbf{\sf Alg.\ \ref{Alg:GlobalClass}}}
			 \rput[rb](6.3,3.9){\scriptsize \textbf{\sf \cite{Oreshkin10-OptimizationAnalysisDistrAveraging}}}
       \rput[lt](3.86,4.21){\scriptsize \textbf{\sf \cite{Zhu09-DistributedInNetworkChannelCoding}}}
       \rput[rb](2.06,4.34){\scriptsize \textbf{\sf \cite{Olshevsky11-ConvergenceSpeedDistributedConsensusAveraging}}}
       %\psgrid
     \end{pspicture}
     }
     \caption[Results for the average consensus problem for all the networks of \tref{Tab:NetworkModels}.]{
			Results for the average consensus problem for all the networks of \tref{Tab:NetworkModels}. The plots,  organized by network type, compare Algorithm~\ref{Alg:GlobalClass} with algorithms~\cite{Zhu09-DistributedInNetworkChannelCoding} (see Algorithm~\ref{Alg:Zhu}), \cite{Ling12-MultiBlockAlternatingDirectionMethodParallelSplittingConsensusOptimization}, \cite{Olshevsky11-ConvergenceSpeedDistributedConsensusAveraging,Oreshkin10-OptimizationAnalysisDistrAveraging}. The algorithm~\cite{Ling12-MultiBlockAlternatingDirectionMethodParallelSplittingConsensusOptimization} does not appear in \text{(e)}, because it always achieved the maximum number of iterations, except for the first network. Note that~\cite{Olshevsky11-ConvergenceSpeedDistributedConsensusAveraging,Oreshkin10-OptimizationAnalysisDistrAveraging} were designed specifically for consensus and cannot solve any other problem in the class~\eqref{Eq:GlobalProblem}.
    }
     \label{Fig:GlobalExpConsensusAllNets}
  \end{figure}

  \mypar{Results}
  The performance of all the above algorithms is compared on the geometric network with $P=2000$ nodes, from \tref{Tab:NetworkParametersADNC}. This is shown in \fref{Fig:GlobalExpConsensusAllAlgs} and constitutes our first set of experiments. The plot in the figure shows the evolution of the relative error as a function of the CSs. The relative error is measured as $\|\bar{x}^k - \theta^\star 1_P\|/(\sqrt{P}|\theta^\star|)$, where $\bar{x}^k = (x_1^k,\ldots,x_P^k)$, $x_p^k$ is the solution estimate of node~$p$ at iteration~$k$, and $\theta^\star = (1/P)\sum_{p=1}^P \theta_p$ is the problem's solution. The augmented Lagrangian parameter~$\rho$ was~$1.1$ for Algorithm~\ref{Alg:GlobalClass}, $0.5$ for~\cite{Zhu09-DistributedInNetworkChannelCoding}, $0.6$ for~\cite{Schizas08-ConsensusAdHocWSNsPartI}, and~$0.4$ for~\cite{Ling12-MultiBlockAlternatingDirectionMethodParallelSplittingConsensusOptimization}, and was computed with precision~$0.1$ for all the algorithms. In \fref{Fig:GlobalExpConsensusAllAlgs}, Algorithm~\ref{Alg:GlobalClass} was the algorithm whose error decreased the fastest; in fact, it required uniformly less CSs than all the other algorithms to achieve any relative error between~$10^{-1}$ and~$10^{-4}$. The algorithms with the second and third best performances were, respectively, the consensus algorithm~\cite{Oreshkin10-OptimizationAnalysisDistrAveraging} and the ADMM-based algorithm~\cite{Ling12-MultiBlockAlternatingDirectionMethodParallelSplittingConsensusOptimization}. Next, the ADMM-based algorithms~\cite{Schizas08-ConsensusAdHocWSNsPartI} and~\cite{Zhu09-DistributedInNetworkChannelCoding} had a very similar performance, requiring about $200$ communication steps to achieve a relative error of $10^{-4}$. Both the consensus algorithm~\cite{Olshevsky11-ConvergenceSpeedDistributedConsensusAveraging} and the general-purpose algorithm~\cite{Nedic09-DistributedSubgradientMethodsMultiAgentOptimization} did not converge, i.e., achieve a $10^{-4}$ relative error in less than $250$ CSs.

  In our second set of experiments, shown in \fref{Fig:GlobalExpConsensusAllNets}, we discarded algorithms~\cite{Nedic09-DistributedSubgradientMethodsMultiAgentOptimization} and~\cite{Schizas08-ConsensusAdHocWSNsPartI}, since they exhibited performances inferior to the other algorithms. There are~$5$ plots in \fref{Fig:GlobalExpConsensusAllNets}, one per network type, i.e., row of \tref{Tab:NetworkParametersADNC}. In contrast with the plot of \fref{Fig:GlobalExpConsensusAllAlgs}, the plots of \fref{Fig:GlobalExpConsensusAllNets} show the number of CSs to achieve a relative error of $10^{-4}$ as a function of the network size. For example, in the Watts-Strogatz network with~$200$ nodes (\fref{SubFig:GlobalExpConsWS}), algorithm~\cite{Ling12-MultiBlockAlternatingDirectionMethodParallelSplittingConsensusOptimization} took~$302$ CSs to converge, while~\cite{Olshevsky11-ConvergenceSpeedDistributedConsensusAveraging} took~$116$, \cite{Zhu09-DistributedInNetworkChannelCoding} took~$73$, Algorithm~\ref{Alg:GlobalClass} took~$52$, and~\cite{Oreshkin10-OptimizationAnalysisDistrAveraging} took~$37$. The type of networks for which Algorithm~\ref{Alg:GlobalClass} performed worst was, in fact, Watts-Strogatz type (\fref{SubFig:GlobalExpConsWS}) and Erd\H os-R\'enyi type (\fref{SubFig:GlobalExpConsER}). For the remaining networks, Algorithm~\ref{Alg:GlobalClass} was always among the best. For example, in Barabasi-Albert network types (\fref{SubFig:GlobalExpConsBarabasi}), Algorithm~\ref{Alg:GlobalClass} was always the algorithm requiring the least amount of CSs to converge. From theses experiments, we can conclude that Algorithm~\ref{Alg:GlobalClass}, a general-purpose distributed algorithm, ranks among the most communication-efficient algorithms for solving the average consensus problem.

	\subsection{Row partition: BP and BPDN}

	We now discuss our experiments on other application problems. In this subsection, we consider compressed sensing problems with a row partition (see \fref{Fig:PartitionOfA}), namely, basis pursuit (BP) and basis pursuit denoising (BPDN). These problems, as well as their reformulation as~\eqref{Eq:GlobalProblem}, are discussed in \ssref{SubSec:GC:SparseSolutions}. Next, we mention the experimental setup and how we implemented the computation of the prox operators. Then, we discuss the experimental results.

	\mypar{Experimental setup}
	In our experiments, we used all the networks with $50$ nodes, i.e., all the networks in the second column of \tref{Tab:NetworkParametersADNC}. The exact solution of BP (resp.\ BPDN) was computed in a centralized way with the Matlab toolbox spgl1~\cite{Friedlander08-ProbingParetofrontierBasisPursuit-spgl1} (resp.\ GPSR~\cite{Figueiredo07-GradientProjectionSparseReconstruction}). Knowing the solutions of these problems, we were able to assess the relative error of each algorithm along its iterations. Let~$x^\star$ denote the solution of either BP or BPDN. The relative error is measured as $\|x^k - x^\star\| /\|x^\star\|$, where~$x^k$ is the estimate of an arbitrary node in the network. The algorithms stopped whenever they reached a relative error of $10^{-4}$, or a maximum number of CSs. The maximum number of CSs was~$1000$ for BP and~$2000$ for BPDN. All the algorithms we compare are based on ADMM and, thus, have a tuning parameter~$\rho$. In these experiments, $\rho$ was always chosen as the best value from the set $\{10^{-4},10^{-3},10^{-2},10^{-1},1,10,10^2\}$. Regarding the data, i.e., the matrix~$A \in \mathbb{R}^{m \times n}$ and the vector~$b \in \mathbb{R}^m$, we used two different types of data, one for BP and other for BPDN. For BP, $A$ had dimensions $500 \times 2000$ and each entry was generated randomly and independently from a Gaussian distribution with $0$ mean and standard deviation $1/\sqrt{500} \simeq 0.045$; since there were~$50$ nodes, each node stored a matrix of size~$10 \times 2000$. The vector~$b$ was generated from a sparse linear combination of the columns of~$A$. For BPDN, we used a matrix from problem $902$ of the Sparco toolbox~\cite{Friedlander07-Sparco}. That matrix has dimensions~$200 \times 1000$ and, thus, each node stored a matrix of size $4\times 1000$. The vector~$b$ was generated from a sparse linear combination of the columns of~$A$, to which we added Gaussian noise. The noise parameter~$\beta$ in BPDN (see~\eqref{Eq:BPDN}) was set to $0.3$.

	\mypar{Computation of the prox operator}
	In ADMM-based algorithms, at each iteration, each node has to compute the prox operator of its function. In the case of BP, the function at node~$p$ is given by $f_p(x) = (1/P)\|x\|_1 + \text{i}_{A_px=b_p}(x)$, as shown in~\eqref{Eq:GlobalRecastBP}. Computing the prox of~$f_p$, in this case, is equivalent to finding the minimizer of:
	\begin{equation}\label{Eq:GlobalProxBP}
		\begin{array}{ll}
			\underset{x}{\text{minimize}} & \|x\|_1 + v^\top x + c\|x\|^2 \\
			\text{subject to} & Ax = b\,,
		\end{array}
	\end{equation}
	for some vector~$v \in \mathbb{R}^n$ and some scalar~$c >0$. To simplify, we dropped the subscripts from the matrix~$A$ and the vector~$b$. Since the objective of~\eqref{Eq:GlobalProxBP} is strictly convex, we can find a primal solution by solving its dual problem:
	\begin{equation}\label{Eq:GlobalProxBPDual}
		\begin{array}{ll}
			\underset{\lambda}{\text{maximize}} & b^\top \lambda + \sum_{i=1}^n \inf_{x_i} \Bigl(|x_i| + u_i(\lambda) x_i + c\,x_i^2\Bigr)\,,
 		\end{array}
	\end{equation}
	where $u(\lambda) = v - A^\top \lambda$. We solve~\eqref{Eq:GlobalProxBPDual} with the algorithm in~\cite{Raydan97-BarzilaiBorweinUnconstrainedMinimization}, which is based on the Barzilai-Borwein method. In our implementation, we used warm-starts, that is, at each iteration and for a given node, the algorithm is initialized with the solution that the node found in the previous iteration.

	Regarding BPDN, the function at node~$p$ is given by $f_p(x) = (1/2)\|A_px -b_p\|^2 + (\beta/P) \|x\|_1$, as shown in~\eqref{Eq:GlobalRecastBPDN}. Computing the prox of function~$f_p$, in this case, is actually equivalent to finding a minimizer of a function with the same format as~$f_p$. An efficient method for doing that is GPSR~\cite{Figueiredo07-GradientProjectionSparseReconstruction}, namely GPSR-BB, which uses the Barzilai-Borwein stepsize.

	\begin{figure}
     \centering
     \subfigure[BP]{\label{SubFig:GlobalExpBPRP}
     \begin{pspicture}(3.75,5.1)
       \rput(2.11,2.547){\includegraphics[scale=0.23]{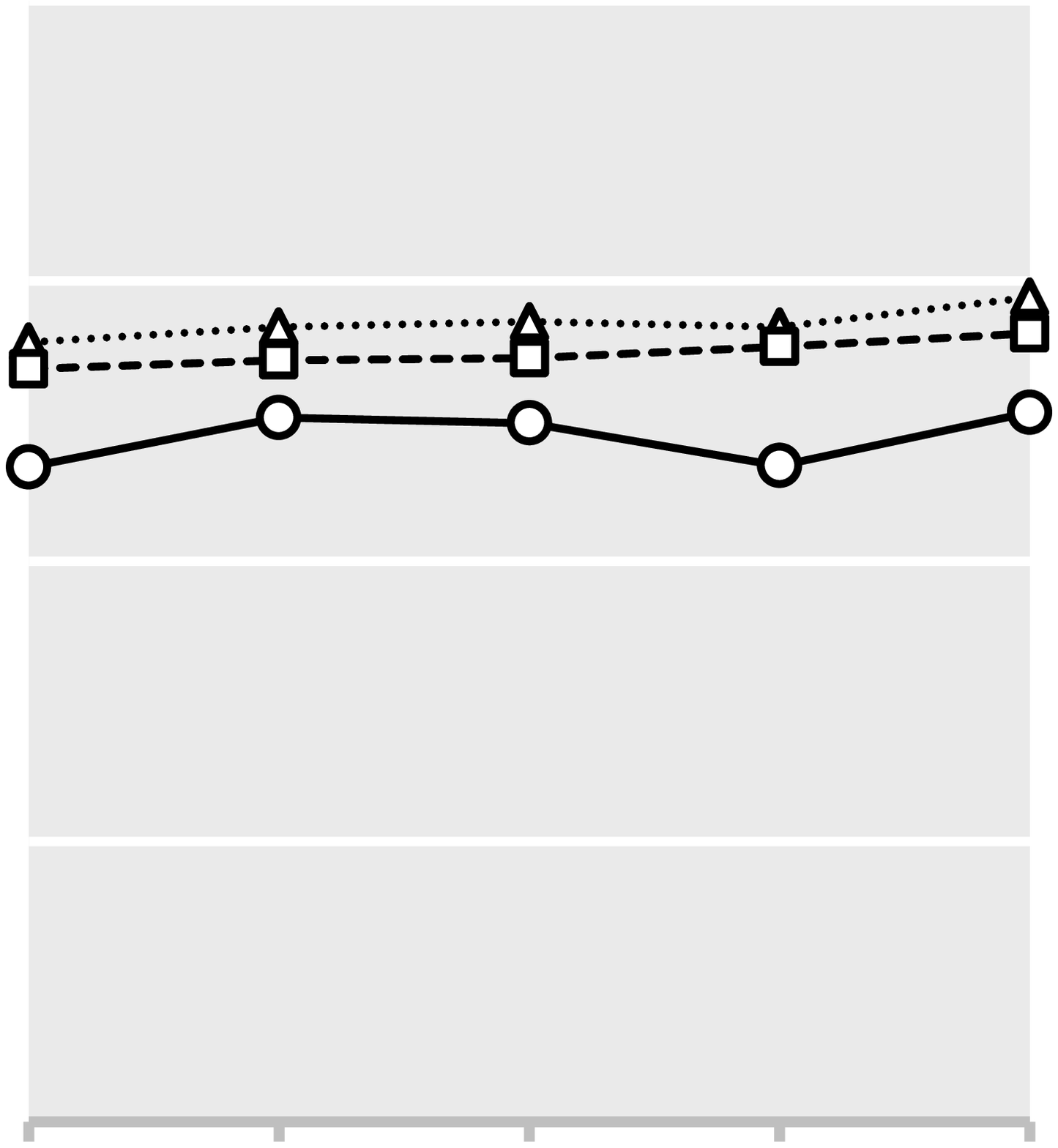}}
       \rput[b](2.1,0.08){\scriptsize \textbf{\sf Network number}}
       \rput[bl](-0.05,4.7){\mbox{\scriptsize \textbf{{\sf Communication steps}}}}

       \rput[r](0.38,4.41){\scriptsize $\mathsf{10^4}$}
       \rput[r](0.38,3.53){\scriptsize $\mathsf{10^3}$}
       \rput[r](0.38,2.60){\scriptsize $\mathsf{10^2}$}
       \rput[r](0.38,1.71){\scriptsize $\mathsf{10^1}$}
       \rput[r](0.38,0.82){\scriptsize $\mathsf{10^0}$}

       \rput[t](0.483,0.58){\scriptsize $\mathsf{1}$}
       \rput[t](1.305,0.58){\scriptsize $\mathsf{2}$}
       \rput[t](2.123,0.58){\scriptsize $\mathsf{3}$}
       \rput[t](2.928,0.58){\scriptsize $\mathsf{4}$}
       \rput[t](3.750,0.58){\scriptsize $\mathsf{5}$}

       \rput[rt](2.8,2.85){\scriptsize \textbf{\sf Alg.\ \ref{Alg:GlobalClass}}}
       \rput[lb](2.6,3.6){\scriptsize \textbf{\sf \cite{Zhu09-DistributedInNetworkChannelCoding}}}
       \psline[linewidth=0.5pt]{-}(2.3,3.29)(2.59,3.59)
       \rput[lb](1.7,3.7){\scriptsize \textbf{\sf \cite{Schizas08-ConsensusAdHocWSNsPartI}}}
       \psline[linewidth=0.5pt]{-}(1.5,3.38)(1.69,3.69)

       %\psgrid
     \end{pspicture}
     }
     \hfill
     \subfigure[BPDN]{\label{SubFig:GlobalExpBPDN}
     \begin{pspicture}(3.75,5.1)
       \rput(2.11,2.547){\includegraphics[scale=0.23]{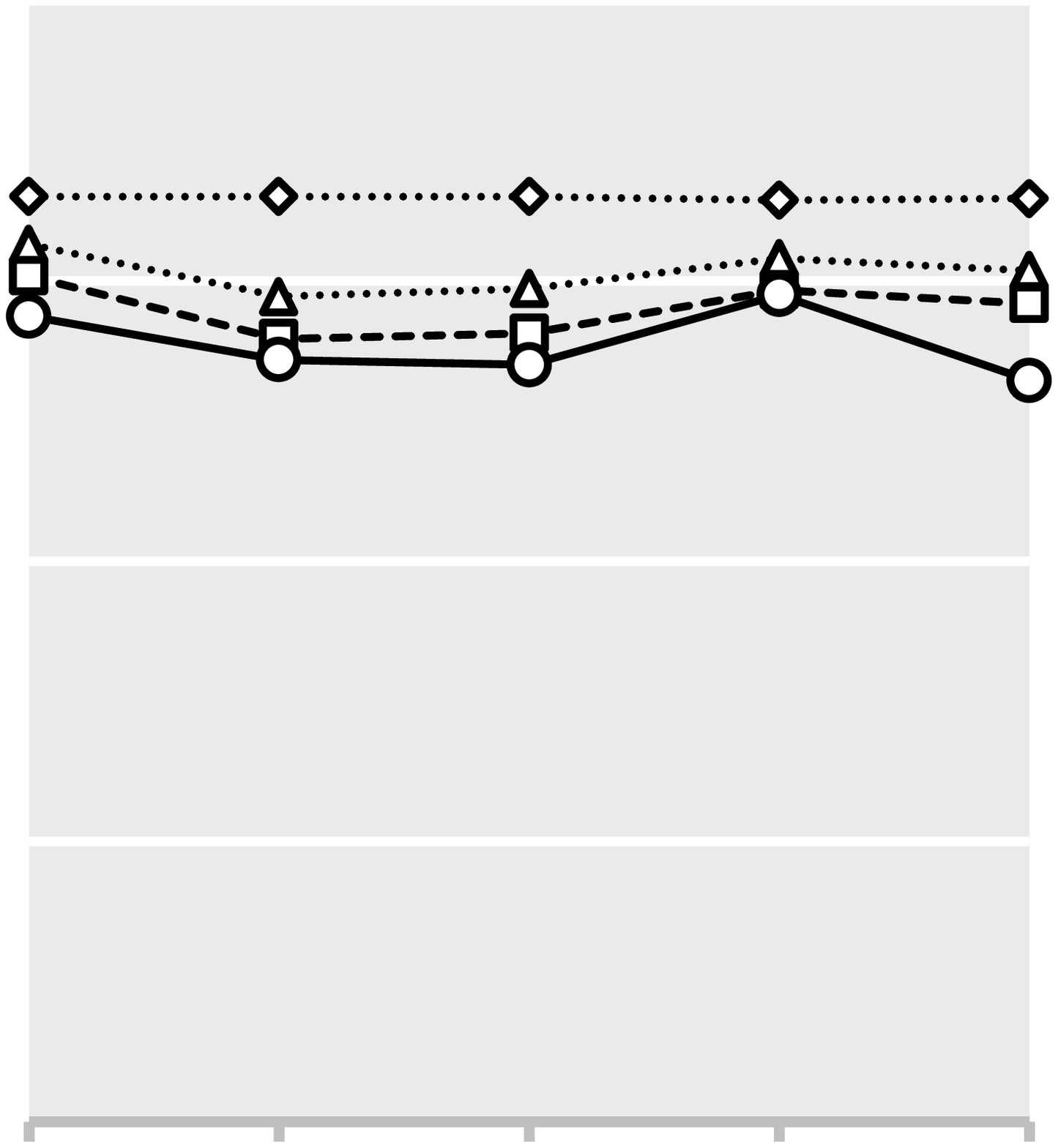}}
       \rput[b](2.1,0.08){\scriptsize \textbf{\sf Network number}}
       \rput[bl](-0.05,4.7){\mbox{\scriptsize \textbf{{\sf Communication steps}}}}

       \rput[r](0.38,4.41){\scriptsize $\mathsf{10^4}$}
       \rput[r](0.38,3.53){\scriptsize $\mathsf{10^3}$}
       \rput[r](0.38,2.60){\scriptsize $\mathsf{10^2}$}
       \rput[r](0.38,1.71){\scriptsize $\mathsf{10^1}$}
       \rput[r](0.38,0.82){\scriptsize $\mathsf{10^0}$}

       \rput[t](0.483,0.58){\scriptsize $\mathsf{1}$}
       \rput[t](1.305,0.58){\scriptsize $\mathsf{2}$}
       \rput[t](2.123,0.58){\scriptsize $\mathsf{3}$}
       \rput[t](2.928,0.58){\scriptsize $\mathsf{4}$}
       \rput[t](3.750,0.58){\scriptsize $\mathsf{5}$}

       \rput[rt](3.7,3.1){\scriptsize \textbf{\sf Alg.\ \ref{Alg:GlobalClass}}}
       \rput[lb](2.6,4.0){\scriptsize \textbf{\sf \cite{Schizas08-ConsensusAdHocWSNsPartI}}}
       \psline[linewidth=0.5pt](2.3,3.53)(2.59,3.99)
       \rput[rt](1.5,3.0){\scriptsize \textbf{\sf \cite{Zhu09-DistributedInNetworkChannelCoding}}}
       \psline[linewidth=0.5pt](1.51,3.01)(1.8,3.31)
			 \rput[lb](1.3,3.85){\scriptsize \textbf{\sf \cite{Mateos10-DistributedSparseLinearRegression}}}

       %\psgrid
     \end{pspicture}
     }
     \hfill
     \subfigure[Reversed lasso]{\label{SubFig:GlobalExpReversedLasso}
     \begin{pspicture}(3.75,5.1)
       \rput(2.11,2.547){\includegraphics[scale=0.23]{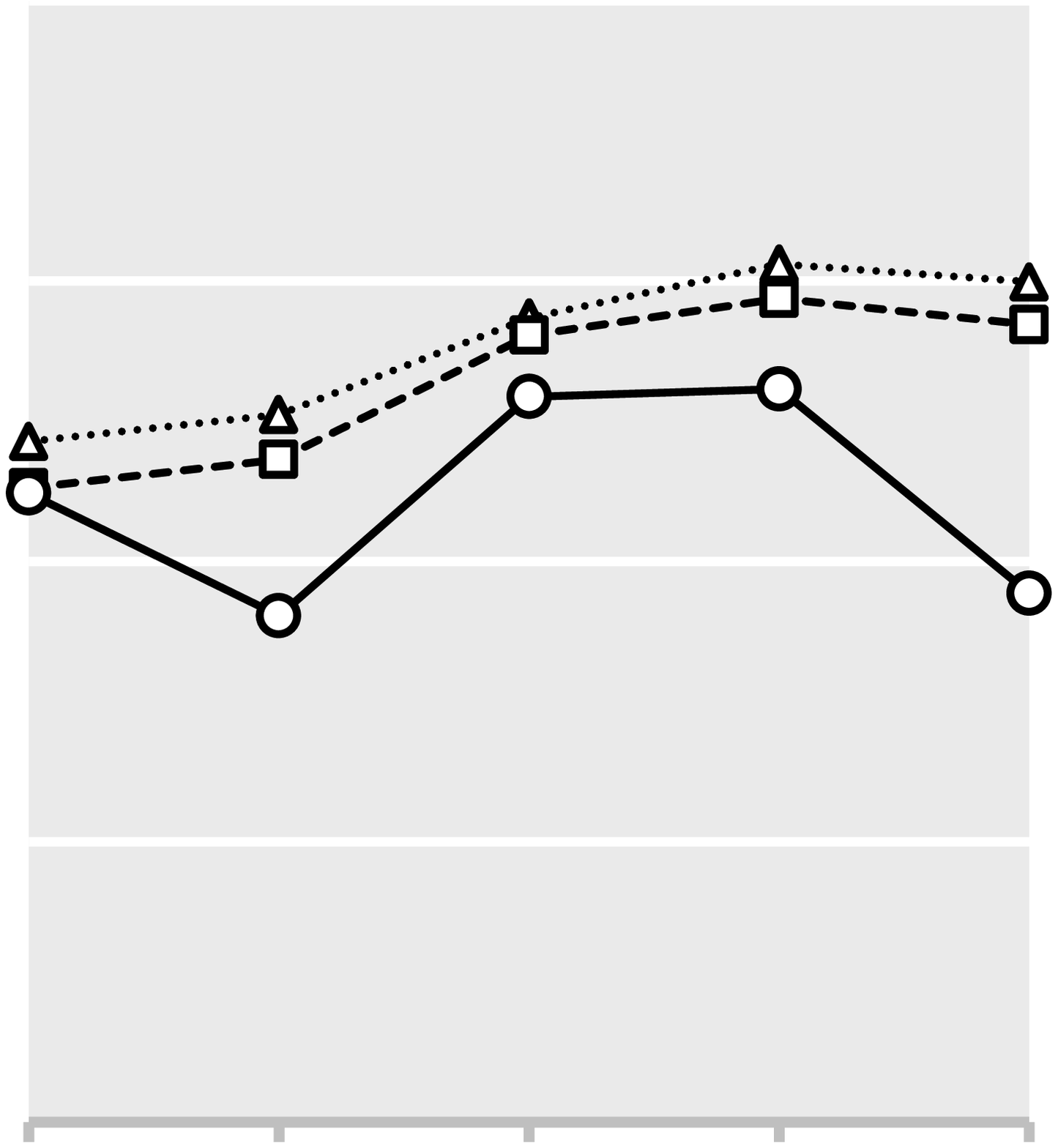}}
       \rput[b](2.1,0.08){\scriptsize \textbf{\sf Network number}}
       \rput[bl](-0.05,4.7){\mbox{\scriptsize \textbf{{\sf Communication steps}}}}

       \rput[r](0.38,4.41){\scriptsize $\mathsf{10^4}$}
       \rput[r](0.38,3.53){\scriptsize $\mathsf{10^3}$}
       \rput[r](0.38,2.60){\scriptsize $\mathsf{10^2}$}
       \rput[r](0.38,1.71){\scriptsize $\mathsf{10^1}$}
       \rput[r](0.38,0.82){\scriptsize $\mathsf{10^0}$}

       \rput[t](0.483,0.58){\scriptsize $\mathsf{1}$}
       \rput[t](1.305,0.58){\scriptsize $\mathsf{2}$}
       \rput[t](2.123,0.58){\scriptsize $\mathsf{3}$}
       \rput[t](2.928,0.58){\scriptsize $\mathsf{4}$}
       \rput[t](3.750,0.58){\scriptsize $\mathsf{5}$}

       \rput[lt](1.4,2.4){\scriptsize \textbf{\sf Alg.\ \ref{Alg:GlobalClass}}}
       \rput[rt](3.76,3.27){\scriptsize \textbf{\sf \cite{Zhu09-DistributedInNetworkChannelCoding}}}
       \rput[lb](3.2,3.6){\scriptsize \textbf{\sf \cite{Schizas08-ConsensusAdHocWSNsPartI}}}

       %\psgrid
     \end{pspicture}
     }
     \hfill
     \subfigure[SVM]{\label{SubFig:GlobalExpSVM}
     \begin{pspicture}(3.75,5.1)
       \rput(2.11,2.547){\includegraphics[scale=0.23]{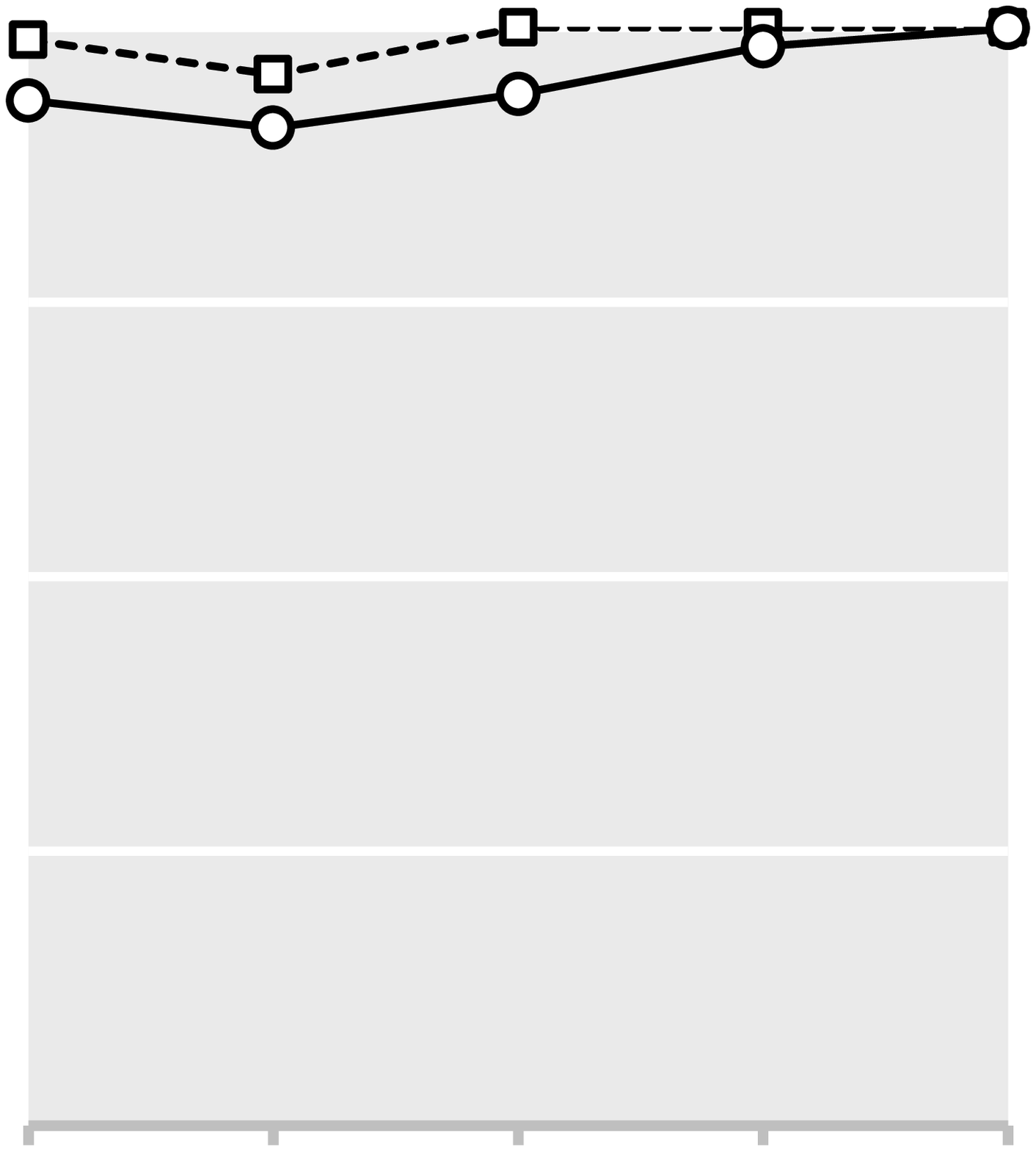}}
       \rput[b](2.1,0.08){\scriptsize \textbf{\sf Network number}}
       \rput[bl](-0.05,4.7){\mbox{\scriptsize \textbf{{\sf Communication steps}}}}

       \rput[r](0.38,4.41){\scriptsize $\mathsf{10^4}$}
       \rput[r](0.38,3.53){\scriptsize $\mathsf{10^3}$}
       \rput[r](0.38,2.60){\scriptsize $\mathsf{10^2}$}
       \rput[r](0.38,1.71){\scriptsize $\mathsf{10^1}$}
       \rput[r](0.38,0.82){\scriptsize $\mathsf{10^0}$}

       \rput[t](0.483,0.58){\scriptsize $\mathsf{1}$}
       \rput[t](1.305,0.58){\scriptsize $\mathsf{2}$}
       \rput[t](2.123,0.58){\scriptsize $\mathsf{3}$}
       \rput[t](2.928,0.58){\scriptsize $\mathsf{4}$}
       \rput[t](3.750,0.58){\scriptsize $\mathsf{5}$}

       \rput[lt](0.67,3.92){\scriptsize \textbf{\sf Alg.\ \ref{Alg:GlobalClass}}}
       \rput[lt](1.92,3.92){\scriptsize \textbf{\sf \cite{Zhu09-DistributedInNetworkChannelCoding}}}
       \psline[linewidth=0.5pt](1.92,3.97)(1.7,4.24)
       %\psgrid
     \end{pspicture}
     }
     \caption[Results of the simulations for BP, reversed lasso, BPDN, and SVM.]{Results of the simulations for $\text{(a)}$ BP, $\text{(b)}$ reversed lasso, $\text{(c)}$ BPDN, and $\text{(d)}$ SVM. The simulations were run on all the networks with $50$ nodes.}
     \label{Fig:ExpGlobalBPBPDNLASSOSVM}
	\end{figure}

	\mypar{Results}
	The results for BP and BPDN are shown in Figures~\ref{SubFig:GlobalExpBPRP} and~\ref{SubFig:GlobalExpBPDN}, respectively. In \fref{SubFig:GlobalExpBPRP}, we compare Algorithm~\ref{Alg:GlobalClass} against the ADMM-based methods~\cite{Schizas08-ConsensusAdHocWSNsPartI} and~\cite{Zhu09-DistributedInNetworkChannelCoding} (written, as Algorithms~\ref{Alg:Schizas} and~\ref{Alg:Zhu}, respectively). The behavior of the algorithms in this figure is very uniform: in all the networks, Algorithm~\ref{Alg:GlobalClass} was the one requiring the least amount of CSs to converge, i.e., to achieve a relative error of $10^{-4}$, the algorithm in~\cite{Schizas08-ConsensusAdHocWSNsPartI} was always the one requiring the largest amount of CSs, and the algorithm in~\cite{Zhu09-DistributedInNetworkChannelCoding} was always in between.

	The exact same behavior can be observed in \fref{SubFig:GlobalExpBPDN} for the BPDN, although the lines, and thus their performance, are closer together. The figure also shows the performance of~\cite[Alg.3]{Mateos10-DistributedSparseLinearRegression}, which is an ADMM-based method specifically designed to solve BPDN. That algorithm has the advantage of requiring simpler computations at each node but, as seen in the figure, at the cost of spending more CSs to converge. In fact, that algorithm achieved the maximum number of CSs, i.e., it failed to converge, in all but the last two networks.

	\subsection{Column partition: reversed lasso}

	In this subsection we give an example of a compressed sensing problem with a column partition. In particular, we consider the reversed lasso~\eqref{Eq:ReverseLasso}, which we showed how to recast as~\eqref{Eq:GlobalProblem} in \ssref{SubSec:GC:SparseSolutions}. As in the previous subsections, we first describe the experimental setup, then how we computed the prox operator at each node and, finally, we present the experimental results.

	\mypar{Experimental setup}
	The set of networks is the same as in the experiments for BP and BPDN, i.e., all the networks with~$50$ nodes. To compute the problem's solution~$x^\star$ beforehand, we used the Matlab toolbox spgl1~\cite{Friedlander08-ProbingParetofrontierBasisPursuit-spgl1}. We ran the algorithms either until they reached a maximum number of~$1000$ CSs or until they reached a relative error of $5\times 10^{-3}$. The relative error has the same expression as before, $\|x^k - x^\star\|/\|x^\star\|$, but now $x^k$ is the concatenation of all the nodes's estimates, i.e., $x^k = (x_1^k,x_2^k,\ldots,x_P^k)$; recall that in the column partition, the variable is partitioned into blocks and each block is estimated by a single node. Recall also that, in order to recast the reversed lasso as~\eqref{Eq:GlobalProblem}, we compute the dual of a regularized version of the problem; see~\eqref{Eq:RegularizarionReversedLasso3}. The regularization parameter~$\delta$ was set to~$10^{-2}$ and the noise tolerance~$\sigma$ to~$0.1$. Again, the parameter~$\rho$ was selected from the set~$\{10^{-4},10^{-3},10^{-2},10^{-1},1,10,10^2\}$. The problem data is the same as in BPDN, i.e., the matrix~$A$ was taken from problem $902$ of the Sparco toolbox~\cite{Friedlander07-Sparco}. This means that~$A$ had dimensions $200 \times 1000$ and, given the column partition, each node stored a matrix of size~$200 \times 20$.

	\mypar{Computation of the prox operator}
	In \ssref{SubSec:GC:SparseSolutions} we manipulated the reversed lasso in order to recast it as~\eqref{Eq:GlobalProblem}.
	More specifically, the dual problem of a regularized version of reversed lasso can be written as~\eqref{Eq:ReversedLassoDualization4}, where the function at node~$p$ is
	$$
		f_p(\lambda) = h_p^\star(A_p^\top \lambda) + \frac{\sigma}{P}\|\lambda\| - \frac{1}{P}b^\top \lambda\,,
	$$
	and~$h_p^\star$ is the convex conjugate of $h_p(x_p) = \|x_p\|_1 + (\delta/2)\|x_p\|^2$. It can be shown that computing the prox of~$f_p$ is equivalent to finding the minimizer of the optimization problem:
	\begin{equation}\label{Eq:GlobalProxReversedLasso}
		\begin{array}{ll}
			\underset{\lambda}{\text{minimize}} & h_p^\star(A_p^\top \lambda) + \frac{\sigma}{P}\|\lambda\| - \frac{1}{P}b^\top \lambda + v^\top \lambda + c\|\lambda\|^2\,,
		\end{array}
	\end{equation}
	for some vector $v \in \mathbb{R}^m$ and some scalar~$c \in \mathbb{R}$. Introducing an epigraph variable~$t$, \eqref{Eq:GlobalProxReversedLasso} becomes equivalent to
	\begin{equation}\label{Eq:GlobalProxReversedLasso2}
		\begin{array}{ll}
			\underset{\lambda,t}{\text{minimize}} & h_p^\star(A_p^\top \lambda) + \frac{\sigma}{P}t - \frac{1}{P}b^\top \lambda + v^\top \lambda + c\|\lambda\|^2 \\
			\text{subject to} & \|\lambda\| \leq t\,.
		\end{array}
	\end{equation}
	Since $h_p$ is strongly convex, its conjugate~$h_p^\star$ is differentiable and its gradient is Lipschitz-continuous. In fact, the entire objective function of~\eqref{Eq:GlobalProxReversedLasso2} is differentiable and its gradient is Lipschitz-continuous with constant~$\sigma_{\max}^2(A_p)/\delta + 2c$, where $\sigma_{\max}(A_p)$ is the largest singular value of~$A_p$. Moreover, given an arbitrary point~$(\lambda,t)$, its projection onto the Lorenz cone $\{(\lambda,t)\,:\, \|\lambda\|\leq t\}$ is given in closed-form by~\cite[A.2.7]{Beck02-ConvergenceRateAnalysisGradientBasedAlgorithms-thesis}
	$$
		\left\{
			\begin{array}{ll}
				(\lambda,t) &,\,\text{if}\,\,\, t\geq \|\lambda\| \\
				(0,0) &,\,\text{if}\,\,\, t \leq -\|\lambda\| \\
				\frac{t+\|\lambda\|}{2}(\frac{\lambda}{\|\lambda\|},1) &,\,\text{if}\,\,\, -\|x\| < t < \|x\|\,.
			\end{array}
		\right.
	$$
	Therefore, \eqref{Eq:GlobalProxReversedLasso} can be solved with projected gradient methods. We solve it with Nesterov's projected gradient method~\eqref{Eq:RelatedWorkNesterovAlg}, also known as FISTA~\cite{Beck09-FISTA}, whose convergence rate is $O(1/k^2)$.

	\mypar{Results}
	The results of the reversed lasso experiments are shown in \fref{SubFig:GlobalExpReversedLasso}. There, Algorithm~\ref{Alg:GlobalClass} is compared against the algorithms in~\cite{Schizas08-ConsensusAdHocWSNsPartI} and in~\cite{Zhu09-DistributedInNetworkChannelCoding}. They exhibit the same behavior we had observed in Figures~\ref{SubFig:GlobalExpBPRP} and~\ref{SubFig:GlobalExpBPDN}: Algorithm~\ref{Alg:GlobalClass} required uniformly less CSs to converge. Also, \cite{Zhu09-DistributedInNetworkChannelCoding} required uniformly less CSs than~\cite{Schizas08-ConsensusAdHocWSNsPartI} to converge.

	\subsection{SVM}

	Finally, we present our experimental results for training an SVM~\eqref{Eq:SVM}. Among all experiments that we performed, the ones for SVM required the largest number of CSs to converge, as can be seen by comparing all the plots in \fref{Fig:ExpGlobalBPBPDNLASSOSVM}. The results for the SVM experiments are shown in \fref{SubFig:GlobalExpSVM}. But before we analyze them, we describe the experimental setup and how we computed the respective prox operator.

	\mypar{Experimental setup}
	As in the other plots in the same figure, the experiments for the SVM problem~\eqref{Eq:SVM} were executed on the networks with~$50$ nodes. Since problem~\eqref{Eq:SVM} can be recast as a quadratic program, we obtained the problem's solution beforehand using the \verb#quadprog# function of the Matlab optimization toolbox~\cite{OptimizationToolboxMatlab}. The algorithms ran until they achieved a maximum number of~$10^4$ CSs, or a relative error of~$10^{-3}$. The relative error in this case was measured exactly as in the compressed sensing problems with a row partition: $\|x^k - x^\star\|/\|x^\star\|$, where~$x^k$ is the estimate at an arbitrary node. And the augmented Lagrangian parameter~$\rho$ was selected exactly as in the previous experiments. Regarding the problem data, i.e., the sets of datapoints~$(x_k, y_k)$ in~\eqref{Eq:SVM}, we used data from~\cite{UCIMachineLearningRepository}, namely two overlapping sets of datapoints from the Iris dataset. In total, there were~$m = 100$ points of size~$n = 4$, which means that each node stored~$2$ datapoints. The parameter~$\beta$ in~\eqref{Eq:SVM} was set to~$1$ in all the experiments.

	\mypar{Computation of the prox operator}
	We showed in~\ssref{SubSec:GC:SparseSolutions} that in the SVM problem the function at each node is given by~\eqref{Eq:GlobalFunctionNodeSVM}. It can be easily seen that computing the prox operator of~\eqref{Eq:GlobalFunctionNodeSVM} is equivalent to finding a minimizer of a quadratic program with inequality constraints. This problem has no closed-form solution, but it can be solved with standard quadratic program solvers, such as Matlab's \verb#quadprog# function. We used this function in our implementation.

	\mypar{Results}
	As mentioned, the results of the experiments for SVM are shown in \fref{SubFig:GlobalExpSVM}. In this case, the algorithm in~\cite{Schizas08-ConsensusAdHocWSNsPartI} achieved always the maximum number of CSs and, thus, is not represented in the plot. Both Algorithm~\ref{Alg:GlobalClass} and the algorithm in~\cite{Zhu09-DistributedInNetworkChannelCoding} required always more than $1000$ CSs to converge for all the networks. Again, Algorithm~\ref{Alg:GlobalClass} required the least number of CSs to converge, never achieving the maximum number of $10^4$ CSs. In contrast, the algorithm in~\cite{Zhu09-DistributedInNetworkChannelCoding} achieved the maximum number of CSs in all but the first two networks.

%% file: 02-MainMatter/partialClass.tex
\chapter{Connected and Non-Connected Classes}
\label{Ch:ConnectedNonConnected}

	In this chapter, we solve problem~\eqref{Eq:IntroProb} with a generic variable, following ideas similar to the ones presented in the previous chapter for the global class. We first address the case of a connected variable, which is simpler, and then we see how to handle a non-connected variable. This chapter is based on the publications~\cite{Mota13-DistributedOptimizationLocalDomains,Mota12-DistributedADMMForMPCAndCongestionControl-CDC,Mota13-UnifiedAlgorithmicApproachDistributedOptimization} and is organized as follows: in \sref{Sec:PartialProblemStatement}, we formally state the problem and outline our assumptions; then, in \sref{Sec:PartialApplications}, we describe some application problems that can be written as~\eqref{Eq:IntroProb} with a non-global variable. These include distributed model predictive control (D-MPC), network flow problems, and the reversed lasso with a row partition. In particular, we propose a new framework for D-MPC that considerably extends the modeling capability of the standard D-MPC; this, for example, will allow us to model scenarios where systems coupled through their dynamics do not necessarily communicate directly. Next, in \sref{Sec:PartialAlgorithmDerivation}, we derive our algorithm, first for a connected variable, and then for a non-connected variable. This will give us the most general algorithm in this thesis. Finally, in \sref{Sec:PartialExpResults}, we show how the performance of the proposed algorithm compares with prior algorithms for some of the problems introduced in \sref{Sec:PartialApplications}.

	\section{Problem statement}
	\label{Sec:PartialProblemStatement}

	As in the global class, here we also minimize the sum of~$P$ functions, where each function is known at one node only. However, each function here, rather than depending on all the components of the variable~$x \in \mathbb{R}^n$, depends only on the ones indexed by the set~$S_p \subseteq \{1,\ldots,n\}$. That is, we solve
	\begin{equation}\tag{P}
		\begin{array}{ll}
			\underset{x \in \mathbb{R}^n}{\text{minimize}} & f_1(x_{S_1}) + f_2(x_{S_2}) + \cdots + f_P(x_{S_P})\,.
		\end{array}
	\end{equation}
	We make the following assumptions:
	\begin{assumption}\label{Ass:PPfunctions}
		Each function~$f_p : \mathbb{R}^{n_p} \xrightarrow{} \mathbb{R} \cup \{+\infty\}$ is closed and convex over~$\mathbb{R}^{n_p}$ and not identically~$+\infty$.
	\end{assumption}
	\begin{assumption}\label{Ass:PPsolvable}
		Problem~\eqref{Eq:GlobalProblem} is solvable, i.e., it has at least one solution~$x^\star \in \mathbb{R}^n$.
	\end{assumption}
	Assumptions~\ref{Ass:PPfunctions} and~\eqref{Ass:PPsolvable} are essentially the same we made for the global class. The only difference is that now each function~$f_p$ is defined over~$\mathbb{R}^{n_p}$, where~$n_p = |S_p|$, and not over the entire domain of the variable~$x$, $\mathbb{R}^n$. Note that the sum of the dimensions of the domains of each function, i.e., $n_1 + \cdots + n_P$, is always less than or equal to the corresponding sum in the case of a global variable, which is~$nP$. In other words, $n_1 + \cdots + n_P \leq nP$. The following assumption makes the problem well-formulated by guaranteeing that, for each component~$x_l$, there is always one node~$p$ that depends on~$x_l$, i.e., $l \in S_p$:
	\begin{assumption}\label{Ass:PPWellForm}
		There holds $\cup_{p=1}^P S_p = \{1,2,\ldots,n\}$.
	\end{assumption}
	This assumption was not required for the global class, because all functions there depended on all the components of the variable. Regarding the network, we make exactly the same assumptions we made for the global class:
	\begin{assumption}\label{Ass:PPConnectedStatic}
		The network is connected and does not vary with time.
	\end{assumption}
	\begin{assumption}\label{Ass:PPColoring}
		A coloring scheme~$\mathcal{C}$ of the network is available; each node knows its own color and the color of its neighbors.
	\end{assumption}
	The comments we made in \sref{Sec:GC:ProblemStatement} about these assumptions also apply here. Next, we describe some application problems that can be written as~\eqref{Eq:IntroProb} with a non-global variable and under Assumptions~\ref{Ass:PPfunctions}-\ref{Ass:PPColoring}.

	\section{Applications}
	\label{Sec:PartialApplications}

	There are many problems in signal processing, control engineering, and machine learning that can be written as~\eqref{Eq:IntroProb}. In the previous chapter, we described some that require a global variable. In this section, we focus on problems that require a variable that is non-global, for example, a star-shaped or a mixed variable. We start with distributed model predictive control (D-MPC), which appears in the literature as an instance of~\eqref{Eq:IntroProb} with a star-shaped variable. One of the contributions of this thesis is a new framework for D-MPC that uses generic connected, and even non-connected, variables. This new framework allows modeling D-MPC scenarios where systems that are coupled through their dynamics need not to communicate directly. The second application we will see is the distributed compressed sensing problem \textit{reversed lasso with a row partition}, which we formulate as~\eqref{Eq:IntroProb} with a mixed variable. Then, we describe three applications that have been solved with distributed algorithms: network flow problems (star-shaped variable), network utility maximization (NUM) (star-shaped and mixed variable), and state estimation in power networks (star-shaped variable). The last application, state estimation in power networks, is described in~\cite{Kekatos12-DistributedRobustPowerStateEstimation}, which also proposes an ADMM-based algorithm to solve it. This is the only algorithm we found in the literature that can be easily generalized to solve~\eqref{Eq:IntroProb} for all types of variables. At the end of this section, we will describe the algorithm in~\cite{Kekatos12-DistributedRobustPowerStateEstimation} for a generic connected variable.

	\subsection{Distributed model predictive control}
	\label{SubSec:DMPC}

	This subsection describes model predictive control (MPC), first from a centralized perspective, and then from a distributed one.

	\mypar{Centralized MPC}
	As mentioned in \cref{Ch:relatedWork}, model predictive control (MPC) is a popular strategy for controlling discrete-time systems. In MPC, a system is described at each time instant~$t$ by its state-space vector~$x[t] \in \mathbb{R}^n$, whose value at time~$t+1$ is determined by the state and control input at time~$t$. Mathematically, $x[t+1] = \Theta^t(x[t],u[t])$, where~$u[t] \in \mathbb{R}^m$ denotes the control input applied to the system at time~$t$ and~$\Theta^t:\mathbb{R}^n\times \mathbb{R}^m \xrightarrow{} \mathbb{R}^n$ is an arbitrary, time-variant map modeling the system. Being a control strategy, the goal of MPC is to take the state vector of the system from an initial point $x[0]$ to some predefined ``goal state.'' To be more concrete, let $\Phi : \mathbb{R}^n \xrightarrow{} \mathbb{R}$ be a function that penalizes deviations from the goal state or, in other words, $\Phi(x)$ increases with the distance of~$x$ to the goal state. Almost always, there are several possible paths from~$x[0]$ to the goal state and, typically, these paths have different energy consumptions, for example, the energy spent on the input signals $u[0], u[1], \ldots$. We model energy consumption at time~$t$ with the function $\Psi^t(x[t],u[t])$. Therefore, we want to choose the path from~$x[0]$ to the goal state that uses the minimum amount of energy; this is actually the problem solved by MPC. However, in MPC, we make the key assumption that the system can measure its state at each time instant. This capability is used to mitigate model inaccuracies and disturbances to the system. It works as follows: instead of solving the problem at once, time is divided into slots of~$T$ units, where~$T$ is called the time-horizon. At each time-instant, the time variable~$t$ is set to zero and the state is measured, say, $x[0] = x^0$, where~$x^0$ is the known measurement. Then, the following optimization problem is solved for a time-horizon~$T$:
	\begin{equation}\label{Eq:MPC}
    \begin{array}{ll}
      \underset{\bar{x},\bar{u}}{\text{minimize}} & \Phi(x[T]) + \sum_{t=0}^{T-1} \Psi^t(x[t], u[t]) \\
      \text{subject to} & x[t+1] = \Theta^t(x[t], u[t])\,,\quad t=0,\ldots,T-1 \\
                        & x[0] = x^0\,,
    \end{array}
  \end{equation}
	where~$(\bar{x},\bar{u}):= (\{x[t]\}_{t=0}^T, \{u[t]\}_{t=0}^{T-1})$ is the optimization variable and  represents the set of states (resp.\ inputs) from time~$t=0$ to time~$T$ (resp.\ $T-1$). In the objective of~\eqref{Eq:MPC}, there is a tradeoff between achieving the goal state at time~$T$, expressed by the term~$\Phi(x[T])$, and minimizing the path energy, expressed by the term $\sum_{t=0}^{T-1} \Psi^t(x[t], u[t])$. While the first constraint in~\eqref{Eq:MPC} enforces the state to satisfy the system dynamics, the second constraint encodes the measurement~$x^0$. After solving problem~\eqref{Eq:MPC}, the first input~$u[0]$ is applied to the system, the time~$t$ is again set to zero, and the process is repeated. This means that, at each time instant, only the first input is used, even though a set of inputs and states are computed for the entire horizon from $t=0$ to~$t=T$. MPC thus provides a conservative strategy to deal with model inaccuracies and system disturbances, which perhaps explains its effectiveness and, consequently, its popularity.

	\begin{figure}
  \centering
  \subfigure[Connected star-shaped variable]{\label{SubFig:PPMPCStar}
    \psscalebox{1.00}{
      \begin{pspicture}(4.6,5.0)
        \def\nodesimp{
          \pscircle*[linecolor=black!65!white](0,0){0.3}
        }
        \def\nodeshigh{
          \pscircle*[linecolor=black!25!white](0,0){0.3}
        }

        \rput(0.4,4.1){\rnode{C1}{\nodesimp}}   \rput(0.4,4.1){\small \textcolor{white}{$1$}}
        \rput(0.9,2.6){\rnode{C2}{\nodesimp}}   \rput(0.9,2.6){\small \textcolor{white}{$2$}}
        \rput(0.8,0.9){\rnode{C3}{\nodesimp}}   \rput(0.8,0.9){\small \textcolor{white}{$3$}}
        \rput(3.0,1.1){\rnode{C4}{\nodesimp}}   \rput(3.0,1.1){\small \textcolor{white}{$4$}}
        \rput(4.0,2.6){\rnode{C5}{\nodesimp}}   \rput(4.0,2.6){\small \textcolor{white}{$5$}}
        \rput(2.2,3.3){\rnode{C6}{\nodesimp}}   \rput(2.2,3.3){\small \textcolor{white}{$6$}}

        \psset{nodesep=0.33cm,linewidth=1.1pt}
        \ncline{-}{C1}{C2}
        \ncline{-}{C1}{C6}
        \ncline{-}{C2}{C3}
        \ncline{-}{C2}{C6}
        \ncline{-}{C3}{C4}
        \ncline{-}{C4}{C5}
        \ncline{-}{C5}{C6}

        \psset{linestyle=dotted,dotsep=0.7pt,arrowsize=7pt,arrowinset=0.2,arrowlength=0.55,linewidth=1.1pt,arcangleA=22,arcangleB=22,ArrowInside=->,ArrowInsidePos=0.56}
        \ncarc{-}{C1}{C2}
        \ncarc{-}{C2}{C1}
        \ncarc{-}{C1}{C6}
        \ncarc{-}{C6}{C1}
        \ncarc{-}{C2}{C3}
        \ncarc{-}{C3}{C2}
        \ncarc{-}{C2}{C6}
        \ncarc{-}{C6}{C2}
        \ncarc{-}{C3}{C4}
        \ncarc{-}{C4}{C3}
        \ncarc{-}{C4}{C5}
        \ncarc{-}{C5}{C4}
        \ncarc{-}{C6}{C5}
        \ncarc{-}{C5}{C6}

				\psset{ncurv=3}
        \nccurve[angleA=65,angleB=125]{-}{C1}{C1}
        \nccurve[angleA=150,angleB=210]{-}{C2}{C2}
        \nccurve[angleA=210,angleB=270]{-}{C3}{C3}
        \nccurve[angleA=270,angleB=330]{-}{C4}{C4}
        \nccurve[angleA=20,angleB=80]{-}{C5}{C5}
        \nccurve[angleA=55,angleB=115]{-}{C6}{C6}

        %\psgrid
      \end{pspicture}
    }
  }
  \hspace{2cm}
  \subfigure[Non-connected variable]{\label{SubFig:PPMPCNonConnected}
    \psscalebox{1.00}{
      \begin{pspicture}(4.6,5.0)
        \def\nodesimp{
          \pscircle*[linecolor=black!65!white](0,0){0.3}
        }

        \rput(0.4,4.1){\rnode{C1}{\nodesimp}}   \rput(0.4,4.1){\small \textcolor{white}{$1$}}
        \rput(0.9,2.6){\rnode{C2}{\nodesimp}}   \rput(0.9,2.6){\small \textcolor{white}{$2$}}
        \rput(0.8,0.9){\rnode{C3}{\nodesimp}}   \rput(0.8,0.9){\small \textcolor{white}{$3$}}
        \rput(3.0,1.1){\rnode{C4}{\nodesimp}}   \rput(3.0,1.1){\small \textcolor{white}{$4$}}
        \rput(4.0,2.6){\rnode{C5}{\nodesimp}}   \rput(4.0,2.6){\small \textcolor{white}{$5$}}
        \rput(2.2,3.3){\rnode{C6}{\nodesimp}}   \rput(2.2,3.3){\small \textcolor{white}{$6$}}

        \psset{nodesep=0.33cm,linewidth=1.1pt}
        \ncline{-}{C1}{C2}
        \ncline{-}{C1}{C6}
        \ncline{-}{C2}{C3}
        \ncline{-}{C2}{C6}
        \ncline{-}{C3}{C4}
        \ncline{-}{C4}{C5}
        \ncline{-}{C5}{C6}

        \psset{linestyle=dotted,dotsep=0.7pt,arrowsize=7pt,arrowinset=0.2,arrowlength=0.55,linewidth=1.1pt,arcangleA=22,arcangleB=22,ArrowInside=->,ArrowInsidePos=0.56}
        \ncarc{-}{C1}{C2}
        \ncarc{-}{C2}{C1}
        \ncarc[arcangleA=38,arcangleB=38]{-}{C3}{C1}
        \ncarc{-}{C2}{C3}
        %\ncarc[arcangleA=0,arcangleB=0]{-}{C2}{C4}
        \ncarc[arcangleA=0,arcangleB=0,ArrowInsidePos=0.65]{-}{C2}{C5}
        \ncarc{-}{C3}{C2}
        \ncarc{-}{C3}{C4}
        \ncarc[arcangleA=0,arcangleB=0]{-}{C4}{C6}
        \ncarc{-}{C1}{C6}
        \ncarc{-}{C6}{C5}
        \ncarc{-}{C5}{C4}

        \psset{ncurv=3}
        \nccurve[angleA=65,angleB=125]{-}{C1}{C1}
        \nccurve[angleA=150,angleB=210]{-}{C2}{C2}
        \nccurve[angleA=210,angleB=270]{-}{C3}{C3}
        \nccurve[angleA=270,angleB=330]{-}{C4}{C4}
        \nccurve[angleA=20,angleB=80]{-}{C5}{C5}
        \nccurve[angleA=55,angleB=115]{-}{C6}{C6}

        %\psgrid
      \end{pspicture}
    }
  }

  \caption[Two D-MPC scenarios. Solid lines represent links in the communication network and dotted arrows represent system interactions.]{
		Two D-MPC scenarios. Solid lines represent links in the communication network and dotted arrows represent system interactions. The optimization variable is star-shaped (and thus connected) in \text{(a)} and is non-connected in \text{(b)}, because node~$2$ influences node~$5$ but not any neighbor of that node.
  }
  \label{Fig:D-MPC}
  \end{figure}
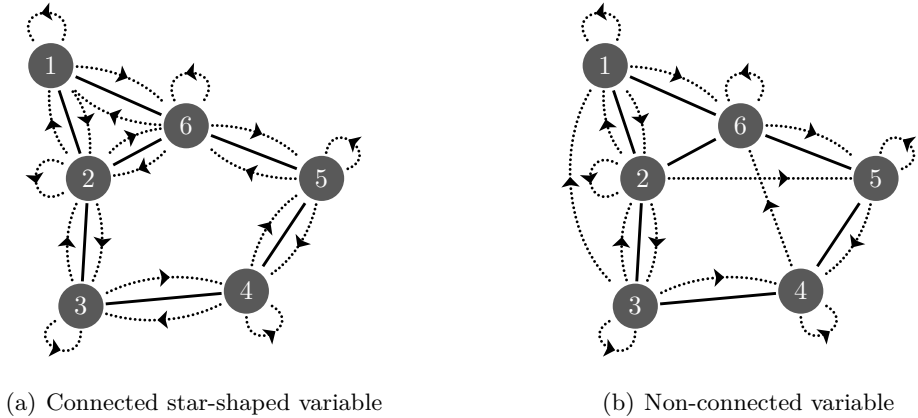

	\mypar{D-MPC}
	We now turn to distributed	 scenarios and focus on solving one instance of~\eqref{Eq:MPC}, i.e., for a fixed MPC iteration. Suppose that, instead of a single system, we now have a network of systems, where each system is described by its own state vector and has a local control input. Let $x_p[t] \in \mathbb{R}^{n_p}$ denote the state of system~$p$ at time~$t$, and~$u_p[t] \in \mathbb{R}^{m_p}$ denote its local input also at time~$t$; we have $n_1 + \cdots + n_P = n$ and $m_1 + \cdots + m_P = m$. Each system is viewed as a node of a communication network $\mathcal{G} = (\mathcal{V},\mathcal{E})$, whose edges determine which systems communicate directly. 	We assume that the state of system~$p$ evolves as
	\begin{equation}	\label{Eq:DMPCFunction}
		x_p[t+1] = \Theta_p^t\bigl(\{x_j[t],u_j[t]\}_{j \in \Omega_p}\bigr)\,,
	\end{equation}
	where $\Omega_p \subseteq \mathcal{V}$ is the set of nodes whose state and/or input influences~$x_p$ (we assume each node~$p$ influences itself, i.e., $\{p\}\subseteq \Omega_p$). In~\eqref{Eq:DMPCFunction}, we used the following notation: given a finite set~$\Omega = \{\omega_1, \omega_2, \ldots,\omega_L\}$ and a vector~$z_\omega$, indexed by a parameter~$\omega \in \Omega$, the symbol $\{z_\omega\}_{\omega \in \Omega}$ denotes the $L$-tuple $(z_{\omega_1}, z_{\omega_2}, \ldots, z_{\omega_L})$. Many times, when~$\Omega$ is represented as $\Omega = \{\omega\,:\, \text{$A(\omega)$ holds}\}$, we will represent~$\{z_\omega\}_{\omega \in \Omega}$ simply as~$\{z_\omega\}_{\text{$A(\omega)$ holds}}$.	In contrast with what is usually assumed, $\Omega_p$ in~\eqref{Eq:DMPCFunction} is not necessarily a subset of the neighbors of node~$p$. This means that two systems that influence each other through their dynamics may be unable to communicate directly. This is illustrated in \fref{SubFig:PPMPCNonConnected} where, for example, the state/input of node~$3$ influences the state of node~$1$ (dotted arrow), but there is no communication link (solid line) between them. Finally, we assume functions~$\Phi$ and~$\Psi^t$ in~\eqref{Eq:MPC} can be decomposed, respectively, as $	\Phi(x[T]) = \sum_{p=1}^P \Phi_p(\{x_j[T]\}_{j \in \Omega_p})$ and $	\Psi^t(x[t],u[t]) = \sum_{p=1}^P \Psi_p^t (\{x_j[t], u_j[t]\}_{j \in \Omega_p})$, where~$\Phi_p$ and~$\Psi_p^t$ are both associated to node~$p$. This means that non-communicating systems can have coupled goals or energy measures. Hence, in our distributed setting, the MPC problem~\eqref{Eq:MPC} becomes
	\begin{equation}\label{Eq:D-MPC}
	  \begin{array}{ll}
	  	\underset{\bar{x},\bar{u}}{\text{minimize}} &
				\sum_{p=1}^P\biggl[\Phi_p(\{x_j[T]\}_{j \in \Omega_p})
				+ \sum_{t=0}^{T-1} \Psi_p^t (\{x_j[t], u_j[t]\}_{j \in \Omega_p})\biggr] \\
			\text{subject to} & x_p[t+1] = \Theta_p^t\bigl(\{x_j[t],u_j[t]\}_{j \in \Omega_p}\bigr)\,,\quad t = 0,\ldots,T-1\,,\quad p = 1,\ldots,P \\
			            & x_p[0] = x_p^0\,, \quad  p = 1,\ldots,P\,,
	  \end{array}
	\end{equation}
	where~$x_p^0$ is the initial measurement at node~$p$. The optimization variable in this case is $(\bar{x},\bar{u}) := \bigl(\{\bar{x}_p\}_{p=1}^{P} , \{\bar{u}_p\}_{p=1}^{P}\bigr)$, where $\bar{x}_p := \{x_p[t]\}_{t=0}^T$ and $\bar{u}_p := \{u_p[t]\}_{t=0}^{T-1}$ represent, respectively, the collection of all the states and inputs of node~$p$ for the time-horizon~$T$. Problem~\eqref{Eq:D-MPC} can be written as~\eqref{Eq:IntroProb} by making
	\begin{multline*}
		f_p\Bigl(\{\bar{x}_j,\bar{u}_j\}_{j \in \Omega_p}\Bigr) = \Phi_p\Bigl(\{x_j[T]\}_{j \in \Omega_p}\Bigr) + \text{i}_{\{x_p[0] = x_p^0\}}(\bar{x}_p)
		\\
									 + \sum_{t=0}^{T-1} \biggl(
																				\Psi_p^t(\{x_j[t],u_j[t]\}_{j \in \Omega_p})
																				+
																				\text{i}_{\Gamma_p^t}(\{\bar{x}_j,\bar{u}_j\}_{j \in \Omega_p})
																			\biggr)\,,
	\end{multline*}
	where $\text{i}_{\Gamma_p^t}$ is the indicator function of the set $\Gamma_p^t := \Bigl\{\{\bar{x}_j,\bar{u}_j\}_{j \in \Omega_p} \,:\, x_p[t+1] = \Theta_p^t\bigl(\{x_j[t],u_j[t]\}_{j \in \Omega_p}\bigr)\Bigr\}$.

	\fref{SubFig:PPMPCStar} illustrates the standard D-MPC scenario, where each system~$p$ is influenced only by itself and by its neighbors, i.e., $\Omega_p \subseteq \mathcal{N}_p \cup \{p\}$. According to the terminology introduced in \cref{Ch:Introduction}, each component~$(\bar{x}_p,\bar{u}_p)$ of the variable is star-shaped and thus the entire variable $(\bar{x},\bar{u})$ is also star-shaped. Several instances of this particular case of~\eqref{Eq:D-MPC} have been addressed, for example, by~\cite{Jia01-DistributedModelPredictiveControl,Camponogara02-DistributedMPC,Keviczky06-DecentralizedRecedingHorizonControl,Venkat05-StabilityOptimalityDistributedMPC}, who propose heuristics that are not guaranteed to solve exactly~\eqref{Eq:D-MPC}, and by~\cite{Fawal98-OptimalControlComplexIrrigationSystemsAugmentedLagrangian,Wakasa08-DecentralizedMPCDualDecomp,Camponogara11-DistributedOptimizationMPCLinearDynamicNetworks,Conte12-ComputationalAspectsDistributedMPC,Summers12-DistributedModelPredictiveControlViaADMM}, who propose algorithms based on distributed optimization methods and thus, in principle, are guaranteed to solve~\eqref{Eq:D-MPC}.

	The model we propose here is significantly more general, since it can handle scenarios where interacting nodes do not necessarily need to communicate, or even scenarios with a non-connected variable. Both cases are shown in \fref{SubFig:PPMPCNonConnected}. For example, the subgraph induced by $(\bar{x}_3,\bar{u}_3)$ consists of the nodes~$\{1,2,3,4\}$ and is connected. (The reference for connectivity is always the communication network which, in the plots, is represented by solid lines.) Nodes~$1$ and~$3$, however, cannot communicate directly. This is an example of an induced subgraph that is not a star. On the other hand, the subgraph induced by~$(\bar{x}_2,\bar{u}_2)$ consists of the nodes~$\{1,2,3,5\}$. This subgraph is not connected, which implies that the optimization variable is non-connected. A connected variable with induced subgraphs that are not stars, or even a non-connected variable, can be useful to model scenarios where communications links are expensive, hard to establish, or simply do not exist. We describe two such applications below.

	\mypar{Applications of our D-MPC model}
	Although D-MPC has been applied to solve many applications, we present here two applications where the scenario of \fref{SubFig:PPMPCNonConnected} might arise naturally, i.e., the variable is either non-connected, or is connected but not star-shaped. The first application is flight formation and the other is temperature regulation of buildings.

	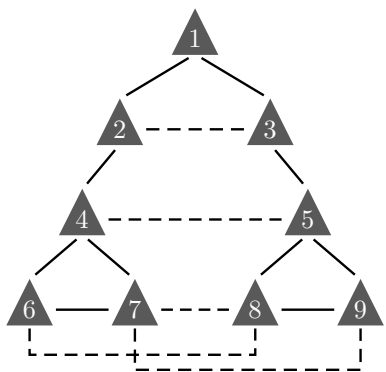
\begin{figure}[t]
  \centering
	\psscalebox{1.0}{
		\begin{pspicture}(5,5)
       \def\triang{
				\pspolygon*[linecolor=black!65!white](-0.31,-0.31)(0.31,-0.31)(0,0.31)
       }

	  	 \rput(2.5,4.5){\triang}  \pnode(2.5,4.2){C1b}
       \rput(1.5,3.3){\triang}  \pnode(1.5,3.6){C2t}  \pnode(1.5,3.0){C2b}  \pnode(1.5,3.2){C2r}
       \rput(3.5,3.3){\triang}  \pnode(3.5,3.6){C3t}  \pnode(3.5,3.0){C3b}  \pnode(3.5,3.2){C3l}
       \rput(1.0,2.1){\triang}  \pnode(1.0,2.4){C4t}  \pnode(1.0,1.8){C4b}  \pnode(1.0,2.0){C4r}
       \rput(4.0,2.1){\triang}  \pnode(4.0,2.4){C5t}  \pnode(4.0,1.8){C5b}  \pnode(4.0,2.0){C5l}
       \rput(0.3,0.9){\triang}  \pnode(0.3,1.2){C6t}  \pnode(0.3,0.8){C6r}
       \rput(1.7,0.9){\triang}  \pnode(1.7,1.2){C7t}  \pnode(1.7,0.8){C7l}
       \rput(3.3,0.9){\triang}  \pnode(3.3,1.2){C8t}  \pnode(3.3,0.8){C8r}
       \rput(4.7,0.9){\triang}  \pnode(4.7,1.2){C9t}  \pnode(4.7,0.8){C9l}

			 \rput(2.5,4.4){\small \textcolor{white}{$1$}}
			 \rput(1.5,3.2){\small \textcolor{white}{$2$}}
			 \rput(3.5,3.2){\small \textcolor{white}{$3$}}
			 \rput(1.0,2.0){\small \textcolor{white}{$4$}}
			 \rput(4.0,2.0){\small \textcolor{white}{$5$}}
			 \rput(0.3,0.8){\small \textcolor{white}{$6$}}
			 \rput(1.7,0.8){\small \textcolor{white}{$7$}}
			 \rput(3.3,0.8){\small \textcolor{white}{$8$}}
			 \rput(4.7,0.8){\small \textcolor{white}{$9$}}

			 \psset{nodesep=0.1cm,linewidth=0.9pt}
			 \ncline{-}{C1b}{C2t}
			 \ncline{-}{C1b}{C3t}
			 \ncline{-}{C2b}{C4t}
			 \ncline{-}{C3b}{C5t}
			 \ncline{-}{C4b}{C6t}
			 \ncline{-}{C4b}{C7t}
			 \ncline{-}{C5b}{C8t}
			 \ncline{-}{C5b}{C9t}

			 \psset{nodesep=0.35cm}
			 \ncline{-}{C6r}{C7l}
			 \ncline{-}{C8r}{C9l}

			 \psset{linestyle=dashed}
			 \ncline{-}{C2r}{C3l}
			 \ncline{-}{C4r}{C5l}
			 \ncline{-}{C7l}{C8r}

			 \psset{nodesep=0.25cm}
			 \ncangle[angleA=-90,angleB=-90,arm=0.35cm]{-}{C6r}{C8r}
			 \ncangle[angleA=-90,angleB=-90,arm=0.55cm]{-}{C7l}{C9l}

			 %\psgrid
		\end{pspicture}
		}
	\smallskip
  \caption{
		Example of a geometrical pattern used in formation for minimizing the effect of drag forces or for escorting a moving object. Solid lines indicate direct communication, while dashed lines indicate dynamic coupling, but not necessarily direct communication.
  }
  \label{Fig:FormationFlight}
  \end{figure}

	\fref{Fig:FormationFlight} shows the setup of flight formation: there is a group of autonomous agents, such as unmanned airplanes, submarines, or robots, whose goal is to form a geometrical pattern while performing some task. This task could be simply flying and, at the same time, trying to minimize the effect of drag forces to reduce fuel consumption; or, for example, to escort a moving object, which might block some communications between the agents. We assume there is a communication network through which the agents communicate. In \fref{Fig:FormationFlight}, the links of this communication network are represented by the solid lines. Also, some agents influence the behavior of other agents with which they do not communicate directly; this is represented by the dashed lines in \fref{Fig:FormationFlight}. Flight formation is a widely studied topic and we refer to~\cite{Boyd04-DistributedOptimizationFormationFlight} for references and related work. In this problem, we extend the optimization model used in~\cite{Boyd04-DistributedOptimizationFormationFlight} to the MPC framework~\eqref{Eq:MPC}. Namely, we write the dynamics of the $p$th agent as~\eqref{Eq:DMPCFunction}, where the relative position between two agents affects their dynamics due, for example, to drag forces. Regarding the objective, while~$\Psi_p^t$ models fuel consumption at time~$t$, $\Phi_p$ models the geometrical pattern to be formed. Note that, in principle, $\Theta_p^t$ and~$\Psi_p^t$ depend only on the state/input of agent~$p$ and of its closest agents (i.e., its neighbors in the communication network); in~$\Phi_p$ we can, in addition, include dependencies on agents that are not within communication reach. For example, suppose we specify agent~$6$ in \fref{Fig:FormationFlight} to have a relative distance of~$\delta_{46}$ and~$\delta_{67}$ from its closest neighbors, agents~$4$ and~$7$, and also a relative distance of~$\delta_{68}$ from agent~$8$, for symmetry reasons. Note that agents~$6$ and~$8$ do not communicate directly. In this case, $\Omega_6 = \{4,6,7,8\}$, and~$\Phi_6$ would have a format similar to~$\Phi_6(x_4,x_6,x_7,x_8) = \frac{1}{2}\|x_4 - x_6 - \delta_{46}\|^2 + \frac{1}{2}\|x_6 - x_7 - \delta_{67}\|^2 + \frac{1}{2}\|x_6 - x_8 - \delta_{68}\|^2$, where~$x_p$ represents the position of agent~$p$; note that we dropped the time index~$T$ for notational simplicity. A similar reasoning can be applied to the remaining agents of \fref{Fig:FormationFlight}, where relative distances are specified for each edge (represented either with continuous or dashed lines).

	We now describe another application of D-MPC where the variable can be connected, not necessarily star-shaped, or even non-connected. The application is temperature regulation of buildings and is described in the context of D-MPC in~\cite{Morosan10-BuildingTemperatureRegulationUsingDMPC}. The algorithm proposed in~\cite{Morosan10-BuildingTemperatureRegulationUsingDMPC}, however, is heuristic and, thus, not guaranteed to solve the original problem. The motivation for using MPC in the control of room temperature stems from its ability to integrate in its model the prediction of future events, in this case, room occupation profiles. This feature is necessary in the regulation of room temperature, because temperature varies very slowly. If it did not, a simple PID controller would be enough. The work in~\cite{Morosan10-BuildingTemperatureRegulationUsingDMPC} models room temperature, viewing it as a state~$x$, which varies linearly with the heating power applied to the room, $u$. According to our notation in~\eqref{Eq:MPC}, the function~$\Phi$ is identically zero, and $\Psi^t(x[t],u[t]) = \beta u[t] + \delta[t] \Bigl|x[t] - r[t]\Bigr|$, where~$r[t]$ is the reference temperature for the room at time~$t$, and~$\delta[t] =1$ if the room is predicted to be occupied at time~$t$ and $\delta[t] = 0$ otherwise. The parameter~$\beta > 0$ sets the tradeoff between energy consumption and comfort. The power~$u[t]$ is constrained to an interval: $0 \leq u[t] \leq u_{\max}$. This is the model for one room. However, \cite{Morosan10-BuildingTemperatureRegulationUsingDMPC} also models buildings, where the rooms are thermally coupled. The model it proposes can be written as~\eqref{Eq:D-MPC}, where the sets~$\Omega_p$'s model coupling between adjacent rooms. Yet, it is assumed that adjacent rooms can communicate or, in other words, that interactions through coupling coincide with interactions through communication. If the communication technology is wired, it might be expensive to connect all the adjacent rooms with cables; if it is wireless, large concrete walls might prevent adjacent rooms from communicating directly. This is clearly an example where our proposed D-MPC model~\eqref{Eq:D-MPC} could be useful, as the problem variable might have induced subgraphs that are not stars and might even be non-connected.

	\subsection{Reversed lasso with a row partition}

	We saw in \cref{Ch:GlobalVariable} how to recast several compressed sensing problems as~\eqref{Eq:GlobalProblem}, that is, as~\eqref{Eq:IntroProb} with a global variable. The only problem for which we were not able to do so was the reversed lasso~\eqref{Eq:ReverseLasso} with a row partition. The goal of this subsection is to complete this missing part of the puzzle, by recasting that problem as~\eqref{Eq:IntroProb} with a mixed variable.

	Recall that the reversed lasso~\eqref{Eq:ReverseLasso} is the problem
	\begin{equation}\label{Eq:ReversedLasso2}
		\begin{array}{ll}
			\underset{x}{\text{minimize}} & \|x\|_1 \\
			\text{subject to} & \|Ax - b\| \leq \sigma\,.
		\end{array}
	\end{equation}
	Since we will use duality, we want to make sure that the primal objective is strictly convex, so that we can recover a primal solution after having solved the dual problem. We will make the primal objective strictly convex by using the regularization~\eqref{Eq:RegularizarionReversedLasso3} (page~\pageref{Eq:RegularizarionReversedLasso3}), which we used for the same problem with a column partition. There, we showed that, for a small~$\delta > 0$, \eqref{Eq:ReversedLasso2} can be approximated by
	\begin{equation}\label{Eq:ReversedLassoRegularizedM}
		\begin{array}{ll}
			\underset{x}{\text{minimize}} & \|x\|_1 + \frac{\delta}{2}\|x\|^2  + \frac{\delta}{4}\|Ax - b\|^2\\
			  	\text{subject to} & \|Ax - b\|^2 \leq \sigma^2\,,
		\end{array}
	\end{equation}
	where we squared both sides of the constraint.	Consider now a row partition, as visualized in \fref{Fig:PartitionOfA}, and rewrite~\eqref{Eq:ReversedLassoRegularizedM} as
	\begin{equation}\label{Eq:ReversedLassoRegularizedM2}
		\begin{array}{ll}
			\underset{x}{\text{minimize}} & \sum_{p=1}^P \biggl(\frac{1}{P}\|x\|_1 + \frac{\delta}{2P}\|x\|^2  + \frac{\delta}{4}\|A_px - b_p\|^2\biggr)\\
			  	\text{subject to} & \sum_{p=1}^P\|A_px - b_p\|^2 \leq \sigma^2\,.
		\end{array}
	\end{equation}
	This problem has the following format
	\begin{equation}\label{Eq:ReversedLassoGeneral}
    \begin{array}{ll}
      \underset{x}{\text{minimize}} & f_1(x) + f_2(x) + \cdots + f_P(x) \\
      \text{subject to} & h_1(x) + h_2(x) + \cdots + h_P(x) \leq r\,,
    \end{array}
  \end{equation}
	with $f_p(x) = \frac{1}{P}\|x\|_1 + \frac{\delta}{2P}\|x\|^2  + \frac{\delta}{4}\|A_px - b_p\|^2$, $h_p(x) = \|A_px - b_p\|^2$, and~$r = \sigma^2$. We will now see how to solve~\eqref{Eq:ReversedLassoGeneral} in a network where each node~$p$ knows~$f_p$, $h_p$, and~$r$. We assume that each function~$f_p$ is strictly convex, as in~\eqref{Eq:ReversedLassoRegularizedM2}. We start by cloning the variable~$x$, rewriting~\eqref{Eq:ReversedLassoGeneral} as
	\begin{equation}\label{Eq:ReversedLassoGeneralManip}
    \begin{array}{ll}
      \underset{x_1,\ldots,x_P}{\text{minimize}} & f_1(x_1) + f_2(x_2) + \cdots + f_P(x_P) \\
      \text{subject to} & h_1(x_1) + h_2(x_2) + \cdots + h_P(x_P) \leq r \\
      & x_i = x_j\,,\quad (i,j) \in \mathcal{E}\,,
    \end{array}
  \end{equation}
  where~$x_p$ is the copy of~$x$ held at node~$p$. Let~$\mu$ be a dual variable associated to the first constraint of~\eqref{Eq:ReversedLassoGeneralManip} and~$\lambda_{ij}$ the dual variable associated to the constraint~$x_i = x_j$, for~$(i,j) \in \mathcal{E}$. The dual problem of~\eqref{Eq:ReversedLassoGeneralManip} is
  \begin{equation}\label{Eq:ReversedLassoGeneralManipFin}
  	\begin{array}{cl}
  		\underset{\mu, \{\lambda_{ij}\}_{(i,j) \in \mathcal{E}}}{\text{minimize}} & g_1(\mu, \{\lambda_{1j}\}_{j \in \mathcal{N}_1}) + g_2(\mu, \{\lambda_{2j}\}_{j \in \mathcal{N}_2}) + \cdots + g_P(\mu, \{\lambda_{Pj}\}_{j \in \mathcal{N}_P}) \\
  		\text{subject to} & \mu \geq 0\,,
  	\end{array}
  \end{equation}
  where, for each~$p$,
	\begin{equation}\label{Eq:ReversedLassoGeneralMConj}
		g_p(\mu,\{\lambda_{pj}\}_{j \in \mathcal{N}_p}) := \sup_{x_p} \,\,\, \Bigl(\sum_{j \in \mathcal{N}_p} \text{sign}(p-j)\lambda_{pj}\Bigr)^\top x_p - \Bigl(f_p(x_p) + \mu^\top (h_p(x_p) - \frac{1}{P}r)\Bigr) \,.
	\end{equation}
	We slightly abused notation in~\eqref{Eq:ReversedLassoGeneralManipFin}, since each~$\lambda_{ij}$ is only defined for~$i < j$. We extended that definition: $\lambda_{ij}  = -\lambda_{ji}$ when $i>j$. We thus can see that~\eqref{Eq:ReversedLassoGeneralManipFin} has the same format as~\eqref{Eq:IntroMixed} or, in other words, it has a mixed variable. The dual variable~$\mu$ is a global component, since it appears in the function of all the nodes, and all the components of~$\lambda = (\ldots,\lambda_{ij},\ldots)$ are non-global. Since we assume that each~$f_p$ is strictly convex, the $p$th block of the primal variable of~\eqref{Eq:ReversedLassoGeneral}, i.e., $x_p^\star$ will be available at the $p$th node as the solution of the optimization problem in~\eqref{Eq:ReversedLassoGeneralMConj}, for $\mu=\mu^\star$ and $\{\lambda_{pj}\}_{j \in \mathcal{N}_p} = \{\lambda_{pj}^\star\}_{j \in \mathcal{N}_p}$, where the starred vectors solve the dual problem~\eqref{Eq:ReversedLassoGeneralManipFin}.

	\begin{figure}
    \centering
    \subfigure[Original network]{\label{SubFig:TCPNetwork}
    \psscalebox{0.85}{
        \begin{pspicture}(9,3.8)
            \def\nodesimp{
                \pscircle*[linecolor=black!70!white](0,0){0.3}
            }
            \def\nodesquare{
              \psframe*[linecolor=black!20!white](-0.21,-0.21)(0.21,0.21)
            }

            \rput(1.0,0.2){\rnode{S1}{\nodesquare}}
            \rput(1.0,1.7){\rnode{S2}{\nodesquare}}
            \rput(1.0,3.2){\rnode{S3}{\nodesquare}}
            \rput(6.5,3.2){\rnode{R1}{\nodesquare}}
            \rput(8.0,1.7){\rnode{R2}{\nodesquare}}
            \rput(8.0,0.2){\rnode{R3}{\nodesquare}}
            \rput(3.00,2.45){\rnode{N1}{\nodesimp}}
            \rput(5.00,2.45){\rnode{N2}{\nodesimp}}
            \rput(6.50,0.95){\rnode{N3}{\nodesimp}}

            \rput(1.0,0.2){\textcolor{black}{$\textsf{s}_3$}} \rput[rb](0.7,3.5){$x_1$}
            \rput(1.0,1.7){\textcolor{black}{$\textsf{s}_2$}} \rput[rb](0.7,2.0){$x_2$}
            \rput(1.0,3.2){\textcolor{black}{$\textsf{s}_1$}} \rput[rb](0.7,0.5){$x_3$}
            \rput(6.5,3.2){\textcolor{black}{$\textsf{r}_1$}}
            \rput(8.0,1.7){\textcolor{black}{$\textsf{r}_2$}}
            \rput(8.0,0.2){\textcolor{black}{$\textsf{r}_3$}}
            \rput(3.00,2.45){\textcolor{white}{$\textsf{n}_1$}}
            \rput(5.00,2.45){\textcolor{white}{$\textsf{n}_2$}}
            \rput(6.50,0.95){\textcolor{white}{$\textsf{n}_3$}}

            \psset{nodesep=0.35cm}
            \ncline{-}{S1}{N2}%\Bput{$x_{12}$}
            \ncline{-}{S2}{N1}
            \ncline{-}{S3}{N1}
            \ncline{-}{N1}{N2}\Aput[0.06]{$c_1$}
            \ncline{-}{N2}{N3}\Aput[0.06]{$c_3$}
            \ncline{-}{N2}{R1}\Aput[0.06]{$c_2$}
            \ncline{-}{N3}{R2}\Bput[0.02]{$c_4$}
            \ncline{-}{N3}{R3}\Aput[0.02]{$c_5$}

            \psset{linestyle=dashed,arrowsize=8pt,arrowlength=1,linewidth=1.4pt,linecolor=black!80!white}
            \pscurve{->}(1.4,3.2)(3,2.9)(5,2.9)(6.6,1.7)(7.6,1.7)
            \pscurve{->}(1.4,1.7)(3,2.1)(5,2.1)(6.1,2.9)
            \pscurve{->}(1.4,0.2)(5,1.8)(6.5,0.6)(7.6,0.2)

            %\psgrid
           \end{pspicture}
      }
    }
    \subfigure[Bipartite graph obtained from $\text{(a)}$]{\label{SubFig:TCPBipartite}
      \psscalebox{0.85}{
        \begin{pspicture}(9,2)

            \def\nodesimp{
                \pscircle*[linecolor=black!70!white](0,0){0.3}
            }
            \def\nodesquare{
              \psframe*[linecolor=black!20!white](-0.21,-0.21)(0.21,0.21)
            }

            \rput(0.200,0.4){\rnode{N2}{\nodesimp}}
            \rput(1.633,0.4){\rnode{S2}{\nodesquare}}
            \rput(3.066,0.4){\rnode{N1}{\nodesimp}}
            \rput(4.499,0.4){\rnode{S1}{\nodesquare}}
            \rput(5.932,0.4){\rnode{N2p}{\nodesimp}}
            \rput(7.365,0.4){\rnode{S3}{\nodesquare}}
            \rput(8.798,0.4){\rnode{N3p}{\nodesimp}}
            \rput(4.499,1.833){\rnode{N3}{\nodesimp}}

            \rput(0.200,0.4){\textcolor{white}{$2$}}\rput[b](0.200,0.85){$\lambda_2$}
            \rput(1.633,0.4){\textcolor{black}{$\text{s}_2$}}
            \rput(3.066,0.4){\textcolor{white}{$1$}}\rput[b](3.066,0.85){$\lambda_1$}
            \rput(4.499,0.4){\textcolor{black}{$\text{s}_1$}}
            \rput(5.932,0.4){\textcolor{white}{$3$}}\rput[b](5.932,0.85){$\lambda_3$}
            \rput(7.365,0.4){\textcolor{black}{$\text{s}_3$}}
            \rput(8.798,0.4){\textcolor{white}{$5$}}\rput[b](8.798,0.85){$\lambda_5$}
            \rput(4.499,1.833){\textcolor{white}{$4$}}\rput[l](4.949,1.833){$\lambda_4$}

            \psset{nodesep=0.35cm}
            \ncline{-}{N2}{S2}
            \ncline{-}{S2}{N1}
            \ncline{-}{N1}{S1}
            \ncline{-}{S1}{N2p}
            \ncline{-}{N2p}{S3}
            \ncline{-}{S3}{N3p}
            \ncline{-}{S1}{N3}

            %\psgrid
            \end{pspicture}
      }
    }
      \caption{\text{(a)} Example network with~$3$ source nodes~$s_1$, $s_2$, and $s_3$, that use predetermined routes to send packets to three recipient nodes~$r_1$, $r_2$, and $r_3$; \text{(b)} Bipartite graph obtained from \text{(a)}: each link from~\text{(a)} with a capacity associated is represented as a circular node in~\text{(b)}.}
      \label{Fig:TCPIP}
  \end{figure}

  \subsection{Network utility maximization}

  Network utility maximization (NUM) is usually used for modeling congestion control in networks. The setup of congestion control is a network with some source nodes sending information, encoded in packets, to other nodes of the network, called recipient nodes. Independently of the network, its links have always finite capacity and, therefore, there is a limit on the rate of packets that can be injected into the network without congesting it. The goal of congestion control is to avoid congesting the network; this is done by implementing a protocol between the source nodes and the nodes through which they send their packets, called intermediate nodes. The communication between these nodes can occur implicitly or explicitly. The algorithm we propose for solving~\eqref{Eq:IntroProb} will require an explicit communication between the source and the intermediate nodes.

  \mypar{Star-shaped variable model}
  Given a network, let~$\mathcal{S}$, $\mathcal{R}$, and~$\mathcal{N}$ represent the source nodes, the recipient nodes, and the intermediate nodes, respectively. Figure~\ref{SubFig:TCPNetwork} shows an example of such network with~$3$ nodes of each kind, i.e., $|\mathcal{S}| = |\mathcal{R}| = |\mathcal{N}| = 3$. The source nodes are the squares on the left side, the recipient nodes are the squares on the right side, and the intermediate nodes are the circles. We assume each source sends packets only to one recipient node. Also, the routes through which each source sends its packets are predetermined (see the arrows in Figure~\ref{SubFig:TCPNetwork}). Each link~$l$ in the network has a finite capacity~$c_l > 0$, and the total number of links will be denoted with~$L$. In Figure~\ref{SubFig:TCPNetwork}, to simplify, we represent the capacity of only some links. Congestion control can be modeled with a problem called \textit{network utility maximization} (NUM):
  \begin{equation}\label{Eq:NUM}
    \begin{array}{ll}
      \underset{\{x_s\}_{s \in \mathcal{S}}}{\text{maximize}} & \sum_{s \in \mathcal{S}} U_s(x_s) \\
      \text{subject to} & \sum_{s \in \mathcal{S}(l)} x_s \leq c_l\,, \quad l=1,\ldots,L\,,
    \end{array}
  \end{equation}
  where~$x_s$ represents the sending rate of source~$s$ and~$U_s(x_s)$ its utility, or ``satisfaction.'' The constraints in~\eqref{Eq:NUM} are simply the link capacity constraints: $\mathcal{S}(l)$ represents the set of sources that use link~$l$ and thus~$\sum_{s \in \mathcal{S}(l)} x_s$ represents the rate of packets flowing in link~$l$, which has to be smaller than the link capacity~$c_l$. The goal in~\eqref{Eq:NUM} is to maximize the aggregate utilities of the sources, while satisfying the link capacity constraints. It is generally assumed that each utility is increasing and strictly concave. For example, TCP Vegas, FAST, and Scalable TCP have been modeled as~\eqref{Eq:NUM} with~$U_s(x_s) = w_s \log x_s$, for some~$w_s > 0$~\cite{Low02-UnderstandingVegas,Chiang07-LayeringAsOptimizationDecomposition}. This makes the objective of~\eqref{Eq:NUM} strictly concave, and hence its dual problem can be solved instead:
  \begin{equation}\label{Eq:DualNUM}
    \begin{array}{cl}
      \underset{\lambda =(\lambda_1,\ldots,\lambda_L)}{\text{minimize}} & \sum_{l \in \mathcal{L}} c_l\lambda_l  + \sum_{s \in \mathcal{S}} \bar{U}_s \Bigl(\sum_{l \in \mathcal{L}(s)} \lambda_l\Bigr) \\
        \text{subject to} & \lambda_l \geq 0\,, \quad l =1,\ldots, L\,,
    \end{array}
  \end{equation}
  where~$\mathcal{L}(s)$ is the set of links source~$s$ uses to route its packets and~$\bar{U}_s(t) := \sup_{x_s} \bigl(U_s(x_s) - t x_s\bigr)$ has always a unique solution~$x_s(t)$, due to the strict concavity of~$U_s$. In the case of TCP Vegas, FAST, and Scalable TCP, $\bar{U}_s(t) = w_s \log w_s - w_s \log t - w_s$, and~$x_s(t) = w_s/t$. After a solution~$\lambda^\star$ to~\eqref{Eq:DualNUM} has been found, the optimal value for the rate of source~$s$ can be found as~$x_s^\star = x_s(\sum_{l \in \mathcal{L}(s)} \lambda_l^\star)$, which is~$w_s / \sum_{l \in \mathcal{L}(s)} \lambda_l^\star$ in the case of TCP Vegas, FAST, and Scalable TCP.

  The ``physical communications'' occur in a network that has a format similar to the one represented in Figure~\ref{SubFig:TCPNetwork}. However, a congestion control protocol establishes direct communications between the source nodes and the intermediate nodes that manage the respective links. Therefore, the communication network it considers is actually the one represented in Figure~\ref{SubFig:TCPBipartite}. This network is constructed as follows: each link~$l$, which we assume is unidirectional for the sake of simplicity, has a node associated (in Figure~\ref{SubFig:TCPBipartite}, a circle node), and each source~$s$ has also a node associated (in Figure~\ref{SubFig:TCPBipartite}, a square node); the recipient nodes are not considered in this new network. If link~$l$ is used in the route assigned to source~$s$, the nodes representing link~$l$ and source~$s$ are connected to each other. Figure~\ref{SubFig:TCPBipartite} shows the network obtained from Figure~\ref{SubFig:TCPNetwork} by considering only the links marked with capacities. The way this network is constructed makes it automatically bipartite. In the model considered here, the intermediate node having link~$l$ as output manages~$\lambda_l$. For example, node~$n_2$ in Figure~\ref{Fig:TCPIP} manages both~$\lambda_2$ and~$\lambda_3$. Note that each communication occurring in the network of Figure~\ref{SubFig:TCPBipartite} corresponds to an arbitrary number of communications in the original network of Figure~\ref{SubFig:TCPNetwork}. Regarding problem~\eqref{Eq:DualNUM}, it can be written as~\eqref{Eq:IntroProb} with the function at node~$p$ given by
  $$
    f_p(\lambda_1,\ldots,\lambda_L)
    =
    \left\{
      \begin{array}{ll}
        c_p \lambda_p + \text{i}_{\mathbb{R}^+}(\lambda_p)\,, &\quad \text{if~$p$ is an intermediate node} \\
        \bar{U}_s\bigl(\sum_{l \in \mathcal{L}(p)} \lambda_l\bigr)\,, &\quad \text{if~$p$ is a source node}
      \end{array}\,,
    \right.
  $$
  where~$\text{i}_{\mathbb{R}^+}(\cdot)$ is the indicator of the set of the nonnegative real numbers, and the variable is~$\lambda = (\lambda_1,\ldots,\lambda_L)$. The variable in this case is star-shaped. The algorithm we propose for~\eqref{Eq:IntroProb} can then be used to inspire a new congestion control protocol. However, it has one disadvantage with respect to gradient-based algorithms: while gradient-based algorithms can work with implicit communication, due to their linearity, the algorithm we propose requires explicit communication between each each source and all the intermediate nodes along its route.

  \mypar{Mixed variable model}
  The NUM problem was introduced as~\eqref{Eq:NUM} to model congestion control in networks. In~\eqref{Eq:NUM}, the utility function~$U_p$ of source/node~$p$ depends only on its sending rate~$x_p$, i.e., $U_p(x_p)$. However, in cooperative or competitive scenarios it might be useful to consider coupled objectives, e.g., $U_p(\{x_l\}_{l \in S_p})$, where~$S_p$ is the set of nodes whose rates influence the utility of source~$p$. Such model was considered in~\cite{Tan06-DistributedOptimizationCoupledSystemsNUM} (see also~\cite{Palomar06-TutorialDecompositionMethods}). For example, in digital subscriber line (DSL) spectrum management, or in wireless power control, the signal-to-interference ratio at one user depends on the transmit powers of other users, making the scenario competitive. A cooperative scenario would be rate allocation in clusters: the higher the rate allocated to one cluster, the higher the rate allocated to each node inside that cluster. In particular, \cite{Tan06-DistributedOptimizationCoupledSystemsNUM} considered the following variation of~\eqref{Eq:NUM}:
  \begin{equation}\label{Eq:NUMCoupled}
    \begin{array}{ll}
      \underset{x_1,\ldots,x_P}{\text{maximize}} & \sum_{p=1}^P U_p(\{x_l\}_{l \in S_p}) \\
      \text{subject to} & \sum_{p=1}^P g_p(x_p) \leq c\,,
    \end{array}
  \end{equation}
	where each~$g_p$ is a convex function and~$c$ is a globally known vector. We slightly changed the notation with respect to~\eqref{Eq:NUM}: we now denote each source by~$p$ and the total number of sources is~$P$. In~\cite{Tan06-DistributedOptimizationCoupledSystemsNUM} it is also assumed that source~$p$ can communicate with all the sources that interfere with its utility, and vice-versa. The work in~\cite{Tan06-DistributedOptimizationCoupledSystemsNUM} proposes a gradient-based algorithm that solves a dual problem of~\eqref{Eq:NUMCoupled}. To arrive at that dual problem, we first perform a splitting (or cloning) of~$x_p$ among all the nodes whose utilities depend on~$x_p$:
	\begin{equation}\label{Eq:NUMCoupled2}
    \begin{array}{ll}
      \underset{\{\bar{x}_l\}_{l=1}^P}{\text{maximize}} & \sum_{p=1}^P U_p(\{x_l^{(p)}\}_{l \in S_p}) \\
      \text{subject to} & \sum_{p=1}^P g_p(x_p^{(p)}) \leq c \\
                        & x_l^{(i)} = x_l^{(j)}\,,\quad l \in S_i \cap S_j\,,\,\,\, (i,j) \in  \mathcal{E}\,,
    \end{array}
  \end{equation}
  where~$x_l^{(p)}$ is the copy of the variable~$x_l$ held by node~$p$, and~$\mathcal{E}$ is the set of edges in the communication network. The variable in~\eqref{Eq:NUMCoupled2} is~$\{\bar{x}_l\}_{l=1}^P$, where~$\bar{x}_l := \{x_l^{(p)}\}_{p \in \mathcal{V}_l}$, and~$\mathcal{V}_l$ is the set of nodes whose utilities depend on~$x_l$. Associating a dual variable~$\mu$ to the first constraint in~\eqref{Eq:NUMCoupled2} and~$\lambda_l^{ij}$ to each constraint of the second set of constraints, the dual problem of~\eqref{Eq:NUMCoupled2} is
  \begin{equation}\label{Eq:DualCoupledNUM}
  	\underset{\mu,\{\lambda_l^{ij}\}}{\text{minimize}} \,\,\,\, h_1(\mu,\{\bar{\lambda}^{1j}\}_{j\in \mathcal{N}_1}) + h_2(\mu,\{\bar{\lambda}^{2j}\}_{j\in \mathcal{N}_2}) + \cdots + h_P(\mu,\{\bar{\lambda}^{Pj}\}_{j\in \mathcal{N}_P})\,,
  \end{equation}
  where the function~$h_p$ is associated to source~$p$ and is given by
  \begin{multline*}
		h_p(\mu,\{\bar{\lambda}^{pj}\}_{j \in \mathcal{N}_p}) = \sup_{x^{(p)}} \,\,\, U_p(\{x_l^{(p)}\}_{l \in S_p}) + \mu^\top g_p(x_p^{(p)}) - \frac{1}{P}\mu^\top c \\+ \sum_{j \in \mathcal{N}_p} \sum_{l \in S_p \cap S_j} \text{sign}(j-p) \bigl(\lambda_l^{pj}\bigr)^\top x_l^{(p)} + \text{i}_{\mathbb{R}^+}(\mu)\,.
  \end{multline*}
	We used the notation~$\bar{\lambda}^{ij} = \{\lambda_l^{ij}\}_{l \in S_i \cap S_j}$. Note that to arrive at~\eqref{Eq:DualCoupledNUM} we used the identity
	$$
		\sum_{(i,j) \in \mathcal{E}} \sum_{l \in S_i \cap S_j} \bigl(\lambda_l^{ij}\bigr)^\top (x_l^{(i)} - x_l^{(j)}) = \sum_{p=1}^P \sum_{j \in \mathcal{N}_p} \sum_{l \in S_p \cap S_j} \text{sign}(j-p) \bigl(\lambda_l^{pj}\bigr)^\top x_l^{(p)}
	$$
	and, with it, we extended the notation~$\lambda_l^{ij}$ for~$i>j$ as $\lambda_l^{ij} := -\lambda_l^{ji}$. The variable in~\eqref{Eq:DualCoupledNUM} has the global components~$\mu$, appearing in all the functions~$h_p$, and non-global components~$\{\lambda_l^{ij}\}$. Thus, \eqref{Eq:DualCoupledNUM} is a particular instance of~\eqref{Eq:IntroProb} with a mixed variable.

	\begin{figure}
  \centering
  \psscalebox{1.0}{
		\begin{pspicture}(7,4.3)
			\def\nodesimp{
          \pscircle*[linecolor=black!65!white](0,0){0.3}
      }

      \rput(1,3.7){\rnode{N1}{\nodesimp}}
      \rput(0.3,2.3){\rnode{N2}{\nodesimp}}
      \rput(0.3,0.3){\rnode{N3}{\nodesimp}}
      \rput(2.7,1.6){\rnode{N4}{\nodesimp}}
      \rput(5,1.3){\rnode{N5}{\nodesimp}}
      \rput(4.0,3.3){\rnode{N6}{\nodesimp}}
      \rput(6.7,2.4){\rnode{N7}{\nodesimp}}

      \rput(1,3.7){\small \textcolor{white}{$1$}}    %\rput[lb](1.282843,3.982843){$f_1$}
      \rput(0.3,2.3){\small \textcolor{white}{$2$}}    %\rput[rb](0.017157,2.282843){$f_2$}
      \rput(0.3,0.3){\small \textcolor{white}{$3$}}  %\rput[rt](0.017157,0.017157){$f_3$}
      \rput(2.7,1.6){\small \textcolor{white}{$4$}}  %\rput[lt](2.782843,1.417157){$f_4$}
      \rput(5,1.3){\small \textcolor{white}{$5$}}      %\rput[lt](5.282843,1.717157){$f_5$}
      \rput(4.0,3.3){\small \textcolor{white}{$6$}}    %\rput[lb](4.082843,3.282843){$f_6$}
      \rput(6.7,2.4){\small \textcolor{white}{$7$}}  %\rput[lb](6.982843,3.482843){$f_7$}

      \psset{nodesep=0.35cm,linestyle=solid,arrowsize=5pt,arrowinset=0.1,labelsep=0.07}
      \ncline[]{->}{N1}{N2}\nbput{$\phi_{12}(x_{12})$}
      \ncline[]{->}{N2}{N3}\nbput{$\phi_{23}(x_{23})$}
      \ncline[]{->}{N2}{N4}\naput{$\phi_{24}(x_{24})$}
      \ncline[]{->}{N4}{N5}\nbput{$\phi_{45}(x_{45})$}
      \ncline[]{->}{N4}{N6}\nbput{$\phi_{46}(x_{46})$}
      %\ncline[nodesep=0.35cm]{-}{N5}{N6}\Bput{$\phi_{56}(x_{56})$}
      \ncline[]{->}{N5}{N7}\nbput{$\phi_{57}(x_{57})$}
      \ncline[]{<-}{N3}{N4}\nbput{$\phi_{43}(x_{43})$}
      \ncline[]{->}{N1}{N6}\naput{$\phi_{16}(x_{16})$}
      \ncline[]{->}{N6}{N7}\naput{$\phi_{67}(x_{67})$}

      %\psgrid
    \end{pspicture}
  }
  \caption{
		Example of a network flow problem. Each edge has associated both a variable~$x_{ij}$ and function of that variable, $\phi_{ij}(x_{ij})$. The goal is to minimize the sum of all the functions, while satisfying conservation of flow constraints.
	}
  \label{Fig:NetworkFlow}
  \end{figure}
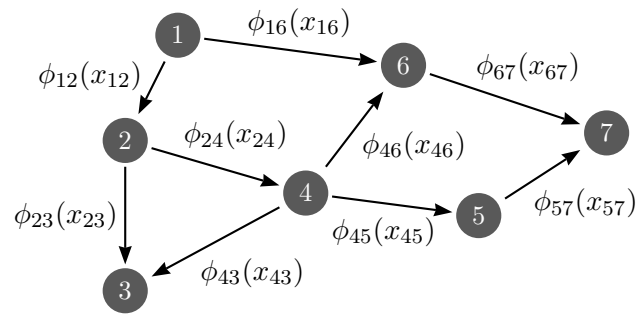

	\subsection{Network flow problems}

	A network flow problem is formulated on a network with arcs~$\mathcal{A}$ (or directed edges), where an arc from node~$i$ to node~$j$, i.e., $(i,j) \in \mathcal{A}$, indicates a flow in that direction. \fref{Fig:NetworkFlow} shows an example which allows, for example, a flow from node~$1$ to node~$6$, but not from node~$6$ to node~$1$. To quantify the flow in an arc~$(i,j) \in \mathcal{A}$, we use a non-negative variable~$x_{ij}$. Also, each arc~$(i,j) \in \mathcal{A}$ has associated a cost function~$\phi_{ij}(x_{ij})$, depending only on~$x_{ij}$, that typically increases with~$x_{ij}$. The goal in network flow problems is to minimize the sum of all these cost functions, while constraining the flows to satisfy conservation laws; namely, the inflows at a given node have to equal the outflows. These inflows/outflows are either caused by neighboring nodes, or are injected/extracted externally at the node itself. A node to which flow is injected (resp.\ extracted) is called source (resp.\ sink). For example, in \fref{Fig:NetworkFlow}, if either~$x_{12}$ or~$x_{16}$ is positive, node~$1$ can only be a source, since all of its edges point outwards. Nodes~$3$ and~$7$, in contrast, can only be sinks, if the flow in their incident arcs is nonzero. Other nodes in that network, for example node~$4$, can be sources, sinks, or neither. A way to represent a network with flows is via the node-arc incidence matrix~$B$, where the column associated to an arc from node~$i$ to node~$j$ has a~$-1$ in the $i$th entry, a $1$ in the $j$th entry, and zeros elsewhere. We assume the components of the variable~$x$ and the columns of~$B$ are in lexicographic order. For example, $x = (x_{12},x_{16},x_{23},x_{24},x_{43},x_{45},x_{46},x_{57},x_{67})$ would be the variable in \fref{Fig:NetworkFlow}. The laws of conservation of flow are expressed as $B x = d$, where~$d \in \mathbb{R}^P$ is the vector of external inputs/outputs. The entries of~$d$ sum up to zero and~$d_p < 0$ (resp. $d_p > 0$) if node~$p$ is a source (resp.\ sink). When node~$p$ is neither a source nor a sink, $d_p = 0$. The problem we solve is
	\begin{equation}\label{Eq:NetworkFlow}
		\begin{array}{ll}
			\underset{x}{\text{minimize}} & \sum_{(i,j) \in \mathcal{A}} \phi_{ij}(x_{ij})\\
			\text{subject to} & Bx = d \\
			                  & x\geq 0\,,
		\end{array}
	\end{equation}
	which can be written as~\eqref{Eq:IntroProb} by setting
	\begin{multline*}
		f_p\biggl(\{x_{pj}\}_{(p,j) \in \mathcal{A}},\{x_{jp}\}_{(j,p) \in \mathcal{A}}\biggr) = \frac{1}{2}\sum_{(p,j) \in \mathcal{A}} \phi_{pj}(x_{pj}) + \frac{1}{2}\sum_{(j,p) \in \mathcal{A}} \phi_{jp}(x_{jp}) \\+ \text{i}_{\{b_p^\top x = d_p\}}\biggl(\{x_{pj}\}_{(p,j) \in \mathcal{A}},\{x_{jp}\}_{(j,p) \in \mathcal{A}}\biggr)\,,
	\end{multline*}
	where~$b_p^\top$ is the $p$th row of~$B$. In words, $f_p$ consists of the sum of the functions associated to all arcs involving node~$p$, plus the indicator function of the set~$\{x\,:\,b_p^\top x = d_p\}$, which enforces the conservation of flow at node~$p$ and only involves the variables~$\{x_{pj}\}_{(p,j) \in \mathcal{A}}$ and $\{x_{jp}\}_{(j,p) \in \mathcal{A}}$.

	Regarding the communication network~$\mathcal{G} = (\mathcal{V},\mathcal{E})$, we assume it consists of the underlying undirected network. This means that nodes~$i$ and~$j$ can exchange messages directly, i.e., $(i,j) \in \mathcal{E}$ for~$i<j$, if there is an arc between these nodes, i.e., $(i,j) \in \mathcal{A}$ or $(j,i) \in \mathcal{A}$. Therefore, in contrast with the flows, messages do not necessarily need to be exchanged satisfying the direction of the arcs. In fact, messages and flows might represent different physical quantities: think, for example, in a network of water pipes controlled by actuators at each pipe junction; while the pipes might enforce a direction in the flow of water (by using valves, for example), there is no reason to impose the same constraint on the electrical signals exchanged by the actuators. In problem~\eqref{Eq:NetworkFlow}, the subgraph induced by~$x_{ij}$, $(i,j) \in \mathcal{A}$, consists only of nodes~$i$ and~$j$ and an edge connecting them. This makes the variable in~\eqref{Eq:NetworkFlow} connected and star-shaped.

	\begin{figure}
  \centering
  \psscalebox{1.25}{
		\begin{pspicture}(5,5.6)
				\def\nodesimp{
            \pscircle*[linecolor=black!50!white](0,0){0.08}
        }
        \def\nodesimpSPEC{
            \pscircle*[linecolor=black!50!white](0,0){0.08}
            \pscircle[linecolor=black!98!white,linewidth=1.0pt](0,0){0.085}
        }

        \def\square{
            \psframe*[linecolor=black!98!white](-0.1,0)(0,0.1)
        }

        %\rput{30}(3.8,4.0){\psellipse[linestyle=none,fillcolor=black!50!white,fillstyle=shape](0,0)(1.7,1.5)}
        \rput{20}(3.06,4.0){\psellipse[linestyle=none,fillcolor=black!18!white,fillstyle=shape](0,0)(0.92,1.3)}
        \rput{52}(4.75,4.45){\psellipse[linestyle=none,fillcolor=black!18!white,fillstyle=shape](0,0)(0.8,0.5)}
        \rput{-20}(1.0,2.95){\psellipse[linestyle=none,fillcolor=black!18!white,fillstyle=shape](0,0)(1.4,0.4)}
        \rput{-45}(0.71,0.85){\psellipse[linestyle=none,fillcolor=black!18!white,fillstyle=shape](0,0)(1.01,1.01)}
        \rput{0}(4.37,1.1){\psellipse[linestyle=none,fillcolor=black!18!white,fillstyle=shape](0,0)(0.75,1.15)}

        \rput(3.315200,3.290900){\rnode{N0}{\nodesimpSPEC}}
        \rput(4.288700,4.029300){\rnode{N1}{\nodesimpSPEC}}
        \rput(2.074500,2.517300){\rnode{N2}{\nodesimpSPEC}}
        \rput(4.000100,1.920500){\rnode{N3}{\nodesimpSPEC}}
        \rput(1.259200,1.418300){\rnode{N4}{\nodesimp}}
        \rput(4.242100,0.762400){\rnode{N5}{\nodesimp}}
        \rput(1.776900,2.468500){\rnode{N6}{\nodesimp}}
        \rput(4.695200,4.614800){\rnode{N7}{\nodesimp}}
        \rput(3.418600,4.171300){\rnode{N8}{\nodesimpSPEC}}
        \rput(0.694600,1.269700){\rnode{N9}{\nodesimp}}
        \rput(3.751900,3.187800){\rnode{N10}{\nodesimp}}
        \rput(0.259800,1.013500){\rnode{N11}{\nodesimp}}
        \rput(3.025500,4.353500){\rnode{N12}{\nodesimpSPEC}}
        \rput(4.557100,1.591800){\rnode{N13}{\nodesimpSPEC}}
        \rput(4.346000,0.165800){\rnode{N14}{\nodesimp}}
        \rput(3.123000,2.914400){\rnode{N15}{\nodesimp}}
        \rput(1.189800,2.841500){\rnode{N16}{\nodesimp}}
        \rput(4.799200,3.973500){\rnode{N17}{\nodesimp}}
        \rput(2.730700,3.936100){\rnode{N18}{\nodesimp}}
        \rput(4.418400,4.524500){\rnode{N19}{\nodesimp}}
        \rput(0.885900,3.074200){\rnode{N20}{\nodesimp}}
        \rput(1.552000,1.099500){\rnode{N21}{\nodesimp}}
        \rput(3.734700,3.439100){\rnode{N22}{\nodesimp}}
        \rput(4.823500,4.316200){\rnode{N23}{\nodesimp}}
        \rput(0.894800,1.006600){\rnode{N24}{\nodesimp}}
        \rput(1.050941,0.881621){\rnode{N25}{\nodesimp}}
        \rput(0.613100,1.538700){\rnode{N26}{\nodesimpSPEC}}
        \rput(2.332400,4.367800){\rnode{N27}{\nodesimp}}
        \rput(5.000000,5.000000){\rnode{N28}{\nodesimp}}
        \rput(3.453400,2.882500){\rnode{N29}{\nodesimp}}
        \rput(1.413155,0.955549){\rnode{N30}{\nodesimp}}
        \rput(0.884293,1.206324){\rnode{N31}{\nodesimp}}
        \rput(1.147400,0.579500){\rnode{N32}{\nodesimpSPEC}}
        \rput(3.063800,3.531600){\rnode{N33}{\nodesimp}}
        \rput(4.112100,0.202900){\rnode{N34}{\nodesimp}}
        \rput(0.480300,3.080900){\rnode{N35}{\nodesimp}}
        \rput(2.792000,3.353000){\rnode{N36}{\nodesimp}}
        \rput(3.383400,4.731400){\rnode{N37}{\nodesimp}}
        \rput(1.002600,0.000000){\rnode{N38}{\nodesimp}}
        \rput(2.997400,4.973000){\rnode{N39}{\nodesimp}}
        \rput(4.509800,0.290000){\rnode{N40}{\nodesimp}}
        \rput(2.409100,3.477100){\rnode{N41}{\nodesimp}}
        \rput(4.654700,0.482200){\rnode{N42}{\nodesimpSPEC}}
        \rput(0.084900,3.421200){\rnode{N43}{\nodesimpSPEC}}
        \rput(4.418600,2.037400){\rnode{N44}{\nodesimp}}
        \rput(3.610300,4.688900){\rnode{N45}{\nodesimp}}
        \rput(4.382500,1.779000){\rnode{N46}{\nodesimp}}
        \rput(0.110200,1.529800){\rnode{N47}{\nodesimp}}
        \rput(0.586200,3.365800){\rnode{N48}{\nodesimp}}
        \rput(4.660903,4.199744){\rnode{N49}{\nodesimp}}
        \rput(1.216640,0.992721){\rnode{N50}{\nodesimp}}
        \rput(4.983900,1.386100){\rnode{N51}{\nodesimp}}
        \rput(4.195100,1.481500){\rnode{N52}{\nodesimp}}
        \rput(2.795900,4.914800){\rnode{N53}{\nodesimp}}
        \rput(0.996900,1.644300){\rnode{N54}{\nodesimp}}
        \rput(3.581900,3.669700){\rnode{N55}{\nodesimp}}
        \rput(0.000000,3.145700){\rnode{N56}{\nodesimp}}
        \rput(3.630492,2.975443){\rnode{N57}{\nodesimp}}
        \rput(3.352200,3.636700){\rnode{N58}{\nodesimp}}
        \rput(3.899800,1.508400){\rnode{N59}{\nodesimp}}

        \psset{nodesep=0.080000cm,linewidth=0.5pt,linecolor=black!50!white}
        \ncline{-}{N0}{N1}
        \lput{:U}{\rput[r](0,-0.04){\square}}
        %\lput{:U}{\rput[r](0.45,-0.04){\square}}
        \ncline{-}{N0}{N2}
        \ncline{-}{N0}{N3}
        \lput{:U}{\rput[r](-0.1,-0.04){\square}}
        %\lput{:U}{\rput[r](0.6,-0.04){\square}}
        \ncline{-}{N0}{N8}
        \lput{:U}{\rput[r](0.2,-0.04){\square}}
        \ncline{-}{N0}{N10}
        \ncline{-}{N0}{N12}
        \ncline{-}{N0}{N15}
        \ncline{-}{N0}{N18}
        \ncline{-}{N0}{N22}
        \lput{:U}{\rput[r](0,-0.04){\square}}
        \ncline{-}{N0}{N29}
        \ncline{-}{N0}{N33}
        \ncline{-}{N0}{N36}
        \lput{:U}{\rput[r](0,-0.04){\square}}
        \ncline{-}{N0}{N55}
        \ncline{-}{N0}{N57}
        \ncline{-}{N0}{N58}
        \ncline{-}{N1}{N7}
        \lput{:U}{\rput[r](0.08,-0.04){\square}}
        \ncline{-}{N1}{N17}
        \ncline{-}{N1}{N19}
        \ncline{-}{N1}{N23}
        \ncline{-}{N1}{N49}
        \ncline{-}{N2}{N3}
        \lput{:U}{\rput[r](-0.6,-0.04){\square}}
        %\lput{:U}{\rput[r](0.81,-0.04){\square}}
        \ncline{-}{N2}{N4}
        \ncline{-}{N2}{N6}
        \ncline{-}{N2}{N16}
        \lput{:U}{\rput[r](0.09,-0.04){\square}}
        \ncline{-}{N3}{N5}
        \lput{:U}{\rput[r](0.2,-0.04){\square}}
        \ncline{-}{N3}{N13}
        \ncline{-}{N3}{N44}
        \ncline{-}{N3}{N46}
        \ncline{-}{N3}{N52}
        \ncline{-}{N3}{N59}
        \lput{:U}{\rput[r](0.06,-0.04){\square}}
        \ncline{-}{N4}{N9}
        \ncline{-}{N4}{N21}
        \ncline{-}{N4}{N24}
        \lput{:U}{\rput[r](0.05,-0.04){\square}}
        \ncline{-}{N4}{N25}
        \ncline{-}{N4}{N26}
        \ncline{-}{N4}{N30}
        \lput{:U}{\rput[r](0.05,-0.04){\square}}
        \ncline{-}{N4}{N31}
        \ncline{-}{N4}{N32}
        \ncline{-}{N4}{N50}
        \ncline{-}{N4}{N54}
        \ncline{-}{N5}{N14}
        \ncline{-}{N5}{N34}
        \ncline{-}{N5}{N40}
        \ncline{-}{N5}{N42}
        \ncline{-}{N7}{N28}
        \lput{:U}{\rput[r](0.05,-0.04){\square}}
        \ncline{-}{N8}{N37}
        \ncline{-}{N8}{N45}
        \ncline{-}{N9}{N11}
        \ncline{-}{N12}{N39}
        \ncline{-}{N12}{N53}
        \ncline{-}{N13}{N51}
        \ncline{-}{N16}{N20}
        \ncline{-}{N16}{N35}
        \lput{:U}{\rput[r](0.2,-0.04){\square}}
        \ncline{-}{N16}{N4}
        %\lput{:U}{\rput[r](-0.42,-0.04){\square}}
        \lput{:U}{\rput[r](0.38,-0.04){\square}}
        \ncline{-}{N18}{N27}
        \ncline{-}{N20}{N48}
        \ncline{-}{N26}{N47}
        \lput{:U}{\rput[r](0.05,-0.04){\square}}
        \ncline{-}{N32}{N38}
        \lput{:U}{\rput[r](0.05,-0.04){\square}}
        \ncline{-}{N35}{N43}
        \lput{:U}{\rput[r](0,-0.04){\square}}
        \ncline{-}{N35}{N56}
        \ncline{-}{N36}{N41}
        \ncline{-}{N32}{N5}
        %\lput{:U}{\rput[r](-1.2,-0.04){\square}}
        \lput{:U}{\rput[r](1.22,-0.04){\square}}

        \rput(3.315200,3.290900){\rnode{N0}{\nodesimpSPEC}}
        \rput(4.288700,4.029300){\rnode{N1}{\nodesimpSPEC}}
        \rput(2.074500,2.517300){\rnode{N2}{\nodesimpSPEC}}
        \rput(4.000100,1.920500){\rnode{N3}{\nodesimpSPEC}}
        \rput(1.259200,1.418300){\rnode{N4}{\nodesimp}}
        \rput(4.242100,0.762400){\rnode{N5}{\nodesimp}}
        \rput(1.776900,2.468500){\rnode{N6}{\nodesimp}}
        \rput(4.695200,4.614800){\rnode{N7}{\nodesimp}}
        \rput(3.418600,4.171300){\rnode{N8}{\nodesimpSPEC}}
        \rput(0.694600,1.269700){\rnode{N9}{\nodesimp}}
        \rput(3.751900,3.187800){\rnode{N10}{\nodesimp}}
        \rput(0.259800,1.013500){\rnode{N11}{\nodesimp}}
        \rput(3.025500,4.353500){\rnode{N12}{\nodesimpSPEC}}
        \rput(4.557100,1.591800){\rnode{N13}{\nodesimpSPEC}}
        \rput(4.346000,0.165800){\rnode{N14}{\nodesimp}}
        \rput(3.123000,2.914400){\rnode{N15}{\nodesimp}}
        \rput(1.189800,2.841500){\rnode{N16}{\nodesimp}}
        \rput(4.799200,3.973500){\rnode{N17}{\nodesimp}}
        \rput(2.730700,3.936100){\rnode{N18}{\nodesimp}}
        \rput(4.418400,4.524500){\rnode{N19}{\nodesimp}}
        \rput(0.885900,3.074200){\rnode{N20}{\nodesimp}}
        \rput(1.552000,1.099500){\rnode{N21}{\nodesimp}}
        \rput(3.734700,3.439100){\rnode{N22}{\nodesimp}}
        \rput(4.823500,4.316200){\rnode{N23}{\nodesimp}}
        \rput(0.894800,1.006600){\rnode{N24}{\nodesimp}}
        \rput(1.050941,0.881621){\rnode{N25}{\nodesimp}}
        \rput(0.613100,1.538700){\rnode{N26}{\nodesimpSPEC}}
        \rput(2.332400,4.367800){\rnode{N27}{\nodesimp}}
        \rput(5.000000,5.000000){\rnode{N28}{\nodesimp}}
        \rput(3.453400,2.882500){\rnode{N29}{\nodesimp}}
        \rput(1.413155,0.955549){\rnode{N30}{\nodesimp}}
        \rput(0.884293,1.206324){\rnode{N31}{\nodesimp}}
        \rput(1.147400,0.579500){\rnode{N32}{\nodesimpSPEC}}
        \rput(3.063800,3.531600){\rnode{N33}{\nodesimp}}
        \rput(4.112100,0.202900){\rnode{N34}{\nodesimp}}
        \rput(0.480300,3.080900){\rnode{N35}{\nodesimp}}
        \rput(2.792000,3.353000){\rnode{N36}{\nodesimp}}
        \rput(3.383400,4.731400){\rnode{N37}{\nodesimp}}
        \rput(1.002600,0.000000){\rnode{N38}{\nodesimp}}
        \rput(2.997400,4.973000){\rnode{N39}{\nodesimp}}
        \rput(4.509800,0.290000){\rnode{N40}{\nodesimp}}
        \rput(2.409100,3.477100){\rnode{N41}{\nodesimp}}
        \rput(4.654700,0.482200){\rnode{N42}{\nodesimpSPEC}}
        \rput(0.084900,3.421200){\rnode{N43}{\nodesimpSPEC}}
        \rput(4.418600,2.037400){\rnode{N44}{\nodesimp}}
        \rput(3.610300,4.688900){\rnode{N45}{\nodesimp}}
        \rput(4.382500,1.779000){\rnode{N46}{\nodesimp}}
        \rput(0.110200,1.529800){\rnode{N47}{\nodesimp}}
        \rput(0.586200,3.365800){\rnode{N48}{\nodesimp}}
        \rput(4.660903,4.199744){\rnode{N49}{\nodesimp}}
        \rput(1.216640,0.992721){\rnode{N50}{\nodesimp}}
        \rput(4.983900,1.386100){\rnode{N51}{\nodesimp}}
        \rput(4.195100,1.481500){\rnode{N52}{\nodesimp}}
        \rput(2.795900,4.914800){\rnode{N53}{\nodesimp}}
        \rput(0.996900,1.644300){\rnode{N54}{\nodesimp}}
        \rput(3.581900,3.669700){\rnode{N55}{\nodesimp}}
        \rput(0.000000,3.145700){\rnode{N56}{\nodesimp}}
        \rput(3.630492,2.975443){\rnode{N57}{\nodesimp}}
        \rput(3.352200,3.636700){\rnode{N58}{\nodesimp}}
        \rput(3.899800,1.508400){\rnode{N59}{\nodesimp}}

        \rput[b](0.4,3.68){\scriptsize Area 1}
        \rput[b](0.4,1.95){\scriptsize Area 2}
        \rput[r](3.6,1.7){\scriptsize Area 3}
				\rput[r](2.4,5.3){\scriptsize Area 4}
				\rput[r](5.1,5.3){\scriptsize Area 5}

        %\psgrid
   \end{pspicture}
  }
  \bigskip
  \caption[Illustration of five connected areas in a power network.]{
		Illustration of five connected areas in a power network. Nodes represent buses, i.e, generators or loads, and they are connected through transmission lines. Circled nodes indicate voltage measurements and squares in the transmission lines indicate current measurements.
	}
  \label{Fig:PowerNetwork}
  \end{figure}
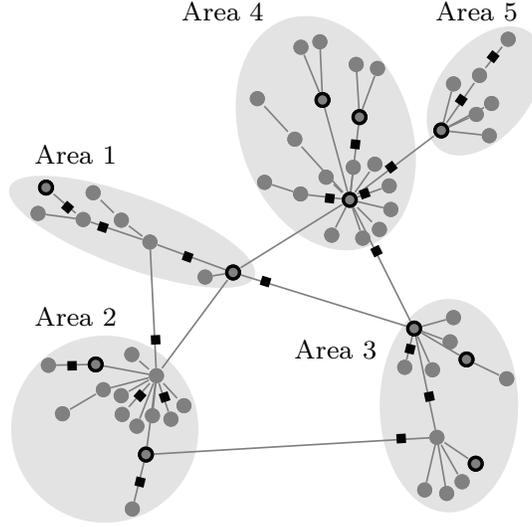

	\subsection{State estimation in the power grid}

	The power grid is the network that connects energy producers to energy consumers. Both its large-scale dimensions and its large number of parameters make it an appropriate application of distributed optimization. For example, state estimation in the power grid~\cite{Abur04-PowerSystemStateEstimation} can be posed as a particular instance of~\eqref{Eq:IntroProb} with a star-shaped variable, as was done by Kekatos and Giannakis in~\cite{Kekatos12-DistributedRobustPowerStateEstimation}. In this subsection, we briefly describe this problem from their point of view and derive the algorithm they propose, but adapted to solve~\eqref{Eq:IntroProb} for a generic connected variable. As we had mentioned in \cref{Ch:relatedWork}, the algorithm proposed in~\cite{Kekatos12-DistributedRobustPowerStateEstimation} is actually the only algorithm we found that can solve~\eqref{Eq:IntroProb} for connected variables that are neither global nor stars. Later, in subsection~\ref{SubSec:DerivationNonConnected}, we will generalize it to solve~\eqref{Eq:IntroProb} with a non-connected variable. For other problems and applications of distributed optimization in the power grid, see, for example, \cite{Giannakis13-MonitoringOptimizationPowerGrids,Boyd12-DynamicNetworkEnergyManagementProximalMessagePassing,Baker12-OptimalIntegrationIntermittent,Hug08-CoordinatedPowerFlowControlSecurityPowerSystems-thesis}.

	\mypar{State estimation}
	To explain state estimation in the power grid, consider the network of \fref{Fig:PowerNetwork}, whose nodes are divided into~$5$ disjoint areas. The nodes represent either generators or loads or, in the terminology of power systems, buses. Buses are connected through transmission lines, represented as the edges of the network, and through which current flows. Each area in the network, although controlled by its own operator, is also connected to other areas for robustness and reliability. It is essential for the proper functioning of the power network to know or, at least, to estimate the state of the system; this includes knowing power flows, voltage and current magnitudes at the buses, and generator outputs. \fref{Fig:PowerNetwork} illustrates a typical scenario where circled nodes indicate buses at which voltage measurements are taken, and edges with squares indicate lines where current measurements are taken. Based on these measurements and on a model for the system, the goal of state estimation is to determine what are the voltages and currents at the other buses and lines of the network. Let us denote the state of the entire network, i.e., the set of all the voltages and all the currents, with~$x \in \mathbb{R}^n$. The state of a given area~$p$ is a set~$S_p \in \{1,\ldots,n\}$ of $n_p = |S_p|$ components of~$x$, i.e., $x_{S_p}$ denotes the state of area~$p$. Let us represent the set of measurements taken at this area with~$y_p \in \mathbb{R}^{m_p}$. These measurements are related with~$x_{S_p}$ through~$y_p = h_p(x_{S_p}) + w_p$, where $h_p:\mathbb{R}^{n_p} \xrightarrow{} \mathbb{R}^{m_p}$ models area~$p$ and is typically a nonlinear function, and~$w_p \in \mathbb{R}^{m_p}$ models measurement noise and model inaccuracies. The areas that are connected with transmission lines will share some state variables, i.e., $S_i \cap S_j$. The problem of state estimation in power systems can then be formulated as
	\begin{equation}\label{Eq:PowerEstimationNolinear}
		\underset{x \in \mathbb{R}^n}{\text{minimize}} \,\,\, \frac{1}{2}\sum_{p=1}^P \bigl\|y_p - h_p(x_{S_p}) \bigr\|^2\,.
	\end{equation}
	Since each function~$h_p$ is nonlinear, problem~\eqref{Eq:PowerEstimationNolinear} is nonconvex. Yet, as mentioned in~\cite{Kekatos12-DistributedRobustPowerStateEstimation}, either using Gauss-Newton methods to solve~\eqref{Eq:PowerEstimationNolinear} directly or using a DC-approximation model for the system, one usually ends up with a linearized version of the system, i.e., $y_p = H_p x_{S_p} + w_p$, where~$H_p \in \mathbb{R}^{m_p \times n_p}$ is the Jacobian of~$h_p$ at some nominal operating point. So, instead of solving the nonconvex problem~\eqref{Eq:PowerEstimationNolinear}, we can solve
	\begin{equation}\label{Eq:PowerEstimationConvex}
		\underset{x \in \mathbb{R}^n}{\text{minimize}} \,\,\, \frac{1}{2}\sum_{p=1}^P \bigl\|y_p - H_p x_{S_p} \bigr\|^2\,,
	\end{equation}
	which is convex. If we view each area in \fref{Fig:PowerNetwork} as a node of a network, then~\eqref{Eq:PowerEstimationConvex} (and also~\eqref{Eq:PowerEstimationNolinear}) have the format of~\eqref{Eq:IntroProb}, where each~$f_p$ is given by $f_p(x_{S_p}) = (1/2)\|y_p - H_p x_{S_p}\|^2$. In this case, the variable is star-shaped because each set~$S_p$ indexes components of the state of area~$p$ and, possibly, of areas adjacent (i.e., neighbors) of area~$p$.

	\mypar{The algorithm in~\cite{Kekatos12-DistributedRobustPowerStateEstimation}}
	The algorithm proposed in~\cite{Kekatos12-DistributedRobustPowerStateEstimation} solves~\eqref{Eq:PowerEstimationConvex} in a distributed way and assumes that neighboring areas communicate, i.e., that they are able to exchange estimates of their common state variables. However, that algorithm can be easily generalized to solve~\eqref{Eq:IntroProb} under a generic connected variable, i.e., all induced subgraphs are connected. The algorithm is presented as Algorithm~\ref{Alg:Kekatos} and is derived in \aref{App:ADMMAlgsPartialClass}.

	\begin{algorithm}
		\small
    \caption{\cite{Kekatos12-DistributedRobustPowerStateEstimation}}
    \label{Alg:Kekatos}
    \algrenewcommand\algorithmicrequire{\textbf{Initialization:}}
    \begin{algorithmic}[1]
    \Require Choose $\rho \in \mathbb{R}$; for all $p \in \mathcal{V}$ and $l \in S_p$, set $\gamma_{l}^{(p),0} = x_l^{(p),0} = 0$; set $k=0$
    \Repeat

      \ForAll{$p \in \mathcal{V}$ [in parallel]}

        \Statex

        \State Compute \,
					 $
						 v_l^{(p),k} = \gamma_l^{(p),k} - \frac{\rho}{2} \Bigl(D_{p,l}\, x_l^{(p),k} + \sum_{j \in \mathcal{N}_p \cap \mathcal{V}_l} x_l^{(j),k}\Bigr)
           $,\,
           for all $ l \in S_p$
           \label{SubAlg:Kekatosv}

        \Statex

        \State Compute \,
           $
             x_{S_p}^{(p),k+1} =
             \underset{x_{S_p}^{(p)}=\{x_l^{(p)}\}_{l \in S_p}}{\arg\min}
             \,\,
					 	 f_p(x_{S_p}^{(p)}) + \sum_{l \in S_p} {v_l^{(p),k}}^\top x_l^{(p)}  + \frac{\rho}{2}\sum_{l \in S_p} D_{p,l}\bigl(x_l^{(p)}\bigr)^2
           $
           \label{SubAlg:KekatosProx}

        \Statex

        \State For each component~$l \in S_p$, exchange $x_l^{(p),k+1}$ with neighbors $\mathcal{N}_p \cap \mathcal{V}_l$
					 \label{SubAlg:KekatosComm}

				\Statex

				\State Update the dual variables \,
					 $
						\gamma_l^{(p),k+1} = \gamma_l^{(p),k} + \frac{\rho}{2} \sum_{j \in \mathcal{N}_p \cap \mathcal{V}_l} (x_l^{(p),k+1} -  x_l^{(j),k+1})
					 $,\,
					 for all $l \in S_p$
					 \label{SubAlg:KekatosDual}

				\vspace{0.2cm}
				\State $k \gets k+1$

			\EndFor
    \Until{some stopping criterion is met}
    \end{algorithmic}
  \end{algorithm}

  Algorithm~\ref{Alg:Kekatos} solves~\eqref{Eq:IntroProb} by first reformulating it as
  \begin{equation}\label{Eq:PartialReformulationEdge}
  	\begin{array}{ll}
  		\underset{\{\bar{x}_l\}_{l=1}^n}{\text{minimize}} & f_1(x_{S_1}^{(1)}) + f_1(x_{S_2}^{(2)}) + \cdots + f_1(x_{S_P}^{(P)}) \\
  		\text{subject to} & x_l^{(p)} = x_l^{(j)}\,, \quad l \in S_p \cap S_j\,, \quad j \in \mathcal{N}_p\,, \quad p = 1,\ldots,P\,,
  	\end{array}
  \end{equation}
  where we created a copy of the component~$x_l$ in all the nodes whose functions depend on~$x_l$, i.e., on all $p \in \mathcal{V}_l$, where $\mathcal{G}_l = (\mathcal{V}_l,\mathcal{E}_l)$ is the subgraph induced by~$x_l$. The copy at node~$p$ is~$x_l^{(p)}$. All the copies held by node~$p$ are denoted with~$x_{S_p}^{(p)}:= \{x_l^{(p)}\}_{l \in S_p}$. The constraints in~\eqref{Eq:PartialReformulationEdge} enforce copies of the component~$x_l$ to be equal for neighboring nodes that depend on it. That is, if nodes~$i$ and~$j$ are neighbors, $(i,j) \in \mathcal{E}$, and both depend on~$x_l$, $l \in S_i$ and~$l \in S_j$, then $x_l^{(i)} = x_l^{(j)}$ will be on the constraints of~\eqref{Eq:PartialReformulationEdge}. Since we assume a connected variable, problems~\eqref{Eq:IntroProb} and~\eqref{Eq:PartialReformulationEdge} are equivalent. We also denote by~$\bar{x}_l$ the set of all copies of the component~$x_l$, i.e., $\bar{x}_l = \{x_l^{(p)}\}_{p \in \mathcal{V}_l}$. Similarly to what was done for Algorithm~\ref{Alg:Zhu}, \cite{Kekatos12-DistributedRobustPowerStateEstimation} introduces a variable per network edge and writes~\eqref{Eq:PartialReformulationEdge} equivalently as
  \begin{equation}\label{Eq:PartialReformulationKekatos}
		\begin{array}{cl}
			\underset{\{\bar{x}_l\}_{l=1}^n, \{\bar{z}_l\}_{l=1}^n}{\text{minimize}} & f_1(x_{S_1}^{(1)}) + f_1(x_{S_2}^{(2)}) + \cdots + f_1(x_{S_P}^{(P)}) \\
  		\text{subject to} & x_l^{(p)} = z_l^{\{p,j\}}\,, \quad l \in S_p \cap S_j\,, \quad j \in \mathcal{N}_p\,, \quad p = 1,\ldots,P\,,
  	\end{array}
  \end{equation}
  where $z_l^{\{i,j\}} = z_l^{\{j,i\}}$ is associated to the common component~$x_l$ between nodes~$i$ and~$j$, for~$(i,j) \in \mathcal{E}_l$. We used~$\bar{z}_l$ to denote the set of variables~$z_{l}^{\{i,j\}}$ associated to the component~$x_l$, i.e., $\bar{z}_l = \{z_l^{\{i,j\}}\}_{(i,j) \in \mathcal{E}_l}$. Problem~\eqref{Eq:PartialReformulationKekatos} has two sets of variables, $\{\bar{x}_l\}_{l=1}^n$ and~$\{\bar{z}_l\}_{l=1}^n$, and linear constraints. Therefore, the $2$-block ADMM~\eqref{Eq:RelatedWorkADMMIter1}-\eqref{Eq:RelatedWorkADMMIter3} can be applied and yields Algorithm~\ref{Alg:Kekatos}, as shown in \aref{App:ADMMAlgsPartialClass}. Algorithm~\ref{Alg:Kekatos} has a structure very similar to Algorithm~\ref{Alg:Zhu}; indeed, it is derived using the same principles, but adapted to the problem~\eqref{Eq:IntroProb}. In particular, all nodes perform the same tasks in parallel. These tasks consist of solving an optimization problem in step~\ref{SubAlg:KekatosProx} and sending components of the respective solution to the neighbors that have common components, in step~\ref{SubAlg:KekatosComm}. We used~$D_{p,l}$ to denote the degree of node~$p$ in the subgraph induced by component~$x_l$, $\mathcal{G}_l$. Steps~\ref{SubAlg:Kekatosv}, \ref{SubAlg:KekatosProx}, and~\ref{SubAlg:KekatosComm} in Algorithm~\ref{Alg:Kekatos} correspond to step~\ref{SubAlg:RelatedWorkZhuUpdateX} of Algorithm~\ref{Alg:Zhu}: now, however, the prox notation is not as convenient as it was for the global class algorithms. Also, if in Algorithm~\ref{Alg:Zhu} each node broadcasts all the components of its new update, in Algorithm~\ref{Alg:Kekatos} each node~$p$ needs only to transmit to its neighbor~$j \in \mathcal{N}_p$ their common components $x_{S_p \cap S_j}$. After these exchanges occur, node~$p$ can update its set of dual variables~$\gamma_l$, $l \in \mathcal{S}_p$, as in step~\ref{SubAlg:KekatosDual}.

	The algorithm we propose for~\eqref{Eq:IntroProb} relates to the algorithm we proposed for the global class~\eqref{Eq:GlobalProblem} in the same way that Algorithm~\ref{Alg:Kekatos} relates to Algorithm~\ref{Alg:Zhu}. We will derive it in the next section, first for a connected variable, and then for a general variable, connected or not. Similarly to the algorithms for the global class, our algorithm outperforms Algorithm~\ref{Alg:Kekatos} in terms of the number of communications, as will be observed in \sref{Sec:PartialExpResults}.

	\section{Algorithm derivation}
	\label{Sec:PartialAlgorithmDerivation}

	In this section, we derive our algorithm for~\eqref{Eq:IntroProb}. First, we consider a connected variable, i.e., every induced subgraph is connected, and then we propose a way to address a non-connected variable, i.e., when there is at least one induced subgraph that is non-connected.

	\subsection{Connected variable}

	The idea we use to derive an algorithm for~\eqref{Eq:IntroProb} when the variable is connected is the same we used before to derive Algorithm~\ref{Alg:GlobalClass} for the global class: we manipulate~\eqref{Eq:IntroProb} to make the multi-block ADMM~\eqref{Eq:RelatedWorkADMMProbExtendedADMMIter1}-\eqref{Eq:RelatedWorkADMMProbExtendedADMMIter4} applicable. The difference is in the way we manipulate the problem, more specifically, in how we create copies of the variables. Recall our notation on the coloring scheme: $\mathcal{C}_c \subset \mathcal{V}$ denotes the nodes that have color~$c$ and~$\mathcal{C}(p)$ denotes the color of node~$p$; also, $C_c = |\mathcal{C}_c|$ is the number of nodes with color~$c$. As in the derivation of the global class algorithm, we will assume, without loss of generality, that the nodes are numbered according to their colors: the first~$C_1$ nodes have color~$1$, $\mathcal{C}_1 = \{1,\ldots,C_1\}$, the next~$C_2$ nodes have color~$2$, $\mathcal{C}_2 = \{C_1 +1,\ldots,C_1+C_2\}$, and so on.

	\mypar{Problem manipulation}
	Recall that~$\mathcal{G}_l = (\mathcal{V}_l,\mathcal{E}_l)$ denotes the subgraph induced by component~$x_l$. In this subsection, we assume each~$\mathcal{G}_l$ is connected. Similarly to~\cite{Kekatos12-DistributedRobustPowerStateEstimation}, we create a copy of the component~$x_l$ only in the nodes that are interested in it, which are precisely the nodes in~$\mathcal{G}_l$; let~$x_l^{(p)}$ be the copy at node~$p$. Since a given node~$p$ depends on the components~$x_{S_p}$ of the variable~$x$, it will have~$|S_p|$ different  (scalar) copies; let~$x_{S_p}^{(p)} := \{x_l^{(p)}\}_{l \in S_p}$ be the set of all these copies, at node~$p$. We now rewrite~\eqref{Eq:IntroProb} in a way slightly different than~\eqref{Eq:PartialReformulationEdge}:
	\begin{equation}\label{Eq:PartialManip1}
		\begin{array}{ll}
			\underset{\{\bar{x}_l\}_{l=1}^n}{\text{minimize}} & f_1(x_{S_1}^{(1)}) + f_2(x_{S_2}^{(2)}) + \cdots + f_P(x_{S_P}^{(P)}) \\
			\text{subject to} & x_l^{(i)} = x_l^{(j)}\,, \quad (i,j) \in \mathcal{E}_l\,,\,\,\, l = 1,\ldots,n\,,
		\end{array}
	\end{equation}
	where the optimization variable is~$\{\bar{x}_l\}_{l=1}^L$ and it represents the set of all copies. We used~$\bar{x}_l$ to denote all copies of the component~$x_l$, which are located only in the nodes of~$\mathcal{G}_l$: $\bar{x}_l := \{x_l^{(p)}\}_{p \in \mathcal{V}_l}$. Although problems~\eqref{Eq:PartialReformulationEdge} and~\eqref{Eq:PartialManip1} are both equivalent to~\eqref{Eq:IntroProb}, \eqref{Eq:PartialReformulationEdge} has twice the constraints of~\eqref{Eq:PartialManip1}. While in~\eqref{Eq:PartialReformulationEdge} each constraint appears twice (to make the introduction of the~$z$'s in~\eqref{Eq:PartialReformulationKekatos} possible), our reformulation~\eqref{Eq:PartialManip1} uses less constraints. Recall that our notation~$(i,j) \in \mathcal{E}_l$ (or~$\mathcal{E}$) implies that~$i < j$, and therefore there are no repeated equations in~\eqref{Eq:PartialManip1}. Finally, note that our assumption that the variable is connected is what makes problems~\eqref{Eq:IntroProb} and~\eqref{Eq:PartialManip1} equivalent, since each induced subgraph is connected. When the variable is non-connected, this equivalence no longer holds.

	Let~$A_l$ denote the transpose of the node-arc incidence matrix of the subgraph~$\mathcal{G}_l$. Then, the constraint~$x_l^{(i)} = x_l^{(j)}$, $(i,j) \in \mathcal{E}_l$ can be written as~$A_l \bar{x}_l = 0$. We now use the coloring scheme (cf.\ Assumption~\ref{Ass:PPColoring}) to partition each variable~$\bar{x}_l$ as~$\bar{x}_l = (\bar{x}_l^1,\ldots,\bar{x}_l^C)$, where
	$$
		\bar{x}_l^c =
		\left\{
			\begin{array}{ll}
				\{x_l^{(p)}\}_{p \in \mathcal{V}_l \cap \mathcal{C}_c}\,, &\quad \text{if $\mathcal{V}_l \cap \mathcal{C}_c \neq \emptyset$} \\
				\emptyset\,, &\quad \text{if $\mathcal{V}_l \cap \mathcal{C}_c = \emptyset$}
			\end{array}\,.
		\right.
	$$
	Recall that~$\mathcal{C}_c$ is the set of nodes that have color~$c$. In words, $\bar{x}_l^c$ represents the set of copies of~$x_l$ held by the nodes that have color~$c$. If no node with color~$c$ depends on~$x_l$, then $\bar{x}_l^c$ is empty. Using a similar notation for the columns of the matrix~$A_l$, we write $A_l \bar{x}_l$ as $\bar{A}_l^1 \bar{x}_l^1 + \cdots + \bar{A}_l^C \bar{x}_l^C$, for all~$l$. Therefore, \eqref{Eq:PartialManip1} is equivalent to
  \begin{equation}\label{Eq:PartialManip2}
  	\begin{array}{ll}
			\underset{\bar{x}^1,\ldots,\bar{x}^C}{\text{minimize}} & \sum_{p \in \mathcal{C}_1} f_p(x_{S_p}^{(p)}) +  \cdots + \sum_{p \in \mathcal{C}_C} f_p(x_{S_p}^{(p)})\\
			\text{subject to} & \bar{A}^1 \bar{x}^1 + \cdots + \bar{A}^C \bar{x}^C = 0\,,
		\end{array}
  \end{equation}
  where~$\bar{x}^c = \{\bar{x}_l^c\}_{l=1}^n$, and~$\bar{A}^c$ is the diagonal concatenation of the matrices~$\bar{A}_1^c$, $\bar{A}_2^c$, \ldots, $\bar{A}_n^c$, i.e., $\bar{A}^c= \text{diag}(\bar{A}_1^c,\bar{A}_2^c,\ldots,\bar{A}_n^c)$. For better visualization, we wrote the constraint in~\eqref{Eq:PartialManip2} as
  \begin{equation}\label{Eq:PartialStructureMatrices}
    \underbrace{
		\begin{bmatrix}
			\bar{A}_1^1 &              &        &             \\
			            & \bar{A}_2^1  &        &             \\
			            &              & \ddots &             \\
			            &              &        & \bar{A}_n^1
		\end{bmatrix}
		}_{\bar{A}^1}
		\underbrace{
		\begin{bmatrix}
			\bar{x}_1^1 \\
			\bar{x}_2^1 \\
			\vdots \\
			\bar{x}_n^1
		\end{bmatrix}
		}_{\bar{x}^1}
		+
		\underbrace{
		\begin{bmatrix}
			\bar{A}_1^2 &              &        &             \\
			            & \bar{A}_2^2  &        &             \\
			            &              & \ddots &             \\
			            &              &        & \bar{A}_n^2
		\end{bmatrix}
		}_{\bar{A}^2}
		\underbrace{
		\begin{bmatrix}
			\bar{x}_1^2 \\
			\bar{x}_2^2 \\
			\vdots \\
			\bar{x}_n^2
		\end{bmatrix}
		}_{\bar{x}^2}
		+
		\cdots
		+
		\underbrace{
		\begin{bmatrix}
			\bar{A}_1^C &              &        &             \\
			            & \bar{A}_2^C  &        &             \\
			            &              & \ddots &             \\
			            &              &        & \bar{A}_n^C
		\end{bmatrix}
		}_{\bar{A}^C}
		\underbrace{
		\begin{bmatrix}
			\bar{x}_1^C \\
			\bar{x}_2^C \\
			\vdots \\
			\bar{x}_n^C
		\end{bmatrix}
		}_{\bar{x}^C}
		=
		0\,.
  \end{equation}
  Note that the~$c$th term in the objective of~\eqref{Eq:PartialManip2} depends only on~$\bar{x}^c$, the set of copies associated with nodes with color~$c$. Thus, \eqref{Eq:PartialManip2} has the format of~\eqref{Eq:RelatedWorkADMMProbExtended}, the problem solved by the multi-block ADMM, and thus the iterations~\eqref{Eq:RelatedWorkADMMProbExtendedADMMIter1}-\eqref{Eq:RelatedWorkADMMProbExtendedADMMIter4} can be applied.

  \mypar{Applying multi-block ADMM}
  To apply the multi-block ADMM iterations~\eqref{Eq:RelatedWorkADMMProbExtendedADMMIter1}-\eqref{Eq:RelatedWorkADMMProbExtendedADMMIter4} to~\eqref{Eq:PartialManip2}, we first need to write the augmented Lagrangian. Let~$\lambda_l^{ij}$ be the dual variable associated to the constraint~$x_l^{(i)} = x_l^{(j)}$, for some~$l \in \{1,\ldots,n\}$ and~$(i,j) \in \mathcal{E}_l$ (cf.\ \eqref{Eq:PartialManip1}). The augmented Lagrangian of~\eqref{Eq:PartialManip2} is then
  \begin{equation}\label{Eq:PartialAugmentedLagrangian}
		L_\rho(\bar{x}^1,\ldots,\bar{x}^C;\lambda) = \sum_{c=1}^C \sum_{p \in \mathcal{C}_c} f_p(x_{S_p}^{(p)}) + \sum_{c = 1}^C \lambda^\top \bar{A}^c \bar{x}^c + \frac{\rho}{2}\Bigl\|\sum_{c = 1}^C \bar{A}^c \bar{x}^c\Bigr\|^2\,,
  \end{equation}
  where~$\lambda = (\lambda_1,\ldots,\lambda_n)$ is the dual variable, whose $l$th block is~$\lambda_l := \{\lambda_l^{ij}\}_{(i,j) \in \mathcal{E}_l}$. The multi-block ADMM consists of a sequence of subproblems, obtained by minimizing~$L_\rho$ with respect to each block~$\bar{x}^c$, and then updating each dual variable with
  \begin{equation}\label{Eq:PartialDualVariableUpdate}
				\lambda_l^{ij,k+1} = \lambda_l^{ij,k} + \rho \bigl(x_l^{(i),k+1} - x_l^{(j),k+1}\bigr)\,,
  \end{equation}
  for every~$(i,j) \in \mathcal{E}_l$ and every~$l=1,\ldots,n$. In~\eqref{Eq:PartialDualVariableUpdate}, $k$ denotes the iteration number and~$x_l^{(p),k+1}$ is the estimate of the component~$x_l$, by node~$p$, after iteration~$k$. We now analyze the subproblem each node solves to find those estimates. In particular, we will see that minimizing~\eqref{Eq:PartialAugmentedLagrangian} with respect to~$\bar{x}^c$ yields~$|\mathcal{C}_c|$ problems that can be solved in parallel, i.e., all nodes with color~$c$ ``work'' in parallel. For example, the copies of the nodes with color~$1$ are updated according to~\eqref{Eq:RelatedWorkADMMProbExtendedADMMIter1}:
  \begin{align}
		  \bar{x}^{1,k+1}
		&=
		  \underset{\bar{x}^1}{\arg\min} \,
		  \sum_{p \in \mathcal{C}_1} f_p(x_{S_p}^{(p)}) + {\lambda^k}^\top \bar{A}^1 \bar{x}^1
		  + \frac{\rho}{2}\biggl\|\bar{A}^1 \bar{x}^{1} + \sum_{c = 2}^C \bar{A}^c \bar{x}^{c,k}\biggr\|^2
		\label{Eq:PartialC1_1}
		\\
		&=
		  \underset{\bar{x}^1}{\arg\min} \,
				\sum_{p \in \mathcal{C}_1} \biggl( f_p(x_{S_p}^{(p)})
				+
				\sum_{l \in S_p}
				\sum_{j \in \mathcal{N}_p \cap \mathcal{V}_l}
				\Bigl(
					\lambda_l^{pj,k} - \rho\, x_l^{(j),k}
				\Bigr)^\top x_l^{(p)}
				+
				\frac{\rho}{2}\sum_{l \in S_p} D_{p,l}\Bigl(x_l^{(p)}\Bigr)^2 \biggr)\,,
		\label{Eq:PartialC1_2}
	\end{align}
	whose equivalence is established in Lemma~\ref{Lem:EquivalenceNodesC1} below. 	As in Algorithm~\ref{Alg:Kekatos}, $D_{p,l}$ is the degree of node~$p$ in the subgraph~$\mathcal{G}_l$, i.e., the number of neighbors of node~$p$ that also depend on~$x_l$. Of course, $D_{p,l}$ is only defined when~$l \in S_p$. Note that all the dual variables~$\lambda_l^{pj}$ are well-defined because of our assumption that the nodes are numbered according to their colors; namely, any neighbor~$j$ of a node~$ p \in\mathcal{C}_1$ will have a color larger than~$1$, and hence~$p < j$, making~$\lambda_l^{pj}$ well-defined. Recall our convention that~$(i,j) \in \mathcal{E}$ implies $i<j$.	Before we establish the equivalence between~\eqref{Eq:PartialC1_1} and~\eqref{Eq:PartialC1_2}, note that~\eqref{Eq:PartialC1_2} actually consists of~$|\mathcal{C}_1|$ problems that can be solved in parallel. This is because nodes with the same color are not neighbors and, thus, none of the components of the optimization variable~$\bar{x}^1$, which corresponds to all the copies of the nodes with color~$1$, appears as~$x_l^{(j),k}$ in the second term of~\eqref{Eq:PartialC1_2}. This means that all nodes~$p \in \mathcal{C}_1$ solve, in parallel,
	\begin{multline}\label{Eq:PartialProbSolvedAtEachNodeC1}
		x_{S_p}^{(p),k+1}
		=
		\underset{x_{S_p}^{(p)}=\{x_l^{(p)}\}_{l \in S_p}}{\arg\min} \, f_p(x_{S_p}^{(p)})
		+
		\sum_{l \in S_p} \sum_{j \in \mathcal{N}_p \cap \mathcal{V}_l}
				\bigl(\lambda_l^{pj,k} - \rho\, x_l^{(j),k}\bigr)^\top x_l^{(p)}
		+
		\frac{\rho}{2}\sum_{l \in S_p} D_{p,l}\Bigl(x_l^{(p)}\Bigr)^2\,.
	\end{multline}
	Node~$p$ can only solve~\eqref{Eq:PartialProbSolvedAtEachNodeC1} if it knows~$x_l^{(j),k}$ and~$\lambda_l^{pj,k}$, for $j \in \mathcal{N}_p \cap \mathcal{V}_l$ and $l \in S_p$. This is possible if, in the previous iteration, it received the respective copies of~$x_l$ from its neighbors. This is also enough for knowing~$\lambda_l^{pj,k}$, although we will see later that no node needs	to know each~$\lambda_l^{pj,k}$ individually.	We finally show how to obtain~\eqref{Eq:PartialC1_2} from~\eqref{Eq:PartialC1_1}.
	\begin{lemma}\label{Lem:EquivalenceNodesC1}
		\eqref{Eq:PartialC1_1} and \eqref{Eq:PartialC1_2} are equivalent.
	\end{lemma}

	\begin{proof}
	\hfill

	\noindent
		To go from~\eqref{Eq:PartialC1_1} to~\eqref{Eq:PartialC1_2}, we first develop the last two terms of~\eqref{Eq:PartialC1_1}, respectively,
		\begin{equation}\label{App:Eq_1}
			{\lambda^k}^\top \bar{A}^1 \bar{x}^1
		\end{equation}
		and
		\begin{equation}\label{App:Eq_2}
			\frac{\rho}{2}\Bigl\|\bar{A}^1 \bar{x}^{1} + \sum_{c = 2}^C \bar{A}^c \bar{x}^{c,k}\Bigr\|^2\,.
		\end{equation}
		We first address~\eqref{App:Eq_1}. 	Given the structure of~$\bar{A}^1$, as seen in \eqref{Eq:PartialStructureMatrices}, we can write~\eqref{App:Eq_1} as $\sum_{l=1}^n ((\bar{A}_l^1)^\top \lambda_l^k)^\top \bar{x}_l^1$.	Recall that~$(\bar{A}_l^1)^\top$, if it exists (i.e., if there is a node with color~$1$ that depends on component~$x_l$), consists of the block of rows of the node-arc incidence matrix of~$\mathcal{G}_l$ corresponding to the nodes with color~$1$. Therefore, if there exists $p \in \mathcal{C}_1 \cap \mathcal{V}_l$, the vector~$(\bar{A}_l^1)^\top \lambda_l^k$ will have an entry $\sum_{j \in \mathcal{N}_p \cap \mathcal{V}_l} \text{sign}(j - p)\lambda_l^{pj,k}$. The sign function appears here because the column of the node-arc incidence matrix corresponding to $x_l^{(i)} - x_l^{(j)} = 0$, for a pair~$(i,j) \in \mathcal{E}_l$, contains $1$ in the $i$th entry and~$-1$ in the $j$th entry, where~$i<j$. In the previous expression, we used an extension of the definition of~$\lambda_l^{ij}$, which was only defined for~$i<j$ (due to our convention that for any edge~$(i,j) \in \mathcal{E}$ we have always~$i<j$). Assume~$\lambda_l^{ij}$ is initialized with zero; switching~$i$ and~$j$ in~\eqref{Eq:PartialDualVariableUpdate}, we obtain~$\lambda_l^{ji,k} = -\lambda_l^{ij,k}$, which holds for all iterations~$k$. To be consistent with the previous equation, we define~$\lambda_l^{ij}$ as $\lambda_l^{ij} := -\lambda_l^{ji}$ whenever~$i>j$. Therefore, \eqref{App:Eq_1} develops as
	\begin{align}
	    {\lambda^k}^\top \bar{A}^1 \bar{x}^1
	  &=
		  \sum_{l=1}^n ((\bar{A}_l^1)^\top \lambda_l^k)^\top \bar{x}_l^1
		  \notag
		\\
		&=
		  \sum_{l=1}^n \sum_{p \in \mathcal{C}_1} \sum_{j \in \mathcal{N}_p \cap \mathcal{V}_l} \!\!\! \text{sign}(j-p) \Bigl(\lambda_l^{pj,k}\Bigr)^\top x_l^{(p)}
		\notag
		\\
		&=
		  \sum_{p \in \mathcal{C}_1} \sum_{l=1}^n \sum_{j \in \mathcal{N}_p \cap \mathcal{V}_l} \!\!\! \text{sign}(j-p) \Bigl(\lambda_l^{pj,k}\Bigr)^\top x_l^{(p)}\,.
		\label{Eq:AppFirstTerm_1}
	\end{align}

	Regarding~\eqref{App:Eq_2}, it can be written as
	\begin{align}
	      &\frac{\rho}{2}\Bigl\|\bar{A}^1 \bar{x}^1 + \sum_{c = 2}^C \bar{A}^c \bar{x}^{c,k}\Bigr\|^2
		  =
		    \frac{\rho}{2}\Bigl\|\bar{A}^1 \bar{x}^1\Bigr\|^2 + \rho (\bar{A}^1\bar{x}^1)^\top \sum_{c=2}^C \bar{A}^c \bar{x}^{c,k}
		    + \frac{\rho}{2}\Bigl\|\sum_{c=2}^C \bar{A}^c \bar{x}^{c,k}\Bigr\|^2\,.
		  \label{Eq:AppSecondTerm_1}
	\end{align}
	Since the last term does not depend on~$\bar{x}^1$, it can be dropped from the optimization problem. We now use the structure of~$\bar{A}^1$ to rewrite the first term of~\eqref{Eq:AppSecondTerm_1}:
	\begin{align}
		  \frac{\rho}{2}\Bigl\|\bar{A}^1 \bar{x}^1\Bigr\|^2
		&=
		  \frac{\rho}{2} \sum_{l=1}^n (\bar{x}^1_l)^\top (\bar{A}_l^1)^\top \bar{A}_l^1 \bar{x}^1_l
		\label{Eq:AppSecondTerm_2}
		\\
		&=
		  \frac{\rho}{2} \sum_{l=1}^n \sum_{p \in \mathcal{C}_1} D_{p,l} \Bigl( x_l^{(p)}\Bigr)^2
		\label{Eq:AppSecondTerm_3}
		\\
		&=
		  \frac{\rho}{2} \sum_{p \in \mathcal{C}_1} \sum_{l \in S_p} D_{p,l} \Bigl( x_l^{(p)}\Bigr)^2\,.
		\label{Eq:AppSecondTerm_4}
	\end{align}
	From~\eqref{Eq:AppSecondTerm_2} to~\eqref{Eq:AppSecondTerm_3} we used the structure of~$\bar{A}_l^1$. Namely, if it exists, $(\bar{A}_l^1)^\top \bar{A}_l^1$ is a diagonal matrix, where each diagonal entry is extracted from the diagonal of~$A_l^\top A_l$, the Laplacian matrix for~$\mathcal{G}_l$. Since each entry in the diagonal of a Laplacian matrix contains the degrees of the respective nodes, the diagonal of $(\bar{A}_l^1)^\top \bar{A}_l^1$ contains $D_{p,l}$ for all~$p \in \mathcal{C}_1$. The reason why~$(\bar{A}_l^1)^\top \bar{A}_l^1$ is diagonal is because nodes with the same color are never neighbors. As in~\eqref{Eq:AppSecondTerm_1}, we exchanged the order of the summations from~\eqref{Eq:AppSecondTerm_3} to~\eqref{Eq:AppSecondTerm_4}.

	Finally, we develop the second term of~\eqref{Eq:AppSecondTerm_1}:
	\begin{align}
		  \rho (\bar{A}^1\bar{x}^1)^\top \sum_{c=2}^C \bar{A}^c \bar{x}^{c,k}
		&=
		  \rho \sum_{c=2}^C \sum_{l=1}^n (\bar{x}^1_l)^\top (\bar{A}^1_l)^\top (\bar{A}^c_l) \,\bar{x}_l^{c,k}
		\label{Eq:AppSecondTerm_5}
		\\
		&=
		  -\rho \sum_{c=2}^C \sum_{l=1}^n \sum_{p \in \mathcal{C}_1} \sum_{j \in \mathcal{N}_p \cap \mathcal{C}_c \cap \mathcal{V}_l} {x_l^{(p)}}^\top x_l^{(j),k}
		\label{Eq:AppSecondTerm_6}
		\\
		&=
		  -\rho \sum_{p \in \mathcal{C}_1} \sum_{l \in S_p} {x_l^{(p)}}^\top \sum_{c=2}^C \sum_{j \in \mathcal{N}_p \cap \mathcal{C}_c \cap \mathcal{V}_l} x_l^{(j),k}
		\label{Eq:AppSecondTerm_7}
		\\
		&=
		  -\rho \sum_{p \in \mathcal{C}_1} \sum_{l \in S_p}  \sum_{j \in \mathcal{N}_p \cap \mathcal{V}_l}  {x_l^{(p)}}^\top x_l^{(j),k}\,.
		\label{Eq:AppSecondTerm_8}
	\end{align}
	In~\eqref{Eq:AppSecondTerm_5} we just used the structure of~$\bar{A}^1$ and~$\bar{A}^c$, as visualized in~\eqref{Eq:PartialStructureMatrices}. From~\eqref{Eq:AppSecondTerm_5} to~\eqref{Eq:AppSecondTerm_6} we used the fact that $(\bar{A}^1_l)^\top \bar{A}^c_l$ is a submatrix of~$A_l^\top A_l$, the Laplacian of~$\mathcal{G}_l$, containing some of its off-diagonal elements. More concretely, $(\bar{A}^1_l)^\top \bar{A}^c_l$ contains the entries of~$A_l^\top A_l$ corresponding to all the nodes~$i \in \mathcal{C}_1 \cap \mathcal{V}_l$ and~$j \in \mathcal{C}_c \cap \mathcal{V}_l$. And, for such nodes, the corresponding entry in~$A_l^\top A_l$ is~$-1$ if~$i$ and~$j$ are neighbors, and~$0$ otherwise. From~\eqref{Eq:AppSecondTerm_7} to~\eqref{Eq:AppSecondTerm_8} we just used the fact that the set $\{\mathcal{C}_c\}_{c=2}^C$ is nothing but a partition of the set of neighbors of any node with color~$1$.
	Using~\eqref{Eq:AppFirstTerm_1}, \eqref{Eq:AppSecondTerm_1}, \eqref{Eq:AppSecondTerm_4}, and~\eqref{Eq:AppSecondTerm_8} in~\eqref{Eq:PartialC1_1}, we get~\eqref{Eq:PartialC1_2}.
	\end{proof}

	The optimization problem~\eqref{Eq:PartialC1_2} decomposes into~$|\mathcal{C}_1|$ decoupled optimization problems, each one solved by a node with color~$1$. For node~$p$, the problem is~\eqref{Eq:PartialProbSolvedAtEachNodeC1}. For the other colors, the same reasoning and equations apply, just with one small difference: in the second term of~\eqref{Eq:PartialProbSolvedAtEachNodeC1} we have~$x_l^{(j),k+1}$ from the neighbors with a smaller color and~$x_l^{(j),k}$ from the nodes with a larger color.

	\begin{algorithm}
    \caption{Algorithm for a connected variable}
    \algrenewcommand\algorithmicrequire{\textbf{Initialization:}}
    \label{Alg:Conn}
    \begin{algorithmic}[1]
    \small
    \Require Choose $\rho \in \mathbb{R}$; for all~$p \in \mathcal{V}$ and $l \in S_p$, set $\gamma_{l}^{(p),0} = x_l^{(p),0} = 0$; set $k=1$
    \Repeat
    \For{$c =1,\ldots,C$}  \label{SubAlg:Conn_ForColors}
        \ForAll{$p \in \mathcal{C}_c$ [in parallel]}

				\Statex

        \State Compute \,
            $
                v_l^{(p),k} = \gamma_l^{(p),k}-
                \rho \sum_{\begin{subarray}{c}
                             j \in \mathcal{N}_p \cap \mathcal{V}_l \\
                             C(j) < c
                           \end{subarray}
                }x_l^{(j),k+1} - \rho \sum_{\begin{subarray}{c}
                             j \in \mathcal{N}_p \cap \mathcal{V}_l \\
                             C(j) > c
                           \end{subarray}
                }x_l^{(j),k}
            $,\, for all $l \in S_p$
            \label{SubAlg:Conn_Vec}

        \Statex

        \State Compute \,
					$
            x_{S_p}^{(p),k+1} =
            \underset{x_{S_p}^{(p)}=\{x_l^{(p)}\}_{l \in S_p}}{\arg\min}
            \,
						f_p(x_{S_p}^{(p)}) + \sum_{l \in S_p} {v_l^{(p),k}}^\top x_l^{(p)}  + \frac{\rho}{2}\sum_{l \in S_p} D_{p,l}\bigl(x_l^{(p)}\bigr)^2
            $
           \label{SubAlg:Conn_Prob}

        \State For each component~$l \in S_p$, exchange~$x_l^{(p),k+1}$ with neighbors $\mathcal{N}_p \cap \mathcal{V}_l$
        \label{SubAlg:Conn_Comm}

    \EndFor
    \EndFor \label{SubAlg:Conn_EndForColors}

    \ForAll{$p \in \mathcal{V}$ and $l \in S_p$ [in parallel]} \vspace{0.15cm}
    \hfill

        $
            \gamma_l^{(p),k+1} = \gamma_l^{(p),k} + \rho \sum_{j \in \mathcal{N}_p \cap \mathcal{V}_l} (x_l^{(p),k+1} -  x_l^{(j),k+1})
        $\label{SubAlg:Conn_DualVar}  \vspace{0.15cm}

    \EndFor
    \State $k \gets k+1$
    \Until{some stopping criterion is met}
    \end{algorithmic}
  \end{algorithm}

  Algorithm~\ref{Alg:Conn} shows the resulting algorithm. As in the global class algorithm (Algorithm~\ref{Alg:GlobalClass}), the coloring scheme functions as a schedule: the nodes with color~$1$ work first, the nodes with color~$2$ work next, and so on. Each ``work'' consists of computing $v_l^{(p),k}$ for all $l \in S_p$, as in step~\ref{SubAlg:Conn_Vec}, solving the optimization problem in step~\ref{SubAlg:Conn_Prob}, and then sending the new component estimates to the neighbors that also depend on those components, as in step~\ref{SubAlg:Conn_Comm}. After a given node~$p$ has received the new estimates from all its neighbors, it can update each dual variable~$\gamma_l^{(p)}$ as in step~\ref{SubAlg:Conn_DualVar}. Note that the edge-wise dual variables~$\lambda_l^{ij}$ were replaced by the node-wise dual variables~$\gamma_l^{(p)}$. The reason is because the optimization solved by node~$p$ (see~\eqref{Eq:PartialProbSolvedAtEachNodeC1}) depends on $\gamma_l^{(p),k} := \sum_{j \in \mathcal{N}_p \cap \mathcal{V}_l} \text{sign}(j-p) \lambda_l^{pj,k}$ and not on the individual~$\lambda_l^{ij}$'s. The update of step~\ref{SubAlg:Conn_DualVar} is obtained by replacing
  \begin{equation}\label{Eq:PartialNewDualUpdate}
  	\lambda_l^{ij,k+1} = \lambda_l^{ij,k} + \rho \,\,\text{sign}(j-i)\bigl(x_l^{(i),k+1} - x_l^{(j),k+1}\bigr)
  \end{equation}
  in the definition of~$\gamma_l^{(p),k}$. Note that~\eqref{Eq:PartialNewDualUpdate} differs from~\eqref{Eq:PartialDualVariableUpdate} in the extra ``sign.'' This is because we extended the definition of the dual variable~$\lambda_l^{ij}$ for~$i>j$ (see the proof of Lemma~\ref{Lem:EquivalenceNodesC1}).

  Note that if we make the variable global, i.e., if $S_p = \{1,\ldots,n\}$ for all~$p$, Algorithm~\ref{Alg:Conn} becomes the global class algorithm, that is, Algorithm~\ref{Alg:GlobalClass}. This means that Algorithm~\ref{Alg:Conn} is a generalization of Algorithm~\ref{Alg:GlobalClass} since, in fact, it cannot be obtained from it. The comments about the coordination of the nodes we made for Algorithm~\ref{Alg:GlobalClass} also apply to its generalization, Algorithm~\ref{Alg:Conn}. Namely, if each node knows its own color and the color of its neighbors, as specified in Assumption~\ref{Ass:PPColoring}, then the algorithm becomes automatically distributed, because each node can work immediately after it has received estimates from its neighbors with smaller colors. See \fref{Fig:SelfCoordination} from \cref{Ch:GlobalVariable} for an illustration. Regarding the convergence of Algorithm~\ref{Alg:Conn}, we have:
  \begin{theorem}\label{Teo:ConnectedConvergence}
		\hfill

		\medskip
		\noindent
    Let Assumptions~\ref{Ass:PPfunctions}-\ref{Ass:PPColoring} hold and let the variable of
    \begin{equation}\label{Eq:IntroProb}\tag{P}
			\begin{array}{ll}
				\underset{x \in\mathbb{R}^n}{\text{minimize}} & f_1(x_{S_1}) + f_2(x_{S_2}) + \cdots + f_P(x_{S_P})\,,
			\end{array}
		\end{equation}
		be connected. Then, Algorithm~\ref{Alg:Conn} produces a sequence~$(x_{S_1}^k,\ldots,x_{S_P}^k)$ convergent to~$(x_{S_1}^\star, \ldots,x_{S_P}^\star)$, where~$x^\star$ solves~\eqref{Eq:IntroProb}, when at least one of the following conditions is satisfied:
    \newcounter{TeoConvDADMM2}
		\begin{list}{(\alph{TeoConvDADMM2})}{\usecounter{TeoConvDADMM2}}
      \item the coloring scheme uses two colors only (which implies that the network is bipartite);
      \item each function~$f_p$ is strongly convex with modulus~$\mu_p$ and
      \begin{equation}\label{Eq:ThmConditionRho2}
				0 < \rho < \underset{c=1,\ldots,C}{\min}\,\,\, \frac{2\sum_{p \in \mathcal{C}_c}\mu_p}{3\,(C-1)\,\max_{p \in \mathcal{C}_c,\, l \in S_p}D_{p,l}}\,.
      \end{equation}
    \end{list}
  \end{theorem}
  \begin{proof}
  	\hfill

  	\medskip
  	\noindent
  	As in the proof of Theorem~\ref{Teo:GlobalConvergence}, we have to show that~\eqref{Eq:PartialManip2}, which is the problem to which we apply the multi-block ADMM, satisfies the conditions of Theorem~\ref{Teo:RelatedWorkConvergenceADMM}. In fact, Assumptions~\ref{Ass:PPfunctions}, \ref{Ass:PPsolvable}, and~\ref{Ass:PPWellForm}, together with the equivalence between~\eqref{Eq:IntroProb} and~\eqref{Eq:PartialManip2} (for a connected variable), imply that each function $\sum_{p \in \mathcal{C}_c} f_p(x_{S_p}^{(p)})$ in~\eqref{Eq:PartialManip2} is closed and convex over the full space. Next we see that condition \text{(a)} (resp.\ \text{(b)}) implies condition \text{(a)} (resp.\ \text{(b)}) of Theorem~\ref{Teo:RelatedWorkConvergenceADMM}.
  	\newcounter{TeoConvDADMMProof2}
		\begin{list}{(\alph{TeoConvDADMMProof2})}{\usecounter{TeoConvDADMMProof2}}
			\item We first see that Assumption~\ref{Ass:PPConnectedStatic} together with the fact that the variable is connected implies that each~$\bar{A}^c$ has full column rank. Let~$c$ be any color in~$\{1,2,\ldots,C\}$. By definition, $\bar{A}^c = \text{diag}(\bar{A}_1^c,\bar{A}_2^c,\ldots,\bar{A}_n^c)$; therefore, we have to prove that each~$\bar{A}_l^c$ has full column rank, for~$l=1,2,\ldots,n$. Let then~$c$ and~$l$ be fixed. We are going to prove that $(\bar{A}_l^c)^\top \bar{A}_l^c$, a square matrix, has full rank, and therefore $\bar{A}_l^c$ has full column rank. Since~$\bar{A}_l = \begin{bmatrix}\bar{A}_1^c & \bar{A}_2^c & \cdots & \bar{A}_n^c\end{bmatrix}$, $(\bar{A}_l^c)^\top \bar{A}_l^c$ corresponds to the $l$th block in the diagonal of the matrix~$A_l^\top A_l$, the Laplacian matrix of the induced subgraph~$\mathcal{G}_l$. 	Recall that each induced subgraph~$\mathcal{G}_l$ is connected, because the variable is connected. Consequently, each node in~$\mathcal{G}_l$ has at least one neighbor also in~$\mathcal{G}_l$ and hence each entry in the diagonal of~$A_l^\top A_l$ is greater than zero.\footnote{Implicitly, we are assuming that there is no component~$x_l$ that appears in only one node, say node~$p$; this would lead to a Laplacian matrix~$A_l^\top A_l$ equal to~$0$. This can be easily addressed by redefining~$f_p$, the function at node~$p$, to~$\tilde{f}_p(\cdot) = \inf_{x_l} f_p(\ldots,x_l,\ldots)$.} The same happens to the entries in the diagonal of~$(\bar{A}_l^c)^\top \bar{A}_l^c$. In fact, these are the only nonzero entries of~$(\bar{A}_l^c)^\top \bar{A}_l^c$, since this matrix is diagonal. The reason is because~$(\bar{A}_l^c)^\top \bar{A}_l^c$ corresponds to the Laplacian entries of nodes that have the same color, which are never neighbors. Therefore, $(\bar{A}_l^c)^\top \bar{A}_l^c$ has full rank. This shows that, independently of the coloring scheme, each matrix~$\bar{A}_c$ has full column rank. As a consequence, when the network is bipartite and the coloring scheme has two colors, point \text{(a)} of Theorem~\ref{Teo:RelatedWorkConvergenceADMM} holds.

			\item When each function~$f_p$ is strongly convex with modulus~$\mu_p$ and~$\rho$ satisfies~\eqref{Eq:ThmConditionRho2}, then~$\sum_{p \in \mathcal{C}_c}f_p$ is strongly convex with modulus~$\sum_{p \in \mathcal{C}_c} \mu_p$ \cite[Lem.\ 2.1.4]{Nesterov04-IntroductoryLecturesConvexOptimization} and conditions~\eqref{Eq:RelatedWorkExtendedADMMConditionRho} and~\eqref{Eq:ThmConditionRho2} are equivalent. To see this, note that
				\begin{align*}
					\sigma_{\max}(\bar{A}^c)^2
					=
					\lambda_{\max}((\bar{A}^c)^\top \bar{A}^c)
					=
					\max_{l=1,\ldots,n}\, \lambda_{\max}((\bar{A}_l^c)^\top \bar{A}_l^c)
					=
					\max_{l=1,\ldots,n, \, p \in \mathcal{C}_c}\, D_{p,l}
					=
					\max_{p \in \mathcal{C}_c,\, l \in S_p}D_{p,l}\,,
				\end{align*}
				since each $(\bar{A}^c_l)^\top \bar{A}^c_l$ is a diagonal matrix whose entries are the degrees of the nodes with color~$c$ that depend on component~$x_l$.
		\end{list}
  \end{proof}

	\subsection{Non-connected variable}
	\label{SubSec:DerivationNonConnected}

	In this subsection we drop the assumption that the variable is connected. This means that there exists at least one component~$x_l$ for which the induced subgraph~$\mathcal{G}_l$ is non-connected. In this case, problem~\eqref{Eq:PartialManip1} is no longer equivalent to problem~\eqref{Eq:IntroProb}, because its constraints fail to enforce equality between all the copies of~$x_l$. We propose a trick to make these problems equivalent, based on the following assumption:
	\begin{assumption}\label{Ass:Nonconnected}
		When the variable is non-connected, the communication network and all the sets~$S_p$ are known before the execution of the algorithm.
	\end{assumption}
	The reason we require both the communication network and the sets~$S_p$ to be known beforehand is to allow some preprocessing: first, we identify the non-connected components of the variable, and then, we select which nodes should retransmit them. Note that this assumption only requires knowing beforehand the components each node depends on, but not the functions~$f_p$. In other words, this preprocessing can be done before any data arrives.

	Let~$x_l$ be a non-connected component, i.e.,  the induced subgraph~$\mathcal{G}_l = (\mathcal{V}_l,\mathcal{E}_l)$ is non-connected. As we have seen, the constraint~$x_l^{(i)} = x_l^{(j)}$, $(i,j) \in \mathcal{E}_l$, in~\eqref{Eq:PartialManip1} is not enough to enforce equality of all the copies of~$x_l$. We propose enlarging the subgraph~$\mathcal{G}_l$ by selecting other nodes in the network that will retransmit estimates of~$x_l$. In other words, we will add to~$\mathcal{G}_l$ some nodes (and edges) so that that the induced subgraph becomes connected. Since our goal is to minimize the overall number of communications, we should add the least number of edges to this subgraph. It turns out that this is exactly the problem of finding an optimal \textit{Steiner tree} in the communication network.

	\begin{figure}
  \centering
  \psscalebox{1.1}{
    \begin{pspicture}(6.0,5.3)
        \def\nodesimp{
          \pscircle*[linecolor=black!25!white](0,0){0.15}
        }
        \def\nodeterm{
          \pscircle*[linecolor=black!90!white](0,0){0.15}
        }
         \def\nodestein{
           \pscircle[fillstyle=hlines*,linecolor=black!80!white,hatchcolor=black!80!white,hatchsep=2pt,hatchangle=160](0,0){0.15}
         }

        \psrotate(2.5,3.0){45}{
       	\rput(1.247200,2.922000){\rnode{N0}{\nodesimp}}	  %\rput(1.247200,2.922000){0}
				\rput(2.019781,3.556917){\rnode{N1}{\nodestein}}	  %\rput(2.019781,3.556917){1}
				\rput(0.929200,4.236700){\rnode{N2}{\nodesimp}}	  %\rput(0.929200,4.236700){2}
				\rput(2.615100,1.917500){\rnode{N3}{\nodestein}}	  %\rput(2.615100,1.917500){3}
				\rput(2.205218,1.005361){\rnode{N4}{\nodesimp}}	  %\rput(2.205218,1.005361){4}
				\rput(2.879054,2.882035){\rnode{N5}{\nodesimp}}	  %\rput(2.879054,2.882035){5}
				\rput(3.867600,1.535600){\rnode{N6}{\nodesimp}}	  %\rput(3.867600,1.535600){6}
				\rput(1.882690,4.547475){\rnode{N7}{\nodestein}}	  %\rput(1.882690,4.547475){7}
				\rput(0.658078,5.199245){\rnode{N8}{\nodesimp}}	  %\rput(0.658078,5.199245){8}
				\rput(4.602102,2.214207){\rnode{N9}{\nodeterm}}	  %\rput(4.602102,2.214207){9}
				\rput(4.619707,0.876559){\rnode{N10}{\nodesimp}}	  %\rput(4.619707,0.876559){10}
				\rput(3.159143,0.705314){\rnode{N11}{\nodesimp}}  %\rput(3.159143,0.705314){11}
				\rput(-0.165185,4.631585){\rnode{N12}{\nodesimp}}	%\rput(-0.165185,4.631585){12}
				\rput(1.755478,5.539351){\rnode{N13}{\nodeterm}}	  %\rput(1.755478,5.539351){13}
				\rput(2.923960,4.312936){\rnode{N14}{\nodesimp}}	  %\rput(2.923960,4.312936){14}
				\rput(5.358214,1.559764){\rnode{N15}{\nodesimp}}	  %\rput(5.358214,1.559764){15}
				\rput(0.268552,3.127543){\rnode{N16}{\nodestein}}	%\rput(0.268552,3.127543){16}
				\rput(3.968349,0.117789){\rnode{N17}{\nodesimp}}	  %\rput(3.968349,0.117789){17}
				\rput(2.717151,5.813546){\rnode{N18}{\nodeterm}}	  %\rput(2.717151,5.813546){18}
				\rput(1.753251,0.113327){\rnode{N19}{\nodesimp}}	  %\rput(1.753251,0.113327){19}
				\rput(-0.468233,2.451415){\rnode{N20}{\nodeterm}}	%\rput(-0.468233,2.451415){20}
				\rput(2.707175,-0.186721){\rnode{N21}{\nodesimp}}	%\rput(2.707175,-0.186721){21}
				\rput(3.878797,2.904735){\rnode{N22}{\nodeterm}}	  %\rput(3.878797,2.904735){22}
				\rput(1.482241,1.696233){\rnode{N23}{\nodesimp}}	  %\rput(1.482241,1.696233){23}
				\rput(5.252491,0.102231){\rnode{N24}{\nodesimp}}	  %\rput(5.252491,0.102231){24}
				}

				\psset{nodesep=0.140000cm}
				\ncline{-}{N16}{N2}
				\ncline{-}{N4}{N21}
				\ncline[linewidth=1.8pt,nodesep=0.17cm]{-}{N9}{N22}
				\ncline{-}{N0}{N1}
				\ncline{-}{N1}{N2}
				\ncline[linewidth=1.8pt,nodesep=0.17cm]{-}{N1}{N3}
				\ncline[linewidth=1.8pt,nodesep=0.17cm]{-}{N1}{N7}
				\ncline[linewidth=1.8pt,nodesep=0.17cm]{-}{N7}{N13}
				\ncline{-}{N1}{N14}
				\ncline[linewidth=1.8pt,nodesep=0.17cm]{-}{N1}{N16}
				\ncline{-}{N2}{N8}
				\ncline{-}{N2}{N12}
				\ncline{-}{N3}{N4}
				\ncline{-}{N3}{N5}
				\ncline{-}{N3}{N6}
				\ncline{-}{N3}{N11}
				\ncline[linewidth=1.8pt,nodesep=0.17cm]{-}{N3}{N22}
				\ncline{-}{N3}{N23}
				\ncline{-}{N4}{N19}
				\ncline{-}{N6}{N9}
				\ncline{-}{N6}{N10}
				\ncline{-}{N6}{N15}
				\ncline{-}{N6}{N17}
				\ncline{-}{N10}{N24}
				\ncline{-}{N11}{N21}
				\ncline[linewidth=1.8pt,nodesep=0.17cm]{-}{N13}{N18}
				\ncline[linewidth=1.8pt,nodesep=0.17cm]{-}{N16}{N20}

        %\psgrid
    \end{pspicture}
  }
  \vspace{-0.2cm}
  \caption[Example of an optimal Steiner tree.]{
			Example of an optimal Steiner tree. The required nodes~$\mathcal{R}$ are black, and the Steiner nodes~$\mathcal{S}$ are striped. The Steiner tree edges are represented with thicker lines.
  }
  \label{Fig:SteinerTree}
  \end{figure}
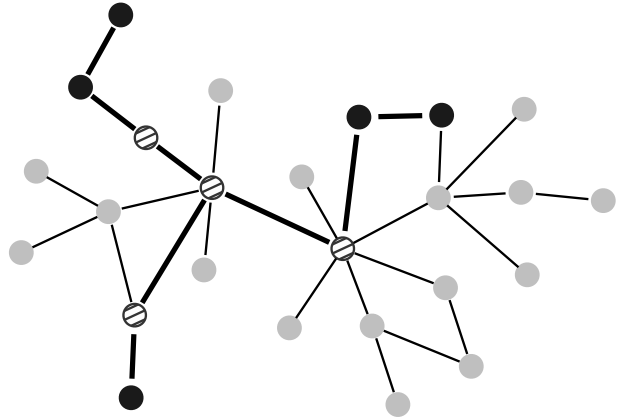

\mypar{Steiner tree problem}
	To describe the Steiner tree problem, consider an undirected graph~$\mathcal{G} = (\mathcal{V},\mathcal{E})$, in our case the communication network, and let~$\mathcal{R} \subset \mathcal{V}$ be a set of \textit{required nodes}, in our case, the nodes~$\mathcal{V}_l$ of a non-connected induced subgraph~$\mathcal{G}_l$. \fref{Fig:SteinerTree} shows an example where~$\mathcal{G}$ is the entire network, and $\mathcal{R}$ are the black nodes. A Steiner tree in~$\mathcal{G}$ is any tree in that contains the required nodes~$\mathcal{R}$; in other words, it is an acyclic connected subgraph~$(\mathcal{T},\mathcal{F}) \subseteq \mathcal{G}$ such that $\mathcal{R} \subseteq \mathcal{T}$ and $\mathcal{F} \subseteq \mathcal{E}$. The \textit{Steiner nodes}, which will be represented with~$\mathcal{S}$, are the nodes in that tree that are not required, i.e., $\mathcal{S}:= \mathcal{T} \backslash \mathcal{R}$. For example, in \fref{Fig:SteinerTree}, the Steiner nodes~$\mathcal{S}$ are striped and the Steiner tree edges~$\mathcal{F}$ are thicker. Note that the set of black and striped nodes and the thicker edges form a subgraph that is a tree.

	Now we can state the Steiner tree problem: \textit{given an undirected graph~$\mathcal{G} = (\mathcal{V},\mathcal{E})$, a set of required nodes~$\mathcal{R} \subset \mathcal{V}$, and a set of costs~$c_{ij}$ for each edge of the network $(i,j) \in \mathcal{E}$, find a Steiner tree whose edges have a minimal cost}. In our case, since we want to minimize the total number of communications, all edges are viewed equal, that is, they all have the same cost, for example, $c_{ij} = 1$. The set of required nodes in our case are the nodes in the subgraph induced by a non-connected component~$x_l$, i.e., $\mathcal{R} = \mathcal{V}_l$. Of course, we have to solve a Steiner tree problem for each non-connected component. Unfortunately, solving Steiner tree problems is NP-hard~\cite{Garey77-ComplexityComputingSteinerMinimalTrees}. However, many approximation algorithms are available, some of which have approximation guarantees. For example, the Steiner tree problem can be formulated as the following optimization problem~\cite{Williamson02-PrimalDualMethodApproximation}:
	\begin{equation}\label{Eq:SteinerProb}
		\begin{array}{ll}
			\underset{\{z_{ij}\}_{(i,j) \in \mathcal{E}}}{\text{minimize}} & \underset{(i,j)\in \mathcal{E}}{\sum} c_{ij} z_{ij} \vspace{0.1cm}\\
			\text{subject to} & \underset{\begin{subarray}{c}i \in \mathcal{U} \\ j \not\in \mathcal{U} \end{subarray}}{\sum} z_{ij} \geq 1\,,\quad \forall_\mathcal{U} \,:\, 0<|\mathcal{U} \cap \mathcal{R}| <|\mathcal{R}| \\
			& z_{ij} \in \{0,1\}\,,\quad (i,j) \in \mathcal{E}\,.
		\end{array}
	\end{equation}
	In the first constraint of~\eqref{Eq:SteinerProb}, $\mathcal{U}$ represents any subset of nodes that separates at least two required nodes, i.e., $\mathcal{U}$ contains at least one node in~$\mathcal{R}$, but not all of them. The optimization variable of problem~\eqref{Eq:SteinerProb} is~$z \in \mathbb{R}^E$ and each~$z_{ij}$ is associated to edge~$(i,j) \in \mathcal{E}$. If the optimal value is $z_{ij}^\star = 1$, then edge~$(i,j)$ is in the selected Steiner tree. Note that the last constraint of~\eqref{Eq:SteinerProb} imposes each component of~$z$ to be either~$0$ or~$1$. Let us denote the objective of problem~\eqref{Eq:SteinerProb} by~$h(z):= \sum_{(i,j)\in \mathcal{E}} c_{ij} z_{ij}$. We say that an algorithm for~\eqref{Eq:SteinerProb} has an \textit{approximation ratio of}~$\alpha$ if it produces a feasible point~$\bar{z}$ such that $h(\bar{z}) \leq \alpha h(z^\star)$, for any problem instance. The primal-dual algorithm for combinatorial problems~\cite{Williamson02-PrimalDualMethodApproximation,Goemans97-PrimalDualMethod}, for example, has an approximation ratio of~$2$. To the best of our knowledge, \cite{Robins00-ImprovedSteinerTreeApproximation} proposed the algorithm for computing Steiner trees that has the smallest approximation ratio, namely~$1+\text{ln}\,	3/2\simeq 1.55$.

	\mypar{Application to our problem}
	Based on Assumption~\ref{Ass:Nonconnected} and on the concept of Steiner tree problem, we now propose a modification to Algorithm~\ref{Alg:Conn} to make it applicable to a non-connected variable. This modification applies to Algorithm~\ref{Alg:Kekatos} exactly the same way. According to Assumption~\ref{Ass:Nonconnected}, both the communication network and the sets~$S_p$ are known before the execution of the algorithm. This allows solving a Steiner tree problem for each non-connected component, as a preprocessing step, which can be done in a distributed or in a centralized way (for distributed algorithms computing Steiner trees see for example \cite{Drummond09-DistributedDualAscentAlgorithmSteinerProblems,Sadeh08-DistributedPrimalDualApproximationAlgorithmsNetworkDesign-thesis}). More concretely, for every non-connected component~$x_l$ with induced subgraph~$\mathcal{G}_l=(\mathcal{V}_l,\mathcal{E}_l)$, we can compute a Steiner tree~$(\mathcal{T}_l,\mathcal{F}_l) \subseteq \mathcal{G}$ using~$\mathcal{V}_l$ as the set of required nodes. Let~$\mathcal{S}_l := \mathcal{T}_l \backslash \mathcal{V}_l$ denote the Steiner nodes in that tree. The functions associated to these Steiner nodes do not depend on~$x_l$, i.e., $l \not\in S_p$ for all~$p \in \mathcal{S}_l$. But we artificially force them to depend on it by defining a new induced graph as $\mathcal{G}_l'=(\mathcal{V}_l',\mathcal{E}_l')$, with $\mathcal{V}_l' := \mathcal{T}_l$ and~$\mathcal{E}_l' := \mathcal{E}_l\cup\mathcal{F}_l$. Then, we can create copies of~$x_l$ in all nodes in~$\mathcal{V}_l'$, and write~\eqref{Eq:IntroProb} equivalently as
	\begin{equation}\label{Eq:NonConnectedManip}
	  \begin{array}{ll}
		  \underset{\{\bar{x}_l\}_{l=1}^n}{\text{minimize}} & f_1(x_{S_1}^{(1)}) + f_2(x_{S_2}^{(2)}) + \cdots + f_P(x_{S_P}^{(P)}) \\
		  \text{subject to} & x_l^{(i)} = x_l^{(j)},\quad (i,j) \in \mathcal{E}_l'\,,\,\,l=1,\ldots,n\,,
		\end{array}
	\end{equation}
	where~$\{\bar{x}_l\}_{l=1}^L$ is the optimization variable, and~$\bar{x}_l := \{x_l^{(p)}\}_{p \in \mathcal{V}_l'}$ denotes the set of all copies of~$x_l$. If node~$p$ is a Steiner node for any component of the variable, it will hold ``extra'' copies, but its function~$f_p$ remains unchanged. In particular, it has the copies $x_{S_p \cup S_p'}^{(p)}$, where $S_p'$ is the set of components of which node~$p$ is a Steiner node, but its function~$f_p$ depends only on~$x_{S_p}^{(p)} := \{x_l^{(p)}\}_{l \in S_p}$. Of course, if a component~$x_l$ is connected, we set $\mathcal{G}_l' = \mathcal{G}_l$, and if node~$p$ is not Steiner for any component, we set~$S_p' = \emptyset$. If we replace problem~\eqref{Eq:PartialManip1} by the modified problem~\eqref{Eq:NonConnectedManip} and repeat the derivation that followed problem~\eqref{Eq:PartialManip1}, we get Algorithm~\ref{Alg:NonConn}.

  \begin{algorithm}
    \caption{Algorithm for a generic variable, connected or non-connected}
    \label{Alg:NonConn}
    \begin{algorithmic}[1]
    \small
    \algrenewcommand\algorithmicrequire{\textbf{Preprocessing:}}
    \Require
    \State Set $S_p' = \emptyset$ for all~$p \in \mathcal{V}$, and $\mathcal{V}_l'=\mathcal{V}_l$ for all~$l = \{1,\ldots,n\}$
    \ForAll{non-connected components $x_l$, $l \in \{1,\ldots,n\}$}
			\State Compute a Steiner tree $(\mathcal{T}_l, \mathcal{F}_l)$, setting $\mathcal{V}_l$ as the set of required nodes
			\State Set $\mathcal{V}_l'= \mathcal{T}_l$ and $\mathcal{S}_l := \mathcal{T}_l \backslash \mathcal{V}_l$ (Steiner nodes)
			\State For all $p \in \mathcal{S}_l$, $S_p' = S_p' \cup \{x_l\}$
    \EndFor
    \Statex
    \algrenewcommand\algorithmicrequire{\textbf{Main algorithm:}}
    \Require
    \algrenewcommand\algorithmicrequire{\textbf{Initialization:}}
    \Require Choose~$\rho \in \mathbb{R}$; for all $p \in \mathcal{V}$ and $l \in S_p \cup S_p'$, set $\gamma_{l}^{(p),0}\! = x_l^{(p),0}\! = 0$; set $k=0$
    \Repeat
		\label{SubAlg:NonConn_FirstStep}
    \For{$c =1,\ldots,C$}  \label{SubAlg:NonConn_ForColors}
        \ForAll{$p \in \mathcal{C}_c$ [in parallel]}

					\Statex

					\State Compute \,
						$
							v_l^{(p),k} = \gamma_l^{(p),k}-
                \rho \sum_{\begin{subarray}{c}
                             j \in \mathcal{N}_p \cap \mathcal{V}_l' \\
                             C(j) < c
                           \end{subarray}
                }x_l^{(j),k+1} - \rho \sum_{\begin{subarray}{c}
                             j \in \mathcal{N}_p \cap \mathcal{V}_l' \\
                             C(j) > c
                           \end{subarray}
                }x_l^{(j),k}
						$
						,\, for all $l \in S_p \cup S_p'$
            \label{SubAlg:NonConn_Vec}

					\Statex

					\State Compute \,
						$
							x_{S_p \cup S_p'}^{(p),k+1}
							=
							\underset{x_{S_p\cup S_p'}^{(p)}}{\arg\min}
            \,\,
						f_p(x_{S_p}^{(p)}) + \sum_{l \in S_p\cup S_p'} \Bigl( {v_l^{(p),k}}^\top x_l^{(p)}  + \frac{\rho}{2} D_{p,l}'\Bigl(x_l^{(p)}\Bigr)^2 \Bigr)
						$
					  \label{SubAlg:NonConn_Prob}

        \State For each component~$l \in S_p\cup S_p'$, exchange~$x_l^{(p),k+1}$ to neighbors $\mathcal{N}_p \cap \mathcal{V}_l'$
        \label{SubAlg:NonConn_Comm}

    \EndFor
    \EndFor \label{SubAlg:NonConn_EndForColors}

    \ForAll{$p \in \mathcal{V}$ and $l \in S_p \cup S_p'$ [in parallel]} \vspace{0.15cm}
    \hfill

        $
            \gamma_l^{(p),k+1} = \gamma_l^{(p),k} + \rho \sum_{j \in \mathcal{N}_p \cap \mathcal{V}_l'} (x_l^{(p),k+1} -  x_l^{(j),k+1})
        $\label{SubAlg:NonConn_DualVar}  \vspace{0.15cm}

    \EndFor
    \State $k \gets k+1$
    \Until{some stopping criterion is met}
    \label{SubAlg:NonConn_FinalStep}
    \end{algorithmic}
  \end{algorithm}

  Algorithm~\ref{Alg:NonConn} is essentially an adapted version of Algorithm~\ref{Alg:Conn}, with a preprocessing step, which can be computed in a centralized or in a distributed way. The preprocessing step relies on Assumption~\ref{Ass:Nonconnected} by assuming that both the communication network and the dependency sets~$S_p$ are known. Note that the specific functions~$f_p$ are not required for this preprocessing step. Regarding the main algorithm, it is similar to Algorithm~\ref{Alg:Conn} except that each node, in addition to estimating the components its function originally depends on, it also estimates the components for which it is a Steiner node. The computation for these additional components can, however, be found in closed-form: if node~$p$ is a Steiner node for component~$x_l$, it updates it as~$x_l^{(p),k+1} = -(1/(\rho \,D_{p,l}))v_l^{(p),k}$ in step~\ref{SubAlg:NonConn_Prob}. In Algorithm~\ref{Alg:NonConn}, $D_{p,l}'$ is defined as the degree of node~$p$ in the subgraph~$\mathcal{G}_l'$. The steps we took to generalize Algorithm~\ref{Alg:Conn} to a non-connected variable can be easily applied the same way to Algorithm~\ref{Alg:Kekatos}, the algorithm proposed by~\cite{Kekatos12-DistributedRobustPowerStateEstimation}.

  \section{Experimental results}
	\label{Sec:PartialExpResults}

	In this section, we assess experimentally the performance of the proposed algorithms, namely Algorithm~\ref{Alg:Conn} and Algorithm~\ref{Alg:NonConn}, with respect to prior distributed algorithms. We focus on two applications: networks flow problems and D-MPC. While network flow problems are formulated as~\eqref{Eq:IntroProb} with a star-shaped variable, D-MPC has more flexibility, since it can be formulated with any type of variable (see \ssref{SubSec:DMPC}). As mentioned before, most of the prior distributed optimization algorithms solve~\eqref{Eq:IntroProb} only when the variable is global or star-shaped. The only exception is the algorithm proposed by~\cite{Kekatos12-DistributedRobustPowerStateEstimation}, which we presented as Algorithm~\ref{Alg:Kekatos}. Indeed, that algorithm can solve~\eqref{Eq:IntroProb} with any connected variable and, if using the adaptation we proposed in the previous section, it can also solve it with a non-connected variable.

	\mypar{Communication steps}
	The performance metric we use in our experiments is the number of communication steps (CSs). The concept is the same we introduced in \cref{Ch:GlobalVariable} for the global class: after all nodes have updated their estimates of the components they depend on and broadcast them to their neighbors, we say that a CS has occurred. The only difference with respect to the CS concept in \cref{Ch:GlobalVariable} is in the size of the messages exchanged between nodes: here, two neighbors~$(i,j) \in \mathcal{E}$ only exchange the common components their functions depend on, i.e., $x_{S_i \cap S_j}$, rather than the entire vector~$x$. This applies to all the algorithms we compare in this chapter. The only exception is Algorithm~\ref{Alg:GlobalClass}, the algorithm we proposed for the global class, which we show here for comparison purposes. In fact, we will see that, even ignoring the difference in the size of the exchanged messages, Algorithm~\ref{Alg:GlobalClass} takes more CSs to converge than any of the algorithms solving~\eqref{Eq:IntroProb} with a non-global variable. This effectively illustrates how important it is to explore the structure of the problem in order to design communication-efficient algorithms.

	\subsection{Network flow problems}

	We start with the experiments on network flow problems. First, we describe the model we used in our experiments, then the experimental setup and the algorithms we compare, and finally we present our results.

	\mypar{Model}
	Recall that a network flow problem has the format of~\eqref{Eq:NetworkFlow}. Its objective consists of the sum of the costs~$\phi_{ij}(x_{ij})$ associated to all the arcs of the directed network. The constraint~$Bx = d$ enforces the laws of conservation of flow, whereas the constraint~$x\geq 0$ forbids negative flows on each arc. We consider two scenarios for problem~\eqref{Eq:NetworkFlow}:
	\begin{align*}
		\text{Scenario 1:}& \qquad\phi_{ij}(x_{ij}) = \frac{1}{2}(x_{ij}-a_{ij})^2\,,\,\,\, \text{and the constraint~$x \geq 0$ is dropped,} \\
		\text{Scenario 2:}& \qquad\phi_{ij}(x_{ij}) = \frac{x_{ij}}{c_{ij} - x_{ij}} + \text{i}_{x_{ij}\leq c_{ij}}(x_{ij})\,.
	\end{align*}
	In scenario~1, the cost function associated to each arc~$(i,j) \in \mathcal{A}$ is quadratic, $\phi_{ij}(x_{ij}) = \frac{1}{2}(x_{ij}-a_{ij})^2$, where~$a_{ij}$ is positive. Also, we drop the nonnegativity constraint~$x\geq 0$ in order to make the algorithm in~\cite{Zargham12-AcceleratedDualDescentForNetworkFlowOptimization} applicable. Scenario~1 is thus very simple: it solves
	\begin{equation}\label{Eq:NetworkFlowScen1}
		\begin{array}{cl}
			\underset{x = \{x_{ij}\}_{(i,j) \in \mathcal{A}}}{\text{minimize}} & \sum_{(i,j) \in \mathcal{A}} \frac{1}{2}(x_{ij} - a_{ij})^2 \\
			\text{subject to} & Bx = d\,.
		\end{array}
	\end{equation}
	Regarding scenario~2, besides the cost function being more complicated, $\phi_{ij}(x_{ij}) = x_{ij}/(c_{ij} - x_{ij})$, where~$c_{ij}>0$, is the maximum capacity of arc~$(i,j)$, it also has the constraints~$0\leq x_{ij} \leq c_{ij}$, for each arc. That is, scenario~2 solves
	\begin{equation}\label{Eq:NetworkFlowScen2}
		\begin{array}{cl}
			\underset{x = \{x_{ij}\}_{(i,j) \in \mathcal{A}}}{\text{minimize}} & \sum_{(i,j) \in \mathcal{A}} \frac{x_{ij}}{c_{ij} - x_{ij}} \\
			\text{subject to} & Bx = d \\
			                  & 0 \leq x_{ij} \leq c_{ij}\,,
		\end{array}
	\end{equation}
	which can be used to model aggregate system delays in multicommodity flow problems~\cite[Ch.4]{Ahuja93-NetworkFlows}.

	The problem each node has to solve at each iteration, for example, at step~\ref{SubAlg:Conn_Prob} of Algorithm~\ref{Alg:Conn}, has a closed-form solution in scenario~1, but not in scenario~2. In scenario~2, node~$p$ has to solve a problem with the following format:
	\begin{equation}\label{Eq:NetFlowEachNode}
			\begin{array}{cl}
				\underset{y= (y_1,\ldots,y_{D_p})}{\text{minimize}} & \sum_{i=1}^{D_p} (\frac{y_i}{c_i - y_i} + v_i y_i + a_i y_i^2)
				\\
				\text{subject to} & b_p^\top y = d_p \\
				                  & 0 \leq y \leq c\,,
			\end{array}
	\end{equation}
	where each~$y_i$ corresponds to~$x_{pj}$ if $(p,j) \in \mathcal{A}$, or to~$x_{jp}$ if~$(j,p) \in \mathcal{A}$. Since projecting a point onto the set of constraints of~\eqref{Eq:NetFlowEachNode} can be done in closed-form~\cite{Vandenberghe11-TheProximalMapping-lecs}, any projected gradient method is easy to apply. In our implementation, we chose~\cite{Birgin00-NonmonotoneSpectralGradient}, a gradient projection method with a Barzilai-Borwein step.

	\mypar{Experimental setup}
	In both instances of the network flow problem we solve, we use a network with~$P = 2000$ nodes and $E = 3996$ edges, generated randomly in Network X~\cite{Hagberg08-NetworkX} according to the Barabasi-Albert model~\cite{Barabasi99-EmergenceOfScalingInRandomNetworks}; see \tref{Tab:NetworkModels} of \cref{Ch:GlobalVariable} for a brief description. As in the network flow problem illustrated in \fref{Fig:NetworkFlow}, we consider that there is at most one arc between any pair of nodes. As a consequence, the size of the problem variable, $x$, is equal to the number of edges~$E$, in this case~$3996$. The diameter of the generated network was~$8$, it had an average node degree of~$3.996$, and it was colored with~$3$ colors in Sage~\cite{Stein13-Sage}. We then assigned a direction to each edge of this network: for each edge~$(i,j)$, we assigned the directions~$i\xrightarrow{} j$ and~$i\xleftarrow{} j$ with equal probability, thus creating a set of arcs~$\mathcal{A}$ from the set of edges~$\mathcal{E}$. To each edge, we also assigned a number drawn randomly from the set~$\{10, 20, 30, 40, 50, 100\}$. The probabilities were $0.2$ for the first four elements, and~$0.1$ for~$50$ and~$100$. These numbers played the role of the $a_{ij}$'s in scenario~1 and the role of the capacities~$c_{ij}$ in scenario~2. To generate the vector~$d$ or, in other words, to determine which nodes are sources or sinks, we proceeded as follows. For each~$k = 1,\ldots,100$, we picked a source~$s_k$ randomly (uniformly) out of the set of~$2000$ nodes and then picked a sink~$r_k$ randomly (uniformly) out of the set of reachable nodes of~$s_k$. For example, if we were considering the network of \fref{Fig:NetworkFlow} and picked~$s_k = 4$ as a source node, the set of its reachable nodes would be~$\{3, 5, 6, 7\}$. Then, we added to the entries~$s_k$ and~$r_k$ of~$d$ the values $-f_k/100$ and~$f_k/100$, respectively, where~$f_k$ is a number drawn randomly exactly as~$c_{ij}$ (or~$a_{ij}$). This corresponds to injecting a flow of quantity~$f_k/100$ at node~$s_k$ and extracting the same quantity at node~$r_k$. After repeating this process~$100$ times, for~$k = 1,\ldots,K$, we obtained vector~$d$.

	Before executing the distributed algorithms and to assess their error, we computed the solutions of~\eqref{Eq:NetworkFlowScen1}, from scenario~1, and~\eqref{Eq:NetworkFlowScen2}, from scenario~2, in a centralized way. In scenario~1, the solution can be computed in closed-form, because the problem is quadratic with linear constraints. In scenario~2, we used CVXOPT~\cite{CVXOPT} to obtain a solution of~\eqref{Eq:NetworkFlowScen2}.

	\mypar{Algorithms for comparison}
	The network flow problems~\eqref{Eq:NetworkFlowScen1} and~\eqref{Eq:NetworkFlowScen2} are formulated as~\eqref{Eq:IntroProb} with a star-shaped variable (see also~\eqref{Eq:NetworkFlow}). As discussed before, in this case, the ADMM-based algorithm~\cite[\S7.2]{Boyd11-ADMM} becomes distributed. In fact, for network flow problems it becomes exactly algorithm~\cite{Kekatos12-DistributedRobustPowerStateEstimation} (Algorithm~\ref{Alg:Kekatos}); this is not surprising, since both are based on the same underlying algorithm, the $2$-block ADMM. Also, a star-shaped variable makes gradient methods directly applicable. We then also consider Nesterov's fast gradient method~\cite{Nesterov04-IntroductoryLecturesConvexOptimization}, more precisely, the algorithm~\eqref{Eq:RelatedWorkNesterovAlg}. Finally, we consider the distributed Newton method~\cite{Zargham12-AcceleratedDualDescentForNetworkFlowOptimization}, which was designed specifically for network flow problems. All these methods, including ours, have tuning parameters: $\rho$ for the ADMM-based algorithms, a Lipschitz constant~$L$ for Nesterov's algorithm, and a stepsize~$\alpha$ for the distributed Newton algorithm. Note that Nesterov's algorithm requires the objective function to be differentiable and have a Lipschitz-continuous gradient. While this is true for~\eqref{Eq:NetworkFlowScen1}, in scenario~1, it is not true for~\eqref{Eq:NetworkFlowScen2}, in scenario~2. Namely, the gradient of the objective of~\eqref{Eq:NetworkFlowScen2} is not Lipschitz-continuous in all the domain, although it is near the solution. Therefore, in scenario~2, we have to estimate a Lipschitz constant the same way we estimate the parameters of the other algorithms. To do that, we use the concept of \textit{precision}, defined in \cref{Ch:GlobalVariable}: for example, $\bar{\rho}$ \textit{has precision}~$\gamma$ for an ADMM-based algorithm if both~$\bar{\rho} - \gamma$ and~$\bar{\rho} + \gamma$ lead to worse results, i.e., to more CSs. Regarding the number of CSs each of these algorithms takes per iteration, all the ADMM-based ones (Algorithms~\ref{Alg:Kekatos}~\cite{Kekatos12-DistributedRobustPowerStateEstimation}, \ref{Alg:Conn}, and~\cite[\S7.2]{Boyd11-ADMM}) and Nesterov's algorithm~\cite{Nesterov04-IntroductoryLecturesConvexOptimization} take one CS per iteration. Our implementation of the distributed Newton method~\cite{Zargham11-AcceleratedDualDescentNetworkOptimization}, in turn, takes~$3$ CSs per iteration, since we used a fixed stepsize~$\alpha$ and set the parameter~$N$, the order of the approximation of Newton's direction, to~$2$. We will also show the performance of Algorithm~\ref{Alg:GlobalClass}, our proposed algorithm for the global class, in scenario~1. That algorithm makes all the nodes compute the full solution~$x^\star$, which has dimensions~$3996$ in this case. Hence, each message exchanged in one CS of Algorithm~\ref{Alg:GlobalClass} is~$3996$ times larger than the messages exchanged by the other algorithms.

	\begin{figure}
     \centering
     \subfigure[Scenario 1]{\label{SubFig:NF_Scenario1}
     \begin{pspicture}(7.9,5.2)
       \rput[bl](0.59,0.70){\includegraphics[width=7.1cm]{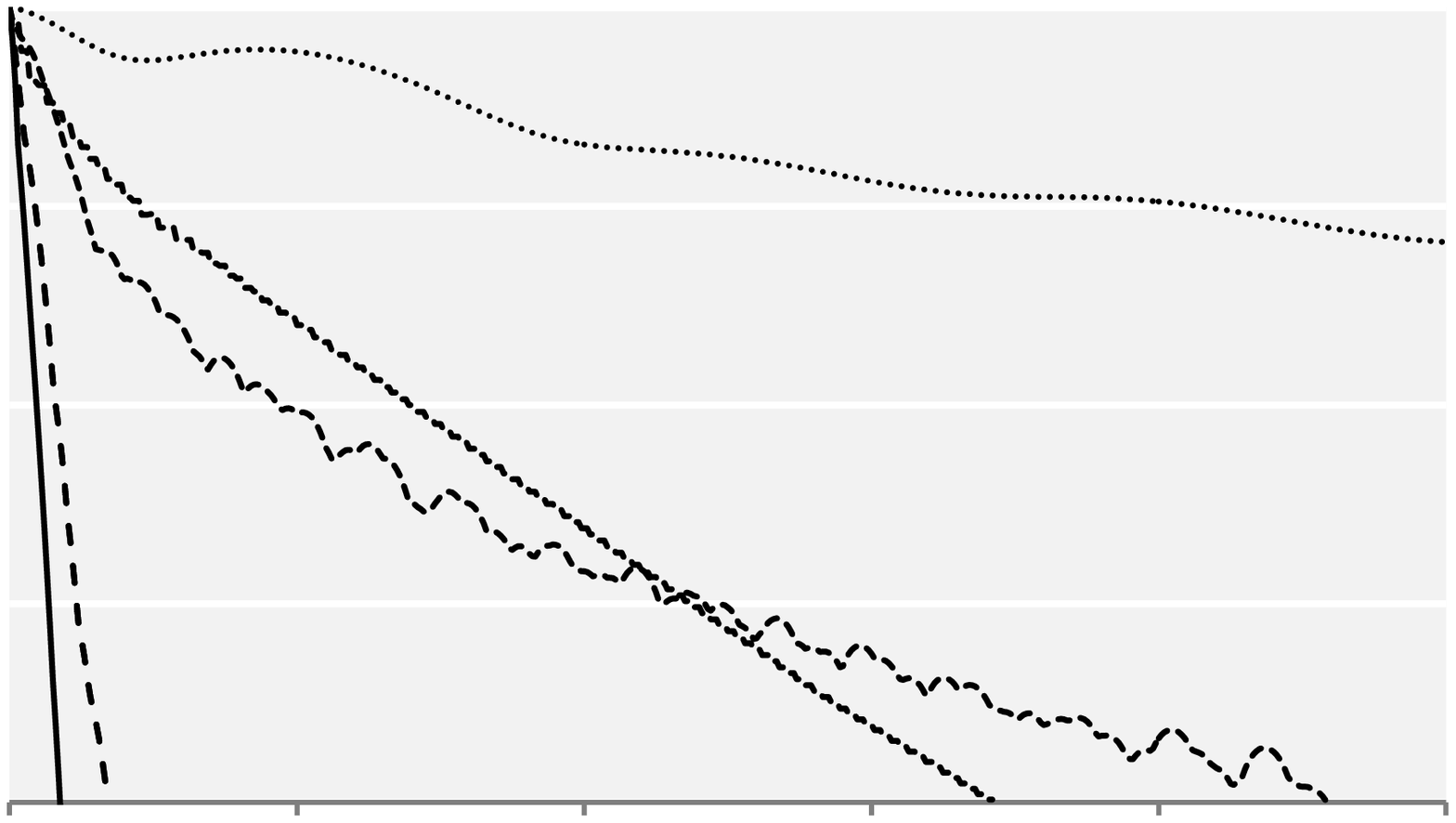}}
       \rput[b](4.00,-0.001){\footnotesize \textbf{\sf Communication steps}}
       \rput[bl](-0.05,5.0){\mbox{\footnotesize \textbf{{\sf Relative error}}}}

       \rput[r](0.56,4.68){\scriptsize $\mathsf{10^{0\phantom{-}}}$}
       \rput[r](0.56,3.74){\scriptsize $\mathsf{10^{-1}}$}
       \rput[r](0.56,2.77){\scriptsize $\mathsf{10^{-2}}$}
       \rput[r](0.56,1.79){\scriptsize $\mathsf{10^{-3}}$}
       \rput[r](0.56,0.83){\scriptsize $\mathsf{10^{-4}}$}

       \rput[t](0.614,0.59){\scriptsize $\mathsf{0}$}
       \rput[t](2.030,0.59){\scriptsize $\mathsf{100}$}
       \rput[t](3.436,0.59){\scriptsize $\mathsf{200}$}
       \rput[t](4.852,0.59){\scriptsize $\mathsf{300}$}
       \rput[t](6.265,0.59){\scriptsize $\mathsf{400}$}
       \rput[t](7.672,0.59){\scriptsize $\mathsf{500}$}

       \rput[lb](1.45,1.14){\scriptsize \textbf{\sf Alg.\ref{Alg:Conn}}}
       \psline[linewidth=0.5pt](1.40,1.22)(0.88,1.1)
       \rput[lb](0.98,1.90){\scriptsize \textbf{\sf \cite{Kekatos12-DistributedRobustPowerStateEstimation,Boyd11-ADMM}}}
       \rput[lb](2.1,3.1){\scriptsize \textbf{\sf \cite{Zargham12-AcceleratedDualDescentForNetworkFlowOptimization}}}
			 \rput[lb](5.45,1.30){\scriptsize \textbf{\sf \cite{Nesterov04-IntroductoryLecturesConvexOptimization}}}
			 \rput[lb](4.0,4.0){\scriptsize \textbf{\sf Alg.\ref{Alg:GlobalClass} (global class)}}

       %\psgrid
     \end{pspicture}
     }
     \hfill
     \subfigure[Scenario 2]{\label{SubFig:NF_Scenario2}
     \begin{pspicture}(7.9,5.2)
       \rput[bl](0.59,0.70){\includegraphics[width=7.1cm]{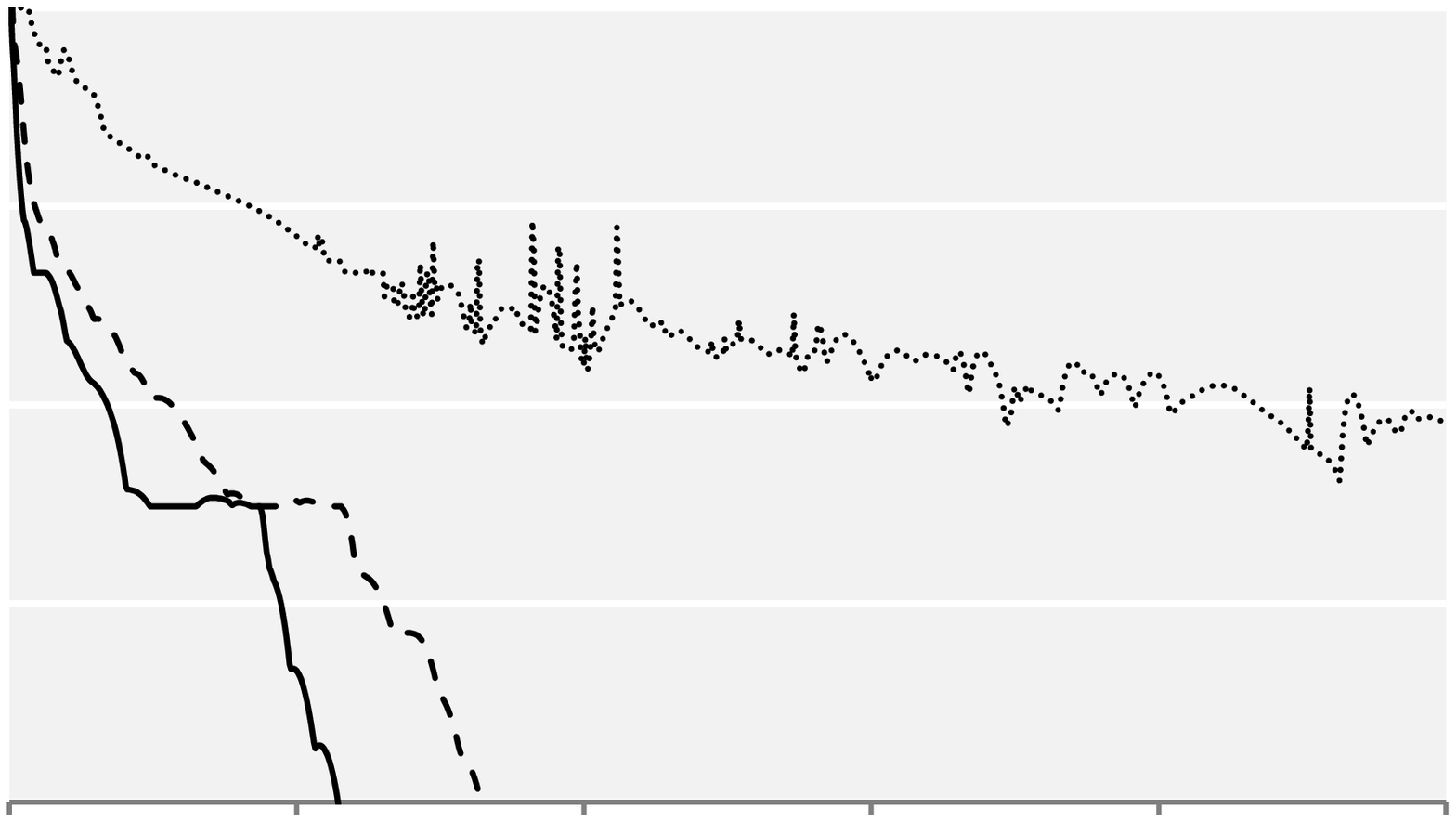}}
       \rput[b](4.00,-0.001){\footnotesize \textbf{\sf Communication steps}}
       \rput[bl](-0.05,5.0){\mbox{\footnotesize \textbf{{\sf Relative error}}}}

       \rput[r](0.56,4.68){\scriptsize $\mathsf{10^{0\phantom{-}}}$}
       \rput[r](0.56,3.74){\scriptsize $\mathsf{10^{-1}}$}
       \rput[r](0.56,2.77){\scriptsize $\mathsf{10^{-2}}$}
       \rput[r](0.56,1.79){\scriptsize $\mathsf{10^{-3}}$}
       \rput[r](0.56,0.83){\scriptsize $\mathsf{10^{-4}}$}

       \rput[t](0.614,0.59){\scriptsize $\mathsf{0}$}
       \rput[t](2.030,0.59){\scriptsize $\mathsf{200}$}
       \rput[t](3.436,0.59){\scriptsize $\mathsf{400}$}
       \rput[t](4.852,0.59){\scriptsize $\mathsf{600}$}
       \rput[t](6.265,0.59){\scriptsize $\mathsf{800}$}
       \rput[t](7.672,0.59){\scriptsize $\mathsf{1000}$}

       \rput[rt](1.9,1.5){\scriptsize \textbf{\sf Alg.\ref{Alg:Conn}}}
       \rput[lb](2.30,2.20){\scriptsize \textbf{\sf \cite{Kekatos12-DistributedRobustPowerStateEstimation,Boyd11-ADMM}}}
       \rput[lb](4.40,3.22){\scriptsize \textbf{\sf \cite{Nesterov04-IntroductoryLecturesConvexOptimization}}}

       %\psgrid
     \end{pspicture}
     }
     \caption[Results of our experiments for the network flow problems.]{
			Results of our experiments for the network flow problems. The problem solved in~\text{(a)} is~\eqref{Eq:NetworkFlowScen1}, a simple quadratic program. The problem solved in~\text{(b)} is~\eqref{Eq:NetworkFlowScen2}, which models aggregate delays in multicommodity flow problems. In both cases, the network has~$P =2000$ nodes and~$E = 3996$ edges, and was generated randomly according to the Barabasi-Albert model.
		}
     \label{Fig:Exp_NF}
	\end{figure}

	\mypar{Results}
	The results of our experiments for scenarios~1 and~2 are shown, respectively, in Figures~\ref{SubFig:NF_Scenario1} and~\ref{SubFig:NF_Scenario2}. These show the relative error on the primal variable $\|x^k - x^\star\|_{\infty}/\|x^\star\|_{\infty}$, where~$x^k$ is the concatenation of the estimates of all nodes, versus the number of CSs. It can be seen in~\fref{SubFig:NF_Scenario1} that Algorithm~\ref{Alg:Conn} in scenario~1 was the one requiring the least amount of CSs to achieve any relative error between~$1$ and~$10^{-4}$. It was closely followed by the ADMM-based algorithms~\cite{Kekatos12-DistributedRobustPowerStateEstimation} and~\cite[\S7.2]{Boyd11-ADMM}, whose lines coincide because they become the same algorithm when applied to network flows. Nesterov's method~\cite{Nesterov04-IntroductoryLecturesConvexOptimization} and the Newton-based method~\cite{Zargham12-AcceleratedDualDescentForNetworkFlowOptimization} had a performance very similar to each other, but worse than the ADMM-based algorithms. In the same plot we can also see that Algorithm~\ref{Alg:GlobalClass}, which solves the global class, had the worst performance; furthermore, each message exchange by that algorithm is~$3996$ times larger than a message exchanged by the other algorithms. This clearly shows that if we want to derive communication-efficient algorithms, we have to explore the structure of~\eqref{Eq:IntroProb}. Regarding the parameters for each algorithm in these experiments, we used~$\rho =2$ for all the ADMM-based algorithms (precision~$1$), a Lipschitz constant~$L = 70$ for~\cite{Nesterov04-IntroductoryLecturesConvexOptimization} (precision~$5$), and a stepsize~$\alpha=0.4$ for~\cite{Zargham12-AcceleratedDualDescentForNetworkFlowOptimization} (precision~$0.1$).

	The results for scenario~2, i.e., for problem~\eqref{Eq:NetworkFlowScen2}, are shown in \fref{SubFig:NF_Scenario2}. We were not able to make the algorithm in~\cite{Zargham12-AcceleratedDualDescentForNetworkFlowOptimization} converge for this scenario (actually, that algorithm is not guaranteed to converge for problem~\eqref{Eq:NetworkFlowScen2}). Overall, scenario~2 looks more challenging to solve, since all algorithms took more CSs to achieve the same relative error. Again, Algorithm~\ref{Alg:Conn} was the algorithm with the best performance. This time we could not find any choice for~$L$ that made Nesterov's algorithm~\cite{Nesterov04-IntroductoryLecturesConvexOptimization} achieve the relative error of~$10^{-4}$ in less than~$1000$ CSs. The best result, obtained for~$L = 15000$, is shown in \fref{SubFig:NF_Scenario2}. The augmented Lagrangian parameter~$\rho$ was~$0.08$ for Algorithm~\ref{Alg:Conn} and~$0.12$ for algorithms~\cite{Kekatos12-DistributedRobustPowerStateEstimation,Boyd11-ADMM}, both computed with precision~$0.02$.

	\subsection{D-MPC}

	We now describe our experiments for distributed model predictive control (D-MPC). Recall that D-MPC can have a variable of any type, either connected or non-connected. We start by describing the particular MPC model we used, and then the experimental setup.

	\mypar{Model}
	For convenience, we reproduce here our D-MPC model~\eqref{Eq:D-MPC}, which was proposed earlier in \ssref{SubSec:DMPC}:
  \begin{equation}\label{Eq:DMPCReprod}
	  \begin{array}{ll}
	  	\underset{\bar{x},\bar{u}}{\text{minimize}} &
				\sum_{p=1}^P\biggl[\Phi_p(\{x_j[T]\}_{j \in \Omega_p})
				+ \sum_{t=0}^{T-1} \Psi_p^t (\{x_j[t], u_j[t]\}_{j \in \Omega_p})\biggr] \\
			\text{subject to} & x_p[t+1] = \Theta_p^t\bigl(\{x_j[t],u_j[t]\}_{j \in \Omega_p}\bigr)\,,\quad t = 0,\ldots,T-1\,,\quad p = 1,\ldots,P \\
			            & x_p[0] = x_p^0\,, \quad  p = 1,\ldots,P\,.
	  \end{array}
	\end{equation}
	Problem~\eqref{Eq:DMPCReprod} is associated to a network with~$P$ dynamic systems where each dynamic system is viewed as a node of that network. The $p$th system is described at each time instant~$t$ by the state vector~$x_p[t] \in \mathbb{R}^{n_p}$ and has a control input~$u_p[t] \in \mathbb{R}^{m_p}$. The D-MPC model~\eqref{Eq:DMPCReprod} generalizes prior D-MPC models in the sense that it allows the state of any system be influenced by the state or input of any other system in the network, and not only by its neighbors; see also \fref{Fig:D-MPC} for a visual comparison between these two scenarios. Therefore, the optimization variable in~\eqref{Eq:DMPCReprod} is arbitrary and not necessarily star-shaped. In our experiments, we consider a simple instance of~\eqref{Eq:DMPCReprod} that preserves this feature. Namely, we assume linear coupling through the inputs, i.e., $x_p[t+1] = A_p x_p[t] + \sum_{j \in \Omega_p} B_{pj} u_j[t]$, where~$A_p \in \mathbb{R}^{n_p \times n_p}$ and each~$B_{pj} \in \mathbb{R}^{n_p \times m_j}$ are arbitrary matrices (in fact, randomly generated), known only at node~$p$. The set~$\Omega_p \subseteq \mathcal{V}$ is the set of nodes whose control input influences the state of node~$p$, $x_p$. We assume that the control input at node~$p$ influences always its own state, i.e., $\{p\} \subset \Omega_p$, for all~$p \in \mathcal{V}$. We also assume there is no coupling through the objective functions. In particular, we consider~$\Phi_p(\{x_j[T]\}_{j \in \Omega_p}) = x_p[T]^\top \bar{Q}_p^f x_p[T]$ and~$\Psi_p^t(\{x_j[t]\}_{j \in \Omega_p}) = x_p[t]^\top \bar{Q}_p x_p[t] + u_p[t]^\top \bar{R}_p$, where~$\bar{Q}_p$ and~$\bar{Q}_p^f$ are positive semidefinite matrices, and~$\bar{R}_p$ is positive definite. With this choice, problem~\eqref{Eq:DMPCReprod} becomes
	\begin{equation}\label{Eq:DistributedMPC_SimpleModel}
  	\begin{array}{cl}
  		\underset{\begin{subarray}{c}x_1,\ldots,x_P\\u_1,\ldots,u_P\end{subarray}}{\text{minimize}} &
  		\sum_{p=1}^P u_p^\top R_p u_p + x_p^\top Q_p x_p \\
  		\text{subject to} & x_p = C_p \{u_j\}_{j \in \mathcal{S}_p} + D_p^0 \,,\,\, p = 1,\ldots,P\,,
  	\end{array}
  \end{equation}
	where, $x_p = (x_p[0],\ldots,x_p[T])$, $u_p = (u_p[0],\ldots,u_p[T-1])$, for each~$p$, and
	\begin{align*}
	  Q_p &= \begin{bmatrix}
		      	I_{T} \otimes \bar{Q}_p & 0\\
		      	0 & \bar{Q}_p^f
		      \end{bmatrix}\,,
		&
		R_p &= I_{T} \otimes \bar{R}_p\,,
		\\
		C_p &=
		\begin{bmatrix}
			0                  & 0                 & \cdots & 0     \\
      B_{p}              & 0                 & \cdots & 0     \\
      A_{pp}B_{p}        & B_{p}             & \cdots & 0     \\
       \vdots            & \vdots            & \ddots & \vdots\\
      A_{pp}^{T-1}B_{p}  & A_{pp}^{T-2}B_{p} & \cdots & B_{p} \\
		\end{bmatrix}\,,
		&D_p^0
		&=
		\begin{bmatrix}
      I           \\
      A_{pp}      \\
      A_{pp}^2    \\
      \vdots      \\
      A_{pp}^{T}  \\
    \end{bmatrix}
    x_p^0\,.
  \end{align*}
  We defined the matrix~$B_p$ (in the entries of~$C_p$) as the horizontal concatenation of the matrices $B_{pj}$, for all~$j \in \Omega_p$. Note that the variables~$x_p$ and~$u_p$ in~\eqref{Eq:DistributedMPC_SimpleModel} now contain the states and inputs for the entire horizon. For this reason, we changed from the notation~$j \in \Omega_p$ to the notation~$j\in S_p$; while~$\Omega_p$ is a subset of the set of nodes~$\mathcal{V}$, $S_p$ is a subset of components of the optimization variable, i.e., $S_p \in \{1,\ldots,(T+1)\sum_{p=1}^P n_p + T\sum_{p=1}^Pm_p\}$. One reason we chose this simple linear model is that all the state variables~$x_p$ in~\eqref{Eq:DistributedMPC_SimpleModel} can be eliminated; indeed, \eqref{Eq:DistributedMPC_SimpleModel} can be written equivalently as
	\begin{equation}\label{Eq:DistributedMPC_SimpleModelFinal}
  	\underset{u_1,\ldots,u_P}{\text{minimize}} \,\,\, \sum_{p=1}^P \{u_j\}_{j \in S_p}^\top E_p \{u_j\}_{j \in S_p} + w_p^\top \{u_j\}_{j \in S_p}\,,
  \end{equation}
  where~$w_p = 2C_p^\top Q_p D_p^0$ and each~$E_p$ is obtained by summing~$R_p$ with $C_p^\top Q_p C_p$ in the correct entries. Note that~\eqref{Eq:DistributedMPC_SimpleModelFinal} is an unconstrained quadratic program. Therefore, in a centralized scenario, where all matrices~$E_p$ and all vectors~$w_p$ are known at the same location, the solution of~\eqref{Eq:DistributedMPC_SimpleModelFinal} is simply the solution of a linear system. For the same reason, the solution of the problem each node has to solve at each iteration, for example in step~\ref{SubAlg:Conn_Prob} of Algorithm~\ref{Alg:Conn}, can be found by solving a linear system.

	\begin{table}[h]
    \centering
    \caption{
             Networks used in the D-MPC experiments.
            }
    \label{Tab:NetworksDMPC}
    \smallskip
    \footnotesize
        %\resizebox{0.7\linewidth}{!}{
        \begin{tabular}{@{}cccccccp{4.5cm}@{}}
          \toprule[1pt]
          Name & Source & \# Nodes & \# Edges & Diam. & \# Colors & Av. Deg. & Description\\
          \midrule
           A & \cite{Barabasi99-EmergenceOfScalingInRandomNetworks}         & $\phantom{4}100$  & $\phantom{6}196$  & $\phantom{4}6$  & $3$ & $3.92$ & Barabasi-Albert (parameter~$2$) \vspace{0.2cm}\\
           B & \cite{Watts98-CollectiveDynamicsSmallWorldNetworks}  & $4941$ & $6594$ & $46$ & $6$ & $2.67$ & US Western states power grid\\
          \bottomrule[1pt]
        \end{tabular}
        %}
  \end{table}

	\mypar{Experimental setup}
	We solved problem~\eqref{Eq:DistributedMPC_SimpleModelFinal} in the two networks of \tref{Tab:NetworksDMPC}. Network~A has~$100$ nodes, $196$ edges, and was generated randomly according to the Barabasi-Albert model~\cite{Barabasi99-EmergenceOfScalingInRandomNetworks}, as briefly described in \tref{Tab:NetworkModels} of \cref{Ch:GlobalVariable}. A parameter of~$2$ means that every time a node is added to the network it connects to other~$2$ nodes. Network~B is considerably larger, having~$4941$ nodes and~$6594$ edges, and it represents the topology of the power grid of the US Western states~\cite{Watts98-CollectiveDynamicsSmallWorldNetworks}. \tref{Tab:NetworksDMPC} also shows the diameter of each network, the average degree of each node, and the number of colors they are colored with. To color these networks, we used a built-in function in Sage~\cite{Stein13-Sage}.

	In all our experiments we considered a time horizon~$T$ of dimension~$5$, the state~$x_p$ of each node~$p$ always had dimensions~$n_p = 3$, and the control input~$u_p$ was always scalar, $m_p = 1$, for all~$p$. Since the size of the variable in~\eqref{Eq:DistributedMPC_SimpleModelFinal} is~$m_p T P$, network~A implied a variable of size~$500$ and network~B implied a variable of size~$24705$. While each dynamical system in network~A could be unstable, each dynamical system in network~B was always guaranteed stable. More specifically, for both networks, we generated the entries of the dynamics matrix~$A_p$ of each system~$p$ from the normal distribution (independently); however, for network~B, after generating each~$A_p$, we always ``shrunk'' its eigenvalues to the interval $[-1,1]$, making the corresponding system stable. Regarding the input-state matrices~$B_{pj}$, each of its entries were also drawn from the normal distribution.

	We now describe how we generated the system couplings, i.e., the sets~$\Omega_p \in \mathcal{V}$; see also the dotted arrows in \fref{Fig:D-MPC}. We generated three types of couplings, and thus of variables. We generated star-shaped variables, where the state of system~$p$ is influenced by the inputs of all its neighbors, that is, $\Omega_p = \mathcal{N}_p$, for all~$p$. This case is illustrated in \fref{SubFig:PPMPCStar} and was considered so that we could compare Algorithms~\ref{Alg:Kekatos} and~\ref{Alg:Conn} with other prior D-MPC algorithms. We also generated instances of the system couplings to make the variable connected (not necessarily star-shaped), and non-connected. To generate a connected variable we proceeded as follows: given a node~$p$, we make it depend on~$u_p$ (recall our assumption that $\{p\} \subset \Omega_p$). Then, we initialize a set~$\mathcal{F}_p$, which we will call the ``fringe,'' with the neighbors of node~$p$, i.e., $\mathcal{F}_p = \mathcal{N}_p$. Next, we select randomly (uniformly) a node~$q$ from the fringe, $q \in \mathcal{F}_p$, and make its state depend on~$u_p$, i.e., $p \in \Omega_q$. Then, we add its set of neighbors to the fringe and remove node~$q$ from it, since it already depends on~$u_p$: $\mathcal{F}_p = \Bigl(\mathcal{F}_p \backslash \{q\}\Bigr) \cup \mathcal{N}_q$. This process is repeated~$3$ times for each node~$p$, and is done for all the nodes in the network. To generate a non-connected variable, the process is exactly the same, including the concept of fringe. The difference is that, at each iteration, any node in the entire network can be selected, not just the nodes in the fringe; however, the nodes in the fringe have twice the probability of being selected with respect to the remaining nodes in the network. We generated a non-connected variable only for network~A, running the described algorithm for each one of its~$500$ components (the size of the variable for this network is $m_p T P = 1\times 5 \times 100 = 500$). As a result, we obtained~$400$ components for which the respective induced subgraphs were non-connected. According to the preprocessing step of Algorithm~\ref{Alg:NonConn}, we have to compute a Steiner tree for each of these~$400$ components. To do that, we used a built-in function in Sage~\cite{Stein13-Sage}. We ended up with~$44$ nodes in the network (out of~$100$) that were Steiner nodes for at least one component.

	\begin{figure}[t]
     \centering
     \subfigure[Network A with a star-shaped variable]{\label{SubFig:MPC_Barab100_Unstable}
     \begin{pspicture}(7.9,5.2)
       \rput[bl](0.59,0.70){\includegraphics[width=7.1cm]{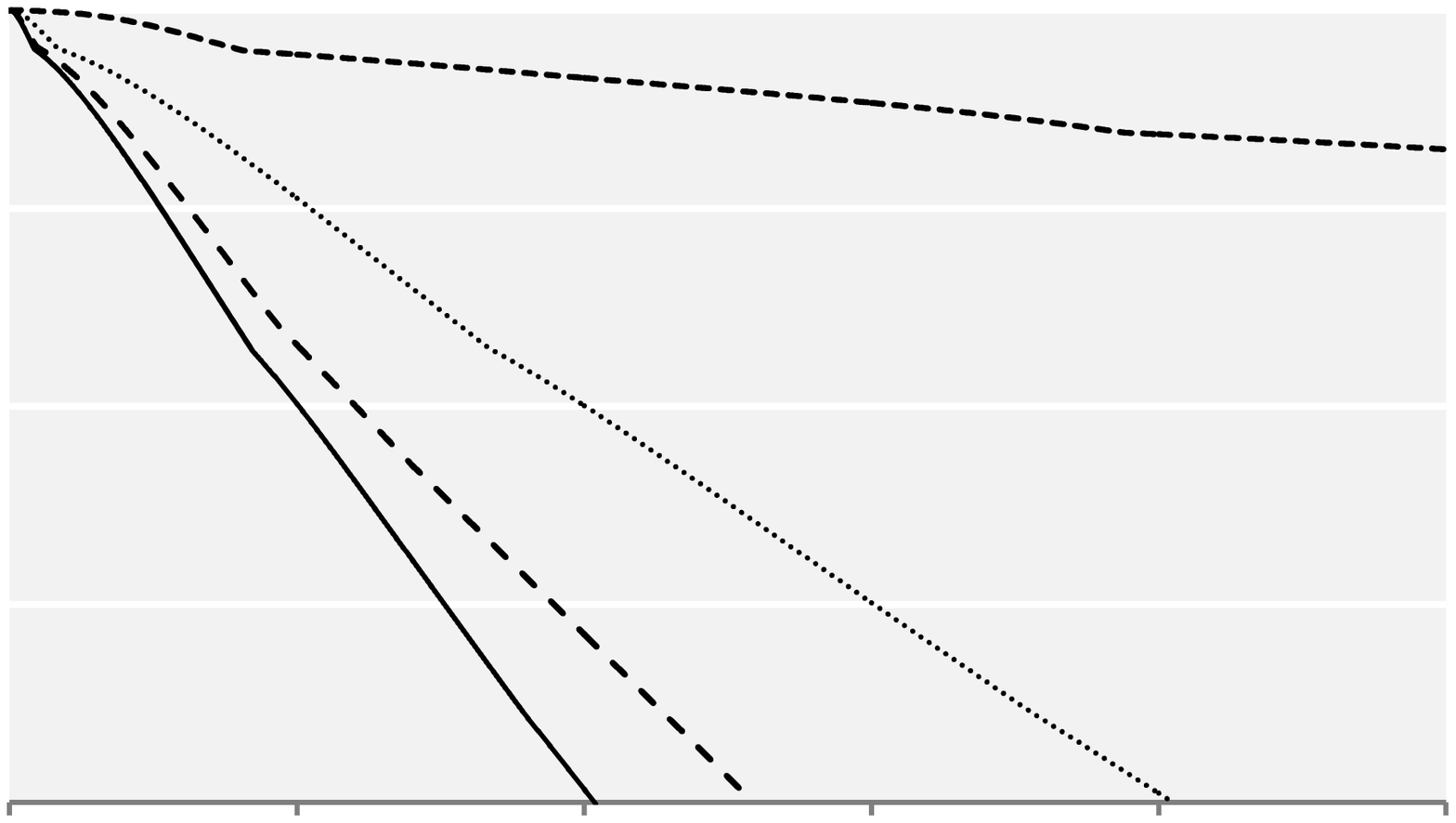}}
       \rput[b](4.00,-0.001){\footnotesize \textbf{\sf Communication steps}}
       \rput[bl](-0.05,5.0){\mbox{\footnotesize \textbf{{\sf Relative error}}}}

       \rput[r](0.56,4.68){\scriptsize $\mathsf{10^{0\phantom{-}}}$}
       \rput[r](0.56,3.74){\scriptsize $\mathsf{10^{-1}}$}
       \rput[r](0.56,2.77){\scriptsize $\mathsf{10^{-2}}$}
       \rput[r](0.56,1.79){\scriptsize $\mathsf{10^{-3}}$}
       \rput[r](0.56,0.83){\scriptsize $\mathsf{10^{-4}}$}

       \rput[t](0.614,0.59){\scriptsize $\mathsf{0}$}
       \rput[t](2.030,0.59){\scriptsize $\mathsf{200}$}
       \rput[t](3.436,0.59){\scriptsize $\mathsf{400}$}
       \rput[t](4.852,0.59){\scriptsize $\mathsf{600}$}
       \rput[t](6.265,0.59){\scriptsize $\mathsf{800}$}
       \rput[t](7.672,0.59){\scriptsize $\mathsf{1000}$}

			 \rput[rt](2.92,1.40){\scriptsize \textbf{\sf Alg.\ref{Alg:Conn}}}
       \rput[bl](3.21,1.88){\scriptsize \textbf{\sf \cite{Boyd11-ADMM}}}
       \rput[bl](4.20,2.25){\scriptsize \textbf{\sf \cite{Kekatos12-DistributedRobustPowerStateEstimation}}}
       \rput[lb](6.20,4.14){\scriptsize \textbf{\sf \cite{Nesterov04-IntroductoryLecturesConvexOptimization}}}

       %\psgrid
     \end{pspicture}
     }
     \hfill
     \subfigure[Network B with a star-shaped variable]{\label{SubFig:MPC_PowerGridStable}
     \begin{pspicture}(7.9,5.2)
       \rput[bl](0.59,0.70){\includegraphics[width=7.1cm]{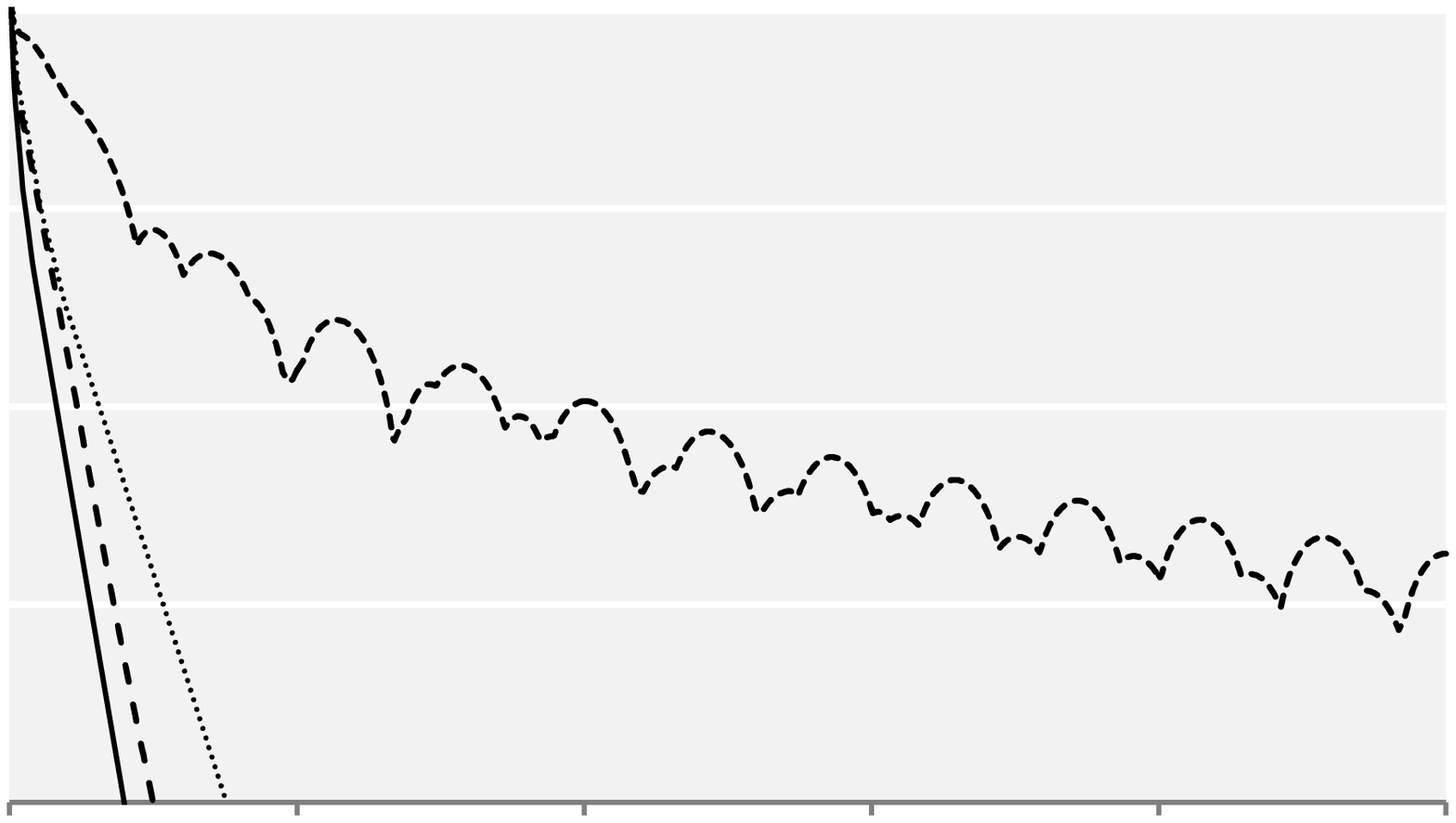}}
       \rput[b](4.00,-0.001){\footnotesize \textbf{\sf Communication steps}}
       \rput[bl](-0.05,5.0){\mbox{\footnotesize \textbf{{\sf Relative error}}}}

       \rput[r](0.56,4.68){\scriptsize $\mathsf{10^{0\phantom{-}}}$}
       \rput[r](0.56,3.74){\scriptsize $\mathsf{10^{-1}}$}
       \rput[r](0.56,2.77){\scriptsize $\mathsf{10^{-2}}$}
       \rput[r](0.56,1.79){\scriptsize $\mathsf{10^{-3}}$}
       \rput[r](0.56,0.83){\scriptsize $\mathsf{10^{-4}}$}

       \rput[t](0.614,0.59){\scriptsize $\mathsf{0}$}
       \rput[t](2.030,0.59){\scriptsize $\mathsf{200}$}
       \rput[t](3.436,0.59){\scriptsize $\mathsf{400}$}
       \rput[t](4.852,0.59){\scriptsize $\mathsf{600}$}
       \rput[t](6.265,0.59){\scriptsize $\mathsf{800}$}
       \rput[t](7.672,0.59){\scriptsize $\mathsf{1000}$}

       \rput[bl](1.74,1.32){\scriptsize \textbf{\sf Alg.\ref{Alg:Conn}}}
       \psline[linewidth=0.5pt](1.70,1.40)(1.13,1.20)
       \rput[bl](1.74,1.92){\scriptsize \textbf{\sf \cite{Boyd11-ADMM}}}
       \psline[linewidth=0.5pt](1.70,2.0)(1.15,1.8)
       \rput[lb](1.21,2.38){\scriptsize \textbf{\sf \cite{Kekatos12-DistributedRobustPowerStateEstimation}}}
       \rput[lb](2.0,3.19){\scriptsize \textbf{\sf \cite{Nesterov04-IntroductoryLecturesConvexOptimization}}}

       %\psgrid
     \end{pspicture}
     }

     \hfill
     \subfigure[Network A with a generic connected variable]{\label{SubFig:MPC_Barab100NonStarUnstable}
     \begin{pspicture}(7.9,5.2)
       \rput[bl](0.59,0.70){\includegraphics[width=7.1cm]{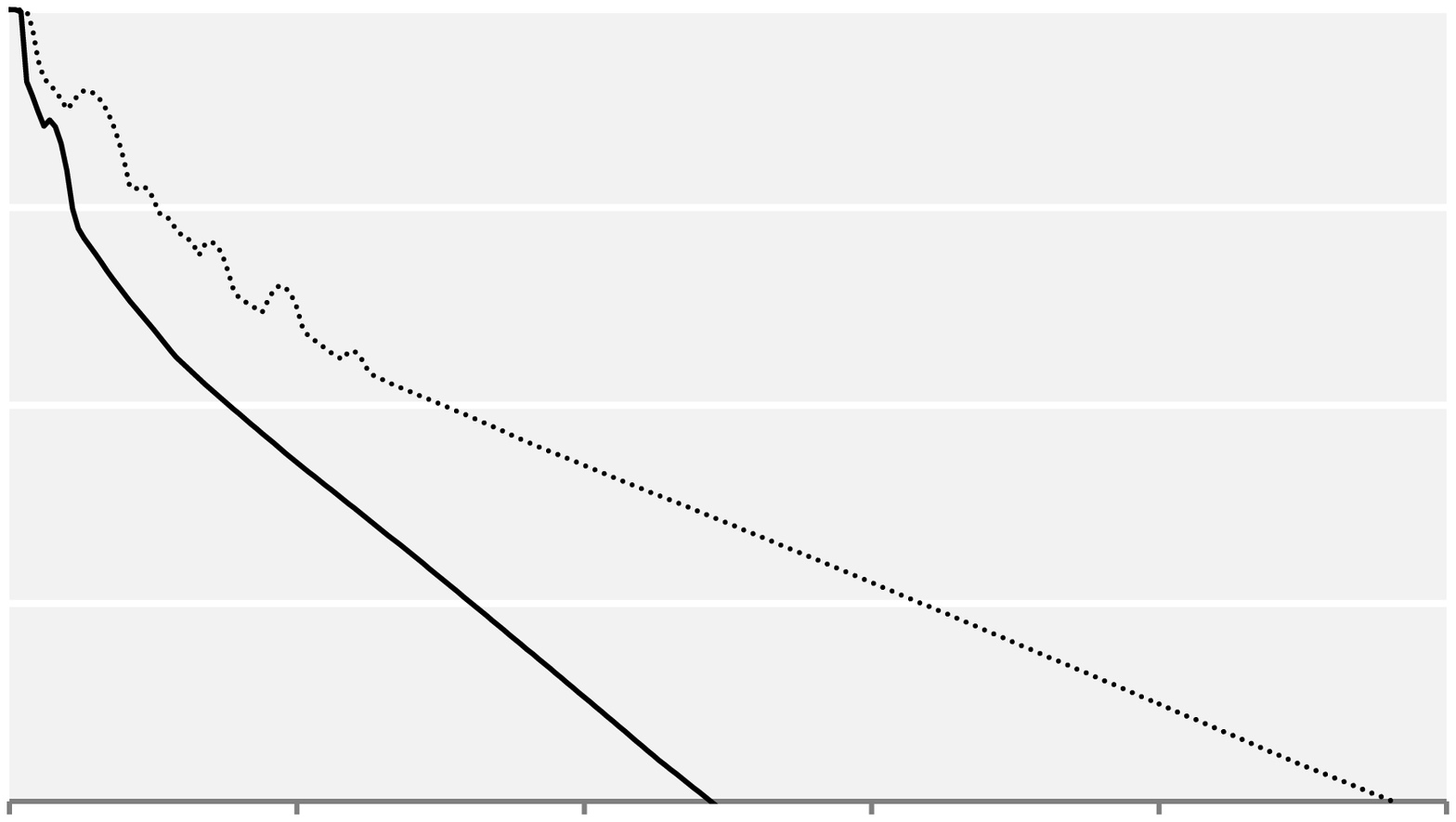}}
       \rput[b](4.00,-0.001){\footnotesize \textbf{\sf Communication steps}}
       \rput[bl](-0.05,5.0){\mbox{\footnotesize \textbf{{\sf Relative error}}}}

       \rput[r](0.56,4.68){\scriptsize $\mathsf{10^{0\phantom{-}}}$}
       \rput[r](0.56,3.74){\scriptsize $\mathsf{10^{-1}}$}
       \rput[r](0.56,2.77){\scriptsize $\mathsf{10^{-2}}$}
       \rput[r](0.56,1.79){\scriptsize $\mathsf{10^{-3}}$}
       \rput[r](0.56,0.83){\scriptsize $\mathsf{10^{-4}}$}

       \rput[t](0.614,0.59){\scriptsize $\mathsf{0}$}
       \rput[t](2.030,0.59){\scriptsize $\mathsf{50}$}
       \rput[t](3.436,0.59){\scriptsize $\mathsf{100}$}
       \rput[t](4.852,0.59){\scriptsize $\mathsf{150}$}
       \rput[t](6.265,0.59){\scriptsize $\mathsf{200}$}
	     \rput[t](7.672,0.59){\scriptsize $\mathsf{250}$}

			 \rput[rt](2.21,2.1){\scriptsize \textbf{\sf Alg.\ref{Alg:Conn}}}
       \rput[bl](4.1,2.2){\scriptsize \textbf{\sf \cite{Kekatos12-DistributedRobustPowerStateEstimation}}}

       %\psgrid
     \end{pspicture}
     }
     \hfill
     \subfigure[Network B with a generic connected variable]{\label{SubFig:MPC_PowerGridStableNonStar}
     \begin{pspicture}(7.9,5.2)
       \rput[bl](0.59,0.70){\includegraphics[width=7.1cm]{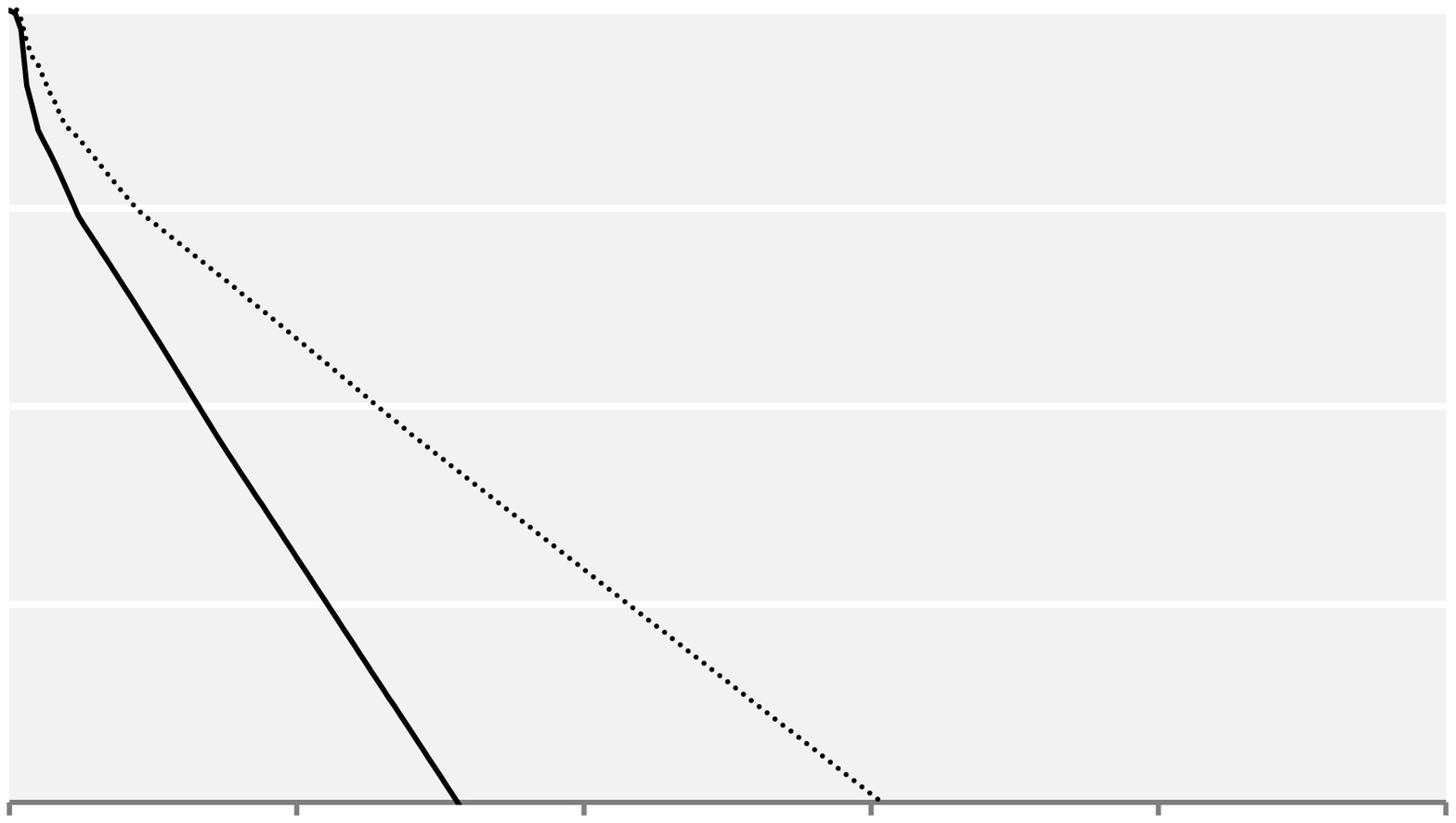}}
       \rput[b](4.00,-0.001){\footnotesize \textbf{\sf Communication steps}}
       \rput[bl](-0.05,5.0){\mbox{\footnotesize \textbf{{\sf Relative error}}}}

       \rput[r](0.56,4.68){\scriptsize $\mathsf{10^{0\phantom{-}}}$}
       \rput[r](0.56,3.74){\scriptsize $\mathsf{10^{-1}}$}
       \rput[r](0.56,2.77){\scriptsize $\mathsf{10^{-2}}$}
       \rput[r](0.56,1.79){\scriptsize $\mathsf{10^{-3}}$}
       \rput[r](0.56,0.83){\scriptsize $\mathsf{10^{-4}}$}

       \rput[t](0.614,0.59){\scriptsize $\mathsf{0}$}
       \rput[t](2.030,0.59){\scriptsize $\mathsf{50}$}
       \rput[t](3.436,0.59){\scriptsize $\mathsf{100}$}
       \rput[t](4.852,0.59){\scriptsize $\mathsf{150}$}
       \rput[t](6.265,0.59){\scriptsize $\mathsf{200}$}
			 \rput[t](7.672,0.59){\scriptsize $\mathsf{250}$}

			 \rput[rt](1.78,2.20){\scriptsize \textbf{\sf Alg.\ref{Alg:Conn}}}
       \rput[bl](3.2,2.20){\scriptsize \textbf{\sf \cite{Kekatos12-DistributedRobustPowerStateEstimation}}}

       %\psgrid
     \end{pspicture}
     }

     \caption[Results for D-MPC with a connected variable.]{
				Results for D-MPC with a connected variable. On the left, \text{(a)} and~\text{(c)} show the results for network~A, and, on the right, \text{(b)} and~\text{(d)} show the results for network~B. The optimization variable is star-shaped on the top plots, \text{(a)} and~\text{(b)}, and is non-star-shaped (and non-global) on the bottom plots, \text{(c)} and \text{(d)}.
     }
     \label{Fig:Exp_MPC}
	\end{figure}

	\mypar{Results}
	The results of our experiments are shown in \fref{Fig:Exp_MPC} for connected variables, and in \fref{Fig:Exp_MPC_NonConnected} for a non-connected variable. Each plot shows how the relative error as a function of the number of CSs. The relative error is measured the same way as in the network flow experiments: $\|x^k - x^\star\|_{\infty}/\|x^\star\|_{\infty}$, where~$x^k$ is the concatenation of all the nodes' control input estimates. The results for networks~A and~B, both with a star-shaped variable, are shown in Figures~\ref{SubFig:MPC_Barab100_Unstable} and~\ref{SubFig:MPC_PowerGridStable}, respectively. The relative behavior of all the compared algorithms is the same: the proposed Algorithm~\ref{Alg:Conn} required uniformly less CSs to achive any relative error between~$1$ and~$10^{-4}$; it was followed by the ADMM-based algorithms~\cite[\S7.2]{Boyd11-ADMM} and~\cite{Kekatos12-DistributedRobustPowerStateEstimation} (shown as Algorithm~\ref{Alg:Kekatos}), with~\cite[\S7.2]{Boyd11-ADMM} being more efficient than~\cite{Kekatos12-DistributedRobustPowerStateEstimation}. Finally, Nesterov's algorithm~\cite{Nesterov04-IntroductoryLecturesConvexOptimization} failed to converge in both cases. A curious fact is that all algorithms required more CSs to converge in the network of \fref{SubFig:MPC_Barab100_Unstable}, which has~$100$ nodes, than in the network of \fref{SubFig:MPC_PowerGridStable}, which is considerably larger, with nearly~$5000$ nodes. In fact, what influenced the performance of all the algorithms was the stability of the systems: while each system in \fref{SubFig:MPC_PowerGridStable} was guaranteed to be stable, no system in \fref{SubFig:MPC_Barab100_Unstable} was guaranteed to be stable. The difficulty of each problem instance can be measured by the magnitude of the Lipschitz constant of the gradient of the objective function of~\eqref{Eq:DistributedMPC_SimpleModelFinal}: $1.63\times 10^6$ for \fref{SubFig:MPC_Barab100_Unstable} and~$3395$ for \fref{SubFig:MPC_PowerGridStable}. Note that this Lipschitz constant can be computed in closed-form. Regarding the augmented Lagrangian parameter~$\rho$, its was computed, with precision~$5$, for \fref{SubFig:MPC_Barab100_Unstable} as~$120$ for~\cite[\S7.2]{Boyd11-ADMM} and as~$135$ for the other algorithms. For \fref{SubFig:MPC_PowerGridStable}, it was computed as~$25$ for Algorithm~\ref{Alg:Conn} and~\cite[\S7.2]{Boyd11-ADMM} and as~$30$ for~\cite{Kekatos12-DistributedRobustPowerStateEstimation}, also with precision~$5$.

	Figures~\ref{SubFig:MPC_Barab100NonStarUnstable} and~\ref{SubFig:MPC_PowerGridStableNonStar} show the results for generic, non-star-shaped variables for networks~A and~B, respectively. Since the ADMM-based algorithm~\cite[\S7.2]{Boyd11-ADMM} and Nesterov's algorithm~\cite{Nesterov04-IntroductoryLecturesConvexOptimization} are distributed only for star-shaped variables, they do not appear in these plots. Only the proposed Algorithm~\ref{Alg:Conn} and the algorithm in~\cite{Kekatos12-DistributedRobustPowerStateEstimation} (see Algorithm~\ref{Alg:Kekatos}) can handle generic connected variables. In both plots, Algorithm~\ref{Alg:Conn} required uniformly less CSs than~\cite{Kekatos12-DistributedRobustPowerStateEstimation} to achieve any relative error between~$1$ and~$10^{-4}$. Again, both algorithms required more CSs to converge in the smaller network~A than in the larger network~B. The reason, as we saw for the other plots, is because each system in network~A can be unstable, while all systems in network~B are stable. The value of~$\rho$ was the same for both algorithms: $40$ for network~A in \fref{SubFig:MPC_Barab100NonStarUnstable} (precision~$5$), and~$23$ for network~B (precision~$1$).

	\begin{figure}
		\centering
		\begin{pspicture}(7.9,5.2)
       \rput[bl](0.59,0.70){\includegraphics[width=7.1cm]{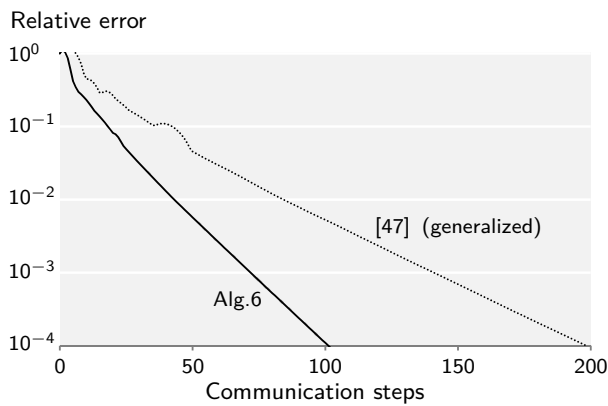}}
       \rput[b](4.00,-0.001){\footnotesize \textbf{\sf Communication steps}}
       \rput[bl](-0.05,5.0){\mbox{\footnotesize \textbf{{\sf Relative error}}}}

       \rput[r](0.56,4.68){\scriptsize $\mathsf{10^{0\phantom{-}}}$}
       \rput[r](0.56,3.74){\scriptsize $\mathsf{10^{-1}}$}
       \rput[r](0.56,2.77){\scriptsize $\mathsf{10^{-2}}$}
       \rput[r](0.56,1.79){\scriptsize $\mathsf{10^{-3}}$}
       \rput[r](0.56,0.83){\scriptsize $\mathsf{10^{-4}}$}

       \rput[t](0.614,0.59){\scriptsize $\mathsf{0}$}
       \rput[t](2.382,0.59){\scriptsize $\mathsf{50}$}
       \rput[t](4.140,0.59){\scriptsize $\mathsf{100}$}
       \rput[t](5.910,0.59){\scriptsize $\mathsf{150}$}
       \rput[t](7.675,0.59){\scriptsize $\mathsf{200}$}

       \rput[rt](3.3,1.5){\scriptsize \textbf{\sf Alg.\ref{Alg:NonConn}}}
       \rput[lb](4.8,2.2){\scriptsize \textbf{\sf \cite{Kekatos12-DistributedRobustPowerStateEstimation} \,(generalized)}}

       %\psgrid
     \end{pspicture}

     \caption[Results for D-MPC with a non-connected variable.]{
				Results for D-MPC with a non-connected variable. All the dynamic systems were designed stable in this case, and the network was~A.}
     \label{Fig:Exp_MPC_NonConnected}
	\end{figure}

	Finally, we present the results for a non-connected variable in \fref{Fig:Exp_MPC_NonConnected}. Neither Algorithm~\ref{Alg:Conn} nor the algorithm in~\cite{Kekatos12-DistributedRobustPowerStateEstimation} are applicable in this case. However, they can be adapted to non-connected variables, as described in \ssref{SubSec:DerivationNonConnected}. The generalization of Algorithm~\ref{Alg:Conn} yields Algorithm~\ref{Alg:NonConn}, and the exact same generalization can be applied to the algorithm in~\cite{Kekatos12-DistributedRobustPowerStateEstimation}. \fref{Fig:Exp_MPC_NonConnected} shows that the behavior we had seen for the non-generalized versions of the algorithms in the previous experiments translates into the generalized versions: Algorithm~\ref{Alg:NonConn} requires uniformly less CSs than the generalized version of~\cite{Kekatos12-DistributedRobustPowerStateEstimation} to achieve any relative error between~$1$ and~$10^{-4}$. Note that, although we used network~A in these experiments, we guaranteed that all the systems were stable.

%% file: 02-MainMatter/conclusions.tex
\chapter{Conclusions and Future Work}
\label{Ch:conclusions}

	We restate our main problem
	\begin{equation}\label{Eq:IntroProb}\tag{P}
		\begin{array}{ll}
			\underset{x \in\mathbb{R}^n}{\text{minimize}} & f_1(x_{S_1}) + f_2(x_{S_2}) + \cdots + f_P(x_{S_P})\,.
		\end{array}
	\end{equation}
	and recall the main goals of this thesis, as presented before in \cref{Ch:Introduction}:
	\begin{center}
		\begin{minipage}{0.92\textwidth}
			\singlespace
			\begin{shaded}
				\medskip
				We aim to design, analyze, and implement algorithms that solve optimization problems of the form~\eqref{Eq:IntroProb} on networks. The algorithms should be
				\begin{list}{}{\setlength\itemindent{-0.15in}\setlength\itemsep{0.1in}}
          \item \textbf{Distributed:} no node has complete knowledge about the problem data and no central node is allowed; also, each node communicates only with its neighbors;
          \item \textbf{Communication-efficient:} the number of communications they use is minimized;
					\item \textbf{Network-independent:} the algorithms run on networks with arbitrary topology and their output is independent of the network.
				\end{list}
			\smallskip
			\end{shaded}
		\end{minipage}
	\end{center}

	First, we summarize our contributions to achieve this goal and discuss current limitations; then, we describe potential future work.

	\section{Major contributions}

	We group the contributions of the thesis into the following categories:
	\begin{itemize}
		\item \textbf{Classification scheme.} The optimization problem~\eqref{Eq:IntroProb} is quite generic because each function may depend on an arbitrary subset of components of the optimization variable. This makes the design of a distributed algorithm a challenging task. We solve this problem with a classification scheme that allowed us to first identify particular instances of~\eqref{Eq:IntroProb} that are easier to solve in a distributed way. After that, we generalized the algorithms to solve larger classes and eventually all problems of the form~\eqref{Eq:IntroProb}. Besides helping us develop our algorithms, our classification scheme is also useful to categorize applications and to organize prior work on distributed optimization.

		\item \textbf{Algorithms.} Based on the proposed classification scheme, we developed a set of algorithms that solve subclasses of distributed optimization problems of the form~\eqref{Eq:IntroProb}. Each algorithm was built from a previous one, by modifying it to increase generality. Our most general algorithm solves~\eqref{Eq:IntroProb} in full generality.	Our algorithms satisfy all the requirements we had set forth: they are distributed, network-independent and,  most significantly, they are communication-efficient. Under certain conditions, they are proven to converge to the same solution as a centralized algorithm and, as shown through several experiments, they usually outperform prior distributed optimization algorithms; namely, they use systematically less communications to achieve a prescribed solution accuracy. A surprising fact is that, despite their generality, they sometimes even outperform distributed algorithms that were designed for specific applications.

		\item \textbf{Applications.} We applied our algorithms to several known distributed problems, and also proposed new applications for them, such as several instances of compressed sensing (or sparse approximation) problems. Namely, we solve the three most important optimization problems in compressed sensing in both the cases where the sensing matrix is partitioned vertically (by rows) and horizontally (by columns). We also propose a new, more general framework for distributed model predictive control (D-MPC). This framework models scenarios where, for example, two dynamical systems that are coupled through their dynamics do not communicate directly. Thus, it is useful in scenarios where establishing communications between systems is expensive.

		\item \textbf{Implementation and benchmarking.} Since there are no tight lower bounds on how many communications are needed to solve~\eqref{Eq:IntroProb} in a distributed setting, the performance assessment of our algorithms had to be done by comparing them to other prior distributed algorithms. This involved implementing both our algorithms and the algorithms for which no implementation was publicly available. We performed several experiments on different types of networks and for different applications where all the algorithms were compared. The size of both the data and the networks varied considerably. For example, the smallest network had only~$10$ nodes, while the largest one had around~$5000$ nodes. As mentioned before, these experiments enabled us to confirm the communication-efficiency of our algorithms.
	\end{itemize}

	\section{Current limitations}

	Despite the excellent communication-efficiency of our algorithms, they still have several limitations:

	\begin{itemize}
		\item \textbf{Selection of \boldmath{$\rho$}.} The algorithms we proposed are based on an augmented Lagrangian method called multi-block alternating direction method of multipliers (ADMM). Augmented Lagrangian methods are generally parametrized by a scalar parameter, which we denote with~$\rho$, and their performance is strongly dependent on that parameter. Currently, there is no known method for selecting~$\rho$ before the execution of the algorithm. And, although there are some heuristics to adapt~$\rho$ while the algorithm is running, implementing those heuristics in distributed algorithms destroys their distributivity, since it requires aggregating information that is spread over the entire network. Therefore, the performance of the algorithms we proposed are conditionally dependent on a good choice for the parameter~$\rho$. While in some situations it is possible to select beforehand a good~$\rho$ using training data, this is still a current limitation.

		\item \textbf{Convergence results.} As mentioned, our algorithms are based on the centralized multi-block ADMM algorithm. There is a proof of the convergence of this algorithm only in the case where all the cost functions are strongly convex. Yet, it has been observed experimentally, including in this thesis, that the multi-block ADMM converges for generic closed convex functions. Proving its convergence for this case is, however, still a well-known open problem. The lack of theoretical results for the multi-block ADMM transfers directly to our algorithms. In particular, we could only prove their convergence for generic closed convex functions when the network is bipartite. When it is not, our algorithms are only (theoretically) guaranteed to converge when the functions associated to each node are strongly convex.

		\item \textbf{Coloring scheme.} All our algorithms use the concept of network coloring and require a coloring scheme to be available before their execution. This coloring scheme is used by our algorithms to synchronize the order of operation of the nodes. In many platforms, most notably, in wireless networks, the nodes already have to operate with such a synchronization scheme in order to avoid packet collisions. In those cases, our algorithms integrate naturally with these low-level protocols. There are, however, some platforms that use other types of protocols or that even all fully parallel communication. In those cases, the coloring scheme required by our algorithms is clearly a limitation.
	\end{itemize}

	\section{Future work}

	We see three main future research directions, as described next:

	\begin{itemize}
		\item \textbf{Algorithm analysis.} We mentioned as a limitation of our algorithms the lack of convergence results. This is closely related to the lack of convergence of the multi-block ADMM, a currently well-known open problem. Therefore, results on this direction would have a significant impact on the distributed algorithms we proposed. Also in this category is the task of developing an heuristic to adapt the augmented Lagrangian parameter~$\rho$ during the execution of the algorithm, and in a distributed way.

		\item \textbf{New distributed algorithms.} Another possible research direction is the development of new distributed optimization algorithms. The current most efficient algorithms are based on ADMM, which can be viewed as an application of a monotone operator splitting method to an optimization problem. Therefore, exploring monotone operator theory and devising new splitting methods may yield new and more efficient distributed optimization algorithms. A topic that became more relevant with the advent of the ``big data'' is privacy. In our view, it would be interesting to study privacy guarantees offered by distributed algorithms in the processing of distributed data.

		\item \textbf{New applications.} Although there are many applications for distributed optimization, including the ones presented in this thesis, the majority of them involve convex problems. Yet, many optimization problems formulated on networks are inherently nonconvex, for example, network coloring or the computation of Steiner trees. An interesting area to explore is the design of distributed approximation schemes for these types of nonconvex problems.

	\end{itemize}

%% file: 03-BackMatter/derivationRelatedWork.tex
\chapter{ADMM-based Algorithms For The Global Class: Derivation}
\label{App:ADMMAlgsGlobalClass}

  In this appendix, we derive Algorithms~\ref{Alg:Schizas} and~\ref{Alg:Zhu}, from \cref{Ch:relatedWork}. Although these algorithms were proposed in~\cite{Schizas08-ConsensusAdHocWSNsPartI} and~\cite{Zhu09-DistributedInNetworkChannelCoding}, respectively, they were derived there for particular instances of the global class~\eqref{Eq:GlobalProblem}. Here, we generalize them to solve the entire class. Before their derivation, we need some identities for quantities defined on the edges of a network.

\section{Network identities}

	Recall that we adopted the convention in~\sref{SubSec:CommunicationNetwork} that if~$(i,j) \in \mathcal{E}$, then $i<j$. The following lemma will be useful for exchanging between ``edge notation'' and ``node notation.''

	\begin{lemma}
	\label{lem:AppNetworkIdentities}
	\hfill

	\medskip
	\noindent
		\newcounter{NetworkLemmaApp}
		\begin{list}{(\alph{NetworkLemmaApp})}{\usecounter{NetworkLemmaApp}}
			\item
				Let~$a_{ij}$ be any quantity associated with the edge~$(i,j) \in \mathcal{E}$. Then,
			\begin{align}
				  \sum_{(i,j) \in \mathcal{E}} a_{ij}
				&=
					\sum_{p=1}^P
					\biggl(
						\sum_{
							\begin{subarray}{c}
								j \in \mathcal{N}_p \\
                p < j
              \end{subarray}
            } a_{pj}
						+
						\sum_{
							\begin{subarray}{c}
								j \in \mathcal{N}_p \\
                j < p
              \end{subarray}
              } a_{jp}
					\biggr)\,.
					\label{Eq:AppNetworkIdentity1}
				\\
				\intertext{Furthermore, if $a_{ij} = a_{ji}$ for all $(i,j) \in \mathcal{E}$, \eqref{Eq:AppNetworkIdentity1} becomes}
				  \sum_{(i,j) \in \mathcal{E}} a_{ij}
				&=
					\sum_{p=1}^P \sum_{j \in \mathcal{N}_p} a_{pj}\,.
					\label{Eq:AppNetworkIdentity2}
			\end{align}

			\item Let~$a_{ij}$ and~$a_{ji}$ be associated with edge~$(i,j) \in \mathcal{E}$. Then,
			\begin{equation}\label{Eq:AppNetworkIdentity3}
				\sum_{p=1}^P \sum_{j \in \mathcal{N}_p} a_{pj} = \sum_{p = 1}^P \sum_{j \in \mathcal{N}_p} a_{jp}\,.
			\end{equation}
    \end{list}
  \end{lemma}

  \begin{proof}

    \vspace{-0.2cm}
    \noindent
    \newcounter{NetworkLemmaAppProof}
		\begin{list}{(\alph{NetworkLemmaAppProof})}{\usecounter{NetworkLemmaAppProof}}
    	\item

    	We have
				\begin{align*}
							\sum_{(i,j) \in \mathcal{E}} a_{ij}
            &=
							\sum_{(p,j) \in \mathcal{E}} a_{pj}
							+
							\sum_{(i,p) \in \mathcal{E}} a_{ip}
							+
							\sum_{
                \begin{subarray}{c}
                    (i,j) \in \mathcal{E} \\
                    i,j \neq p
                \end{subarray}
              } a_{ij}
            \\
            &=
              \sum_{
                \begin{subarray}{c}
                    j \in \mathcal{N}_{p} \\
                    p < j
                \end{subarray}
                   }a_{pj}
              +
              \sum_{
                \begin{subarray}{c}
                    j \in \mathcal{N}_{p} \\
                    j < p
                \end{subarray}
                   }a_{jp}
              +
              \sum_{
                \begin{subarray}{c}
                    (i,j) \in \mathcal{E} \\
                    i,j \neq p
                \end{subarray}
              } a_{ij}\,,
						\intertext{and repeating iteratively for all $P$ nodes,}
						&=
						  \sum_{p=1}^P \biggl(
														\sum_{
															\begin{subarray}{c}
																j \in \mathcal{N}_{p} \\
																p < j
															\end{subarray}
														}a_{pj}
														+
														\sum_{
															\begin{subarray}{c}
																j \in \mathcal{N}_{p} \\
																j < p
															\end{subarray}
														}a_{jp}
													\biggr)\,.
					\end{align*}
					When~$a_{pj} = a_{jp}$, then
					$$
						\sum_{(i,j) \in \mathcal{E}} a_{ij}
						 =
						 \sum_{p=1}^P \sum_{j \in \mathcal{N}_p} a_{pj}\,.
					$$

			\item

			\noindent
			There holds
				\begin{align*}
					  \sum_{p=1}^P \sum_{j \in \mathcal{N}_p} a_{pj}
					&=
					  \sum_{i \in \mathcal{N}_1} a_{i1}
					  +
					  \sum_{p=1}^P \sum_{\begin{subarray}{c} j \in \mathcal{N}_p \\ j \neq 1 \end{subarray}} a_{pj}
					\\
					&=
					  \sum_{i \in \mathcal{N}_1} a_{i1}
					  +
					  \sum_{i \in \mathcal{N}_2} a_{i2}
					  +
					  \sum_{p=1}^P \sum_{\begin{subarray}{c} j \in \mathcal{N}_p \\ j \neq 1,2 \end{subarray}} a_{pj}\,,
					\\
					\intertext{and repeating for all nodes,}
					&=
					  \sum_{p=1}^P \sum_{j \in \mathcal{N}_p} a_{jp}\,.
				\end{align*}
    \end{list}
    \end{proof}

\section{Derivation of Algorithm~\ref{Alg:Schizas}}

	We reproduce here problem~\eqref{Eq:RelatedWorkSchizasReformulation}, which was obtained as a reformulation of~\eqref{Eq:GlobalProblem}:
	$$
		\begin{array}{ll}
			\underset{\bar{x},\bar{z}}{\text{minimize}} & f_1(x_1) + f_2(x_2) + \cdots + f_P(x_P) \\
			\text{subject to} & x_p = z_j\,, \quad j \in \mathcal{N}_p^+\,,\quad p = 1,\ldots,P\,.
		\end{array}
	$$
	Recall that each node~$p$ has two copies of the original problem variable~$x \in \mathbb{R}^n$: $x_p \in \mathbb{R}^n$ and~$z_p \in \mathbb{R}^n$. The collection of the~$x_p$'s and of the~$z_p$'s are $\bar{x} = (x_1,\ldots,x_P)$ and~$\bar{z} = (z_1,\ldots,z_P)$, respectively. We can apply the $2$-block ADMM~\eqref{Eq:RelatedWorkADMMIter1}-\eqref{Eq:RelatedWorkADMMIter3} to this problem, seeing~$\bar{x}$ and~$\bar{z}$ as the two block variables. The augmented Lagrangian is
	\begin{equation}\label{Eq:AppSchizasAugmentedLagrangian}
		L_\rho(z,x;\lambda) = \sum_{p=1}^P f_p(x_p) + \sum_{p=1}^P \sum_{j \in \mathcal{N}_p^+} \lambda_{pj}^\top (x_p - z_j) + \frac{\rho}{2}\sum_{p=1}^P \sum_{j \in \mathcal{N}_p^+} \|x_p - z_j\|^2\,,
	\end{equation}
	where~$\lambda_{pj}$ is the dual variable associated to the constraint~$x_p - z_j = 0$, and $\lambda = (\ldots,\lambda_{ij},\ldots)$ is the collection of dual variables. We consider~$\bar{z}$ as the first block variable, and~$\bar{x}$ as the second block variable.

	\mypar{Minimization in \boldmath{$\bar{z}$}} Fixing~$\bar{x}$ and~$\lambda$ at~$\bar{x}^k$ and~$\lambda^k$, respectively, $\bar{z}$ is updated as
	\begin{equation}\label{Eq:AppSchizasZInterm}
		\bar{z}^{k+1} = \underset{\bar{z}}{\arg\min}\,\,\, \sum_{p=1}^P \sum_{j \in \mathcal{N}_p^+} {\lambda_{jp}^k}^\top (x_j^k - z_p) + \frac{\rho}{2}\sum_{p=1}^P \sum_{j \in \mathcal{N}_p^+} \bigl\|x_j^k - z_p\bigr\|^2\,,
	\end{equation}
	where we used the identity~\eqref{Eq:AppNetworkIdentity3}. Note that we also dropped the first term in~\eqref{Eq:AppSchizasAugmentedLagrangian}, since it does not depend on~$\bar{z}$. Now, \eqref{Eq:AppSchizasZInterm} decouples into~$P$ problems that can be solved in parallel. The problem associated to node~$p$ is
	\begin{align*}
		  z_p^{k+1}
		&=
		  \underset{z_p}{\arg\min} \,\, \sum_{j \in \mathcal{N}_p^+} {\lambda_{jp}^k}^\top \bigl(x_j^k - z_p\bigr) + \frac{\rho}{2} \sum_{j \in \mathcal{N}_p^+} \bigl\|z_p - x_j^k\bigr\|^2
		\\
		&=
		  \underset{z_p}{\arg\min} \,\, -\biggl(\sum_{j \in \mathcal{N}_p^+} \lambda_{jp}^k +  \rho \sum_{j \in \mathcal{N}_p^+} x_j^k\biggr)^\top z_p + \frac{\rho(D_p+1)}{2} \|z_p\|^2\,,
	\end{align*}
	which has the closed-form solution
	\begin{align}
		  z_p^{k+1}
		&=
		  \frac{1}{\rho(D_p+1)}\Bigl(\sum_{j \in \mathcal{N}_p^+} \lambda_{jp}^k +  \rho \sum_{j \in \mathcal{N}_p^+} x_j^k\Bigr)
		\notag
		\\
		&=
		  \tau_p \sum_{j \in \mathcal{N}_p^+} \lambda_{jp}^k + \frac{1}{D_p +1} \sum_{j \in \mathcal{N}_p^+} x_j^k	\,,
		\label{Eq:AppSchizasZ}
	\end{align}
	where~$\tau_p = 1/(\rho(D_p + 1))$.

	\mypar{Minimization in \boldmath{$\bar{x}$}}
	Fixing~$\bar{z}$ and~$\lambda$ at~$\bar{z}^{k+1}$ and~$\lambda^k$, respectively, $\bar{x}$ is updated as
	$$
		\bar{x}^{k+1}
		=
		\underset{\bar{x}}{\arg\min}\,\,
		\sum_{p=1}^P f_p(x_p) + \sum_{p=1}^P\sum_{j \in \mathcal{N}_p^+} {\lambda_{pj}^k}^\top \bigl(x_p - z_j^{k+1}\bigr)
		+
		\frac{\rho}{2}\sum_{p=1}^P \sum_{j \in \mathcal{N}_p^+} \bigl\|x_p - z_j^{k+1}\bigr\|^2\,,
	$$
	which decouples into~$P$ optimization problems that can be solved in parallel. The problem associated to node~$p$ is
	\begin{align}
		  x_p^{k+1}
		&=
		  \underset{x_p}{\arg\min}\,\, f_p(x_p) + \sum_{j \in \mathcal{N}_p^+} {\lambda_{pj}^k}^\top \bigl(x_p - z_j^{k+1}\bigr) + \frac{\rho}{2}\sum_{j \in \mathcal{N}_p^+}\bigl\|x_p - z_j^{k+1}\bigr\|^2
		\notag
		\\
		&=
		  \underset{x_p}{\arg\min}\,\, f_p(x_p) + \biggl(\sum_{j \in \mathcal{N}_p^+} \lambda_{pj}^k - \rho \sum_{j \in \mathcal{N}_p^+ }z_j^{k+1}\biggr)^\top x_p + \frac{\rho (D_p+1)}{2}\|x_p\|^2\,,
		\notag
		\intertext{and after completing the square,}
		&=
		  \underset{x_p}{\arg\min}\,\, f_p(x_p) + \frac{1}{2\tau_p}\biggl\|x_p + \tau_p \biggl(\sum_{j \in \mathcal{N}_p^+} \lambda_{pj}^k - \rho \sum_{j \in \mathcal{N}_p^+}z_j^{k+1}\biggr)\biggr\|^2
		\notag
		\\
		&=
		  \text{prox}_{\tau_p f_p}\biggl(\frac{1}{D_p + 1}\sum_{j \in \mathcal{N}_p^+ }z_j^{k+1} - \tau_p \sum_{j \in \mathcal{N}_p^+} \lambda_{pj}^k \biggr)\,,
		\label{Eq:AppSchizasX}
	\end{align}
	where the operator~$\text{prox}$ is defined in~\eqref{Eq:DefinitionProxOperator}.

	\mypar{Update of the dual variables}
	According to ADMM (cf.\ \eqref{Eq:RelatedWorkADMMIter3}), each dual variable~$\lambda_{pj}$, for $j \in \mathcal{N}_p^+$ and~$p=1,\ldots,P$, is updated as $\lambda_{pj}^{k+1} = \lambda_{pj}^{k} + \rho (x_p^{k+1} - z_j^{k+1})$. Node~$p$, however, does not need to know each individual~$\lambda_{ij}$. In fact, \eqref{Eq:AppSchizasZ} and~\eqref{Eq:AppSchizasX} only depend on the sums $\mu_p^k := \sum_{j \in \mathcal{N}_p^+} \lambda_{jp}^k$ and~$\eta_p^k := \sum_{j \in \mathcal{N}_p^+}\lambda_{pj}^k$, respectively. And these sums (or better, these new dual variables~$\mu_p$ and~$\eta_p$) can be updated as
	\begin{align*}
			\mu_p^{k+1}
		&=
			\sum_{j \in \mathcal{N}_p^+} \lambda_{jp}^{k+1}
			=
		  \underbrace{\sum_{j \in \mathcal{N}_p^+} \lambda_{jp}^{k}}_{\mu_p^k} + \rho \sum_{j \in \mathcal{N}_p^+} (x_j^{k+1} - z_p^{k+1})
		  =
		  \mu_p^k + \frac{1}{\tau_p}\Bigl(\frac{1}{D_p +1}\sum_{j \in \mathcal{N}_p^+} x_j^{k+1} - z_p^{k+1}\Bigr)
		\\
		  \eta_p^{k+1}
		&=
		  \sum_{j \in \mathcal{N}_p^+} \lambda_{pj}^{k+1}
			=
			\underbrace{
		  \sum_{j \in \mathcal{N}_p^+} \lambda_{pj}^{k}
			}_{\eta_p^k}
		  +
		  \rho \sum_{j \in \mathcal{N}_p^+} (x_p^{k+1} - z_j^{k+1})
		  =
		  \eta_p^k + \frac{1}{\tau_p}\Bigl(x_p^{k+1} - \frac{1}{D_p+1}\sum_{j \in \mathcal{N}_p^+} z_j^{k+1} \Bigr)\,.
	\end{align*}
	These updates constitute step~\ref{SubAlg:RelatedWorkSchizasUpdateDual} of Algorithm~\ref{Alg:Schizas}. If we replace $\sum_{j \in \mathcal{N}_p^+}\lambda_{jp}^k$ in~\eqref{Eq:AppSchizasZ} and~$\sum_{j \in \mathcal{N}_p^+}\lambda_{pj}^k$ in~\eqref{Eq:AppSchizasX} by~$\mu_p^k$ and~$\eta_p^k$, respectively, we get steps~\ref{SubAlg:RelatedWorkSchizasUpdateZ} and~\ref{SubAlg:RelatedWorkSchizasUpdateX}.

\section{Derivation of Algorithm~\ref{Alg:Zhu}}
\label{SubApp:DerivationZhu}

	The reformulation~\cite{Zhu09-DistributedInNetworkChannelCoding} makes of~\eqref{Eq:GlobalProblem} is~\eqref{Eq:RelatedWorkZhuReformulation}, which we reproduce here:
	$$
		\begin{array}{ll}
    		\underset{\bar{x},\bar{z}}{\text{minimize}} & f_1(x_1) + f_2(x_2) + \cdots + f_P(x_P) \\
    		\text{subject to} & x_i = z_{ij}\,, \quad (i,j) \in \mathcal{E} \\
    		                  & x_j = z_{ij}\,, \quad (i,j) \in \mathcal{E}\,.
    \end{array}
	$$
	Associating the dual variables~$\lambda_{ij}$ to the first set of constraints and~$\eta_{ij}$ to the second one, the augmented Lagrangian is
	\begin{equation}\label{Eq:AppZhuAugmentedLagrangian}
		L_\rho(\bar{x},\bar{z};\lambda,\eta) =
			\sum_{p=1}^P f_p(x_p) +
							\sum_{(i,j) \in \mathcal{E}}
								\left(
										\lambda_{ij}^\top (x_i - z_{ij}) + \eta_{ij}^\top (x_j - z_{ij})
										+ \frac{\rho}{2} \|x_i - z_{ij}\|^2+ \frac{\rho}{2}\|x_j - z_{ij}\|^2
								\right)\,,
	\end{equation}
	where $\lambda$ (resp.\ $\eta$) is the collection of the dual variables~$\lambda_{ij}$ (resp.\ $\eta_{ij}$). The $2$-block ADMM~\eqref{Eq:RelatedWorkADMMIter1}-\eqref{Eq:RelatedWorkADMMIter3} applied to this problem translates into
  \begin{align}
			\bar{x}^{k+1} &= \underset{\bar{x}}{\arg\min}\,\, L_\rho(\bar{x},\bar{z}^k;\lambda^k,\eta^k)
			\label{Eq:AppZhuMinX}
			\\
      \bar{z}^{k+1} &= \underset{\bar{z}}{\arg\min}\,\, L_\rho(\bar{x}^{k+1},\bar{z};\lambda^k,\eta^k)
      \label{Eq:AppZhuMinZ}
      \\
      \lambda_{ij}^{k+1} &= \lambda_{ij}^{k} + \rho(x_i^{k+1} - z_{ij}^{k+1})\,,\quad (i,j) \in \mathcal{E}
      \label{Eq:AppZhuUpdateLambda}
      \\
      \eta_{ij}^{k+1} &= \eta_{ij}^{k} + \rho(x_j^{k+1} - z_{ij}^{k+1})\,,\quad (i,j) \in \mathcal{E}\,.
      \label{Eq:AppZhuUpdateEta}
  \end{align}
	We first analyze the minimization with respect to~$\bar{z}$, \eqref{Eq:AppZhuMinZ}; then, we analyze the minimization with respect to~$\bar{x}$, \eqref{Eq:AppZhuMinX}; and, finally, we will see how to simplify the updates of the dual variables~\eqref{Eq:AppZhuUpdateLambda} and~\eqref{Eq:AppZhuUpdateEta}.

  \mypar{Minimization in \boldmath{$\bar{z}$}}
  Since the augmented Lagrangian is quadratic in~$\bar{z}$, problem~\eqref{Eq:AppZhuMinZ} has a closed form solution. To compute it component-wise, just select~$(i,j) \in \mathcal{E}$, and
  \begin{align}
		  \frac{\partial}{\partial z_{ij}} L_\rho(\bar{x}^{k+1},\bar{z};\lambda^k,\eta^k)\biggr|_{z_{ij} = z_{ij}^{k+1}} = 0
    &\quad\Longleftrightarrow\quad
      -(\lambda_{ij}^k + \eta_{ij}^k) - \rho (x_i^{k+1} - z_{ij}^{k+1}) - \rho (x_j^{k+1} - z_{ij}^{k+1}) = 0
      \notag
    \\
    &\quad\Longleftrightarrow\quad
      z_{ij}^{k+1} = \frac{\lambda_{ij}^k + \eta_{ij}^k}{2\rho} + \frac{x_i^{k+1} + x_j^{k+1}}{2}\,.
      \label{Eq:AppZhuMinZ1}
  \end{align}
    Replacing~\eqref{Eq:AppZhuMinZ1} in~\eqref{Eq:AppZhuUpdateLambda} and~\eqref{Eq:AppZhuUpdateEta}, we get, respectively,
    \begin{align}
          \lambda_{ij}^{k+1}
        &=
          \lambda_{ij}^{k} + \rho\,\biggl(x_i^{k+1} - \frac{\lambda_{ij}^k + \eta_{ij}^k}{2\rho} - \frac{x_i^{k+1} + x_j^{k+1}}{2}\biggr)
          =
					\frac{\lambda_{ij}^k - \eta_{ij}^k}{2} + \rho\, \frac{x_i^{k+1} - x_j^{k+1}}{2}
					\label{Eq:AppZhuUpdateLambda1}
			  \\
          \eta_{ij}^{k+1}
        &=
          \eta_{ij}^{k} + \rho\,\biggl(x_j^{k+1} - \frac{\lambda_{ij}^k + \eta_{ij}^k}{2\rho} - \frac{x_i^{k+1} + x_j^{k+1}}{2}\biggr)
          =
					\frac{\eta_{ij}^k - \lambda_{ij}^k}{2} + \rho \,\frac{x_j^{k+1} - x_i^{k+1}}{2}\,.
					\label{Eq:AppZhuUpdateEta1}
    \end{align}
    Note that if we sum up~\eqref{Eq:AppZhuUpdateLambda1} and~\eqref{Eq:AppZhuUpdateEta1}, we get
		\begin{equation}\label{Eq:AppZhuSymmetricDualVars}
			\lambda_{ij}^{k+1} + \eta_{ij}^{k+1} = 0\,,
		\end{equation}
		which holds for all~$k\geq 1$. Let us assume that it also holds for~$k=0$, i.e., $\lambda_{ij}^0$ and~$\eta_{ij}^0$ are initialized with symmetric values. Then, the first term in~\eqref{Eq:AppZhuMinZ1} is zero, and updating~$z_{ij}$ simplifies to
		\begin{equation}\label{Eq:AppZhuMinZ2}
			z_{ij}^{k+1} =  \frac{x_i^{k+1} + x_j^{k+1}}{2}\,.
		\end{equation}
		Similarly, the updates of the dual variables, \eqref{Eq:AppZhuUpdateLambda1} and~\eqref{Eq:AppZhuUpdateEta1} simplify, respectively, to
		\begin{align}
			\lambda_{ij}^{k+1} &= \lambda_{ij}^{k} + \rho\,\frac{x_i^{k+1} - x_j^{k+1}}{2}
			\label{Eq:AppZhuUpdateLambda2}
			\\
			\eta_{ij}^{k+1} &= \eta_{ij}^{k} + \rho\,\frac{x_j^{k+1} - x_i^{k+1}}{2}\,,
			\label{Eq:AppZhuUpdateEta2}
		\end{align}
		Since we assume that writing~$(i,j) \in \mathcal{E}$ means that $i<j$, the sets of dual variables~$\lambda_{ij}$ and~$\eta_{ij}$ are only defined for~$i<j$. Let us extend their definition in a meaningful way, i.e., such that~\eqref{Eq:AppZhuUpdateLambda2} and~\eqref{Eq:AppZhuUpdateEta2} make sense. Then, for~$i>j$, we define~$\lambda_{ij}^k$ and~$\eta_{ij}^k$, respectively, as
		\begin{align}
			\lambda_{ij}^k &= \eta_{ji}^k\,,\qquad i>j
			\label{Eq:AppZhuDefinitionLambdaSwitched}
			\\
			\eta_{ij}^k &= \lambda_{ji}^k\,,\qquad i>j\,.
			\label{Eq:AppZhuDefinitionEtaSwitched}
		\end{align}
		Next, we use the identity~\eqref{Eq:AppZhuSymmetricDualVars}, which holds for all~$k$, and the simplified updates~\eqref{Eq:AppZhuMinZ2}, \eqref{Eq:AppZhuUpdateLambda2}, and~\eqref{Eq:AppZhuUpdateEta2} to find a simple expression for the minimization in~$\bar{x}$, \eqref{Eq:AppZhuMinX}.

		\mypar{Minimization in \boldmath{$\bar{x}$}}
		If we set~$z_{ij} = z_{ij}^k$, $\lambda_{ij} = \lambda_{ij}^k$, and~$\eta_{ij} = \eta_{ij}^k$ in the augmented Lagrangian~\eqref{Eq:AppZhuAugmentedLagrangian}, the second term becomes
		\begin{align}
			&
			  \sum_{(i,j) \in \mathcal{E}}
				\left(
            {\lambda_{ij}^k}^\top (x_i - z_{ij}^k) + {\eta_{ij}^k}^\top (x_j - z_{ij}^k)
            + \frac{\rho}{2} \|x_i - z_{ij}^k\|^2+ \frac{\rho}{2}\|x_j - z_{ij}^k\|^2
        \right)
      \label{Eq:AppZhuAux1}
      \\
      =&
			  \sum_{(i,j) \in \mathcal{E}}
          \biggl(
            {\lambda_{ij}^k}^\top x_i + {\eta_{ij}^k}^\top x_j - (\underbrace{\lambda_{ij}^k + \eta_{ij}^k}_{=0})^\top z_{ij}^k
            +
            \frac{\rho}{2}\|x_i - z_{ij}^k\|^2
            + \frac{\rho}{2}\|x_j - z_{ij}^k\|^2
          \biggr)\,,
      \label{Eq:AppZhuAux2}
      \\
      =&
        \sum_{(i,j) \in \mathcal{E}}
        \Bigl(
					\underbrace{
					{\lambda_{ij}^k}^\top x_i + {\lambda_{ji}^k}^\top x_j
					}_{:=a_{ij}}
        \Bigr)
        +
        \frac{\rho}{2}
        \sum_{(i,j) \in \mathcal{E}}
        \Bigl(
					\underbrace{
					\|x_i - z_{ij}^k\|^2 + \|x_j - z_{ij}^k\|^2
          }_{:=b_{ij}}
        \Bigr)\,.
      \label{Eq:AppZhuMinX2ndTerm}
		\end{align}
		From~\eqref{Eq:AppZhuAux1} to~\eqref{Eq:AppZhuAux2}, we just rearranged the first two terms in the sum and used identity~\eqref{Eq:AppZhuSymmetricDualVars}. From~\eqref{Eq:AppZhuAux2} to~\eqref{Eq:AppZhuMinX2ndTerm}, we used definition~\eqref{Eq:AppZhuDefinitionLambdaSwitched}. Now note that~$a_{ij} = a_{ji}$ and also that~$b_{ij} = b_{ji}$ (since, by~\eqref{Eq:AppZhuMinZ2}, $z_{ij}^k = z_{ji}^k$). By identity~\eqref{Eq:AppNetworkIdentity2} in Lemma~\ref{lem:AppNetworkIdentities}, we can write~\eqref{Eq:AppZhuMinX2ndTerm} as
		\begin{align}
		  &
				\sum_{p=1}^P\sum_{j \in \mathcal{N}_p}
				\Bigl(
					{\lambda_{pj}^k}^\top x_p + {\lambda_{jp}^k}^\top x_j
				\Bigr)
        +
				\frac{\rho}{2}
				\sum_{p=1}^P\sum_{j \in \mathcal{N}_p}
        \Bigl(
					\|x_p - z_{pj}^k\|^2 + \|x_j - z_{pj}^k\|^2
        \Bigr)
			\notag
      \\
      =&
				\sum_{p=1}^P\sum_{j \in \mathcal{N}_p} {\lambda_{pj}^k}^\top x_p
				+
				\sum_{p=1}^P\sum_{j \in \mathcal{N}_p} {\lambda_{jp}^k}^\top x_j
				+
				\frac{\rho}{2}	\sum_{p=1}^P\sum_{j \in \mathcal{N}_p} \|x_p - z_{pj}^k\|^2
				+
				\frac{\rho}{2}	\sum_{p=1}^P\sum_{j \in \mathcal{N}_p} \|x_j - z_{pj}^k\|^2
			\label{Eq:AppZhuAux3}
			\\
			=&
			  2\sum_{p=1}^P\sum_{j \in \mathcal{N}_p} {\lambda_{pj}^k}^\top x_p
			  +
			  \rho	\sum_{p=1}^P\sum_{j \in \mathcal{N}_p} \|x_p - z_{pj}^k\|^2\,.
			\label{Eq:AppZhuMinX2ndTerm2}
		\end{align}
		From~\eqref{Eq:AppZhuAux3} to~\eqref{Eq:AppZhuMinX2ndTerm2}, we used identity~\eqref{Eq:AppNetworkIdentity3} from Lemma~\ref{lem:AppNetworkIdentities} in the second and fourth terms, and also that $z_{ij}^k = z_{ji}^k$. We can now write~\eqref{Eq:AppZhuMinX2ndTerm2} as
		\begin{align}
			  \sum_{p=1}^P\sum_{j \in \mathcal{N}_p}
				\biggl(
					2\,{\lambda_{pj}^k}^\top x_p + \rho \bigl\|x_p - z_{pj}^k\bigr\|^2
			  \biggr)
			&=
			  \sum_{p=1}^P\sum_{j \in \mathcal{N}_p}
				\biggl(
				  \Bigl(2\,{\lambda_{pj}^k} - 2\rho \,z_{pj}^k\Bigr)^\top x_p + \rho \|x_p\|^2 + \rho \bigl\|z_{pj}^k\bigr\|^2
				\biggr)
			\notag
			\\
			&=
				\sum_{p=1}^P
				\biggl(
					\Bigl(
						\underbrace{2\sum_{j \in \mathcal{N}_p}\lambda_{pj}^k}_{:=\mu_p^k}
						-2\rho \sum_{j \in \mathcal{N}_p}z_{pj}^k
					\Bigr)^\top x_p
					+
					\rho D_p\Bigl(\|x_p\|^2 + \|z_{pj}^k\|^2\Bigr)
				\biggr)
			\label{Eq:AppZhuDefinitionMu}
			\\
			&=
			  \sum_{p=1}^P
				\biggl(
					\Bigl(
						\mu_p^k
						-2\rho \sum_{j \in \mathcal{N}_p}z_{pj}^k
					\Bigr)^\top x_p
					+
					\rho D_p\Bigl(\|x_p\|^2 + \|z_{pj}^k\|^2\Bigr)
				\biggr)\,.
			\label{Eq:AppZhuMinX2ndTerm3}
		\end{align}
		Therefore, updating~$\bar{x}$ as in~\eqref{Eq:AppZhuMinX} amounts to
		\begin{equation}\label{Eq:AppZhuProbX1}
			\bar{x}^{k+1}
			=
			\underset{\bar{x}}{\arg\min}\,\, \sum_{p=1}^P
			\biggl(
				f_p(x_p) + \Bigl(\mu_p^k - 2\rho \sum_{j \in \mathcal{N}_p}z_{pj}^k \Bigr)^\top x_p
				+	\rho D_p\|x_p\|^2
			\biggr)\,,
		\end{equation}
		where we dropped $\|z_{pj}^k\|^2$ in the last term in~\eqref{Eq:AppZhuMinX2ndTerm3}, since it is independent of the problem variable~$\bar{x}$. Problem~\eqref{Eq:AppZhuProbX1} yields~$P$ independent optimization problems, each depending only on an~$x_p$, which can be executed in parallel. The problem associated to node~$p$ is
		\begin{align}
			  x_p^{k+1}
			&=
			  \underset{x_p}{\arg\min}\,\,\,
				  f_p(x_p) + \Bigl(\mu_p^k - 2\rho \sum_{j \in \mathcal{N}_p}z_{pj}^k \Bigr)^\top x_p
				  +	\rho D_p\|x_p\|^2
			\notag
			\\
			&=
			  \underset{x_p}{\arg\min}\,\,\,
			    f_p(x_p) + \rho D_P \biggl\| x_p - \Bigl(\frac{1}{D_p}\sum_{j \in \mathcal{N}_p}z_{pj}^k - \frac{1}{2\rho D_p}\mu_p^k\Bigr)\biggr\|^2
			\label{Eq:AppZhuProbX2}
			\\
			&=
			  \text{prox}_{\tau_p f_p}\biggl(\frac{1}{D_p}\sum_{j \in \mathcal{N}_p}z_{pj}^k - \tau_p\mu_p^k\biggr)
			\label{Eq:AppZhuProbX3}
			\\
			&=
			  \text{prox}_{\tau_p f_p}\biggl(\frac{1}{2D_p}\sum_{j \in \mathcal{N}_p}(x_{p}^k + x_j^k) - \tau_p\mu_p^k\biggr)\,.
			\label{Eq:AppZhuProbX4}
		\end{align}
		From~\eqref{Eq:AppZhuProbX2} to~\eqref{Eq:AppZhuProbX3}, we used the definition of the prox operator~\eqref{Eq:DefinitionProxOperator} and $\tau_p = 1/(2\rho D_p)$. From~\eqref{Eq:AppZhuProbX3} to~\eqref{Eq:AppZhuProbX4}, we replaced~$z_{pj}^k$ as in~\eqref{Eq:AppZhuMinZ2}.

    \mypar{Update of the dual variables}
    Each node~$p$ does not need to know each individual~$\lambda_{ij}$ associated to its incident edges. In fact, as shown in~\eqref{Eq:AppZhuDefinitionMu}, it only need to know~$\mu^k_p = \sum_{j \in \mathcal{N}_p} \lambda_{pj}^k$. According to~\eqref{Eq:AppZhuUpdateLambda2}, this variable is updated as
		\begin{equation}\label{Eq:AppZhuMuUpdate}
			  \mu_p^{k+1}
			  =
			  2\sum_{j \in \mathcal{N}_p} \lambda_{pj}^{k+1}
			  =
			  \underbrace{2\sum_{j \in \mathcal{N}_p} \lambda_{pj}^{k}}_{\mu_p^k} + \rho \sum_{j \in \mathcal{N}_p} (x_{p}^{k+1} - x_j^{k+1})
			  =
			  \mu_p^k + \frac{1}{2\tau_p} \Bigl( x_p^{k+1} - \frac{1}{D_p}\sum_{j \in \mathcal{N}_p} x_j^{k+1}\Bigr)\,.
		\end{equation}
		We thus see that~\eqref{Eq:AppZhuMuUpdate} corresponds to step~\ref{SubAlg:RelatedWorkZhuUpdateDual} in Algorithm~\ref{Alg:Zhu}, while~\eqref{Eq:AppZhuProbX4} corresponds to step~\ref{SubAlg:RelatedWorkZhuUpdateX}.

%% file: 03-BackMatter/conjugateFunctions.tex
\chapter{Some Conjugate Functions}
\label{App:ConjugateFunctions}

	In this appendix, we compute some conjugate functions that appear throughout the thesis, especially in compressed sensing problems.

	\mypar{\boldmath{$\ell_1$}-norm plus quadratic regularization}
	In \ssref{SubSec:GC:SparseSolutions}, we reformulate BP~\eqref{Eq:BasisPursuit} as a problem in the global class~\eqref{Eq:GlobalProblem}. That reformulation uses duality and, in~\eqref{Eq:BPDual4}, we use the convex conjugate of the function $h(x) = \|x\|_1 + (c/2)\|x\|^2$ where the term~$\|x\|^2$ plays the role of a regularization function. We now show that the convex conjugate of~$h$ has a closed-form expression. Suppose~$x 	\in \mathbb{R}^n$. We have
	\begin{align}
			h^\star(\lambda) &= \underset{x}{\sup}\,\,\, \lambda^\top x - \|x\|_1 - \frac{c}{2}\|x\|^2
		\notag
		\\
		&=
			-\underset{x}{\inf}\,\,\, \|x\|_1 + \frac{c}{2}\|x\|^2 - \lambda^\top x
		\notag
		\\
		&=
			-\sum_{i=1}^n \underset{x_i}{\inf}\,\,\, |x_i| + \frac{c}{2}x_i^2 - \lambda_i x_i\,.
 	\label{Eq:AppCVXConjugateSoft}
	\end{align}
	Applying the optimality condition for convex problems to the problem in the $i$th component,
	\begin{equation}\label{Eq:AppCVXConjugateSoft2}
		0 \in \partial |x_i| + c\,x_i - \lambda_i\,.
	\end{equation}
	When~$x_i > 0$, $\partial |x_i| = \{1\}$, and~\eqref{Eq:AppCVXConjugateSoft2} becomes $x_i = (\lambda_i - 1)/c$. This happens when $\lambda_i > 1$, otherwise the expression would give a negative~$x_i$. Similarly, when $x_i < 0$, $\partial |x_i| = \{-1\}$, and~\eqref{Eq:AppCVXConjugateSoft2} becomes $x_i = (\lambda_i + 1)/c$. This expression is negative when $\lambda_i < -1$. Finally, when~$x_i=0$, $\partial |x_i| = [-1,1]$, and~\eqref{Eq:AppCVXConjugateSoft2} becomes the condition under which~$x_i=0$\,: $|\lambda_i| \leq 1$. This explains expression~\eqref{Eq:BPSoftSolution}.

	\mypar{\boldmath{$\ell_2$}-norm plus quadratic regularization}
	Here, we derive~\eqref{Eq:ConvexConjugateNormPlusNormSquared}, which is a closed-form expression for the conjugate of the function~$h(x) = \|x\| + (c/2)\|x\|^2$, where $\|\cdot\|$ is the $\ell_2$-norm. The convex conjugate of~$h$ is
	\begin{align}
		  h^\star(\lambda)
		&=
		  \underset{x}{\sup}\,\,\, \lambda^\top x - \|x\| - \frac{c}{2}\|x\|^2
		\\
		&=
		  -\underset{x}{\inf}\,\,\, \|x\| + \frac{c}{2}\|x\|^2 - \lambda^\top x\,.
		\label{Eq:AppCVXConjugateNorm2Regularized}
	\end{align}
	The subgradient of the norm function is
	$$
		\partial \|x\|
		=
		\left\{
			\begin{array}{ll}
				B(0,1) &,\, x=0 \\
				\frac{x}{\|x\|} &,\, x\neq 0\,,
			\end{array}
		\right.
	$$
	where~$B(0,1) = \{x\,:\, \|x\|\leq 1\}$ is the ball with radius~$1$, centered at the origin. The optimality conditions for~\eqref{Eq:AppCVXConjugateNorm2Regularized} then tell us that $x = 0$ if $\|\lambda\| \leq 1$ and that, for~$x\neq 0$,
	\begin{equation}\label{Eq:AppCVXConjugateNorm2Regularized2}
		0 = \frac{x}{\|x\|} + c\,x - \lambda\,.
	\end{equation}
	From~\eqref{Eq:AppCVXConjugateNorm2Regularized2}, we first find the norm of~$x$ and then compute an expression for~$x$. To find the norm of~$x$, first rewrite~\eqref{Eq:AppCVXConjugateNorm2Regularized2} as $\lambda = (1/\|x\| + c)x$, and compute the squared norm of both sides of the equation. This yields
	\begin{align*}
		\biggl(\frac{1}{\|x\|} + c\biggr)^2 \|x\|^2 = \|\lambda\|^2
		\qquad
		\Longleftrightarrow
		\qquad
		1 + 2c\|x\| + c^2 \|x\|^2 = \|\lambda\|^2\,,
	\end{align*}
	which is a quadratic expression on~$\|x\|^2$. Solving the quadratic equation, gives us $\|x\| = (\|\lambda\| - 1)/c$, which is positive because $\|\lambda\| > 1$. Replacing in~\eqref{Eq:AppCVXConjugateNorm2Regularized2} gives
	$$
		x = \frac{1}{c}\Bigl(1 - \frac{1}{\|\lambda\|}\Bigr)\lambda\,.
	$$
	To compute the value~\eqref{Eq:AppCVXConjugateNorm2Regularized}, just take the inner product of~\eqref{Eq:AppCVXConjugateNorm2Regularized2} with~$x$ and subtract~$(c/2)\|x\|^2$ to both sides of the equation. This gives
	\begin{align*}
		-\frac{c}{2}\|x\|^2 = \|x\| + \frac{c}{2}\|x\|^2 - \lambda^\top x\,.
	\end{align*}
	Using the expression for the norm of~$x$, we get
	$$
		h^\star(\lambda) = \frac{1}{2c}\Bigl(\|\lambda\|^2 - 2\|\lambda\| + 1\Bigr)\,,
	$$
	for $\|\lambda\| > 1$. This explains~\eqref{Eq:ConvexConjugateNormPlusNormSquared}.

%% file: 03-BackMatter/derivationRelatedWorkKekatos.tex
\chapter{ADMM-based Algorithm For The Connected Class: Derivation}
\label{App:ADMMAlgsPartialClass}

  In this appendix, we derive Algorithm~\ref{Alg:Kekatos}, an ADMM-based algorithm presented in \cref{Ch:ConnectedNonConnected} that was proposed in~\cite{Kekatos12-DistributedRobustPowerStateEstimation} to solve~\eqref{Eq:IntroProb} with a star-shaped variable. That algorithm can be easily generalized to a generic connected variable, as we do next. To do that, we apply the $2$-block ADMM to the reformulation~\eqref{Eq:PartialReformulationKekatos}, reproduced here for convenience:
  \begin{equation}\label{Eq:AppPartialReformulationKekatos}
		\begin{array}{cl}
			\underset{\{\bar{x}_l\}_{l=1}^n, \{\bar{z}_l\}_{l=1}^n}{\text{minimize}} & f_1(x_{S_1}^{(1)}) + f_1(x_{S_2}^{(2)}) + \cdots + f_1(x_{S_P}^{(P)}) \\
  		\text{subject to} & x_l^{(p)} = z_l^{\{p,j\}}\,, \quad l \in S_p \cap S_j\,, \quad j \in \mathcal{N}_p\,, \quad p = 1,\ldots,P\,,
  	\end{array}
  \end{equation}
	whose variable consists of~$\bar{x} := \{\bar{x}_l\}_{l=1}^n$, and~$\bar{z} := \{\bar{z}_l\}_{l=1}^n$. The augmented Lagrangian of~\eqref{Eq:AppPartialReformulationKekatos} is
	\begin{equation}\label{Eq:AppDerivationKekatosAugmentedLagrangian}
		L_\rho(\bar{x},\bar{z};\bar{\lambda}) = \sum_{p=1}^P f_p\Bigl(x_{S_p}^{(p)}\Bigr) + \sum_{p=1}^P\sum_{j \in \mathcal{N}_p} \sum_{l \in S_p \cap S_j} \biggl[ {\lambda_l^{pj}}^\top (x_l^{(p)} - z_l^{\{p,j\}}) + \frac{\rho}{2} \Bigl\|x_l^{(p)} - z_l^{\{p,j\}}\Bigr\|^2 \biggr]\,,
	\end{equation}
	Note that~$\lambda_l^{pj}$ and~$\lambda_l^{jp}$ are associated to different constraints. The $2$-block ADMM~\eqref{Eq:RelatedWorkADMMIter1}-\eqref{Eq:RelatedWorkADMMIter3} applied to this problem translates into
	\begin{align}
		\bar{x}^{k+1} &= \underset{\bar{x}}{\arg\min}\,\,\, L_{\rho}(\bar{x},\bar{z}^k; \bar{\lambda}^k)
		\label{Eq:AppDerivationKekatosX}
		\\
		\bar{z}^{k+1} &= \underset{\bar{z}}{\arg\min}\,\,\, L_{\rho}(\bar{x}^{k+1},\bar{z}; \bar{\lambda}^k)
		\label{Eq:AppDerivationKekatosZ}
		\\
		\lambda_l^{pj,k+1} &= \lambda_l^{pj,k} + \rho \Bigl(x_l^{(p),k+1} - z_l^{\{p,j\}, k+1}\Bigr)
		\label{Eq:AppDerivationKekatosLambda1}
		\\
		\lambda_l^{jp,k+1} &= \lambda_l^{jp,k} + \rho \Bigl(x_l^{(j),k+1} - z_l^{\{p,j\}, k+1}\Bigr)\,.
		\label{Eq:AppDerivationKekatosLambda2}
	\end{align}
	As in the derivation of Algorithm~\ref{Alg:Zhu} in \aref{App:ADMMAlgsGlobalClass}, we first analyze the minimization with respect to~$\bar{z}$, \eqref{Eq:AppDerivationKekatosZ}; then, we analyze the minimization with respect to~$\bar{x}$, \eqref{Eq:AppDerivationKekatosX}; and, finally, we will see how to simplify the updates of the dual variables~\eqref{Eq:AppDerivationKekatosLambda1} and~\eqref{Eq:AppDerivationKekatosLambda2}.

	\mypar{Minimization in \boldmath{$\bar{z}$}}
	No function~$f_p$ depends on any component of~$\bar{z}$, which means that~\eqref{Eq:AppDerivationKekatosZ} is an unconstrained quadratic program and, thus, it has a closed-form solution. Furthermore, it decomposes across each component. In particular the minimization with respect to~$z_l^{\{p,j\}} = z_l^{\{j,p\}}$ is
	\begin{align}
		  z_l^{\{p,j\},k+1}
		&=
		  \underset{z_l^{\{p,j\}}}{\arg\min}\,\,\, {\lambda_l^{pj,k}}^\top \bigl(x_l^{(p),k+1} - z_l^{\{p,j\}}\bigr) +
		  \frac{\rho}{2} \Bigl\|x_l^{(p),k+1} - z_l^{\{p,j\}}\Bigr\|^2
		  \notag
		  \\
		  &\phantom{sssssssssssssssssssssssssssssssssssss}
		  +
		  {\lambda_l^{jp,k}}^\top \bigl(x_l^{(j),k+1} - z_l^{\{p,j\}}\bigr)
		  + \frac{\rho}{2} \Bigl\|x_l^{(j),k+1} - z_l^{\{p,j\}}\Bigr\|^2
		\notag
		\\
		&=
		  \underset{z_l^{\{p,j\}}}{\arg\min}\,\,\, -\bigl(\lambda_l^{pj,k} + \lambda_l^{jp,k}\bigr)^\top z_l^{\{p,j\}}
		  +
		  \rho \Bigl\|z_l^{\{p,j\}}\Bigl\|^2 - \rho \,{x_l^{(p),k+1}}^\top z_l^{\{p,j\}} - \rho \,{x_l^{(j),k+1}}^\top z_l^{\{p,j\}}
		\notag
		\\
		&=
		  \underset{z_l^{\{p,j\}}}{\arg\min}\,\,\, -\Bigl(\lambda_l^{pj,k} + \lambda_l^{jp,k} + \rho(x_l^{(p),k+1} + x_l^{(j),k+1} )\Bigr)^\top z_l^{\{p,j\}}
		  +
		  \rho \Bigl\|z_l^{\{p,j\}}\Bigr\|^2
		\notag
		\\
		&=
		  \frac{\lambda_l^{pj,k} + \lambda_l^{jp,k}}{2\rho} + \frac{x_l^{(p),k+1} + x_l^{(j),k+1} }{2}\,.
		\label{Eq:AppDerivationKekatosZ2}
	\end{align}
	Replacing~\eqref{Eq:AppDerivationKekatosZ2} in~\eqref{Eq:AppDerivationKekatosLambda1} and~\eqref{Eq:AppDerivationKekatosLambda2}, we get, respectively,
	\begin{align}
		  \lambda_l^{pj,k+1}
		&=
		  \lambda_l^{pj,k}
		  +
		  \rho \,
				\Bigl(
					x_l^{(p),k+1} - \frac{\lambda_l^{pj,k} + \lambda_l^{jp,k}}{2\rho} - \frac{x_l^{(p),k+1} + x_l^{(j),k+1} }{2}
				\Bigr)
		\notag
		\\
 		&=
 		  \frac{\lambda_l^{pj,k} - \lambda_l^{jp,k}}{2}
 		  +
 		  \frac{\rho}{2}\bigl(x_l^{(p),k+1} - x_l^{(j),k+1}\bigr)
 		\label{Eq:AppDerivationKekatosLambda3}
 		\\
 			\lambda_l^{jp,k+1}
 		&=
 		  \lambda_l^{jp,k} + \rho \Bigl(x_l^{(j),k+1} - \frac{\lambda_l^{pj,k} + \lambda_l^{jp,k}}{2\rho} - \frac{x_l^{(p),k+1} + x_l^{(j),k+1} }{2}\Bigr)
 		\notag
 		\\
 		&=
 		  \frac{\lambda_l^{jp,k} - \lambda_l^{pj,k}}{2}
 		  +
 		  \frac{\rho}{2}\bigl(x_l^{(j),k+1} - x_l^{(p),k+1}\bigr)
 		\label{Eq:AppDerivationKekatosLambda4}
	\end{align}
	Summing~\eqref{Eq:AppDerivationKekatosLambda3} with~\eqref{Eq:AppDerivationKekatosLambda4}, we get
	$$
		\lambda_l^{pj,k+1} + \lambda_l^{jp,k+1} = 0\,,
	$$
	which holds for all $k \geq 1$. Let us assume that it also holds for $k = 0$, i.e., $\lambda_l^{pj}$ and~$\lambda_l^{jp}$ are initialized with symmetric values. Then, the first term in~\eqref{Eq:AppDerivationKekatosZ2} is zero, and updating~$z^{\{p,j\}}_l$ simplifies to
	\begin{equation}\label{Eq:AppDerivationKekatosZ3}
		z_l^{\{p,j\},k+1} = \frac{x_l^{(p),k+1} + x_l^{(j),k+1} }{2}\,.
	\end{equation}

	\mypar{Minimization in \boldmath{$\bar{x}$}}
	We now turn to the minimization in~$\bar{x}$~\eqref{Eq:AppDerivationKekatosX}. If we fix each $z_l^{\{p,j\}}$ at~$z_l^{\{p,j\},k}$ and each~$\lambda_l^{pj}$ at~$\lambda_l^{pj,k}$, the augmented Lagrangian~\eqref{Eq:AppDerivationKekatosAugmentedLagrangian} is the sum of~$p$ terms, where the $p$th term depends only on~$x_{S_p}$. Thus, problem~\eqref{Eq:AppDerivationKekatosX} decomposes into~$P$ optimization problems that can be solved in parallel. The problem associated with node~$p$ is
	\begin{align}
			x_{S_p}^{(p),k+1}
		&=
			\underset{x_{S_p}^{(p)}}{\arg\min}\,\,\, f_p\Bigl(x_{S_p}^{(p)}\Bigr) + \sum_{j \in \mathcal{N}_p} \sum_{l \in S_p \cap S_j} \biggl[{\lambda_l^{pj,k}}^\top (x_l^{(p)} - z_l^{\{p,j\},k}) + \frac{\rho}{2} \Bigl\|x_l^{(p)} - z_l^{\{p,j\},k}\Bigr\|^2 \biggr]
			\notag
		\\
		&=
		  \underset{x_{S_p}^{(p)}}{\arg\min}\,\,\, f_p\Bigl(x_{S_p}^{(p)}\Bigr) + \sum_{j \in \mathcal{N}_p} \sum_{l \in S_p \cap S_j} \biggl[
									{\lambda_l^{pj,k}}^\top x_l^{(p)} + \frac{\rho}{2} \Bigl\|x_l^{(p)}\Bigr\|^2 - \rho {z_l^{\{p,j\},k}}^\top x_l^{(p)}
		            \biggr]
		\label{Eq:AppDerivationKekatosXAux1}
		\\
		&=
		  \underset{x_{S_p}^{(p)}}{\arg\min}\,\,\, f_p\Bigl(x_{S_p}^{(p)}\Bigr) + \sum_{l \in S_p} \sum_{j \in \mathcal{N}_p \cap \mathcal{V}_l} \biggl[
									\Bigl(\lambda_l^{pj,k} - \rho z_l^{\{p,j\},k}\Bigr)^\top x_l^{(p)} + \frac{\rho}{2} \Bigl\|x_l^{(p)}\Bigr\|^2
		            \biggr]
		\label{Eq:AppDerivationKekatosXAux2}
		\\
		&=
		  \underset{x_{S_p}^{(p)}}{\arg\min}\,\,\, f_p\Bigl(x_{S_p}^{(p)}\Bigr) + \sum_{l \in S_p} \sum_{j \in \mathcal{N}_p \cap \mathcal{V}_l}	\Bigl(\lambda_l^{pj,k} - \rho z_l^{\{p,j\},k}\Bigr)^\top x_l^{(p)}
		  +
		  \frac{\rho}{2}\sum_{l \in S_p} \sum_{j \in \mathcal{N}_p \cap \mathcal{V}_l} \Bigl\|x_l^{(p)}\Bigr\|^2
		  \notag
		\\
		&=
		  \underset{x_{S_p}^{(p)}}{\arg\min}\,\,\, f_p\Bigl(x_{S_p}^{(p)}\Bigr) + \sum_{l \in S_p} \sum_{j \in \mathcal{N}_p \cap \mathcal{V}_l}	\Bigl(\lambda_l^{pj,k} - \rho z_l^{\{p,j\},k}\Bigr)^\top x_l^{(p)}
		  +
		  \frac{\rho}{2}\sum_{l \in S_p} D_{p,l} \Bigl\|x_l^{(p)}\Bigr\|^2
		\label{Eq:AppDerivationKekatosXAux3}
		\\
		&=
		  \underset{x_{S_p}^{(p)}}{\arg\min}\,\,\, f_p\Bigl(x_{S_p}^{(p)}\Bigr) + \sum_{l \in S_p} \sum_{j \in \mathcal{N}_p \cap \mathcal{V}_l}	\Bigl(\lambda_l^{pj,k} - \frac{\rho}{2} (x_l^{(p),k} + x_l^{(j),k}) \Bigr)^\top x_l^{(p)}
		  +
		  \frac{\rho}{2}\sum_{l \in S_p} D_{p,l} \,\Bigl\|x_l^{(p)}\Bigr\|^2
		\label{Eq:AppDerivationKekatosXAux4}
		\\
		&=
		  \underset{x_{S_p}^{(p)}}{\arg\min}\,\,\, f_p\Bigl(x_{S_p}^{(p)}\Bigr) + \sum_{l \in S_p} 	\Bigl(\gamma_l^{(p),k} - \frac{\rho}{2} ( D_{p,l} x_l^{(p),k} + \sum_{j \in \mathcal{N}_p \cap \mathcal{V}_l} x_l^{(j),k}) \Bigr)^\top x_l^{(p)}
		  +
		  \frac{\rho}{2}\sum_{l \in S_p} D_{p,l} \,\Bigl\|x_l^{(p)}\Bigr\|^2\,,
		\label{Eq:AppDerivationKekatosXAux5}
	\end{align}
	where~$D_{p,l}$ is the degree of node~$p$ in the subgraph induced by~$x_l$, $\mathcal{G}_l$.
	From~\eqref{Eq:AppDerivationKekatosXAux1} to~\eqref{Eq:AppDerivationKekatosXAux2}, we used the fact that, for a fixed node~$p$, $\sum_{j \in \mathcal{N}_p} \sum_{l \in S_p \cap S_j} (\cdot) = \sum_{l \in S_p} \sum_{j \in \mathcal{N}_p \cap \mathcal{V}_l} (\cdot)$. From~\eqref{Eq:AppDerivationKekatosXAux3} to~\eqref{Eq:AppDerivationKekatosXAux4}, we used~\eqref{Eq:AppDerivationKekatosZ3}. Finally, from~\eqref{Eq:AppDerivationKekatosXAux4} to~\eqref{Eq:AppDerivationKekatosXAux5}, we defined $\gamma_l^{(p),k} = \sum_{j \in \mathcal{N}_p \cap \mathcal{V}_l} \lambda_l^{pj,k}$.

	\mypar{Update of the dual variables}
	Note from~\eqref{Eq:AppDerivationKekatosXAux5} that node~$p$ depends only on~$\gamma_l^{(p),k}=\sum_{j \in \mathcal{N}_p \cap \mathcal{V}_l} \lambda_l^{pj,k}$ and not on the individual~$\lambda_l^{pj}$s. Using~\eqref{Eq:AppDerivationKekatosLambda1}, the update of~$\gamma_l^{(p),k}$ comes as
	\begin{align}
		  \gamma_l^{(p),k+1}
		&=
		  \sum_{j \in \mathcal{N}_p \cap \mathcal{V}_l} \lambda_l^{pj,k+1}
		\notag
		\\
		&=
		  \underbrace{\sum_{j \in \mathcal{N}_p \cap \mathcal{V}_l} \lambda_l^{pj,k}}_{=\gamma_l^{(p),k}}
		  +
		  \rho \Bigl( x_l^{(p),k+1} - z_l^{\{p,j\},k+1} \Bigr)
		\notag
		\\
		&=
		  \gamma_l^{(p),k} + \rho \sum_{j \in \mathcal{N}_p \cap \mathcal{V}_l}  \Bigl( x_l^{(p),k+1} - z_l^{\{p,j\},k+1} \Bigr)
		\notag
		\\
		&=
		  \gamma_l^{(p),k} + \rho \sum_{j \in \mathcal{N}_p \cap \mathcal{V}_l} \Bigl( x_l^{(p),k+1} - \frac{x_l^{(p),k+1}+ x_l^{(j),k+1}}{2} \Bigr)
		\notag
		\\
		&=
		  \gamma_l^{(p),k} + \frac{\rho}{2}\sum_{j \in \mathcal{N}_p \cap \mathcal{V}_l} \Bigl( x_l^{(p),k+1} -   x_l^{(j),k+1}\Bigr)\,,
		\label{Eq:AppDerivationKekatosGamma}
	\end{align}
	where we have used~\eqref{Eq:AppDerivationKekatosZ3}. We thus see that~\eqref{Eq:AppDerivationKekatosGamma} corresponds to step~\ref{SubAlg:KekatosDual} of Algorithm~\ref{Alg:Kekatos}, while~\eqref{Eq:AppDerivationKekatosXAux5} corresponds to step~\ref{SubAlg:KekatosProx}.